\documentclass[11pt, twoside, leqno]{article}
\usepackage{mathrsfs}
\usepackage{amssymb}
\usepackage{amsmath}
\usepackage{mathrsfs}
\usepackage{amsthm}
\usepackage{amsfonts}
\usepackage{color}
\usepackage{latexsym}
\usepackage{txfonts}
\usepackage{indentfirst}
\usepackage{soul}

\allowdisplaybreaks
\newtheorem{theorem}{Theorem}[section]
\newtheorem{lemma}[theorem]{Lemma}
\newtheorem{corollary}[theorem]{Corollary}
\newtheorem{proposition}[theorem]{Proposition}

\theoremstyle{definition}
\newtheorem{remark}[theorem]{Remark}
\newtheorem{definition}[theorem]{Definition}

\numberwithin{equation}{section}

\begin{document}

\title{\bf\Large Generalized
Brezis--Van Schaftingen--Yung Formulae and Their Applications
in Ball Banach Sobolev Spaces
\footnotetext{\hspace{-0.35cm} 2020
{\it Mathematics Subject Classification}.
Primary 46E35; Secondary 26D10, 42B25, 26A33, 35A23.\endgraf
{\it Key words and phrases.}
ball Banach function space, ball Banach Sobolev space,
Brezis--Van Schaftingen--Yung formula, fractional Sobolev-type inequality,
fractional Gagliardo--Nirenberg-type inequality.
\endgraf
This project is supported by the National Natural Science Foundation of China
(Grant Nos. 12122102, 11971058 and 12071197)
and the National Key Research and Development Program of China
(Grant No. 2020YFA0712900).}}
\date{}
\author{Chenfeng Zhu,
Dachun Yang\footnote{Corresponding author,
E-mail: \texttt{dcyang@bnu.edu.cn}/{\color{red}
March 6, 2023}/Final version.}
\ and Wen Yuan}

\maketitle

\vspace{-0.7cm}

\begin{center}
\begin{minipage}{13cm}
{\small {\bf Abstract}\quad
Let $X$ be a ball Banach function space
on $\mathbb{R}^n$.
In this article, under some mild assumptions about both $X$
and the boundedness of the Hardy--Littlewood maximal operator on
both $X$ and the associate
space of its convexification, the authors
successfully recover the homogeneous ball Banach
Sobolev semi-norm $\|\,|\nabla f|\,\|_X$ via the
functional
$$
\sup_{\lambda\in(0,\infty)}\lambda
\left\|\left[\int_{\{y\in\mathbb{R}^n:\
|f(\cdot)-f(y)|>\lambda|\cdot-y|^{1+\frac{\gamma}{q}}\}}
\left|\cdot-y\right|^{\gamma-n}\,dy\right]^\frac{1}{q}\right\|_X
$$
for any distributions $f$ with $|\nabla f|\in X$,
as well as the corresponding limiting identities
with the limit for $\lambda\to\infty$ when $\gamma\in(0,\infty)$
or the limit for $\lambda\to0^+$ when $\gamma\in(-\infty,0)$,
where $\gamma\in\mathbb{R}\setminus\{0\}$ and
where $q\in(0,\infty)$ is related to $X$.
In particular, some of these results are still new
even when $X:=L^p(X)$ with $p\in[1,\infty)$.
As applications, the authors obtain
some fractional Sobolev-type and some fractional Gagliardo--Nirenberg-type
inequalities in the setting of $X$.
All these results are of quite wide generality and are applied
to various specific function spaces,
including Morrey, mixed-norm (or variable or weighted) Lebesgue,
Lorentz, and Orlicz (or Orlicz-slice) spaces,
some of which are new even in all these special cases.
The novelty of this article is to use both the method of the extrapolation
and the boundedness of the Hardy--Littlewood maximal operator on both
$X$ and the associate space of its convexification to overcome the
essential difficulties caused by the deficiency
of both the translation and the rotation invariance and an explicit
expression of the norm of $X$.
}
\end{minipage}
\end{center}
	
\vspace{0.2cm}
	
\tableofcontents

\vspace{0.2cm}

\section{Introduction}
Recall that, for any given $s\in(0,1)$ and $p\in[1,\infty)$,
the \emph{homogeneous fractional Sobolev space}
$\dot{W}^{s,p}(\mathbb{R}^n)$,
originally introduced by Gagliardo \cite{g1958},
is defined to be the set of all the measurable functions $f$
on $\mathbb{R}^n$ with the following finite \emph{Gagliardo semi-norm}
\begin{align}\label{Gnorm}
\|f\|_{\dot{W}^{s,p}(\mathbb{R}^n)}:&=
\left[\int_{\mathbb{R}^n}\int_{\mathbb{R}^n}
\frac{|f(x)-f(y)|^p}{|x-y|^{sp+n}}\,dx\,dy\right]^{\frac{1}{p}}\\
&=:\left\|\frac{f(x)-f(y)}{|x-y|^{s+\frac{n}{p}}}
\right\|_{L^p(\mathbb{R}^n\times\mathbb{R}^n)}.\nonumber
\end{align}
These spaces play
a major role in many questions involving
both harmonic analysis and partial
differential equations (see, for instance,
\cite{bbm2002,crs2010,cv2011,h2008,ht2008,
ma2011,m2011,npv2012}). Denote the \emph{gradient}
(in the sense of distributions) of a weakly differentiable function $f$
on $\mathbb{R}^n$ by $\nabla f:=(\partial_1f,\ldots,\partial_nf)$,
where, for any $j\in\{1,\ldots,n\}$, $\partial_j f$ denotes
the $j$-th weak derivative of $f$, that is,
for any $\phi\in C_{\mathrm{c}}^\infty(\mathbb{R}^n)$
(the set of all the infinitely differentiable
functions with compact support),
$$
\int_{\mathbb{R}^n}f(x)\partial_j\phi(x)\,dx
=-\int_{\mathbb{R}^n}\partial_jf(x)\phi(x)\,dx.
$$
Recall that, for any given $p\in[1,\infty)$,
the \emph{homogeneous Sobolev space}
$\dot{W}^{1,p}(\mathbb{R}^n)$ is defined to be the
set of all the locally integrable functions $f$ on $\mathbb{R}^n$
having the following finite \emph{homogeneous Sobolev semi-norm}
$$
\|f\|_{\dot{W}^{1,p}(\mathbb{R}^n)}:=\|\,|\nabla f|\,\|_{L^p(\mathbb{R}^n)}.
$$
A well-known \emph{drawback} of the Gagliardo semi-norm in \eqref{Gnorm}
is that it can not recover the above homogeneous Sobolev
semi-norm $\|\,|\nabla f|\,\|_{L^p(\mathbb{R}^n)}$
if we replace $s$ in \eqref{Gnorm} by $1$ directly.
Indeed, as was pointed out by Bourgain et al. \cite{bbm2000,bbm2002} that,
in the case $s=1$,
the integral in \eqref{Gnorm} is finite unless $f$ is a constant function.
Thus, it is natural to ask that, in which way,
the Gagliardo semi-norm in \eqref{Gnorm} with $s=1$
can recover the above homogeneous Sobolev
semi-norm $\|\,|\nabla f|\,\|_{L^p(\mathbb{R}^n)}$.
This question has attracted a lot of attention in recent years.
One    important approach
is due to Bourgain et al. \cite{bbm2001}
who proved that, for any given $p\in[1,\infty)$ and for any
$f\in W^{1,p}(\mathbb{R}^n)
:=\dot{W}^{1,p}(\mathbb{R}^n)\cap L^p(\mathbb{R}^n)$,
\begin{align}\label{1055}
\lim_{s\to1^-}(1-s)\|f\|_{\dot{W}^{s,p}(\mathbb{R}^n)}^p
=C_{(p,n)}\|\,|\nabla f|\,\|_{L^p(\mathbb{R}^n)}^p,
\end{align}
where the positive constant $C_{(p,n)}$
depends only on both $p$ and $n$.
Here and thereafter,
$s\to1^-$ means $s\in(0,1)$ and $s\to1$.
Another important approach
is due to Brezis et al. \cite{bvy2021}
who recently discovered an alternative
and smart way to repair the above defect
via simply replacing the $L^p$ norm
$\|\cdot\|_{L^p(\mathbb{R}^n\times\mathbb{R}^n)}$
in \eqref{Gnorm} by the Marcinkiewicz quasi-norm
[namely the weak product $L^p$ quasi-norm $\|\cdot\|_{
L^{p,\infty}(\mathbb{R}^n\times\mathbb{R}^n)}$].
To be more precise, Brezis et al. \cite{bvy2021} proved that,
for any given $p\in[1,\infty)$ and for
any $f\in C_{\mathrm{c}}^\infty(\mathbb{R}^n)$,
\begin{align}\label{1057}
\left\|\,\left|\nabla f\right|\,\right\|_{L^p(\mathbb{R}^n)}
\sim\left\|\frac{f(x)-f(y)}{|x-y|^{1+\frac{n}{p}}}
\right\|_{L^{p,\infty}(\mathbb{R}^n\times\mathbb{R}^n)}
\end{align}
and
\begin{align}\label{2244}
&\lim_{\lambda\to\infty}\lambda^p
\left|\left\{(x,y)\in\mathbb{R}^n\times\mathbb{R}^n:\
x\neq y,\ \frac{|f(x)-f(y)|}{|x-y|^{1+\frac{n}{p}}}>\lambda\right\}\right|\\
&\quad\sim\left\|\,\left|\nabla f\right|
\,\right\|_{L^p(\mathbb{R}^n)}^p,\nonumber
\end{align}
where the positive equivalence constants in both
\eqref{1057} and \eqref{2244} are independent of $f$
and, for any measurable function $g$ on $\mathbb{R}^n\times\mathbb{R}^n$,
$$
\|g\|_{L^{p,\infty}(\mathbb{R}^n\times\mathbb{R}^n)}
:=\sup_{\lambda\in(0,\infty)}
\lambda\left|\left\{(x,y)\in\mathbb{R}^n\times\mathbb{R}^n:\
|g(x,y)|>\lambda\right\}\right|^\frac{1}{p}.
$$
Here and thereafter, for any given
$m\in\mathbb{N}$ and any Lebesgue measurable
set $E\subset\mathbb{R}^m$, the \emph{symbol} $|E|$
denotes the $m$-dimensional Lebesgue measure of $E$.
Both \eqref{1057} and \eqref{2244}
were further extended to any
$f\in\dot{W}^{1,p}(\mathbb{R}^n)$ with $p\in[1,\infty)$
by Brezis et al. \cite{bsvy.arxiv}
who also proved  more general versions
of both \eqref{1057} and \eqref{2244}. More precisely,
let both $p\in(1,\infty)$ and $\gamma\in\mathbb{R}\setminus\{0\}$
or let both $p=1$ and $\gamma\in\mathbb{R}\setminus[-1,0]$.
Brezis et al. \cite{bsvy.arxiv} proved that,
for any $f\in\dot{W}^{1,p}(\mathbb{R}^n)$,
\begin{align}\label{2118}
\left\|\,\left|\nabla f\right|\,\right\|_{L^p(\mathbb{R}^n)}
\sim\left\|\frac{f(x)-f(y)}{|x-y|^{1+\frac{n}{p}}}
\right\|_{L^{p,\infty}(\mathbb{R}^n\times\mathbb{R}^n,\nu_\gamma)}
\end{align}
and, moreover,
\begin{align}\label{2119}
&\lim_{\lambda\to\infty}\lambda^p
\nu_\gamma\left(\left\{(x,y)\in\mathbb{R}^n\times\mathbb{R}^n:\
x\neq y,\ \frac{|f(x)-f(y)|}{|x-y|^{1+
\frac{\gamma}{p}}}>\lambda\right\}\right)\\
&\quad\sim\left\|\,\left|\nabla f\right|
\,\right\|_{L^p(\mathbb{R}^n)}^p\nonumber
\end{align}
if $\gamma\in(0,\infty)$, and
\begin{align}\label{2020}
&\lim_{\lambda\to0^+}\lambda^p
\nu_\gamma\left(\left\{(x,y)\in\mathbb{R}^n\times\mathbb{R}^n:\
x\neq y,\ \frac{|f(x)-f(y)|}{|x-y|^{1+
\frac{\gamma}{p}}}>\lambda\right\}\right)\\
&\quad\sim\left\|\,\left|\nabla f\right|
\,\right\|_{L^p(\mathbb{R}^n)}^p\nonumber
\end{align}
if $\gamma\in(-\infty,0)$,
where the positive equivalence constants in \eqref{2118},
\eqref{2119}, and \eqref{2020} are independent of $f$ and,
for any $\gamma\in\mathbb{R}$ and any Lebesgue measurable set
$E\subset\mathbb{R}^n\times\mathbb{R}^n$,
\begin{align}\label{nu}
\nu_\gamma(E):=\iint_{\{(x,y)\in E:\ x\neq y\}}
\left|x-y\right|^{\gamma-n}\,dx\,dy.
\end{align}
Here and thereafter,
$\lambda\to0^+$ means $\lambda\in(0,1)$ and $\lambda\to0$.
We refer the reader to \cite[p.\,3, Remark(i)]{bsvy.arxiv}
for more details about these two limits
in both \eqref{2119} and \eqref{2020};
see also the proof of \cite[Lemma~3.2]{bsvy.arxiv}.
These results in \cite{bsvy.arxiv}
unify and substantially extended the previous works
by Nguyen \cite{n2006} and by Brezis et al. \cite{bvy2021},
respectively. Indeed, if $\gamma=n$,
then \eqref{2118} coincides with \eqref{1057},
and \eqref{2119} coincides with \eqref{2244};
if $\gamma=-p\in(-\infty,-1)$, then \eqref{2020}
coincides with \cite[(3)]{n2006}
(see also \cite{bn2018,bn2020}).
Recall that  the aforementioned results obtained in both \cite{bsvy.arxiv} and
\cite{bvy2021} have found
some very interesting applications.
For instance, using \eqref{2118}, Brezis et al. \cite{bsvy.arxiv}
gave a characterization of $\dot{W}^{1,p}(\mathbb{R}^n)$.
In addition, as important applications of \eqref{1057},
Brezis et al. \cite{bvy2021} obtained  several
surprising alternative (weaker) estimates
of fractional Sobolev-type and fractional Gagliardo--Nirenberg-type
inequalities in some exceptional (critical) cases
involving $\dot{W}^{1,1}(\mathbb{R}^n)$,
where the anticipated fractional Sobolev-type
and fractional Gagliaro--Nirenberg-type estimates may fail.
We refer the reader to \cite{bsy2023,bn2016,bsvy,bvyCVPDE,dlyyz2022,
dm.arXiv,dm.arXiv1,dssvy,dt.arXiv,gy2021} for more related works
and to \cite{bm2018,bm2019} for more studies on the
Gagliardo--Nirenberg-type inequality.

On the other hand, recall that the concept of the
ball (quasi-)Banach function space were
originally introduced by Sawano et al. \cite{shyy2017}
to unify the study of several important function spaces,
all of which have wide applications in
many branches of analysis
(see Definition~\ref{1659} below for
the definition of the ball (quasi-)Banach function space).
It is less restrictive than the classical Banach function spaces
introduced by Bennett and Sharpley in \cite[Chapter 1]{bs1988}
and hence has both a wide range of generality and applications.
Many well-known and well-developed function spaces
are ball (quasi-)Banach function spaces,
such as Morrey, mixed-norm (or variable or weighted) Lebesgue,
Lorentz, and Orlicz (or Orlicz-slice) spaces
(see, respectively, Subsections~\ref{5.1} through~\ref{5.6} below
for their histories and definitions),
but some of which might not be the classical Banach function spaces;
see, for instance, \cite[Section~7]{shyy2017}.
After the concept of the
ball (quasi-)Banach function space was put forward,
it aroused a lot of interests of many harmonic analysts,
particularly, in both the study and the generalization of the aforementioned
Bourgain--Brezis--Mironescu (BBM) formula \eqref{1055}
and the aforementioned Brezis--Van Schaftingen--Yung (BVY) formula \eqref{1057}
in the setting of the ball (quasi-)Banach function space.
Very recently, Dai et al. \cite{dlyyz.arxiv} subtly
used the Poincar\'e inequality,
the exquisite geometry of $\mathbb{R}^n$,
a method of extrapolation,
and the exact operator norm
of the Hardy--Littlewood maximal operator
on the associate space of the convexification of $X$,
to get rid of the strong dependence of both the
rotation and the translation invariances of $L^p(\mathbb{R}^n)$
and also to overcome the difficulty coming from the deficiency of
the explicit expression of the norm of the general ball
(quasi-)Banach function space.
As a result, Dai et al. \cite{dlyyz.arxiv} successfully
extended both \eqref{1057}
and \eqref{2244} to the setting of ball Banach function spaces,
and then further established some fractional Sobolev-type and
some fractional Gagliardo--Nirenberg-type
inequalities in various specific function spaces.
At the same time, applying some similar ideas and techniques
to those used in \cite{dlyyz.arxiv},
Dai et al. \cite{dgpyyz2022} also extended the limiting result \eqref{1055}
obtained in \cite{bbm2001} to ball Banach function spaces.
For more studies on ball quasi-Banach function spaces,
we refer the reader to \cite{cwyz2020,lyh2320,s2018,shyy2017,wyy.arxiv,yhyy2022-1,
yhyy2022-2,yyy2020,zhyy2022} for Hardy spaces
associated with ball quasi-Banach function spaces,
to \cite{h2021,wyy2020,zwyy2021} for the
boundedness of operators on ball quasi-Banach function spaces,
and to \cite{hcy2021,ins2019,is2017,tyyz2021,wyyz2021}
for further applications of ball quasi-Banach function spaces.

Motivated by the aforementioned works, in this article,
we establish a suitable extension of
\eqref{2118} as well as the limiting behaviors as in
both \eqref{2119} and \eqref{2020} in the setting of ball
Banach function spaces.
Precisely speaking,
let $X$ be a ball Banach function space
on $\mathbb{R}^n$.
Under some mild assumptions about both $X$
and the boundedness of the Hardy--Littlewood maximal operator on
both $X$ and the associate
space of its convexification,
we establish the
following generalized version of \eqref{2118}:
for any $f\in\dot{W}^{1,X}(\mathbb{R}^n)$
[see Definition~\ref{2.7} below for the
definition of $\dot{W}^{1,X}(\mathbb{R}^n)$],
\begin{align}\label{2213}
\sup_{\lambda\in(0,\infty)}\lambda
\left\|\left[\int_{\mathbb{R}^n}
\mathbf{1}_{E_{\lambda,\frac{\gamma}{q}}[f]}(\cdot,y)
\left|\cdot-y\right|^{\gamma-n}\,dy\right]^\frac{1}{q}\right\|_X
\sim\left\|\,\left|\nabla f\right|\,\right\|_X,
\end{align}
where $q\in(0,\infty)$ is related to $X$,
$\gamma\in\mathbb{R}\setminus\{0\}$,
\begin{align}\label{Elambda}
E_{\lambda,\frac{\gamma}{q}}[f]:=\left\{(x,y)\in\mathbb{R}^n\times\mathbb{R}^n:\
x\neq y,\ \frac{|f(x)-f(y)|}{|x-y|^{1+\frac{\gamma}{q}}}>\lambda\right\}
\end{align}
for any $\lambda\in(0,\infty)$,
and the positive equivalence constants in \eqref{2213} are independent of $f$.
The corresponding generalizations of both \eqref{2119} and \eqref{2020}
on $X$ are also established.
In particular, compared with that in \cite{bsvy.arxiv},
some of these results are still new
even when $X:=L^p(X)$ with $p\in[1,\infty)$.
Then we use \eqref{2213} to further obtain some
fractional Sobolev-type and some fractional Gagliaro--Nirenberg-type
inequalities in the setting of ball
Banach function spaces in a more
general version with parameter $\gamma\in\mathbb{R}\setminus\{0\}$,
which when $\gamma:=n$
is just the fractional Sobolev-type and
the fractional Gagliaro--Nirenberg-type
inequalities in \cite[Corollaries~4.13 and 4.15]{dlyyz.arxiv}
and which when $\gamma\in\mathbb{R}\setminus\{0,n\}$
is new.
All these results are of quite wide generality and are applied
to various specific function spaces,
including Morrey, mixed-norm (or variable or weighted) Lebesgue,
Lorentz, and Orlicz (or Orlicz-slice) spaces,
some of which are new even in all these special cases.
The novelty of this article is to use both the method of extrapolation
and the boundedness of the Hardy--Littlewood maximal operator on both
$X$ and the associate space of its convexification to overcome the
essential difficulties caused by the deficiency
of both the translation and the rotation invariances and an explicit
expression of the norm of $X$.

As was shown in the proof of \cite[Theorem~2.1]{bsvy.arxiv},
the upper bound in \eqref{2118} when $p\in(1,\infty)$ follows directly
from both the Lusin--Lipschitz inequality
(see, for instance, \cite[(2.9)]{bvy2021})
and the boundedness of the Hardy--Littlewood maximal operator
on $L^p(\mathbb{R}^n)$.
Thus, the \emph{hard core of the proof of
the upper bound in \eqref{2118}}
is the case $p=1$.
In this core case,
the proofs of \cite[Propositions~2.1 and~2.2]{bvy2021}
when $\gamma\in(0,\infty)$
use the Vitali covering lemma
and a method of rotation, while
the proof of \cite[Theorem~2.2(b)]{bsvy.arxiv}
when $\gamma\in(-\infty,-1)$
uses a
stopping time argument and also a method of rotation
(both of these proofs actually
work for the full range of $p\in[1,\infty)$).
On the other hand,
note that some of the
aforementioned examples
of ball Banach function spaces are obviously not
the rotation invariance.
Thus, we point out that \eqref{2213},
namely the extension of \eqref{2118} from $L^p(\mathbb{R}^n)$
to the ball Banach function space $X$,
is quite \emph{nontrivial} because
the rotation invariance property does seem to play
an important role in the known proof of the upper bound in \eqref{2118}.
It is also worth noting that \eqref{2213}
gives an equivalence between the
Sobolev semi-norm and the positive functional
involving the difference of the function
under consideration.
Indeed, in approximation theory,
it is quite difficult to find an appropriate way to
characterize the smoothness of functions
via their finite differences
even for some simple weighted Lebesgue spaces in one dimension
(see \cite{k2015,mt2001} and the references therein).
The main obstacle is due to the fact that,
for any $h\in\mathbb{R}^n\setminus\{\bf{0}\}$,
the difference operator
$\Delta_hf:=f(\cdot+h)-f(\cdot)$
might be no longer bounded on the general weighted Lebesgue space.
To overcome this main difficulty,
the extrapolation theorem \cite[Lemmas~4.6 and 4.7]{dlyyz.arxiv}
plays an essential role, which is a bridge
connecting the ball Banach function space and
the weighted Lebesgue space.
With the help of the extrapolation theorem
and the exact operator norm of the Hardy--Littlewood
maximal operator on the associate space of
the convexification of $X$,
the proof of
\eqref{2213} can be reduced to establish
the following estimates in weighted Lebesgue spaces with
$A_1(\mathbb{R}^n)$ weights
[see Definitions~\ref{1557} and~\ref{1556} for the definitions of
both the Muckenhoupt weight class and
the homogeneous weighted Sobolev space
$\dot{W}^{1,p}_\omega(\mathbb{R}^n)$]:
If $p\in[1,\infty)$ and $\omega\in A_1(\mathbb{R}^n)$,
then, for any $f\in\dot{W}^{1,p}_\omega(\mathbb{R}^n)$,
\begin{align}\label{1506}
&\sup_{\lambda\in(0,\infty)}\lambda^p
\int_{\mathbb{R}^n}\left[\int_{\mathbb{R}^n}
\mathbf{1}_{E_{\lambda,\frac{\gamma}{q}}[f]}(x,y)
\left|x-y\right|^{\gamma-n}\,dy\right]^{\frac{p}{q}}
\omega(x)\,dx\\
&\quad\sim\int_{\mathbb{R}^n}\left|\nabla f(x)\right|^p\omega(x)\,dx,
\nonumber
\end{align}
where $\gamma\in\mathbb{R}\setminus\{0\}$,
$q\in(0,\infty)$ is related to $p$,
and the positive equivalence constants are
independent of $f$.
Notice that the integral in \eqref{1506} is not symmetric with respect
to both $x$ and $y$,
which brings a lot of difficulties when proving \eqref{1506}.
Another difficulty comes from the weight $\omega$
when proving \eqref{1506} in the case $p=1$
because the method of rotation used in the proof of \eqref{2118} in
the unweighted case in \cite{bsvy.arxiv,bvy2021}
seems to be inapplicable here.
We overcome these obstacles with the help of several useful tools
such as the Poincar\'e-type inequality, exquisite geometric properties
of $\mathbb{R}^n$ shown via the adjacent system of dyadic cubes of $\mathbb{R}^n$,
and a stopping time argument in the 1-dimensional case.

The organization of the remainder of this article is as follows.

In Section~\ref{section2},
we first recall some concepts of the ball (quasi-)Banach function
space as well as its related convexification and associate space.
Then we recall the concept of the homogeneous ball Banach Sobolev space
$\dot{W}^{1,X}(\mathbb{R}^n)$
and some preliminaries on both the Muckenhoupt weight class and the
weighted Lebesgue space.

The main target of Section~\ref{S3} is to prove
\eqref{2213} as well as the corresponding limiting identities.
In Subsection~\ref{ss3.1},
we give the upper
estimate of \eqref{2213} for any $f\in C^1(\mathbb{R}^n)$
with $|\nabla f|\in C_{\mathrm{c}}(\mathbb{R}^n)$.
To be more precise, using the Lusin--Lipschitz inequality
(or the classical Morrey inequality),
a boundedness assumption of the Hardy--Littlewood
maximal operator on $X$,
and a method of extrapolation,
we obtain the upper estimate of \eqref{2213}
(see Theorem~\ref{upperBound} below).
In the critical case
[for example, $X:=L^1(\mathbb{R}^n)$],
without assuming the boundedness of the Hardy--Littlewood
maximal operator on $X$,
we use both the Poincar\'e-type inequality
and the adjacent system of dyadic cubes in $\mathbb{R}^n$
in the case $\gamma\in(0,\infty)$
and use a stopping time argument in the case both
$\gamma\in(-\infty,-1)$ and $n=1$
to establish the upper estimate of \eqref{2213}
(see Theorems~\ref{gamma>0} and~\ref{n=1} below).
In Subsection~\ref{ss3.2}, when $X$ is only assumed to be
a ball quasi-Banach function space,
we prove the lower estimate of \eqref{2213}
via using both the Lebesgue differentiation
lemma in \cite[Lemma~3.1]{bsvy.arxiv}
and the Fatou property of $X$
(see Theorem~\ref{LimFormulaInf} below).
In particular, if $X$ is a ball Banach function space,
we further obtain two limiting identities
(see Corollary~\ref{LimFormulaEq} below).
In Subsection~\ref{ss3.3},
we extend the results in both Subsections~\ref{ss3.1} and~\ref{ss3.2}
from any $f\in C^1(\mathbb{R}^n)$ with
$|\nabla f|\in C_{\mathrm{c}}(\mathbb{R}^n)$
to any $f\in\dot{W}^{1,X}(\mathbb{R}^n)$ by a density argument
under the additional assumption that $X$ has an absolutely continuous norm
(see Theorems~\ref{2117} and~\ref{4.8} below).
As a corollary of Theorem~\ref{2117},
we obtain the generalized Brezis--Van Schaftingen--Yung formula
and the corresponding limiting identities
in the setting of $L^p(\mathbb{R}^n)$ with $p\in[1,\infty)$
(see Corollary~\ref{2120} below),
some of which are still new.

In Section~\ref{section4},
as applications of \eqref{2213},
we further establish two types of the
fractional Gagliaro--Nirenberg-type inequalities
on the ball Banach function space $X$
(see Theorems~\ref{2255} and~\ref{2256} below).
Theorem~\ref{2255} involves
both $\dot{W}^{1,X}(\mathbb{R}^n)$
and $X^p(\mathbb{R}^n)$ with $p\in[1,\infty)$
[or $L^\infty(\mathbb{R}^n)$ when $p=\infty$],
while Theorem~\ref{2256}
involves both $\dot{W}^{1,X}(\mathbb{R}^n)$
and its fractional version.
As a special case of Theorem~\ref{2255},
we also obtain the
fractional Sobolev-type estimate on $X$.

In Section~\ref{S5},
all these results obtained
in both Sections~\ref{S3} and~\ref{section4} are applied
to various specific ball Banach function spaces, such as
the Morrey space (see Subsection~\ref{5.1} below),
the mixed-norm Lebesgue space (see Subsection~\ref{5.2} below),
the variable Lebesgue space (see Subsection~\ref{5.3} below),
the weighted Lebesgue space (see Subsection~\ref{5.4} below),
the Lorentz space (see Subsection~\ref{5.7} below),
the Orlicz space (see Subsection~\ref{5.5} below),
and the Orlicz-slice space or the generalized amalgam
space (see Subsection~\ref{5.6} below).
Most of these results are new.

Due to the generality and the flexibility,
more applications of the results of this article are predictable.
In particular, the generalized Brezis--Van Schaftingen--Yung formula
\eqref{2213} (see also Theorem~\ref{2117} below)
can further be used to give a new characterization
of the homogeneous ball Banach Sobolev space
$\dot{W}^{1,X}(\mathbb{R}^n)$, which is presented in \cite{zyy2023}
to limit the length of this article.

Recall that, when $\gamma\in(-\infty,-1)$,
the upper estimate of \eqref{2118}
with $p=1$ holds true for any $f\in\dot{W}^{1,1}(\mathbb{R}^n)$
with $n\in\mathbb{N}$
(see, for instance, \cite[Theorem~2.2(b)]{bsvy.arxiv}).
This upper estimate when $p=1$ is generalized
to the ball Banach function case only
when $n=1$ in Theorem~\ref{n=1} below.
When $n\in\mathbb{N}\cap[2,\infty)$, it is still unknown
how to generalize this upper estimate
when $p=1$ to the ball Banach function space case
[see Remark~\ref{1920}(iv) below]. Observe that,
in Theorem \ref{upperBound}, we assume that
the Hardy--Littlewood maximal operator is bounded on $X$
which is not true when $X=L^1(\mathbb{R}^n)$ and hence
Theorem \ref{upperBound} can not give this upper estimate when
$X=L^1(\mathbb{R}^n)$.

At the end of this section, we make some conventions on notation.
We always let $\mathbb{N}:=\{1,2,\ldots\}$,
$\mathbb{Z}_+:=\mathbb{N}\cup\{0\}$, and $\mathbb{S}^{n-1}$
be the unit sphere of $\mathbb{R}^n$.
We denote by $C$ a \emph{positive constant} which is independent
of the main parameters, but may vary from line to line.
We use $C_{(\alpha,\dots)}$ to denote a positive constant depending
on the indicated parameters $\alpha,\, \dots$.
The symbol $f\lesssim g$ means $f\le Cg$
and, if $f\lesssim g\lesssim f$, then we write $f\sim g$.
If $f\le Cg$ and $g=h$ or $g\le h$,
we then write $f\lesssim g=h$ or $f\lesssim g\le h$.
For any index $q\in[1,\infty]$,
we denote by $q'$ its \emph{conjugate index},
that is, $\frac{1}{q}+\frac{1}{q'}=1$.
If $E$ is a subset of $\mathbb{R}^n$, we denote by ${\mathbf{1}}_E$ its
\emph{characteristic function} and
by $E^\complement$ the set $\mathbb{R}^n\setminus E$.
We use $\mathbf{0}$ to denote the \emph{origin} of $\mathbb{R}^n$.
For any $x\in\mathbb{R}^n$ and $r\in(0,\infty)$,
let $B(x,r):=\{y\in\mathbb{R}^n:\ |x-y|<r\}$ and
\begin{align}\label{1654}
\mathbb{B}:=\left\{B(x,r):\ x\in\mathbb{R}^n
\text{ and }r\in(0,\infty)\right\}
\end{align}
be the set of all balls in
$\mathbb{R}^n$. Usually, for any $r\in(0,\infty)$,
we  let $B_r:=B({\bf 0},r)$. For any $\lambda\in(0,\infty)$
and any ball $B:=B(x_B,r_B)$ in
$\mathbb{R}^n$ with both center $x_B\in\mathbb{R}^n$
and radius $r_B\in(0,\infty)$, let $\lambda B:=B(x_B,\lambda r_B)$.
For any measurable function $f$ on $\mathbb{R}^n$,
its support $\mathrm{supp\,}(f)$ is defined by setting
$\mathrm{supp\,}(f):=\{x\in\mathbb{R}^n:\ f(x)\neq0\}$.

Throughout this article, we denote by
$C_{\mathrm{c}}(\mathbb{R}^n)$
[resp. $C^k(\mathbb{R}^n)$ with $k\in\mathbb{N}$]
the set of all the continuous functions with compact support
[resp. all $k$-th continuously differentiable functions] on $\mathbb{R}^n$.
In particular, denote by   $C^\infty(\mathbb{R}^n)$
the set of all infinitely differentiable functions on $\mathbb{R}^n$.

For any $p\in(0,\infty]$,
the \emph{Lebesgue space}
$L^p(\mathbb{R}^n)$
is defined to be the set of
all the measurable functions $f$
on $\mathbb{R}^n$ such that
\begin{align*}
\|f\|_{L^p(\mathbb{R}^n)}:=
\left[\int_{\mathbb{R}^n}|f(x)|^p\,dx\right]^\frac{1}{p}<\infty.
\end{align*}
For any $p\in(0,\infty)$,
the set of all the locally $p$-integrable
functions on $\mathbb{R}^n$,
$L_{\mathrm{loc}}^p(\mathbb{R}^n)$,
is defined to be the set of all the
measurable functions
$f:\ \mathbb{R}^n\to\mathbb{C}$ satisfying that,
for any $x\in \mathbb{R}^n$ and $r\in(0,\infty)$,
$\|f\mathbf{1}_{B(x,r)}\|_{L^p(\mathbb{R}^n)}<\infty$.
For any $f\in L^1_{\mathrm{loc}}(\mathbb{R}^n)$,
its \emph{Hardy--Littlewood maximal function} $\mathcal{M}(f)$
is defined by setting, for any $x\in\mathbb{R}^n$,
\begin{align*}
\mathcal{M}(f)(x):=\sup_{B\ni x}\frac{1}{|B|}\int_B\left|f(y)\right|\,dy,
\end{align*}
where the supremum is taken over all the
balls $B\subset\mathbb{R}^n$ containing $x$.
For any $f\in L^1_{\mathrm{loc}}(\mathbb{R}^n)$ and $E\subset\mathbb{R}^n$
with $|E|<\infty$, let
\begin{align*}
f_E:=\fint_Ef(x)\,dx
:=\frac{1}{|E|}\int_Ef(x)\,dx.
\end{align*}
For any $q\in(0,\infty)$ and $n\in\mathbb{N}$, let
\begin{align}\label{kappaqn}
\kappa(q,n):=\int_{\mathbb{S}^{n-1}}\left|e\cdot\omega\right|^q\,d\omega
=\frac{2\Gamma(\frac{q+1}{2})\pi^{\frac{n-1}{2}}}{\Gamma(\frac{q+n}{2})},
\end{align}
where $e$ is any unit vector in $\mathbb{R}^n$
and $\Gamma$ is the Gamma function.
Finally, when we prove a theorem or the like,
we always use the same symbols in the wanted
proved theorem or the like.

\section{Ball Banach Function Spaces and Homogeneous
Ball Banach\\ Sobolev Spaces}
\label{section2}

In this section, we recall the concepts of
the ball (quasi-)Banach function space
as well as its convexification and its associate space,
the homogeneous ball Banach Sobolev space,
and the Muckenhoupt $A_p(\mathbb{R}^n)$ class.
First, we give some preliminaries on ball quasi-Banach function
spaces introduced in \cite{shyy2017}.
Throughout this article,
denote by the \emph{symbol} $\mathscr{M}(\mathbb{R}^n)$
the set of all measurable functions
on $\mathbb{R}^n$.

\begin{definition}\label{1659}
A quasi-Banach space $X\subset\mathscr{M}(\mathbb{R}^n)$,
equipped with a quasi-norm $\|\cdot\|_X$
which makes sense for all functions
in $\mathscr{M}(\mathbb{R}^n)$,
is called a \emph{ball quasi-Banach function space} if $X$ satisfies that
\begin{enumerate}
\item[\textup{(i)}]
for any $f\in\mathscr{M}(\mathbb{R}^n)$,
if $\|f\|_X=0$, then $f=0$ almost everywhere;
\item[\textup{(ii)}]
if $f,g\in\mathscr{M}(\mathbb{R}^n)$
with $|g|\leq|f|$ almost everywhere,
then $\|g\|_X\leq\|f\|_X$;
\item[\textup{(iii)}]
if a sequence $\{f_m\}_{m\in\mathbb{N}}\subset\mathscr{M}(\mathbb{R}^n)$
and $f\in\mathscr{M}(\mathbb{R}^n)$ satisfy
that $0\leq f_m\uparrow f$ almost everywhere
as $m\to\infty$, then $\|f_m\|_X\uparrow\|f\|_X$ as $m\to\infty$;
\item[\textup{(iv)}]
for any ball $B:=B(x,r)$ with both $x\in\mathbb{R}^n$ and $r\in(0,\infty)$,
$\mathbf{1}_B\in X$.
\end{enumerate}
Moreover, a ball quasi-Banach function space $X$ is called a
\emph{ball Banach function space} if
$X$ satisfies the following additional conditions:
\begin{enumerate}
\item[\textup{(v)}]
for any $f,g\in X$,
$$
\|f+g\|_X\leq\|f\|_X+\|g\|_X;
$$
\item[\textup{(vi)}]
for any ball $B$,
there exists a positive constant $C_{(B)}$, depending on $B$,
such that, for any $f\in X$,
$$
\int_B\left|f(x)\right|\,dx\leq C_{(B)}\|f\|_X.
$$
\end{enumerate}
\end{definition}

\begin{remark}
\begin{enumerate}
\item[\textup{(i)}]
Let $X$ be a ball quasi-Banach function space on $\mathbb{R}^n$.
Then, by \cite[Remark~2.5(i)]{yhyy2022-1}
(see also \cite[Remark~2.6(i)]{yhyy2022-2}),
we find that, for any $f\in\mathscr{M}(\mathbb{R}^n)$,
$\|f\|_X=0$ if and only if $f=0$ almost everywhere.
\item[\textup{(ii)}]
As was mentioned in \cite[Remark~2.5(ii)]{yhyy2022-1}
(see also \cite[Remark~2.6(ii)]{yhyy2022-2}),
we obtain an equivalent formulation
of Definition~\ref{1659} via replacing any
ball $B$ therein by any bounded measurable set $E$.
\item[\textup{(iii)}]
In Definition~\ref{1659}, if we replace any ball $B$ by any
measurable set $E$ with $|E|<\infty$, then we obtain the definition of
(quasi-)Banach function spaces which was originally introduced
by Bennett and Sharpley in
\cite[Definitions~1.1 and~1.3]{bs1988}.
Thus, a (quasi-)Banach function space is always a ball
(quasi-)Banach function space and the converse is
not necessary to be true.
\item[\textup{(iv)}]
By \cite[Proposition~1.2.36]{lyh2320}
(see also \cite[Theorem~2]{dfmn2021}),
we conclude that both (ii) and (iii) of
Definition~\ref{1659} imply that any ball quasi-Banach function
space is complete.
\end{enumerate}
\end{remark}

The following definition of the $p$-convexification of a ball
quasi-Banach function space can
be found in \cite[Definition~2.6]{shyy2017}.

\begin{definition}\label{tuhua}
Let $X$ be a ball quasi-Banach function space and $p\in(0,\infty)$.
The \emph{$p$-convexification} $X^p$ of $X$ is defined by setting
$X^p:=\{f\in\mathscr{M}(\mathbb{R}^n):\ |f|^p\in X\}$
equipped with the quasi-norm $\|f\|_{X^p}:=\|\,|f|^p\|_{X}^\frac{1}{p}$
for any $f\in X^p$.
\end{definition}

Next, we recall the definition of ball Banach function spaces with
absolutely continuous norm
(see, for instance, \cite[Definition~3.1]{bs1988}
and \cite[Definition~3.2]{wyy2020}).

\begin{definition}
A ball Banach function space $X$ is said
to have an \emph{absolutely continuous norm}
if, for any $f\in X$ and any sequence $\{E_j\}_{j\in\mathbb{N}}$
of measurable sets
satisfying that $\mathbf{1}_{E_j}\to0$
almost everywhere as $j\to\infty$, one has
$\|f\mathbf{1}_{E_j}\|_X\to0$ as $j\to\infty$.
\end{definition}

The following concept of the associate space of a ball
Banach function space can be found in \cite[p.\,9]{shyy2017};
see \cite[Chapter 1, Section 2]{bs1988} for more details.

\begin{definition}\label{associte}
Let $X$ be a ball Banach function space.
The \emph{associate space} (also called the \emph{K\"othe dual}) $X'$ of $X$
is defined by setting
\begin{align*}
X':=\left\{f\in\mathscr{M}(\mathbb{R}^n):\
\|f\|_{X'}:=\sup_{\{g\in X:\ \|g\|_{X}=1\}}
\left\|fg\right\|_{L^1(\mathbb{R}^n)}<\infty\right\},
\end{align*}
where $\|\cdot\|_{X'}$ is called the \emph{associate norm} of $\|\cdot\|_X$.
\end{definition}

\begin{remark}\label{2143}
Let $X$ be a ball Banach function space.
Then \cite[Proposition~2.3]{shyy2017} implies that
$X'$ is also a ball Banach function space.
\end{remark}

Now, we recall the concept of the homogeneous ball Banach Sobolev space,
which was introduced in \cite[Definition~2.4]{dlyyz.arxiv}.

\begin{definition}\label{2.7}
Let $X$ be a ball Banach function space.
The \emph{homogeneous ball Banach
Sobolev space} $\dot{W}^{1,X}(\mathbb{R}^n)$
is defined to be the set of all the
distributions $f$ on $\mathbb{R}^n$ such that $|\nabla f|\in X$ equipped
with the quasi-norm
$$
\|f\|_{\dot{W}^{1,X}(\mathbb{R}^n)}:=\left\|\,|\nabla f|\,\right\|_X,
$$
where $\nabla f:=(\partial_1f,\ldots,\partial_nf)$
denotes the distributional gradient of $f$.
\end{definition}

Next, we recall the concept of
the Muckenhoupt $A_p(\mathbb{R}^n)$ weight class
(see, for instance, \cite[Definitions~7.1.1 and~7.1.3]{g2014}).

\begin{definition}\label{1557}
An \emph{$A_p(\mathbb{R}^n)$-weight} $\omega$, with $p\in[1,\infty)$,
is a nonnegative locally integrable function
on $\mathbb{R}^n$ satisfying that,
when $p=1$
\begin{align}\label{A1}
[\omega]_{A_1(\mathbb{R}^n)}:=\sup_{Q\subset\mathbb{R}^n}
\frac{\|\omega^{-1}\|_{L^\infty(Q)}}{|Q|}\int_Q\omega(x)\,dx<\infty,
\end{align}
and, when $p\in(1,\infty)$
\begin{align}\label{Ap}
[\omega]_{A_p(\mathbb{R}^n)}:=\sup_{Q\subset\mathbb{R}^n}
\left[\frac{1}{|Q|}\int_Q\omega(x)\,dx\right]
\left\{\frac{1}{|Q|}\int_Q
\left[\omega(x)\right]^{1-p'}\,dx\right\}^{p-1}<\infty,
\end{align}
where $p'\in[1,\infty]$ denotes the conjugate index of $p$ and
the suprema in both \eqref{A1} and \eqref{Ap}
are taken over all cubes $Q\subset\mathbb{R}^n$.
Moreover, let
$$
A_\infty(\mathbb{R}^n):=\bigcup_{p\in[1,\infty)}A_p(\mathbb{R}^n).
$$
\end{definition}

We recall the concepts of both the weighted Lebesgue space and
the homogeneous weighted Sobolev space as follows.

\begin{definition}\label{1556}
\begin{enumerate}
\item[\textup{(i)}]
Let $p\in(0,\infty]$ and $\omega$ be a weight on $\mathbb{R}^n$.
The \emph{weighted Lebesgue space} $L^p_\omega(\mathbb{R}^n)$
is defined to be the set of all the $f\in\mathscr{M}(\mathbb{R}^n)$ such that
\begin{align*}
\|f\|_{L^p_\omega(\mathbb{R}^n)}:=
\left[\int_{\mathbb{R}^n}\left|f(x)\right|^p
\omega(x)\,dx\right]^{\frac{1}{p}}<\infty.
\end{align*}
\item[\textup{(ii)}]
Let $p\in[1,\infty]$ and $\omega$ be a weight on $\mathbb{R}^n$.
The \emph{homogeneous weighted Sobolev space}
$\dot{W}^{1,p}_\omega(\mathbb{R}^n)$
is defined to be the set of all the distributions $f$ on $\mathbb{R}^n$
whose distributional gradients $\nabla f\in L^p_\omega(\mathbb{R}^n)$.
Moreover, for any $f\in\dot{W}^{1,p}_\omega(\mathbb{R}^n)$, let
\begin{align*}
\|f\|_{\dot{W}^{1,p}_\omega(\mathbb{R}^n)}:=\left\|\,\left|\nabla f\right|\,
\right\|_{L^p_\omega(\mathbb{R}^n)}.
\end{align*}
\end{enumerate}
\end{definition}

For Muckenhoupt $A_p(\mathbb{R}^n)$-weights, we have the following
basic properties and related conclusions
(see, for instance, \cite[(7.3) and (7.5)]{D2000},
\cite[Proposition~7.1.5 and Theorem~7.1.9]{g2014},
and \cite[Theorem~2.7.4]{dhhr2011}).

\begin{lemma}\label{ApProperty}
Let $p\in[1,\infty)$ and $\omega\in A_p(\mathbb{R}^n)$.
Then the following statements hold true:
\begin{enumerate}
\item[\textup{(i)}]
if $p=1$, then $[\omega]_{A_1(\mathbb{R}^n)}\in[1,\infty)$
and, for almost every $x\in\mathbb{R}^n$,
$$
\mathcal{M}(\omega)(x)\leq[\omega]_{A_1(\mathbb{R}^n)}\omega(x);
$$
\item[\textup{(ii)}]
for any cubes $Q,S\subset\mathbb{R}^n$ with $Q\subset S$,
$$
\omega(S)\leq[\omega]_{A_p(\mathbb{R}^n)}
\left(\frac{|S|}{|Q|}\right)^p\omega(Q);
$$
\item[\textup{(iii)}]
$$
[\omega]_{A_p(\mathbb{R}^n)}=
\sup_{Q\subset\mathbb{R}^n}\sup_{\|f\mathbf{1}_Q\|_{L^p_\omega(\mathbb{R}^n)}
\in(0,\infty)}\frac{[\frac{1}{Q}\int_Q|f(t)|\,dt]^p}{\frac{1}{\omega(Q)}
\int_Q|f(t)|^p\omega(t)\,dt},
$$
where the first supremum is taken over all cubes $Q\subset\mathbb{R}^n$;
\item[\textup{(iv)}]
$\omega\in A_q(\mathbb{R}^n)$ for any $q\in[p,\infty)$; moreover,
$[\omega]_{A_q(\mathbb{R}^n)}\leq[\omega]_{A_p(\mathbb{R}^n)}$;
\item[\textup{(v)}]
if $p\in(1,\infty)$ and $\mu:=\omega^{1-p'}$,
then $\mu\in A_{p'}(\mathbb{R}^n)$,
\begin{align*}
[\mu]_{A_{p'}(\mathbb{R}^n)}^{p-1}=[\omega]_{A_p(\mathbb{R}^n)}
\leq[\omega]_{A_1(\mathbb{R}^n)},
\ \text{and}\
\left[L_\omega^p(\mathbb{R}^n)\right]'=L_\mu^{p'}(\mathbb{R}^n);
\end{align*}
\item[\textup{(vi)}]
if $p\in(1,\infty)$, then the Hardy--Littlewood
maximal operator $\mathcal{M}$ is bounded on $L^p_\omega(\mathbb{R}^n)$;
moreover, there exists a positive constant $C$,
independent of $\omega$, such that
$$
\left\|\mathcal{M}\right\|_{L^p_\omega(\mathbb{R}^n)
\to L^p_\omega(\mathbb{R}^n)}
\leq C[\omega]_{A_p(\mathbb{R}^n)}^{p'-1},
$$
where $\|\mathcal{M}\|_{L^p_\omega(\mathbb{R}^n)\to L^p_\omega(\mathbb{R}^n)}$
denotes the operator norm of $\mathcal{M}$ on $L^p_\omega(\mathbb{R}^n)$.
\end{enumerate}
\end{lemma}

\section{Generalized Brezis--Van Schaftingen--Yung Formulae
in Ball\\ Banach Function Spaces}\label{S3}

Let $X$ be a ball Banach function space.
In this section, under some mild assumptions about $X$,
we establish generalized
Brezis--Van Schaftingen--Yung formulae on $X$.
To be more precise, we aim to recover
the ball Banach Sobolev semi-norm
$\|\,|\nabla f|\,\|_X$ for any $f$ in
$\dot{W}^{1,X}(\mathbb{R}^n)$ or its certain density subspace
via the functional
\begin{align}\label{2012}
\sup_{\lambda\in(0,\infty)}\lambda
\left\|\left[\int_{\mathbb{R}^n}
\mathbf{1}_{E_{\lambda,\frac{\gamma}{q}}[f]}(\cdot,y)
\left|\cdot-y\right|^{\gamma-n}\,dy\right]^\frac{1}{q}\right\|_X,
\end{align}
 where $\gamma\in\mathbb{R}\setminus\{0\}$,
$q\in(0,\infty)$, and
$E_{\lambda,\frac{\gamma}{q}}[f]$
for any $\lambda\in(0,\infty)$
is the same as in \eqref{Elambda}.

\subsection{Upper Estimate}
\label{ss3.1}

This subsection is devoted to the upper estimate
of \eqref{2213}.
The following is one of the main result of this subsection.

\begin{theorem}\label{upperBound}
Let $X$ be a ball Banach function space, $\gamma\in\mathbb{R}\setminus\{0\}$,
$p\in[1,\infty)$, and $q\in(0,p]$.
Assume that
\begin{enumerate}
\item[\textup{(i)}]
$X^\frac{1}{p}$ is a ball Banach function space;
\item[\textup{(ii)}]
the Hardy--Littlewood maximal operator $\mathcal{M}$ is bounded on
both $X$ and $(X^\frac{1}{p})'$ with their operator norms denoted,
respectively, by
$\|\mathcal{M}\|_{X\to X}$ and
$\|\mathcal{M}\|_{(X^\frac{1}{p})'\to (X^\frac{1}{p})'}$.
\end{enumerate}
Then there exists a positive constant
$C$,
depending only on $p$, $q$, $\gamma$, $n$, as well as $\|\mathcal{M}\|_{X\to X}$
and $\|\mathcal{M}\|_{(X^\frac{1}{p})'\to(X^\frac{1}{p})'}$,
such that,
for any $f\in C^1(\mathbb{R}^n)$ with
$|\nabla f|\in C_{\mathrm{c}}(\mathbb{R}^n)$,
\begin{align*}
\sup_{\lambda\in(0,\infty)}\lambda
\left\|\left[\int_{\mathbb{R}^n}
\mathbf{1}_{E_{\lambda,\frac{\gamma}{q}}[f]}(\cdot,y)
\left|\cdot-y\right|^{\gamma-n}\,dy\right]^\frac{1}{q}\right\|_X
\leq C\left\|\,\left|\nabla f\right|\,\right\|_{X},
\end{align*}
where, for any $\lambda\in(0,\infty)$,
$E_{\lambda,\frac{\gamma}{q}}[f]$ is the same as in \eqref{Elambda}
and where the positive constant $C$ is continuous, respectively, with
respect to $\|\mathcal{M}\|_{X\to X}$
or $\|\mathcal{M}\|_{(X^\frac{1}{p})'\to (X^\frac{1}{p})'}$
and increases, respectively, as $\|\mathcal{M}\|_{X\to X}$
or $\|\mathcal{M}\|_{(X^\frac{1}{p})'\to(X^\frac{1}{p})'}$ increases.
\end{theorem}

To prove Theorem~\ref{upperBound}, we need three technical lemmas.
The following lemma is just \cite[Lemma~4.6]{dlyyz.arxiv}.

\begin{lemma}\label{4.5}
Let $X\subset\mathscr{M}(\mathbb{R}^n)$ be a
linear normed space, equipped with a norm $\|\cdot\|_X$
which makes sense for all functions in $\mathscr{M}(\mathbb{R}^n)$.
Assume that the Hardy--Littlewood maximal operator $\mathcal{M}$
is bounded on $X$ with its operator norm
denoted by $\|\mathcal{M}\|_{X\to X}$.
For any $g\in X$ and $x\in\mathbb{R}^n$,
let
\begin{align}\label{RXg}
R_{X}g(x):=\sum_{k=0}^\infty
\frac{\mathcal{M}^kg(x)}{2^k\|\mathcal{M}\|^k_{X\to X}},
\end{align}
where, for any $k\in\mathbb{N}$,
$\mathcal{M}^k$ is the $k$ iterations of $\mathcal{M}$
and $\mathcal{M}^0g(x):=|g(x)|$.
Then, for any $g\in X$,
\begin{enumerate}
\item[\textup{(i)}]
for any $x\in\mathbb{R}^n$, $|g(x)|\leq R_{X}g(x)$;
\item[\textup{(ii)}]
$R_Xg\in A_1(\mathbb{R}^n)$ and $[R_Xg]_{A_1(\mathbb{R}^n)}
\leq2\|\mathcal{M}\|_{X\to X}$;
\item[\textup{(iii)}]
$\|R_{X}g\|_X\leq2\|g\|_X$.
\end{enumerate}
\end{lemma}

The following lemma is a simple
corollary of \cite[Lemma~4.7 and Remark~4.8]{dlyyz.arxiv}.

\begin{lemma}\label{4.6}
Let $X$ be a ball Banach function space and $p\in[1,\infty)$.
Assume that $X^\frac{1}{p}$ is a ball Banach function space and
the Hardy--Littlewood maximal operator is bounded on $(X^\frac{1}{p})'$.
Then, for any $f\in X$,
$$
\|f\|_X\leq\sup_{\|g\|_{(X^\frac{1}{p})'}\leq1}
\left[\int_{\mathbb{R}^n}
\left|f(x)\right|^pR_{(X^\frac{1}{p})'}g(x)\,dx\right]^\frac{1}{p}
\leq2^\frac{1}{p}\|f\|_X.
$$
\end{lemma}

The following lemma is just \cite[Lemma~2.6]{zwyy2021}.

\begin{lemma}\label{1555}
Let $X$ be a ball Banach function space.
Then $X$ coincides with its second associate space $X''$.
In other words, a function $f\in X$ if and only if
$f\in X''$ and, in that case,
$$
\|f\|_X=\|f\|_{X''}.
$$
\end{lemma}

Now, we show Theorem~\ref{upperBound}.

\begin{proof}[Proof of Theorem~\ref{upperBound}]
We only show the present theorem in the case $\gamma\in(-\infty,0)$,
because the proof for  the case $\gamma\in(0,\infty)$ is quite similar.
To this end, let $f\in C^1(\mathbb{R}^n)$ and
$|\nabla f|\in C_{\mathrm{c}}(\mathbb{R}^n)$.
By the Lusin--Lipschitz inequality
(see, for instance, \cite[(2.9)]{bvy2021}),
we find that there exists a positive constant $\widetilde{C}$,
independent of $f$, such that,
for any $(x,y)\in E_{\lambda,\frac{\gamma}{q}}[f]$,
\begin{align*}
\lambda|x-y|^\frac{\gamma}{q}<\frac{|f(x)-f(y)|}{|x-y|}
\leq\widetilde{C}\left[\mathcal{M}(|\nabla f|)(x)+
\mathcal{M}(|\nabla f|)(y)\right],
\end{align*}
which further implies that
\begin{align}\label{E12}
E_{\lambda,\frac{\gamma}{q}}[f]&\subset
\left\{(x,y)\in\mathbb{R}^n\times\mathbb{R}^n:\ x\neq y,\
\frac{\mathcal{M}(|\nabla f|)(x)}
{\lambda|x-y|^\frac{\gamma}{q}}>\frac{1}{2\widetilde{C}}\right\}\\
&\quad\cup\left\{(x,y)\in\mathbb{R}^n\times\mathbb{R}^n:\ x\neq y,\
\frac{\mathcal{M}(|\nabla f|)(y)}
{\lambda|x-y|^\frac{\gamma}{q}}>\frac{1}{2\widetilde{C}}\right\}\nonumber\\
&=:E_1\cup E_2.\nonumber
\end{align}
On the one hand,
from the definition of $E_1$, the polar coordinate,
and the assumption that $\mathcal{M}$
is bounded on $X$,
we deduce that, for any $\lambda\in(0,\infty)$,
\begin{align}\label{E1}
&\left\|\left[\int_{\mathbb{R}^n}
\mathbf{1}_{E_1}(\cdot,y)
\left|\cdot-y\right|^{\gamma-n}\,dy\right]^\frac{1}{q}\right\|_X\\
&\quad=\left\|\left[\int_{\{y\in\mathbb{R}^n:\ |\cdot-y|^\frac{\gamma}{q}
<2\widetilde{C}\lambda^{-1}\mathcal{M}(|\nabla f|)(\cdot)\}}
\left|\cdot-y\right|^{\gamma-n}\,dy\right]^\frac{1}{q}\right\|_X\nonumber\\
&\quad=\left\|\left\{\int_{\mathbb{S}^{n-1}}
\int_{[2\widetilde{C}\lambda^{-1}
\mathcal{M}(|\nabla f|)(\cdot)]^\frac{q}{\gamma}}^\infty
r^{\gamma-1}\,dr\,d\sigma\right\}^\frac{1}{q}\right\|_X\nonumber\\
&\quad\sim\frac{1}{\lambda}\left\|\mathcal{M}(|\nabla f|)\right\|_{X}
\leq\frac{\|\mathcal{M}\|_{X\to X}}{\lambda}
\left\|\,\left|\nabla f\right|\,\right\|_{X}.\nonumber
\end{align}
On the other hand, to estimate $E_2$,
we consider the following two cases on both $p$ and $q$.

\emph{Case 1)} $p=q$. In this case,
by Lemma~\ref{4.6}, the Tonelli theorem,
and the definition of $E_2$,
we conclude that
\begin{align}\label{1031}
&\left\|\left[\int_{\mathbb{R}^n}
\mathbf{1}_{E_2}(\cdot,y)
\left|\cdot-y\right|^{\gamma-n}\,dy\right]^\frac{1}{q}\right\|_X^q\\
&\quad\leq\sup_{\|g\|_{(X^\frac{1}{q})'}\leq1}
\int_{\mathbb{R}^n}\left[\int_{\mathbb{R}^n}
\mathbf{1}_{E_2}(x,y)\left|x-y\right|^{\gamma-n}\,dy\right]
R_{(X^\frac{1}{q})'}g(x)\,dx\nonumber\\
&\quad=\sup_{\|g\|_{(X^\frac{1}{q})'}\leq1}
\int_{\mathbb{R}^n}\int_{\{x\in\mathbb{R}^n:\
\frac{\mathcal{M}(|\nabla f|)(y)}
{\lambda|x-y|^\frac{\gamma}{q}}>\frac{1}{2\widetilde{C}}\}}
\left|x-y\right|^{\gamma-n}
R_{(X^\frac{1}{q})'}g(x)\,dx\,dy.\nonumber
\end{align}
From $\gamma\in(-\infty,0)$ and Lemma~\ref{ApProperty}(ii),
we deduce that,
for any $y\in\mathbb{R}^n$, $r\in(0,\infty)$,
and $\omega\in A_1(\mathbb{R}^n)$,
\begin{align}\label{1906}
&\int_{[B(y,r)]^{\complement}}
\left|x-y\right|^{\gamma-n}\omega(x)\,dx\\
&\quad=\sum_{j=1}^\infty\int_{2^jB(y,r)\setminus 2^{j-1}B(y,r)}
\left|x-y\right|^{\gamma-n}\omega(x)\,dx\nonumber\\
&\quad\leq\sum_{j=1}^\infty\left(2^{j-1}r\right)^{\gamma-n}
\omega\left(2^jB(y,r)\right)\nonumber\\
&\quad\lesssim[\omega]_{A_1(\mathbb{R}^n)}
\sum_{j=1}^\infty2^{j\gamma}
r^{\gamma-n}
\omega\left(B(y,r)\right)
\sim[\omega]_{A_1(\mathbb{R}^n)}
r^{\gamma-n}\omega\left(B(y,r)\right),\nonumber
\end{align}
where the implicit positive constants are
independent of $y$, $r$, and $\omega$,
which, combined with \eqref{1031},
Lemmas~\ref{4.5}(ii) and~\ref{ApProperty}(i),
and the assumption that $\mathcal{M}$
is bounded on $X$,
further implies that
\begin{align}\label{1039}
&\left\|\left[\int_{\mathbb{R}^n}
\mathbf{1}_{E_2}(\cdot,y)
\left|\cdot-y\right|^{\gamma-n}\,dy\right]^\frac{1}{q}\right\|_X^q\\
&\quad\lesssim\sup_{\|g\|_{(X^\frac{1}{q})'}\leq1}
\int_{\mathbb{R}^n}
\left[\frac{\mathcal{M}(|\nabla f|)(y)}
{\lambda}\right]^{\frac{q(\gamma-n)}{\gamma}}\nonumber\\
&\qquad\times\int_{B(y,[2\widetilde{C}\lambda^{-1}
\mathcal{M}(|\nabla f|)(y)]^\frac{q}{\gamma})}
R_{(X^\frac{1}{q})'}g(x)\,dx\,dy\nonumber\\
&\quad\leq\sup_{\|g\|_{(X^\frac{1}{q})'}\leq1}
\int_{\mathbb{R}^n}
\left[\frac{\mathcal{M}(|\nabla f|)(y)}{\lambda}
\right]^{\frac{q(\gamma-n)}{\gamma}}\nonumber\\
&\qquad\times\inf_{x\in B(y,[2\widetilde{C}\lambda^{-1}
\mathcal{M}(|\nabla f|)(y)]^\frac{q}{\gamma})}
\mathcal{M}\left(R_{(X^\frac{1}{q})'}g\right)(x)\nonumber\\
&\qquad\times
\left|B(y,[2\widetilde{C}\lambda^{-1}
\mathcal{M}(|\nabla f|)(y)]^\frac{q}{\gamma})\right|
\,dy\nonumber\\
&\quad\lesssim\left[R_{(X^\frac{1}{q})'}g\right]_{A_1(\mathbb{R}^n)}
\sup_{\|g\|_{(X^\frac{1}{q})'}\leq1}
\int_{\mathbb{R}^n}
\left[\frac{\mathcal{M}(|\nabla f|)(y)}{\lambda}\right]^{q}\nonumber\\
&\qquad\times\inf_{x\in B(y,[2\widetilde{C}\lambda^{-1}
\mathcal{M}(|\nabla f|)(y)]^\frac{q}{\gamma})}
R_{(X^\frac{1}{q})'}g(x)\,dy\nonumber\\
&\quad\lesssim\left\|\mathcal{M}\right\|_{(X^\frac{1}{q})'\to(X^\frac{1}{q})'}
\sup_{\|g\|_{(X^\frac{1}{q})'}\leq1}
\int_{\mathbb{R}^n}
\left[\frac{\mathcal{M}(|\nabla f|)(y)}{\lambda}\right]^{q}
R_{(X^\frac{1}{q})'}g(y)\,dy\nonumber\\
&\quad\lesssim\lambda^{-q}\|\mathcal{M}\|_{(X^\frac{1}{q})'\to(X^\frac{1}{q})'}
\left\|\mathcal{M}(|\nabla f|)\right\|_{X}^q\nonumber\\
&\quad\leq\lambda^{-q}\|\mathcal{M}\|_{(X^\frac{1}{q})'\to(X^\frac{1}{q})'}
\|\mathcal{M}\|_{X\to X}^q
\left\|\,\left|\nabla f\right|\,\right\|_{X}^q.\nonumber
\end{align}

\emph{Case 2)} $q\in(0,p)$. In this case,
let $r:=\frac{p}{q}\in(1,\infty)$, $\omega\in A_1(\mathbb{R}^n)$,
and $\mu:=\omega^{1-r'}$.
By both (i) and (v) of Lemma~\ref{ApProperty}, we find that
$\omega\in A_r(\mathbb{R}^n)$, $\mu\in A_{r'}(\mathbb{R}^n)$,
$[\mu]_{A_{r'}(\mathbb{R}^n)}^{r-1}=[\omega]_{A_r(\mathbb{R}^n)}
\leq[\omega]_{A_1(\mathbb{R}^n)}$,
and
$[L_\omega^r(\mathbb{R}^n)]'=L_\mu^{r'}(\mathbb{R}^n)$.
From this, Lemma~\ref{1555} with $X:=L_\omega^r(\mathbb{R}^n)$,
Definition~\ref{associte}, (i), (ii),
and (vi) of Lemma~\ref{4.5}, Lemma \ref{ApProperty}(vi),
and an argument similar to that used in the estimation of \eqref{1039},
we infer that
\begin{align*}
&\left\|\int_{\mathbb{R}^n}
\mathbf{1}_{E_2}(\cdot,y)|\cdot-y|^{\gamma-n}\,dy
\right\|_{L_\omega^r(\mathbb{R}^n)}\\
&\quad=\left\|\int_{\mathbb{R}^n}
\mathbf{1}_{E_2}(\cdot,y)|\cdot-y|^{\gamma-n}\,dy
\right\|_{[L_\omega^r(\mathbb{R}^n)]''}
=\left\|\int_{\mathbb{R}^n}
\mathbf{1}_{E_2}(\cdot,y)|\cdot-y|^{\gamma-n}\,dy
\right\|_{[L_\mu^{r'}(\mathbb{R}^n)]'}\nonumber\\
&\quad=\sup_{\|g\|_{L_\mu^{r'}(\mathbb{R}^n)}=1}
\int_{\mathbb{R}^n}\left[\int_{\mathbb{R}^n}
\mathbf{1}_{E_2}(x,y)|x-y|^{\gamma-n}\,
dy\right]g(x)\,dx\nonumber\\
&\quad\leq\sup_{\|g\|_{L_\mu^{r'}(\mathbb{R}^n)}=1}
\int_{\mathbb{R}^n}\left[\int_{\mathbb{R}^n}
\mathbf{1}_{E_2}(x,y)|x-y|^{\gamma-n}\,
dy\right]Rg(x)\,dx\nonumber\\
&\quad\lesssim\sup_{\|g\|_{L_\mu^{r'}(\mathbb{R}^n)}=1}
[Rg]_{A_1(\mathbb{R}^n)}\lambda^{-q}
\int_{\mathbb{R}^n}\left[\mathcal{M}
\left(\left|\nabla f\right|\right)(x)\right]^qRg(x)\,dx\nonumber\\
&\quad\lesssim\sup_{\|g\|_{L_\mu^{r'}(\mathbb{R}^n)}=1}
\left\|\mathcal{M}\right\|_{L_\mu^{r'}
(\mathbb{R}^n)\to L_\mu^{r'}(\mathbb{R}^n)}
\lambda^{-q}
\left\|\left[\mathcal{M}\left(\left|\nabla f\right|
\right)\right]^q\right\|_{L^r_\omega(\mathbb{R}^n)}
\left\|Rg\right\|_{L_\mu^{r'}(\mathbb{R}^n)}\nonumber\\
&\quad\lesssim
[\mu]_{A_{r'}(\mathbb{R}^n)}^{r-1}
\lambda^{-q}
\left\|\left[\mathcal{M}\left(\left|\nabla f\right|
\right)\right]^q\right\|_{L^r_\omega(\mathbb{R}^n)}
\leq[\omega]_{A_1(\mathbb{R}^n)}
\lambda^{-q}
\left\|\left[\mathcal{M}\left(\left|\nabla f\right|
\right)\right]^q\right\|_{L^r_\omega(\mathbb{R}^n)},
\end{align*}
where $Rg:=R_Xg$ is the same as in \eqref{RXg}
with $X$ replaced by $L_\mu^{r'}(\mathbb{R}^n)$,
which, together with Lemmas~\ref{4.6} and~\ref{4.5}(ii)
and the assumption that $\mathcal{M}$
is bounded on $X$, further implies that
\begin{align*}
&\left\|\left[\int_{\mathbb{R}^n}
\mathbf{1}_{E_2}(\cdot,y)
\left|\cdot-y\right|^{\gamma-n}\,dy\right]^\frac{1}{q}\right\|_X^q\nonumber\\
&\quad\leq\sup_{\|g\|_{(X^\frac{1}{p})'}\leq1}
\left\{\int_{\mathbb{R}^n}\left[\int_{\mathbb{R}^n}
\mathbf{1}_{E_2}(x,y)
\left|x-y\right|^{\gamma-n}\,dy\right]^{\frac{p}{q}}
R_{(X^\frac{1}{p})'}g(x)\,dx\right\}^{\frac{q}{p}}\nonumber\\
&\quad\lesssim\lambda^{-q}\sup_{\|g\|_{(X^\frac{1}{p})'}\leq1}
\left[R_{(X^\frac{1}{p})'}g\right]_{A_1(\mathbb{R}^n)}
\left\{\int_{\mathbb{R}^n}\left[\mathcal{M}
\left(\left|\nabla f\right|\right)(x)\right]^p
R_{(X^\frac{1}{p})'}g(x)\,dx\right\}^{\frac{q}{p}}\\
&\quad\lesssim\lambda^{-q}
\left\|\mathcal{M}\right\|_{(X^\frac{1}{p})'\to(X^\frac{1}{p})'}
\left\|\mathcal{M}(|\nabla f|)\right\|_{X}^q\\
&\quad\leq\lambda^{-q}
\left\|\mathcal{M}\right\|_{(X^\frac{1}{p})'\to(X^\frac{1}{p})'}
\left\|\mathcal{M}\right\|_{X\to X}^q
\left\|\,\left|\nabla f\right|\,\right\|_{X}^q.
\end{align*}
By this, Definition~\ref{1659}(v),
\eqref{E12}, \eqref{E1}, and \eqref{1039},
we conclude that
\begin{align*}
&\sup_{\lambda\in(0,\infty)}\lambda
\left\|\left[\int_{\mathbb{R}^n}
\mathbf{1}_{E_{\lambda,\frac{\gamma}{q}}[f]}(\cdot,y)
\left|\cdot-y\right|^{\gamma-n}\,dy\right]^\frac{1}{q}\right\|_X\\
&\quad\lesssim\sup_{\lambda\in(0,\infty)}\lambda
\left\|\left[\int_{\mathbb{R}^n}
\mathbf{1}_{E_1}(\cdot,y)
\left|\cdot-y\right|^{\gamma-n}\,dy\right]^\frac{1}{q}\right.\\
&\qquad\left.
+\left[\int_{\mathbb{R}^n}
\mathbf{1}_{E_2}(\cdot,y)
\left|\cdot-y\right|^{\gamma-n}\,dy\right]^\frac{1}{q}\right\|_X\\
&\quad\leq\sup_{\lambda\in(0,\infty)}\lambda
\left\|\left[\int_{\mathbb{R}^n}
\mathbf{1}_{E_1}(\cdot,y)
\left|\cdot-y\right|^{\gamma-n}\,dy\right]^\frac{1}{q}\right\|_X\\
&\qquad+\sup_{\lambda\in(0,\infty)}\lambda
\left\|\left[\int_{\mathbb{R}^n}
\mathbf{1}_{E_2}(\cdot,y)
\left|\cdot-y\right|^{\gamma-n}\,dy\right]^\frac{1}{q}\right\|_X\\
&\quad\lesssim
\left[1+\left\|\mathcal{M}\right\|_{(X^\frac{1}{p})'
\to(X^\frac{1}{p})'}^\frac{1}{q}\right]
\left\|\mathcal{M}\right\|_{X\to X}
\left\|\,\left|\nabla f\right|\,\right\|_{X}.
\end{align*}
This finishes the proof of Theorem~\ref{upperBound}.
\end{proof}

\begin{remark}
In Theorem~\ref{upperBound},
if $X:=L^p(\mathbb{R}^n)$ with $p\in(1,\infty)$
and if $q\in(0,p]$, then both $X$ and $X^\frac{1}{p}=L^1(\mathbb{R}^n)$
are ball Banach function spaces and
the Hardy--Littlewood maximal operator is bounded
on both $X$ and $(X^\frac{1}{p})'=L^\infty(\mathbb{R}^n)$.
Thus, Theorem~\ref{upperBound}
in this case
holds true, which when $q=p\in(1,\infty)$
is just \cite[Proposition 2.1]{bsvy.arxiv}
and which when $q\in(0,p)$ is new.
\end{remark}

Recall that the classical Morrey inequality states that,
for any given $r\in(n,\infty)$, any
$f\in C^1(\mathbb{R}^n)$ with $|\nabla f|\in C_{\mathrm{c}}(\mathbb{R}^n)$,
and almost every $x,y\in\mathbb{R}^n$ with $x\neq y$,
\begin{align*}
\frac{|f(x)-f(y)|}{|x-y|}\lesssim
\left[\mathcal{M}(|\nabla f|^r)(x)\right]^\frac{1}{r},
\end{align*}
where the implicit positive constant depends only on both $r$ and $n$
(see, for instance, \cite[Theorem~4.10(i)]{eg2015}).
By this and an argument similar to that used
in the estimation of \eqref{E1},
we obtain the following conclusion of the upper estimate of \eqref{2213}
with a full range of $q\in(0,\infty)$; we omit the details here.

\begin{theorem}\label{2001}
Let $X$ be a ball Banach function space, $\gamma\in\mathbb{R}\setminus\{0\}$,
$r\in(n,\infty)$, and $q\in(0,\infty)$.
Assume that the Hardy--Littlewood maximal operator $\mathcal{M}$
is bounded on $X^\frac{1}{r}$ with its operator norm denoted by
$\|\mathcal{M}\|_{X^\frac{1}{r}
\to X^\frac{1}{r}}$.
Then there exists a positive constant
$C$,
depending only on $r$, $q$, $\gamma$, $n$, and $\|\mathcal{M}\|_{X^\frac{1}{r}
\to X^\frac{1}{r}}$, such that,
for any $f\in C^1(\mathbb{R}^n)$ with
$|\nabla f|\in C_{\mathrm{c}}(\mathbb{R}^n)$,
\begin{align*}
\sup_{\lambda\in(0,\infty)}\lambda
\left\|\left[\int_{\mathbb{R}^n}
\mathbf{1}_{E_{\lambda,\frac{\gamma}{q}}[f]}(\cdot,y)
\left|\cdot-y\right|^{\gamma-n}\,dy\right]^\frac{1}{q}\right\|_X
\leq C\left\|\,\left|\nabla f\right|\,\right\|_{X},
\end{align*}
where, for any $\lambda\in(0,\infty)$,
$E_{\lambda,\frac{\gamma}{q}}[f]$ is the same as in \eqref{Elambda}.
\end{theorem}

\begin{remark}
In Theorem~\ref{2001},
if $X:=L^p(\mathbb{R}^n)$ with $p\in(n,\infty)$
and if $q\in(0,\infty)$, then there exists an $r\in(n,p)$ such that
the Hardy--Littlewood maximal operator is bounded on
$X^\frac{1}{r}=L^\frac{p}{r}(\mathbb{R}^n)$
and hence Theorem~\ref{2001}
in this case
holds true, which when $q=p\in(n,\infty)$
is a part of \cite[Proposition 2.1]{bsvy.arxiv}
and which when $q\in(0,\infty)\setminus\{p\}$ is new.
\end{remark}

Recall that, in the endpoint case, that is, when $p=1$,
the upper estimate in \eqref{2118} also holds true
(see, for instance, \cite[Theorem~2.2]{bsvy.arxiv}).
Notice that, in Theorem~\ref{upperBound},
we assume that the Hardy--Littlewood maximal operator
$\mathcal{M}$ is bounded on $X$
which is not true when $X:=L^1(\mathbb{R}^n)$
and hence Theorem~\ref{upperBound} can not give the
upper estimate in \eqref{2213} when $X:=L^1(\mathbb{R}^n)$.
Thus, to establish \eqref{2213} in the setting of ball Banach function space
$X$ on which $\mathcal{M}$ might be not known to be bounded,
we have two other conclusions
(see Theorems~\ref{gamma>0} and~\ref{n=1} below).
The following is on the case $\gamma\in(0,\infty)$.

\begin{theorem}\label{gamma>0}
Let $X$ be a ball Banach function space,
$\gamma\in(0,\infty)$, $p\in[1,\infty)$,
and $q\in(0,\infty)$ satisfy $n(\frac{1}{p}-\frac{1}{q})<1$.
Assume both that
$X^\frac{1}{p}$ is a ball Banach function space
and that the Hardy--littlewood maximal operator $\mathcal{M}$
is bounded on $(X^\frac{1}{p})'$ with its operator norm
denoted by $\|\mathcal{M}\|_{(X^\frac{1}{p})'\to(X^\frac{1}{p})'}$.
Then there exists a positive constant $C$,
depending only on $p$, $q$, $\gamma$, $n$,
and $\|\mathcal{M}\|_{(X^\frac{1}{p})'\to(X^\frac{1}{p})'}$, such that,
for any $f\in C^1(\mathbb{R}^n)$ with
$|\nabla f|\in C_{\mathrm{c}}(\mathbb{R}^n)$,
\begin{align*}
\sup_{\lambda\in(0,\infty)}\lambda
\left\|\left[\int_{\mathbb{R}^n}
\mathbf{1}_{E_{\lambda,\frac{\gamma}{q}}[f]}(\cdot,y)
\left|\cdot-y\right|^{\gamma-n}\,dy\right]^\frac{1}{q}\right\|_X
\leq C\left\|\,|\nabla f|\,\right\|_X,
\end{align*}
where $\kappa(q,n)$ and $E_{\lambda,\frac{\gamma}{q}}[f]$
with $\lambda\in(0,\infty)$ are the same as, respectively, in
\eqref{kappaqn} and \eqref{Elambda},
and $C$ is continuous with
respect to $\|\mathcal{M}\|_{(X^\frac{1}{p})'\to(X^\frac{1}{p})'}$
and increases as $\|\mathcal{M}\|_{(X^\frac{1}{p})'\to(X^\frac{1}{p})'}$ increases.
\end{theorem}

To prove Theorem~\ref{gamma>0},
we first show the following proposition.

\begin{proposition}\label{2010}
Let $\gamma\in(0,\infty)$,
$p\in[1,\infty)$, and $q\in(0,\infty)$
satisfy $n(\frac{1}{p}-\frac{1}{q})<1$.
Assume that $\omega\in A_1(\mathbb{R}^n)$.
Then there exists a positive constant $C_{([\omega]_{A_1(\mathbb{R}^n)})}$,
depending on $[\omega]_{A_1(\mathbb{R}^n)}$, such that,
for any $f\in C^1(\mathbb{R}^n)$ with
$|\nabla f|\in C_{\mathrm{c}}(\mathbb{R}^n)$,
\begin{align}\label{1607}
&\sup_{\lambda\in(0,\infty)}\lambda^p
\int_{\mathbb{R}^n}\left[\int_{\mathbb{R}^n}
\mathbf{1}_{E_{\lambda,\frac{\gamma}{q}}[f]}
(x,y)|x-y|^{\gamma-n}\,dy\right]^{\frac{p}{q}}\omega(x)\,dx\\
&\quad\leq C\int_{\mathbb{R}^n}\left|\nabla f(x)\right|^p\omega(x)\,dx,\nonumber
\end{align}
where $E_{\lambda,\frac{\gamma}{q}}[f]$
for any $\lambda\in(0,\infty)$ is the same as in \eqref{Elambda}
and $C_{([\omega]_{A_1(\mathbb{R}^n)})}$ is continuous with
respect to $[\omega]_{A_1(\mathbb{R}^n)}$
and increases as $[\omega]_{A_1(\mathbb{R}^n)}$ increases.
\end{proposition}

To prove Proposition~\ref{2010}, we need three technical lemmas.
Let us begin with recalling the
following useful inequality.

\begin{lemma}\label{dis}
Let $r\in(0,1]$
and $\{a_j\}_{j\in\mathbb{N}}\subset\mathbb{C}$.
Then
$(\sum_{j\in\mathbb{N}}|a_j|)^r
\leq\sum_{j\in\mathbb{N}}|a_j|^r$.
\end{lemma}

The following Poincar\'e-type lemma is just \cite[Lemma~3.12]{dlyyz.arxiv}.

\begin{lemma}\label{Poin}
Let $p\in[1,\infty)$, $\omega\in A_p(\mathbb{R}^n)$,
and $f\in\dot{W}^{1,p}_\omega(\mathbb{R}^n)$.
Then there exists a set $A\subset\mathbb{R}^n$ with $|A|=0$
and a positive constant $C_{(n)}$, depending only on $n$, such that,
for any $x\in\mathbb{R}^n\setminus A$, any $r\in(0,\infty)$,
and any ball $B_1\subset B:=B(x,r)\subset 3B_1$,
\begin{align*}
\left|f(x)-f_{B_1}\right|\leq C_{(n)}r\sum_{j=0}^\infty
2^{-j}\fint_{2^{-j}B}\left|\nabla f(y)\right|\,dy.
\end{align*}
\end{lemma}

The following adjacent system of dyadic cubes can be
found in \cite[Section~2.2]{lsu2012}.

\begin{lemma}\label{2115}
For any $\alpha\in\{0,\frac{1}{3},\frac{2}{3}\}^n$, let
$$
\mathcal{D}^\alpha:=\left\{2^j\left[k+(0,1]^n+(-1)^j\alpha\right]:\
j\in\mathbb{Z},\ k\in\mathbb{Z}^n\right\}.
$$
Then
\begin{enumerate}
\item[\textup{(i)}]
for any $Q,P\in\mathcal{D}^\alpha$ with the same
$\alpha\in\{0,\frac{1}{3},\frac{2}{3}\}^n$,
$$
Q\cap P\in\left\{\emptyset,Q,P\right\};
$$
\item[\textup{(ii)}]
for any ball $B\subset\mathbb{R}^n$, there exists an
$\alpha\in\{0,\frac{1}{3},\frac{2}{3}\}^n$ and a $Q\in\mathcal{D}^\alpha$
such that $B\subset Q\subset CB$,
where the positive constant $C$ depends only on $n$.
\end{enumerate}
\end{lemma}

Now, we are ready to show Proposition~\ref{2010}.

\begin{proof}[Proof of Proposition~\ref{2010}]
We borrow some ideas from the proof of \cite[Theorem~3.13]{dlyyz.arxiv}.
Let $f\in C^1(\mathbb{R}^n)$ with $|\nabla f|\in C_{\mathrm{c}}(\mathbb{R}^n)$.
Without loss of generality,
we may assume that $\lambda=1$.
Otherwise, we can replace $f$ by $\frac{f}{\lambda}$ for any $\lambda\in(0,\infty)$.
Let $A\subset\mathbb{R}^n$ be the same as in Lemma~\ref{Poin}.
For any $x,y\in\mathbb{R}^n$ with $x\neq y$,
let $B_{x,y}:=B(\frac{x+y}{2},|x-y|)$.
Then
\begin{align*}
&E_{1,\frac{\gamma}{q}}[f]\setminus
\left[(A\times\mathbb{R}^n)\cup(\mathbb{R}^n\times A)\right]\\
&\quad\subset\left\{(x,y)\in\mathbb{R}^n\times\mathbb{R}^n:\
x\neq y,\
\frac{|f(x)-f_{B_{x,y}}|}{|x-y|^{1+\frac{\gamma}{q}}}>\frac{1}{2}\right\}\setminus
(A\times\mathbb{R}^n)\\
&\qquad\cup\left\{(x,y)\in\mathbb{R}^n\times\mathbb{R}^n:\
x\neq y,\
\frac{|f(y)-f_{B_{x,y}}|}{|x-y|^{1+\frac{\gamma}{q}}}>\frac{1}{2}\right\}\setminus
(\mathbb{R}^n\times A)\\
&\quad=:E^{(1)}\cup E^{(2)}.
\end{align*}
From this, we deduce that, to prove the present theorem,
it suffices to show that there exists a positive constant $C$,
independent of $f$, such that
\begin{align}\label{e(1)}
&\int_{\mathbb{R}^n}\left[\int_{\mathbb{R}^n}\mathbf{1}_{E^{(1)}}
(x,y)|x-y|^{\gamma-n}\,dy\right]^{\frac{p}{q}}\omega(x)\,dx\\
&\quad\leq C\int_{\mathbb{R}^n}
\left|\nabla f(x)\right|^p\omega(x)\,dx\nonumber
\end{align}
and
\begin{align}\label{e(2)}
&\int_{\mathbb{R}^n}\left[\int_{\mathbb{R}^n}\mathbf{1}_{E^{(2)}}
(x,y)|x-y|^{\gamma-n}\,dy\right]^{\frac{p}{q}}\omega(x)\,dx\\
&\quad\leq C\int_{\mathbb{R}^n}
\left|\nabla f(x)\right|^p\omega(x)\,dx.\nonumber
\end{align}

We first prove \eqref{e(1)}. Let $\varepsilon\in(0,1)$.
By the definition of $E^{(1)}$ and
Lemma~\ref{Poin} with both $B:=B(x,2|x-y|)$
and $B_1:=B_{x,y}$, we find that there exists a
positive constant $c_0$, depending only on $n$, such that,
for any given $(x,y)\in E^{(1)}$,
\begin{align}\label{1110}
\frac{|x-y|^{\frac{\gamma}{q}}}{2}<
\frac{|f(x)-f_{B_{x,y}}|}{|x-y|}\leq c_{0}\sum_{j=0}^\infty
2^{-j}\fint_{2^{-j}B}\left|\nabla f(y)\right|\,dy,
\end{align}
which further implies that there exists
a $j_{x,y}\in\mathbb{Z}_+$, depending only on both $x$ and $y$,
such that
\begin{align}\label{2218}
\fint_{B_{j_{x,y}}}\left|\nabla f(z)\right|\,dz
>c_12^{j_{x,y}(1-\varepsilon)}
\left|2^{j_{x,y}}B_{j_{x,y}}\right|^{\frac{\gamma}{qn}},
\end{align}
where $B_{j_{x,y}}:=B(x,2^{-j_{x,y}+1}|x-y|)$ and
$c_1:=\frac{1-2^\varepsilon}{2^{\frac{\gamma}{q}+1}c_0}\in(0,\infty)$.
Indeed, if, for any $j\in\mathbb{Z}_+$,
$$
\fint_{2^{-j}B}\left|\nabla f(z)\right|\,dz
\leq c_12^{j(1-\varepsilon)}
\left|B\right|^{\frac{\gamma}{qn}}
=\frac{1-2^\varepsilon}{2c_0}2^{j(1-\varepsilon)}|x-y|^\frac{q}{\gamma},
$$
then
\begin{align*}
\sum_{j=0}^\infty
2^{-j}\fint_{2^{-j}B}\left|\nabla f(y)\right|\,dy
\leq\frac{|x-y|^{\frac{\gamma}{q}}}{2c_0},
\end{align*}
which contradicts \eqref{1110}. Thus, \eqref{2218} holds true.
Applying Lemma~\ref{2115}(ii) to $B_{j_{x,y}}$, we find that
there exists an $\alpha_{x,y}\in\{0,\frac{1}{3},\frac{2}{3}\}^n$
and a dyadic cube $Q_{x,y}\in\mathcal{D}^{\alpha_{x,y}}$ such that
\begin{align}\label{2212}
B_{j_{x,y}}\subset Q_{j_{x,y}}\subset\widetilde{C}B_{j_{x,y}},
\end{align}
where $\widetilde{C}$ is the same constant as in Lemma~\ref{2115}(ii),
which, combined with \eqref{2218},
further implies that there exists a positive constant $c_2$,
depending only on $q$, $\gamma$, $n$, and $\varepsilon$, such that
\begin{align}\label{3.4}
\fint_{Q_{j_{x,y}}}\left|\nabla f(z)\right|\,dz
>c_22^{j_{x,y}(1-\varepsilon)}\left|2^{j_{x,y}}Q_{j_{x,y}}
\right|^{\frac{\gamma}{qn}}.
\end{align}
For any $\alpha\in\{0,\frac{1}{3},\frac{2}{3}\}^n$ and $j\in\mathbb{Z}_+$,
let
\begin{align*}
\mathscr{A}_{\alpha,j}:=\left\{Q\in\mathcal{D}^\alpha:\
Q\text{ satisfies }\eqref{3.4}\text{ with }j_{x,y}\text{ replaced by }j\right\}.
\end{align*}
From this, \eqref{3.4}, and both (iii) and (iv) of Lemma~\ref{ApProperty},
we deduce that, for any $\alpha\in\{0,\frac{1}{3},\frac{2}{3}\}^n$,
$j\in\mathbb{Z}_+$, and $Q\in\mathscr{A}_{\alpha,j}$,
\begin{align}\label{2230}
\omega(Q)\left|2^jQ\right|^\frac{p\gamma}{qn}
&<[\omega]_{A_p(\mathbb{R}^n)}\frac{\int_Q|\nabla f(z)|^p\omega(z)\,dz}
{[\fint_{Q}|\nabla f(z)|\,dz]^p}
\left[\frac{\fint_{Q}|\nabla f(z)|\,dz}{c_2
2^{j(1-\varepsilon)}}\right]^p\\
&\leq c_2^{-p}[\omega]_{A_1(\mathbb{R}^n)}2^{-jp(1-\varepsilon)}
\int_Q\left|\nabla f(z)\right|^p\omega(z)\,dz.\nonumber
\end{align}
Moreover, observing that, since $f\in C^1(\mathbb{R}^n)$
with $|\nabla f|\in C_{\mathrm{c}}(\mathbb{R}^n)$,
\eqref{3.4} further implies that, for any given
$\alpha\in\{0,\frac{1}{3},\frac{2}{3}\}^n$ and
$j\in\mathbb{Z}_+$,
$\sup_{Q\in\mathscr{A}_{\alpha,j}}l_{(Q)}<\infty$
with $l_{(Q)}$ being the edge length of $Q$.
Thus, each dyadic cube $Q\in\mathscr{A}_{\alpha,j}$ is contained
in a dyadic cube in $\mathscr{A}_{\alpha,j}$ that is
maximal with respect to the set inclusion.
For any $\alpha\in\{0,\frac{1}{3},\frac{2}{3}\}^n$ and $j\in\mathbb{Z}_+$,
denote by  $\mathscr{A}_{\alpha,j}^{\mathrm{max}}$
the set of all the dyadic cubes in $\mathscr{A}_{\alpha,j}$ that are maximal.
Then, by \eqref{2218}, Lemma~\ref{2115}(ii), and
the definitions of both $\mathscr{A}_{\alpha,j}$ and
$\mathscr{A}_{\alpha,j}^{\mathrm{max}}$, we conclude that,
for any $(x,y)\in E^{(1)}$, there exists a $j_{x,y}\in\mathbb{Z}_+$,
an $\alpha_{x,y}\in\{0,\frac{1}{3},\frac{2}{3}\}^n$,
a dyadic cube $Q_{j_{x,y}}\in
\mathcal{D}^{\alpha_{x,y}}$ satisfying \eqref{2212},
and a maximal dyadic cube $\widetilde{Q}_{j_{x,y}}\subset
\mathscr{A}_{\alpha_{x,y},j_{x,y}}^{\mathrm{max}}$ such that
$$
(x,y)\in\left(B_{j_{x,y}}\times2^{j_{x,y}}B_{j_{x,y}}\right)
\subset\left(Q_{j_{x,y}}\times2^{j_{x,y}}Q_{j_{x,y}}\right)
\subset\left(\widetilde{Q}_{j_{x,y}}
\times2^{j_{x,y}}\widetilde{Q}_{j_{x,y}}\right),
$$
which further implies that
\begin{align}\label{2224}
E^{(1)}\subset\bigcup_{j=0}^\infty
\bigcup_{\alpha\in\{0,\frac{1}{3},\frac{2}{3}\}^n}
\bigcup_{Q\in\mathscr{A}_{\alpha,j}^{\mathrm{max}}}\left(Q\times2^jQ\right).
\end{align}
Let $Q\subset\mathbb{R}^n$ be an arbitrary dyadic cube.
Notice that, for any given $x\in Q$
and $j\in\mathbb{Z}_+$ and for any $y\in 2^jQ$,
$$
|x-y|\leq|x-z_Q|+|z_Q-y|
\leq\frac{n^\frac{1}{2}l_{(Q)}}{2}+\frac{2^jn^\frac{1}{2}l_{(Q)}}{2}=:\frac M2,
$$
where $z_Q$ denotes the center of $Q$,
and hence $2^jQ\subset B(x,M)$.
From this and $M\sim2^jl_{(Q)}$, we deduce that,
for any given $x\in Q$ and $j\in\mathbb{Z}_+$,
\begin{align*}
\int_{2^jQ}|x-y|^{\gamma-n}\,dy
&\leq\int_{B(x,M)}|x-y|^{\gamma-n}\,dy\\
&\leq\int_{\mathbb{S}^{n-1}}\int_0^M
r^{\gamma-1}\,dr\,d\sigma
\sim|2^jQ|^{\frac{\gamma}{n}},
\end{align*}
which, together with \eqref{2224},
further implies that, for any $x\in\mathbb{R}^n$,
\begin{align*}
&\int_{\mathbb{R}^n}\mathbf{1}_{E^{(1)}}(x,y)|x-y|^{\gamma-n}\,dy\\
&\quad\leq\sum_{j=0}^\infty
\sum_{\alpha\in\{0,\,\frac{1}{3},\,\frac{2}{3}\}^n}
\sum_{Q\in\mathscr{A}_{\alpha,j}^{\mathrm{max}}}
\int_{2^jQ}|x-y|^{\gamma-n}\,dy\mathbf{1}_Q(x)\\
&\quad\lesssim\sum_{j=0}^\infty
\sum_{\alpha\in\{0,\,\frac{1}{3},\,\frac{2}{3}\}^n}
\sum_{Q\in\mathscr{A}_{\alpha,j}^{\mathrm{max}}}
|2^jQ|^{\frac{\gamma}{n}}\mathbf{1}_Q(x).
\end{align*}
This, combined with \eqref{2230} and
an argument similar to that used in the proofs of
both Case 1) and Case 2) in the proof of \cite[Lemma~3.14]{dlyyz.arxiv}
with $|2^jQ|$ replaced by $|2^jQ|^{\frac{\gamma}{n}}$,
further implies that \eqref{e(1)} holds true; we omit the details here.

Next, we prove \eqref{e(2)}.
Since $n(\frac{1}{p}-\frac{1}{q})<1$, we can take an $\eta\in(0,1)$ such that
\begin{align}\label{1101}
\eta<1-n\left(\frac{1}{p}-\frac{1}{q}\right).
\end{align}
By an argument similar to that used in the proof of \eqref{e(1)},
we conclude that
\begin{enumerate}
\item[(i)]
there exists a positive constant $C_1$ such that, for any $(x,y)\in E^{(2)}$,
there exists a $\widetilde{j}_{x,y}\in\mathbb{Z}_+$ satisfying that
\begin{align*}
\fint_{\widetilde{B}_{j_{x,y}}}\left|\nabla f(z)\right|\,dz
>C_12^{\widetilde{j}_{x,y}(1-\eta)}
\left|2^{\widetilde{j}}\widetilde{B}_{j_{x,y}}\right|^{\frac{\gamma}{q}},
\end{align*}
where $\widetilde{B}_{j_{x,y}}:=B(y,2^{-\widetilde{j}_{x,y}+1}|x-y|)$;
\item[(ii)]
for any given $(x,y)\in E^{(2)}$
with $\widetilde{B}_{j_{x,y}}$ the same as in the above,
there exists an $\widetilde{\alpha}_{x,y}\in\{0,\frac{1}{3},\frac{2}{3}\}^n$
and a dyadic cube $\widetilde{Q}_{x,y}\in
\mathcal{D}^{\widetilde{\alpha}_{x,y}}$ such that
\begin{align*}
\widetilde{B}_{\widetilde{j}_{x,y}}\subset\widetilde{Q}_{\widetilde{j}_{x,y}}
\subset\widetilde{C}\,\widetilde{B}_{\widetilde{j}_{x,y}},
\end{align*}
where $\widetilde{C}$ is the same constant as in Lemma~\ref{2115}(ii),
and hence there exists a positive constant $C_2$,
depending only on $n$, $q$, $\gamma$, and $\varepsilon$, such that
\begin{equation}\label{822}
\fint_{\widetilde{Q}_{j_{x,y}}}\left|\nabla f(z)\right|\,dz
>C_22^{\widetilde{j}_{x,y}(1-\eta)}
\left|2^{\widetilde{j}_{x,y}}\widetilde{Q}_{\widetilde{j}_{x,y}}
\right|^{\frac{\gamma}{qn}};
\end{equation}
\item[(iii)]
for any $\alpha\in\{0,\frac{1}{3},\frac{2}{3}\}^n$ and $j\in\mathbb{Z}_+$,
let
\begin{align*}
\widetilde{\mathscr{A}}_{\alpha,j}:=\left\{Q\in\mathcal{D}^\alpha:\
Q\text{ satisfies }\eqref{822}\text{ with }
\widetilde{j}_{x,y}\text{ replaced by }j\right\}
\end{align*}
and $\widetilde{\mathscr{A}}_{\alpha,j}^{\mathrm{max}}$ denote the set
of all the dyadic cubes in $\widetilde{\mathscr{A}}_{\alpha,j}$ that are maximal
with respect to the set inclusion. Then
there exists a positive constant $C_3$,
depending only on $n$, $q$, $\gamma$, and $\eta$, such that,
for any $j\in\mathbb{Z}_+$, $\alpha\in\{0,\frac{1}{3},\frac{2}{3}\}^n$,
and $Q\in\widetilde{\mathscr{A}}_{\alpha,j}^{\mathrm{max}}$,
\begin{align}\label{1009}
\omega(Q)\left|2^jQ\right|^\frac{p\gamma}{qn}
\leq C_3[\omega]_{A_1(\mathbb{R}^n)}2^{-jp(1-\eta)}
\int_Q\left|\nabla f(z)\right|^p\omega(z)\,dz;
\end{align}
\item[(iv)]
we also have
\begin{align}\label{1777}
E^{(2)}\subset\bigcup_{j=0}^\infty
\bigcup_{\alpha\in\{0,\frac{1}{3},\frac{2}{3}\}^n}
\bigcup_{Q\in\widetilde{\mathscr{A}}_{\alpha,j
}^{\mathrm{max}}}\left(2^jQ\times Q\right).
\end{align}
\end{enumerate}

When $q\in[p,\infty)$,
to prove \eqref{e(2)},
we first claim that
\begin{align}\label{1653}
&\int_{\mathbb{R}^n}\left[\int_{\mathbb{R}^n}
\mathbf{1}_{E^{(2)}}(x,y)|x-y|^{\gamma-n}\,dy
\right]^{\frac{p}{q}}\omega(x)\,dx\\
&\quad\lesssim[\omega]_{A_1(\mathbb{R}^n)}
\sum_{j=0}^\infty\sum_{\alpha\in\{0,\,\frac{1}{3},\,\frac{2}{3}\}^n}
\sum_{Q\in\widetilde{\mathscr{A}}_{\alpha,j}^{\mathrm{max}}}
2^{jn(1-\frac{p}{q})}\left|2^jQ\right|^{\frac{p\gamma}{qn}}\omega(Q).\nonumber
\end{align}
To this end, we consider the following two cases on $\gamma$.

\emph{Case 1)} $\gamma\in[n,\infty)$. In this case,
let $Q\subset\mathbb{R}^n$ be an arbitrary dyadic cube
and $j\in\mathbb{Z}_+$.
Then, for any given $x\in2^jQ$,
\begin{align*}
\int_{Q}|x-y|^{\gamma-n}\,dy\leq|Q|\max_{y\in Q}|x-y|^{\gamma-n}
\lesssim2^{-jn}\left|2^jQ\right|^{\frac{\gamma}{n}}.
\end{align*}
From this, \eqref{1777}, Lemma~\ref{dis}
with $r:=\frac{p}{q}\in(0,1]$,
and Lemma~\ref{ApProperty}(ii),
we further deduce that
\begin{align*}
&\int_{\mathbb{R}^n}\left[\int_{\mathbb{R}^n}
\mathbf{1}_{E^{(2)}}(x,y)|x-y|^{\gamma-n}\,dy\right]^{\frac{p}{q}}\omega(x)\,dx\\
&\quad\leq\int_{\mathbb{R}^n}\left[\sum_{j=0}^\infty
\sum_{\alpha\in\{0,\,\frac{1}{3},\,\frac{2}{3}\}^n}
\sum_{Q\in\widetilde{\mathscr{A}}_{\alpha,j}^{\mathrm{max}}}
\int_{Q}|x-y|^{\gamma-n}\,dy
\mathbf{1}_{2^jQ}(x)\right]^{\frac{p}{q}}\omega(x)\,dx\\
&\quad\lesssim\sum_{j=0}^\infty\sum_{\alpha\in\{0,\,\frac{1}{3},\,\frac{2}{3}\}^n}
\sum_{Q\in\widetilde{\mathscr{A}}_{\alpha,j}^{\mathrm{max}}}
2^{-\frac{jnp}{q}}|2^jQ|^{\frac{p\gamma}{qn}}\omega(2^jQ)\\
&\quad\leq[\omega]_{A_1(\mathbb{R}^n)}
\sum_{j=0}^\infty\sum_{\alpha\in\{0,\,\frac{1}{3},\,\frac{2}{3}\}^n}
\sum_{Q\in\widetilde{\mathscr{A}}_{\alpha,j}^{\mathrm{max}}}
2^{jn(1-\frac{p}{q})}
\left|2^jQ\right|^{\frac{p\gamma}{qn}}\omega(Q).
\end{align*}
This finishes the proof of \eqref{1653} in the case $\gamma\in[n,\infty)$.

\emph{Case 2)} $\gamma\in(0,n)$.
In this case, by \eqref{1777} and Lemma~\ref{dis}
with $r:=\frac{p}{q}\in(0,1]$, we find that
\begin{align*}
&\int_{\mathbb{R}^n}\left[\int_{\mathbb{R}^n}
\mathbf{1}_{E^{(2)}}(x,y)|x-y|^{\gamma-n}\,dy\right]^{\frac{p}{q}}\omega(x)\,dx\\
&\quad\leq\int_{\mathbb{R}^n}\left[\sum_{j=0}^\infty
\sum_{\alpha\in\{0,\,\frac{1}{3},\,\frac{2}{3}\}^n}
\sum_{Q\in\widetilde{\mathscr{A}}_{\alpha,j}^{\mathrm{max}}}\int_{Q}
|x-y|^{\gamma-n}\,dy\mathbf{1}_{2^jQ}(x)\right]^{\frac{p}{q}}\omega(x)\,dx\\
&\quad\leq\sum_{j=0}^\infty
\sum_{\alpha\in\{0,\,\frac{1}{3},\,\frac{2}{3}\}^n}
\sum_{Q\in\widetilde{\mathscr{A}}_{\alpha,j}^{\mathrm{max}}}
\int_{2^jQ}\left(\int_{Q}
|x-y|^{\gamma-n}\,dy\right)^{\frac{p}{q}}\omega(x)\,dx.
\end{align*}
Thus, to prove \eqref{1653} in this case, it suffices to show that,
for any $j\in\mathbb{Z}_+$, $\alpha\in\{0,\,\frac{1}{3},\,\frac{2}{3}\}^n$,
and $Q\in\widetilde{\mathscr{A}}_{\alpha,j}^{\mathrm{max}}$,
\begin{align}\label{2039}
\int_{2^jQ}\left(\int_{Q}
|x-y|^{\gamma-n}\,dy\right)^{\frac{p}{q}}\omega(x)\,dx
\lesssim[\omega]_{A_1(\mathbb{R}^n)}2^{jn(1-\frac{p}{q})}
\left|2^jQ\right|^{\frac{p\gamma}{qn}}\omega(Q),
\end{align}
where the implicit positive constant is independent of $j$, $\alpha$, and $Q$.
To this end, we consider the following three subcases on $j$.

\emph{Subcase 1)} $j=0$. In this subcase,
using $Q\subset B(x,\sqrt{n}l_{(Q)})$ for any $x\in Q$,
the polar coordinate,
and Lemma~\ref{ApProperty}(i),
we find that
\begin{align}\label{2037}
&\int_{Q}\left(\int_{Q}
|x-y|^{\gamma-n}\,dy\right)^{\frac{p}{q}}\omega(x)\,dx\\
&\quad\leq\int_{Q}\left[\int_{B(x,\sqrt{n}l_{(Q)})}
|x-y|^{\gamma-n}\,dy\right]^{\frac{p}{q}}\omega(x)\,dx\nonumber\\
&\quad=\int_{Q}\left[\int_{\mathbb{S}^{n-1}}\int_0^{\sqrt{n}l_{(Q)}}
r^{\gamma-1}\,dr\,d\sigma\right]^{\frac{p}{q}}\omega(x)\,dx\nonumber\\
&\quad\sim|Q|^{\frac{p\gamma}{qn}}\omega(Q)
\leq[\omega]_{A_1(\mathbb{R}^n)}|Q|^{\frac{p\gamma}{qn}}\omega(Q).\nonumber
\end{align}

\emph{Subcase 2)} $j=1$. In this subcase,
from Lemma~\ref{ApProperty}(ii) and
$Q\subset B(x,\frac{3\sqrt{n}}{2}l_{(Q)})$
for any $x\in 2Q$,
we deduce that
\begin{align}\label{2038}
&\int_{2Q}\left(\int_{Q}
|x-y|^{\gamma-n}\,dy\right)^{\frac{p}{q}}\omega(x)\,dx\\
&\quad\leq\int_{2Q}\left[\int_{B(x,\frac{3\sqrt{n}}{2}l_{(Q)})}
|x-y|^{\gamma-n}\,dy\right]^{\frac{p}{q}}\omega(x)\,dx\nonumber\\
&\quad=\int_{2Q}\left[\int_{\mathbb{S}^{n-1}}\int_0^{\frac{3\sqrt{n}}{2}l_{(Q)}}
r^{\gamma-1}\,dr\,d\sigma\right]^{\frac{p}{q}}\omega(x)\,dx\nonumber\\
&\quad\sim|Q|^{\frac{p\gamma}{qn}}\omega(2Q)
\lesssim[\omega]_{A_1(\mathbb{R}^n)}
\left|2Q\right|^{\frac{p\gamma}{qn}}\omega(Q),\nonumber
\end{align}
where the implicit positive constant depends only on
$p$, $q$, $\gamma$, and $n$.

\emph{Subcase 3)} $j\in\mathbb{N}\setminus\{1\}$. In this subcase,
by both $\gamma\in(0,n)$ and Lemma~\ref{ApProperty}(ii), we find that,
for any $k\in\mathbb{N}\cap[2,\infty)$,
\begin{align*}
&\int_{2^kQ\setminus2^{k-1}Q}\left(\int_{Q}
|x-y|^{\gamma-n}\,dy\right)^{\frac{p}{q}}\omega(x)\,dx\\
&\quad\leq\int_{2^kQ\setminus2^{k-1}Q}
|Q|^\frac{p}{q}
\left[\frac{(2^{k-1}-1)l_{(Q)}}{2}\right]^{\frac{p(\gamma-n)}{q}}\omega(x)\,dx\\
&\quad\lesssim2^{\frac{kp(\gamma-n)}{q}}|Q|^{\frac{p\gamma}{qn}}\omega(2^kQ)\leq
[\omega]_{A_1(\mathbb{R}^n)}2^{k[\frac{p\gamma}{q}
+\frac{(q-p)n}{q}]-\frac{jp\gamma}{q}}
\left|2^jQ\right|^{\frac{p\gamma}{qn}}\omega(Q),
\end{align*}
which, together with both \eqref{2038}
and $q\in[p,\infty)$, further implies that
\begin{align*}
&\int_{2^jQ}\left(\int_{Q}
|x-y|^{\gamma-n}\,dy\right)^{\frac{p}{q}}\omega(x)\,dx\\
&\quad=\int_{2Q}\left(\int_{Q}
|x-y|^{\gamma-n}\,dy\right)^{\frac{p}{q}}\omega(x)\,dx\\
&\qquad+\sum_{k=2}^j\int_{2^kQ\setminus2^{k-1}Q}\left(\int_{Q}
|x-y|^{\gamma-n}\,dy\right)^{\frac{p}{q}}\omega(x)\,dx\\
&\quad\lesssim[\omega]_{A_1(\mathbb{R}^n)}\left\{2^{n(1-\frac{p}{q})}
\left|2Q\right|^{\frac{p\gamma}{qn}}\omega(Q)
+\sum_{k=2}^j2^{k[\frac{p\gamma}{q}+\frac{(q-p)n}{q}]-\frac{jp\gamma}{q}}
\left|2^jQ\right|^{\frac{p\gamma}{qn}}\omega(Q)\right\}\\
&\quad\lesssim[\omega]_{A_1(\mathbb{R}^n)}2^{jn(1-\frac{p}{q})}
\left|2^jQ\right|^{\frac{p\gamma}{qn}}\omega(Q).
\end{align*}
Combining this, \eqref{2037}, and \eqref{2038},
we further obtain \eqref{2039},
which completes the proof of \eqref{1653} and hence the above claim.

From the above claim, \eqref{1009}, \eqref{1101},
the definition of $\widetilde{\mathscr{A}}_{\alpha,j}^{\mathrm{max}}$,
and Lemma~\ref{2115}(i),
we infer that
\begin{align}\label{2245}
&\int_{\mathbb{R}^n}\left[\int_{\mathbb{R}^n}
\mathbf{1}_{E^{(2)}}(x,y)|x-y|^{\gamma-n}\,dy
\right]^{\frac{p}{q}}\omega(x)\,dx\\
&\quad\lesssim[\omega]_{A_1(\mathbb{R}^n)}\sum_{j=0}^\infty
\sum_{\alpha\in\{0,\,\frac{1}{3},\,\frac{2}{3}\}^n}
\sum_{Q\in\widetilde{\mathscr{A}}_{\alpha,j}^{\mathrm{max}}}
2^{jn(1-\frac{p}{q})}|2^jQ|^{\frac{p\gamma}{qn}}\omega(Q)\nonumber\\
&\quad\lesssim[\omega]_{A_1(\mathbb{R}^n)}^2\sum_{j=0}^\infty
2^{-jp[1-n(\frac{1}{p}-\frac{1}{q})-\eta]}
\sum_{\alpha\in\{0,\,\frac{1}{3},\,\frac{2}{3}\}^n}
\sum_{Q\in\widetilde{\mathscr{A}}_{\alpha,j}^{\mathrm{max}}}
\int_Q\left|\nabla f(z)\right|^p\omega(z)\,dz\nonumber\\
&\quad\lesssim[\omega]_{A_1(\mathbb{R}^n)}^2
\int_{\mathbb{R}^n}\left|\nabla f(z)\right|^p\omega(z)\,dz,\nonumber
\end{align}
which completes the proof of \eqref{e(2)} when $q\in[p,\infty)$.

When $q\in(0,p)$, by \eqref{2245} with $q:=p$ and
by an argument similar to that used in Case 2)
of the proof of Theorem~\ref{upperBound},
we find that \eqref{e(2)} also holds true in this case,
which, together with \eqref{2245},
then completes the proof of \eqref{e(2)} and hence Proposition~\ref{2010}.
\end{proof}

Now, we use both Proposition~\ref{2010} and a method of extrapolation
to show Theorem~\ref{gamma>0}.

\begin{proof}[Proof of Theorem~\ref{gamma>0}]
By Lemma~\ref{4.6}, Proposition~\ref{2010},
and Lemma~\ref{4.5}(ii), we conclude that,
for any $f\in C^1(\mathbb{R}^n)$ with
$|\nabla f|\in C_{\mathrm{c}}(\mathbb{R}^n)$,
\begin{align*}
&\sup_{\lambda\in(0,\infty)}\lambda
\left\|\left[\int_{\mathbb{R}^n}
\mathbf{1}_{E_{\lambda,\frac{\gamma}{q}}[f]}(\cdot,y)
\left|\cdot-y\right|^{\gamma-n}\,dy\right]^\frac{1}{q}\right\|_X\\
&\quad\leq\sup_{\|g\|_{(X^\frac{1}{p})'}\leq1}
\sup_{\lambda\in(0,\infty)}\lambda^p
\int_{\mathbb{R}^n}\left[\int_{\mathbb{R}^n}
\mathbf{1}_{E_{\lambda,\frac{\gamma}{q}}[f]}
(x,y)|x-y|^{\gamma-n}\,dy\right]^{\frac{p}{q}}
R_{(X^\frac{1}{p})'}g(x)\,dx\nonumber\\
&\quad\leq\sup_{\|g\|_{(X^\frac{1}{p})'}\leq1}
C_{([R_{(X^\frac{1}{p})'}g]_{A_1(\mathbb{R}^n)})}
\int_{\mathbb{R}^n}\left|\nabla f(x)\right|^p
R_{(X^\frac{1}{p})'}g(x)\,dx\\
&\quad\leq C_{(2\|\mathcal{M}\|_{(X^\frac{1}{p})'\to(X^\frac{1}{p})'})}
\sup_{\|g\|_{(X^\frac{1}{p})'}\leq1}
\int_{\mathbb{R}^n}\left|\nabla f(x)\right|^p
R_{(X^\frac{1}{p})'}g(x)\,dx
\lesssim\left\|\,\left|\nabla f\right|\,\right\|_{X},
\end{align*}
where the positive constants $C_{([R_{(X^\frac{1}{p})'}g]_{A_1(\mathbb{R}^n)})}$
and $C_{(2\|\mathcal{M}\|_{(X^\frac{1}{p})'\to(X^\frac{1}{p})'})}$ are
the same constants as in \eqref{1607} with $[\omega]_{A_1(\mathbb{R}^n)}$
replaced, respectively, by
$[R_{(X^\frac{1}{p})'}g]_{A_1(\mathbb{R}^n)}$
and $2\|\mathcal{M}\|_{(X^\frac{1}{p})'\to(X^\frac{1}{p})'}$.
This finishes the proof of Theorem~\ref{gamma>0}.
\end{proof}

\begin{remark}\label{2100}
\begin{enumerate}
\item[\textup{(i)}]
Theorem~\ref{gamma>0} when $\gamma=n$
coincides with \cite[Theorem~4.5]{dlyyz.arxiv}.

\item[\textup{(ii)}]
In Theorem~\ref{gamma>0}(iii), if $X:=L^p(\mathbb{R}^n)$ with $p\in[1,\infty)$
and if $q\in(0,\infty)$ satisfies $n(\frac{1}{p}-\frac{1}{q})<1$,
then both $X$ and $X^\frac{1}{p}=L^1(\mathbb{R}^n)$
are ball Banach function spaces and
the Hardy--Littlewood maximal operator is bounded
on $(X^\frac{1}{p})'=L^\infty(\mathbb{R}^n)$.
Thus, Theorem~\ref{gamma>0}(iii) in this case holds true, which when $p=q\in[1,\infty)$
is just \cite[(2.3) and (1.7)]{bsvy.arxiv} in the case $\gamma\in(0,\infty)$
and which when $p\neq q$ is new.

\item[\textup{(iii)}]
As was pointed out in \cite[Remark~3.6(iii)]{dlyyz.arxiv},
the assumption $n(\frac{1}{p}-\frac{1}{q})<1$
in Proposition~\ref{2010} seems to be \emph{sharp}.
To be more precise, let $p,q\in[1,\infty)$ satisfy
$n\max\{0,\,\frac{1}{p}-\frac{1}{q}\}<s<1$.
On the one hand, it follows from \cite[Theorem~1.3]{h2022} that,
if $f\in L_{{\mathrm{loc}}}^{\min\{p,\,q\}}(\mathbb{R}^n)$,
then $f\in F^s_{p,q}(\mathbb{R}^n)$ if and only if
the following inhomogeneous Gagliaro quasi-norm
\begin{align}\label{1916}
\mathrm{I}:=\|f\|_{L^p(\mathbb{R}^n)}+
\left\|\left[\int_{\mathbb{R}^n}
\frac{|f(\cdot)-f(y)|^q}{|\cdot-y|^{n+q}}
\,dy\right]^\frac{1}{q}\right\|_{L^p(\mathbb{R}^n)}<\infty,
\end{align}
where $F^s_{p,q}(\mathbb{R}^n)$ denotes the classical
Triebel--Lizorkin space (see \cite[Section~2.3]{t1983}
for its definition);
moreover, one has $\mathrm{I}\sim\|f\|_{F^s_{p,q}(\mathbb{R}^n)}$
with the positive equivalence constants independent of $f$.
On the other hand, from \cite[Theorem~1.5]{h2022},
we deduce that, under the assumption that $p,q\in[1,\infty)$,
for any $f\in L_{{\mathrm{loc}}}^{\min\{p,\,q\}}(\mathbb{R}^n)$,
$\mathrm{I}\sim\|f\|_{F^s_{p,q}(\mathbb{R}^n)}$
only if $n\max\{0,\frac{1}{p}-\frac{1}{q}\}\leq s<1$.
In this sense, when $s\in(0,1)$, $n\max\{0,\frac{1}{p}-\frac{1}{q}\}<s<1$
is sharp.
Corresponding to the case $s=1$,
an important observation is that,
for any given $p\in[1,\infty)$ and $q\in[1,p]$,
$$
\left\|\left[\int_{\mathbb{R}^n}
\frac{|f(\cdot)-f(y)|^q}{|\cdot-y|^{n+q}}
\,dy\right]^\frac{1}{q}\right\|_{L^p(\mathbb{R}^n)}<\infty
$$
unless $f$ is a constant,
which can be deduced from \cite[Theorem~4.3]{dlyyz.arxiv}.
Thus, the quasi-norm in \eqref{1916} can not recover
the Triebel--Lizorkin quasi-norm
$\|\cdot\|_{F^1_{p,q}(\mathbb{R}^n)}$
in the case $s=1$.
Indeed, as is shown in Proposition~\ref{2010} with $\omega\equiv1$
(see also \cite[Theorem~3.5]{dlyyz.arxiv}),
a suitable replacement of \eqref{1916} in the case $s=1$ is that,
for any $f\in\dot{W}^{1,p}(\mathbb{R}^n)$,
\begin{align*}
\left\|f\right\|_{F^1_{p,2}(\mathbb{R}^n)}
&\sim\|f\|_{L^p(\mathbb{R}^n)}
+\sup_{\lambda\in(0,\infty)}\lambda
\Bigg\{\int_{\mathbb{R}^n}\bigg[\int_{\mathbb{R}^n}
\mathbf{1}_{E_{\lambda,\frac{\gamma}{q}}[f]}(x,y)\\
&\quad\times|x-y|^{\gamma-n}\,dy
\bigg]^{\frac{p}{q}}\,dx\Bigg\}^\frac{1}{p}
\end{align*}
with any given $\gamma\in(0,\infty)$.
\end{enumerate}
\end{remark}

When $\gamma\in(-\infty,0)$,
we have the following upper estimate of \eqref{2213}
in the case both $n=1$
and $\gamma\in(-\infty,-1)$ without assuming that
the Hardy--Littlewood maximal operator is bounded on $X$.

\begin{theorem}\label{n=1}
Let $X$ be a ball Banach function space
on $\mathbb{R}$, $p\in[1,\infty)$, $q\in[1,p]$,
and $\gamma\in(-\infty,-1)$.
Assume that
$X^\frac{1}{p}$ is a ball Banach function space and that
the Hardy--Littlewood maximal operator $\mathcal{M}$ is bounded on
$(X^\frac{1}{p})'$ with its operator norm denoted by
$\|\mathcal{M}\|_{(X^\frac{1}{p})'\to(X^\frac{1}{p})'}$.
Then there exists a positive constant $C$,
depending only on $p$, $q$, $\gamma$, $n$,
and $\|\mathcal{M}\|_{(X^\frac{1}{p})'\to(X^\frac{1}{p})'}$, such that,
for any $f\in C^1(\mathbb{R})$ with $f'\in C_{\mathrm{c}}(\mathbb{R})$,
\begin{align*}
\sup_{\lambda\in(0,\infty)}\lambda
\left\|\left[\int_{\mathbb{R}}
\mathbf{1}_{E_{\lambda,\frac{\gamma}{q}}[f]}(\cdot,y)
\left|\cdot-y\right|^{\gamma-1}\,dy\right]^\frac{1}{q}\right\|_X
\leq C\left\|f'\right\|_{X},
\end{align*}
where, for any $\lambda\in(0,\infty)$,
$E_{\lambda,\frac{\gamma}{q}}[f]$ is the same as in \eqref{Elambda}
and where the positive constant $C$ is continuous with respect to
$\|\mathcal{M}\|_{(X^\frac{1}{p})'\to(X^\frac{1}{p})'}$ and increases
as $\|\mathcal{M}\|_{(X^\frac{1}{p})'\to(X^\frac{1}{p})'}$ increases.
\end{theorem}

To prove Theorem~\ref{n=1}, we need the following lemma.

\begin{lemma}\label{1558}
Let $X$ be a ball Banach function space on $\mathbb{R}$,
$p\in[1,\infty)$, $q\in(0,p]$,
and $\gamma\in(-\infty,-1)$.
Assume that
$X^\frac{1}{p}$ is a ball Banach function space and that
the Hardy--Littlewood maximal operator $\mathcal{M}$ is bounded on
$(X^\frac{1}{p})'$ with its operator norm denoted by
$\|\mathcal{M}\|_{(X^\frac{1}{p})'\to(X^\frac{1}{p})'}$.
Then there exists a positive constant $C$,
depending only on $p$, $q$, $\gamma$, $n$,
and $\|\mathcal{M}\|_{(X^\frac{1}{p})'\to(X^\frac{1}{p})'}$, such that,
for any $f\in C_{\mathrm{c}}(\mathbb{R})$,
\begin{align}\label{1633}
\left\|\left[\int_{\mathbb{R}}
\mathbf{1}_{E(f,\gamma)}(\cdot,y)
\left|\cdot-y\right|^{\gamma-1}\,dy\right]^\frac{1}{q}\right\|_X
\leq C\left\|\,\left|f\right|^\frac{1}{q}\right\|_{X},
\end{align}
where
\begin{align}\label{1552}
E(f,\gamma):=\left\{(x,y)\in\mathbb{R}\times\mathbb{R}:\
x\neq y,\ \left|\int_x^y|f(s)|\,ds\right|>|x-y|^{\gamma+1}\right\}
\end{align}
and $C$ is continuous with respect to
$\|\mathcal{M}\|_{(X^\frac{1}{p})'\to(X^\frac{1}{p})'}$ and increases
as $\|\mathcal{M}\|_{(X^\frac{1}{p})'\to(X^\frac{1}{p})'}$ increases.
\end{lemma}

\begin{proof}
To prove this lemma, we borrow some ideas from the
proof of \cite[Proposition~2.3(ii)]{bsvy.arxiv}.
Without loss of generality, we may assume that $f$ is both nonnegative
and not identically zero. We write
\begin{align*}
E(f,\gamma)&=\left\{(x,y)\in E(f,\gamma):\ x>y\right\}
\cup\left\{(x,y)\in E(f,\gamma):\ x<y\right\}\nonumber\\
&=:E_+(f,\gamma)\cup E_-(f,\gamma).
\end{align*}
Thus, to show \eqref{1633},
it suffices to prove that
\begin{align}\label{1634}
\left\|\left[\int_{\mathbb{R}}
\mathbf{1}_{E_+(f,\gamma)}(\cdot,y)
\left|\cdot-y\right|^{\gamma-1}\,dy\right]^\frac{1}{q}\right\|_X
\lesssim\left\|f^\frac{1}{q}\right\|_{X}
\end{align}
and
\begin{align}\label{1635}
\left\|\left[\int_{\mathbb{R}}
\mathbf{1}_{E_-(f,\gamma)}(\cdot,y)
\left|\cdot-y\right|^{\gamma-1}\,dy\right]^\frac{1}{q}\right\|_X
\lesssim\left\|f^\frac{1}{q}\right\|_{X}.
\end{align}

We first show \eqref{1634}. To this end,
by Lemma~\ref{4.6} and the assumptions of the present lemma,
we find that it suffices to prove that, for any $\omega\in A_1(\mathbb{R})$
with $[\omega]_{A_1(\mathbb{R})}\leq
2\|\mathcal{M}\|_{(X^\frac{1}{p})'\to(X^\frac{1}{p})'}$,
\begin{align}\label{1626}
\int_{\mathbb{R}}
\left[\int_{\mathbb{R}}\mathbf{1}_{E_+(f,\gamma)}(x,y)
(x-y)^{\gamma-1}\,dy\right]^{\frac{p}{q}}
\omega(x)\,dx\lesssim
\int_{\mathbb{R}}\left|f(x)\right|^{\frac{p}{q}}
\omega(x)\,dx.
\end{align}
Moreover, by an argument similar to that used in
Case 2) in the proof of Theorem~\ref{upperBound}
(briefly speaking, a duality argument of both the
Muckenhoupt weight and the
weighted Lebesgue space),
we conclude that, to prove \eqref{1626} for any $q\in(0,p]$,
we only need to show \eqref{1626}
in the case $q=p$.

Assume that $\mathrm{supp\,}(f)\subset[a,b]$
with $-\infty<a<b<\infty$.
By $\gamma\in(-\infty,-1)$ and a stopping time
argument similar to that used
in the proof of \cite[Proposition~2.3(ii)]{bsvy.arxiv},
we find that exists a $K\in\mathbb{N}$,
depending only on both $f$ and $\gamma$,
and a sequence $\{a_i\}_{i=1}^{K+1}\subset\mathbb{R}$ satisfying
\begin{enumerate}
\item[\textup{(i)}]
for any $i\in\mathbb{N}\cap[1,K]$,
$a_i<a_{i+1}$;
\item[\textup{(ii)}]
$a_1=a$ and $a_K<b\leq a_{K+1}$;
\item[\textup{(iii)}]
for any $i\in\mathbb{N}\cap[1,K]$,
\begin{align}\label{2042}
(a_{i+1}-a_i)^{-(\gamma+1)}\int_{a_i}^{a_{i+1}}f(s)\,ds=\frac{1}{2}.
\end{align}
\end{enumerate}
We write $I_0:=(-\infty,a_1]$, $I_{K+1}:=[a_{K+1},\infty)$, and
$I_i:=[a_i,a_{i+1}]$ for any $i\in\mathbb{N}\cap[1,K]$.
Moreover, the proof of \cite[Proposition~2.3(ii)]{bsvy.arxiv} also
shows that
\begin{align}\label{2140}
E_+(f,\gamma)=\bigcup_{i=1}^{K+1}\left\{E_+(f,\gamma)\cap
\left[\left(a_i,\infty\right)\times\left(-\infty,a_i\right)\right]\right\}
=:\bigcup_{i=1}^{K+1}\mathcal{E}_i
\end{align}
and, for any given $i\in\mathbb{N}\cap[1,K+1]$ and for any
$(x,y)\in E_+(f,\gamma)$ satisfying $y<a_i<x$,
$$
|x-y|\ge\min\left\{|I_i|,|I_{i-1}|\right\}.
$$
From this,
we deduce that, for any given $i\in\mathbb{N}\cap[1,K+1]$
and $\omega\in A_1(\mathbb{R})$,
\begin{align}\label{Z12}
&\int_{\mathbb{R}}
\left[\int_{\mathbb{R}}\mathbf{1}_{\mathcal{E}_i}(x,y)
(x-y)^{\gamma-1}\,dy\right]
\omega(x)\,dx\\
&\quad\leq\int_{a_i}^\infty
\left[\int_{-\infty}^{\min\{a_i,x-\min\{|I_i|,|I_{i-1}|\}\}}
(x-y)^{\gamma-1}\,dy\right]\omega(x)\,dx\nonumber\\
&\quad=-\frac{1}{\gamma}\int_{a_i}^\infty
\left[\max\{x-a_i,\min\{|I_i|,|I_{i-1}|\}\}
\right]^\gamma\omega(x)\,dx\nonumber\\
&\quad=-\frac{1}{\gamma}
\left\{\int_{a_i}^{a_i+\min\{|I_i|,|I_{i-1}|\}}
\left[\min\{|I_i|,|I_{i-1}|\}\right]^\gamma\omega(x)\,dx\right.\nonumber\\
&\qquad+\left.\int_{a_i+\min\{|I_i|,|I_{i-1}|\}}^\infty
(x-a_i)^\gamma\omega(x)\,dx\right\}\nonumber\\
&\quad=:-\frac{1}{\gamma}(\mathrm{Z}_1+\mathrm{Z}_2).\nonumber
\end{align}

To deal with $\mathrm{Z}_1$,
we consider the following two cases on both $|I_i|$ and $|I_{i-1}|$.

\emph{Case 1)} $|I_i|\leq|I_{i-1}|$. In this case,
by the definition of $\mathcal{M}$,
\eqref{2042}, and Lemma~\ref{ApProperty}(i),
we conclude that
\begin{align}\label{2101}
\mathrm{Z}_1&=|I_i|^\gamma
\omega\left(I_i\right)
\leq|I_i|^{\gamma+1}\inf_{t\in I_i}\mathcal{M}(\omega)(t)\\
&\lesssim[\omega]_{A_1(\mathbb{R})}
\inf_{t\in I_i}\omega(t)\int_{I_i}f(s)\,ds
\leq[\omega]_{A_1(\mathbb{R})}
\int_{I_i}f(s)\omega(s)\,ds.\nonumber
\end{align}

\emph{Case 2)} $|I_i|>|I_{i-1}|$. In this case,
from \eqref{2042}, the definition of $\mathcal{M}$,
and Lemma~\ref{ApProperty}(i),
we infer that
\begin{align*}
\mathrm{Z}_1
&=|I_{i-1}|^\gamma\omega\left((a_i,a_{i}+|I_{i-1}|)\right)
\leq|I_{i-1}|^{\gamma}
\omega\left((a_{i-1},2a_i-a_{i-1})\right)\\
&\leq2[\omega]_{A_1(\mathbb{R})}|I_{i-1}|^\gamma
\omega(I_{i-1})\leq2[\omega]_{A_1(\mathbb{R})}|I_{i-1}|^{\gamma+1}
\inf_{t\in I_{i-1}}\mathcal{M}(\omega)(t)\\
&\leq4[\omega]^2_{A_1(\mathbb{R})}
\inf_{t\in I_{i-1}}\omega(t)\int_{I_{i-1}}f(s)\,ds
\leq4[\omega]^2_{A_1(\mathbb{R})}\int_{I_{i-1}}
f(s)\omega(s)\,ds,
\end{align*}
which, combined with \eqref{2101}, further implies that
\begin{align}\label{Z1}
\mathrm{Z}_1\lesssim\max\left\{[\omega]_{A_1(\mathbb{R})},\,
[\omega]^2_{A_1(\mathbb{R})}\right\}\int_{I_{i-1}\cup I_i}
f(s)\omega(s)\,ds,
\end{align}
where the implicit positive constant
is independent of $i$, $\omega$, and $f$.

To deal with $\mathrm{Z}_2$, we also
consider the following two cases on both $|I_i|$ and $|I_{i-1}|$.

\emph{Case 1)} $|I_i|\leq|I_{i-1}|$. In this case,
by both (i) and (ii) of Lemma~\ref{ApProperty},
$\gamma\in(-\infty,-1)$,
and \eqref{2042},
we conclude that
\begin{align}\label{2102}
\mathrm{Z}_2
&=\int_{a_i+|I_i|}^\infty(x-a_i)^\gamma\omega(x)\,dx
=\int_{|I_i|}^\infty x^\gamma\omega(x+a_i)\,dx\\
&=\sum_{n=1}^\infty\int_{2^{n-1}|I_i|}^{2^n|I_i|}
x^\gamma\omega(x+a_i)\,dx
\leq\sum_{n=1}^\infty2^{(n-1)\gamma}|I_i|^\gamma
\int_{2^{n-1}|I_i|}^{2^n|I_i|}
\omega(x+a_i)\,dx\nonumber\\
&\leq2^{-\gamma}\sum_{n=1}^\infty
2^{n\gamma}|I_i|^\gamma
\omega\left((a_i,a_i+2^n|I_i|)\right)\nonumber\\
&\leq2^{-\gamma}[\omega]_{A_1(\mathbb{R})}\sum_{n=1}^\infty
2^{n(\gamma+1)}|I_i|^\gamma
\omega(I_i)\nonumber\\
&\leq2^{-\gamma}[\omega]_{A_1(\mathbb{R})}\sum_{n=1}^\infty
2^{n(\gamma+1)}|I_i|^{\gamma+1}
\inf_{t\in I_i}\mathcal{M}(\omega)(t)\nonumber\\
&\lesssim[\omega]_{A_1(\mathbb{R})}^2
\inf_{t\in I_i}\omega(t)\int_{I_i}f(s)\,ds
\leq[\omega]_{A_1(\mathbb{R})}^2
\int_{I_i}f(s)\omega(s)\,ds.\nonumber
\end{align}

\emph{Case 2)} $|I_i|>|I_{i-1}|$. In this case,
from both (i) and (ii) of Lemma~\ref{ApProperty},
$\gamma\in(-\infty,-1)$,
and \eqref{2042},
it follows that
\begin{align*}
\mathrm{Z}_2
&=\int_{a_i+|I_{i-1}|}^\infty(x-a_i)^\gamma\omega(x)\,dx
=\int_{|I_{i-1}|}^\infty x^\gamma\omega(x+a_i)\,dx\nonumber\\
&=\sum_{n=1}^\infty\int_{2^{n-1}|I_{i-1}|}^{2^n|I_{i-1}|}
x^\gamma\omega(x+a_i)\,dx
\leq\sum_{n=1}^\infty2^{(n-1)\gamma}|I_{i-1}|^\gamma
\int_{2^{n-1}|I_{i-1}|}^{2^n|I_{i-1}|}
\omega(x+a_i)\,dx\nonumber\\
&\lesssim\sum_{n=1}^\infty
2^{n\gamma}|I_{i-1}|^{\gamma}
\omega\left((a_{i-1},a_i+2^n|I_{i-1}|)\right)\nonumber\\
&\leq[\omega]_{A_1(\mathbb{R})}\sum_{n=1}^\infty
2^{n\gamma}(2^n+1)|I_{i-1}|^{\gamma}
\omega(I_{i-1})\nonumber\\
&\lesssim[\omega]_{A_1(\mathbb{R})}
\inf_{t\in I_{i-1}}\mathcal{M}(\omega)(t)\int_{I_{i-1}}f(s)\,ds
\nonumber\\
&\leq[\omega]_{A_1(\mathbb{R})}^2
\inf_{t\in I_{i-1}}\omega(t)\int_{I_{i-1}}f(s)\,ds
\leq[\omega]_{A_1(\mathbb{R})}^2
\int_{I_{i-1}}f(s)\omega(s)\,ds,
\end{align*}
which, together with \eqref{2102}, further implies that
\begin{align*}
\mathrm{Z}_2\lesssim[\omega]_{A_1(\mathbb{R})}^2
\int_{I_{i-1}\cup I_i}f(s)\omega(s)\,ds,
\end{align*}
where the implicit positive constant is independent of $i$, $\omega$, and $f$.
By this, \eqref{2140}, the Tonelli theorem,
\eqref{Z12}, and \eqref{Z1},
we conclude that,
for any $\omega\in A_1(\mathbb{R})$
with $[\omega]_{A_1(\mathbb{R})}
\leq2\|\mathcal{M}\|_{(X^\frac{1}{p})'\to(X^\frac{1}{p})'}$,
\begin{align*}
&\int_{\mathbb{R}}
\left[\int_{\mathbb{R}}\mathbf{1}_{E_+(f,\gamma)}(x,y)
(x-y)^{\gamma-1}\,dy\right]
\omega(x)\,dx\\
&\quad=\sum_{i=1}^{K+1}\int_{\mathbb{R}}\left[\int_{\mathbb{R}}
\mathbf{1}_{\mathcal{E}_i}(x,y)
(x-y)^{\gamma-1}\,dy\right]
\omega(x)\,dx\nonumber\\
&\quad\lesssim
\max\left\{[\omega]_{A_1(\mathbb{R})},\,
[\omega]_{A_1(\mathbb{R})}^2\right\}
\sum_{i=1}^{K+1}
\int_{I_{i-1}\cup I_i}f(s)\omega(s)\,ds\nonumber\\
&\quad\lesssim
\max\left\{\|\mathcal{M}\|_{(X^\frac{1}{p})'\to(X^\frac{1}{p})'},\,
\|\mathcal{M}\|_{(X^\frac{1}{p})'\to(X^\frac{1}{p})'}^2\right\}
\int_{\mathbb{R}}f(s)\omega(s)\,ds,\nonumber
\end{align*}
which completes the proof of \eqref{1626} in the case $q=p$
and hence \eqref{1626} for any $q\in(0,p]$.
This finishes the proof of \eqref{1634}.
From an argument similar
to that used in the proof of \eqref{1634},
we further deduce that \eqref{1635} holds true,
which, combined with \eqref{1634},
then completes the proof of Lemma~\ref{1558}.
\end{proof}

\begin{remark}
In Lemma~\ref{1558}, if $X:=L^p(\mathbb{R})$ with $p\in[1,\infty)$
and if $q\in(0,p]$,
then both $X$ and $X^\frac{1}{p}=L^1(\mathbb{R}^n)$
are ball Banach function spaces and
the Hardy--Littlewood maximal operator is bounded
on $(X^\frac{1}{p})'=L^\infty(\mathbb{R}^n)$.
Thus, Lemma~\ref{1558}
in this case holds true,
which when $q=p\in[1,\infty)$
coincides with \cite[Proposition~2.3(ii)]{bsvy.arxiv}
in the 1-dimensional Euclidean space
and which when $q\in(0,p)$ is new.
\end{remark}

Now, we turn to show Theorem~\ref{n=1}.

\begin{proof}[Proof of Theorem~\ref{n=1}]
Let $f\in C^1(\mathbb{R})$ with $f'\in C_{\mathrm{c}}(\mathbb{R})$.
By the H\"older inequality and
$q\in[1,\infty)$, we find that, for any $x,y\in\mathbb{R}$
with $x<y$,
\begin{align}\label{2122}
\left|f(y)-f(x)\right|
\leq\int_x^y\left|f'(s)\right|\,ds
\leq\left(y-x\right)^{1-\frac{1}{q}}
\left[\int_x^y\left|f'(s)\right|^q\,ds\right]^\frac{1}{q},
\end{align}
which further implies that, for any given $\lambda\in(0,\infty)$ and
for any $x,y\in E_{\lambda,\frac{\gamma}{q}}[f]$ with $x<y$,
$$
\int_x^y\left|f'(s)\right|^q\,ds\ge
\frac{|f(y)-f(x)|^q}{(y-x)^{q-1}}
>\lambda^q(y-x)^{\gamma+1}.
$$
From this and \eqref{1552}, we infer that,
for any $\lambda\in(0,\infty)$,
$$
E_{\lambda,\frac{\gamma}{q}}[f]\subset
E\left(\frac{|f'|^q}{\lambda^q},\gamma\right),
$$
where $E(g,\gamma)$ is the same as in \eqref{1552}
for any $g\in C_{\mathrm{c}}(\mathbb{R})$ and $\gamma\in\mathbb{R}$.
By this and Lemma~\ref{1558}
with $f:=\lambda^{-q}|f'|^q$,
we conclude that,
for any $\lambda\in(0,\infty)$,
\begin{align*}
&\lambda\left\|\left[\int_{\mathbb{R}}
\mathbf{1}_{E_{\lambda,\frac{\gamma}{q}}[f]}(\cdot,y)
\left|\cdot-y\right|^{\gamma-1}\,dy\right]^\frac{1}{q}\right\|_X\\
&\quad\leq\lambda\left\|\left[\int_{\mathbb{R}}
\mathbf{1}_{E(\frac{|f'|^q}{\lambda^q},\gamma)}(\cdot,y)
\left|\cdot-y\right|^{\gamma-1}\,dy\right]^\frac{1}{q}\right\|_X
\lesssim\left\|f'\right\|_{X}.
\end{align*}
Taking the supremum over all $\lambda\in(0,\infty)$,
we then complete the proof of Theorem~\ref{n=1}.
\end{proof}

\begin{remark}\label{1920}
\begin{enumerate}
\item[(i)]
We should point out that,
in Theorem~\ref{n=1},
the assumption $q\in[1,\infty)$ is only used
in the estimation of \eqref{2122}.
\item[(ii)]
In Theorem~\ref{n=1}, if $X:=L^p(\mathbb{R})$ with $p\in[1,\infty)$
and if $q\in[1,p]$,
then both $X$ and $X^\frac{1}{p}=L^1(\mathbb{R}^n)$
are ball Banach function spaces and
the Hardy--Littlewood maximal operator is bounded
on $(X^\frac{1}{p})'=L^\infty(\mathbb{R}^n)$.
Thus, Theorem~\ref{n=1}
in this case holds true,
which when $q=p\in[1,\infty)$
is a part of \cite[Theorem~2.2(b)]{bsvy.arxiv}
and which when $q\in[1,p)$ is new.
\item[(iii)]
When $\gamma\in(-\infty,0)$,
the assumption $\gamma\in(-\infty,-1)$ in Theorem~\ref{n=1}
is sharp in the case $X:=L^1(\mathbb{R})$.
Indeed, on the one hand, note that $L^1(\mathbb{R})$
satisfies all the assumptions on $X$ in Theorem~\ref{n=1}
and hence Theorem~\ref{n=1}
with $X:=L^1(\mathbb{R})$ holds true.
On the other hand, the proof of \cite[Proposition 6.1]{bsvy.arxiv}
gives several counterexamples such that
Theorem~\ref{n=1} with $X:=L^1(\mathbb{R})$ and any $\gamma\in[-1,0)$
fails.

\item[(iv)]
When $n\in\mathbb{N}\cap[2,\infty)$,
it is still unknown whether or not
Theorem~\ref{n=1} with $\mathbb{R}$ replaced by $\mathbb{R}^n$
still holds true.
\end{enumerate}
\end{remark}

At the end of this subsection,
we give a remark on \eqref{2213} in the case $\gamma=0$.

\begin{remark}
As was shown in \cite[Corollary~1.6]{bsvy.arxiv},
if $f\in L_{\mathrm{loc}}^1(\mathbb{R}^n)$
with $|\nabla f|\in L_{\mathrm{loc}}^1(\mathbb{R}^n)$
and if
$$
\iint_{\{(x,y)\in\mathbb{R}^n\times\mathbb{R}^n:\
\frac{|f(x)-f(y)|}{|x-y|}>\lambda\}}|x-y|^{-n}\,dx\,dy<\infty
$$
for any $\lambda\in(0,\infty)$,
then $f$ equals to a constant function almost everywhere.
Thus, even in the case $X:=L^p(\mathbb{R}^n)$ with $p\in[1,\infty)$
and $q=p$,
\eqref{2213} with $\gamma=0$ does not hold true unless $f$
equals to a constant function almost everywhere.
\end{remark}

\subsection{Lower Estimate and
Two Limiting Identities}
\label{ss3.2}

The  main target of this subsection is to establish the lower estimate
of \eqref{2213}. In addition, as generalizations of both
\eqref{2119} and \eqref{2020} from $L^p(\mathbb{R}^n)$
to the ball Banach function space $X$, we also obtain two
limiting identities with the limit $\lambda\to\infty$
when $\gamma\in(0,\infty)$ and the limit $\lambda\to0^+$
when $\gamma\in(-\infty,0)$. We begin with the following conclusion.

\begin{theorem}\label{LimFormulaInf}
Let $X$ be a ball quasi-Banach function space, $q\in(0,\infty)$,
and $\gamma\in\mathbb{R}\setminus\{0\}$.
For any $f\in C^1(\mathbb{R}^n)$
with $|\nabla f|\in C_{\mathrm{c}}(\mathbb{R}^n)$,
\begin{enumerate}
\item[\textup{(i)}]
if $\gamma\in(0,\infty)$, then
\begin{align*}
\liminf_{\lambda\to\infty}
\lambda
\left\|\left[\int_{\mathbb{R}^n}
\mathbf{1}_{E_{\lambda,\frac{\gamma}{q}}[f]}(\cdot,y)
\left|\cdot-y\right|^{\gamma-n}\,dy\right]^\frac{1}{q}\right\|_X
\ge\left[\frac{\kappa(q,n)}{\gamma}\right]^\frac{1}{q}
\left\|\,\left|\nabla f\right|\,\right\|_{X};
\end{align*}
\item[\textup{(ii)}]
if $\gamma\in(-\infty,0)$, then
\begin{align*}
\liminf_{\lambda\to0^+}
\lambda
\left\|\left[\int_{\mathbb{R}^n}
\mathbf{1}_{E_{\lambda,\frac{\gamma}{q}}[f]}(\cdot,y)
\left|\cdot-y\right|^{\gamma-n}\,dy\right]^\frac{1}{q}\right\|_X
\ge\left[-\frac{\kappa(q,n)}{\gamma}\right]^\frac{1}{q}
\left\|\,\left|\nabla f\right|\,\right\|_{X},
\end{align*}
\end{enumerate}
where $\kappa(q,n)$ and $E_{\lambda,\frac{\gamma}{q}}[f]$
with $\lambda\in(0,\infty)$ are the same as, respectively, in
\eqref{kappaqn} and \eqref{Elambda}.
\end{theorem}

To prove Theorem~\ref{LimFormulaInf},
we need the following  Fatou lemma
on ball quasi-Banach function spaces;
see, for instance, \cite[Lemma~2.4]{wyy.arxiv}.

\begin{lemma}\label{FatouX}
Let $X$ be a ball quasi-Banach function space and
$\{f_k\}_{k\in\mathbb{N}}\subset X$.
Then
\begin{align*}
\left\|\liminf_{k\to\infty}\left|f_k\right|\right\|_X
\leq\liminf_{k\to\infty}\left\|f_k\right\|_X.
\end{align*}
\end{lemma}

Now, we prove Theorem~\ref{LimFormulaInf}.

\begin{proof}[Proof of Theorem~\ref{LimFormulaInf}]
By an argument similar to that
used in the proof of \cite[Lemma~3.2]{bsvy.arxiv}
with the Fatou lemma on $L^1(\mathbb{R}^n)$
replaced by Lemma~\ref{FatouX},
we obtain the desired result of the present theorem;
we omit the details here.
This finishes the proof of Theorem~\ref{LimFormulaInf}.
\end{proof}

\begin{remark}
\begin{enumerate}
\item[(i)]
In Theorem~\ref{LimFormulaInf},
if $X:=L^p(\mathbb{R}^n)$ with $p\in[1,\infty)$
and if $q\in(0,\infty)$,
then Theorem~\ref{LimFormulaInf}
in this case holds true,
which when $q=p\in[1,\infty)$ is just \cite[Lemma~3.2]{bsvy.arxiv}
and which when $q\in(0,\infty)$ with $q\neq p$ is new.
\item[(ii)]
Theorem~\ref{LimFormulaInf} when $\gamma=n$ coincides with
\cite[(3.25)]{dlyyz.arxiv}.
\end{enumerate}
\end{remark}

Moreover, if $X$ is a ball Banach function space,
we then have the following upper bounds on $X$
in two limiting cases that $\lambda\to\infty$
when $\gamma\in(0,\infty)$ and $\lambda\to0^+$
when $\gamma\in(-\infty,0)$.

\begin{theorem}\label{LimFormulaSup}
Let $X$ be a ball Banach function space.
\begin{enumerate}
\item[\textup{(i)}]
If $\gamma,q\in(0,\infty)$,
then, for any $f\in C^1(\mathbb{R}^n)$
with $|\nabla f|\in C_{\mathrm{c}}(\mathbb{R}^n)$,
\begin{align}\label{2014}
\limsup_{\lambda\to\infty}
\lambda
\left\|\left[\int_{\mathbb{R}^n}
\mathbf{1}_{E_{\lambda,\frac{\gamma}{q}}[f]}(\cdot,y)
\left|\cdot-y\right|^{\gamma-n}\,dy\right]^\frac{1}{q}\right\|_X
\leq\left[\frac{\kappa(q,n)}{\gamma}\right]^\frac{1}{q}
\left\|\,\left|\nabla f\right|\,\right\|_{X},
\end{align}
where $\kappa(q,n)$ and $E_{\lambda,\frac{\gamma}{q}}[f]$
with $\lambda\in(0,\infty)$ are the same as, respectively, in
\eqref{kappaqn} and \eqref{Elambda}.
\item[\textup{(ii)}]
Let $p\in[1,\infty)$, $\gamma\in(-\infty,0)$,
and $q\in(0,\frac{n-\gamma}{n}p)$.
Assume that $X^\frac{1}{p}$ is a ball Banach function space
and that the Hardy--Littlewood maximal operator $\mathcal{M}$
is bounded on $(X^\frac{1}{p})'$.
Then, for any $f\in C^1_{\mathrm{c}}(\mathbb{R}^n)$,
\begin{align}\label{1144}
\limsup_{\lambda\to0^+}
\lambda
\left\|\left[\int_{\mathbb{R}^n}
\mathbf{1}_{E_{\lambda,\frac{\gamma}{q}}[f]}(\cdot,y)
\left|\cdot-y\right|^{\gamma-n}\,dy\right]^\frac{1}{q}\right\|_X
\leq\left[-\frac{\kappa(q,n)}{\gamma}\right]^\frac{1}{q}
\left\|\,\left|\nabla f\right|\,\right\|_{X}.
\end{align}
\end{enumerate}
\end{theorem}

To prove Theorem~\ref{LimFormulaSup},
we need
the following well-known inequality
(see, for instance, \cite[p.\,699]{b2002}).

\begin{lemma}\label{1111}
Let $q\in(0,\infty)$.
Then, for any $\theta\in(0,1)$,
there exists a positive constant $C_\theta$ such that,
for any $a,b\in(0,\infty)$,
$$
(a+b)^q\leq(1+\theta)a^q+C_\theta b^q.
$$
\end{lemma}

\begin{proof}[Proof of Theorem~\ref{LimFormulaSup}]
We first show (i). Let $f\in C^1(\mathbb{R}^n)$
with $|\nabla f|\in C_{\mathrm{c}}(\mathbb{R}^n)$.
Since $\lambda\to\infty$, without loss of generality,
we may first assume that $\lambda\in[L,\infty)$ and then
let $\lambda\to\infty$, where
$L:=\|\,|\nabla f|\,\|_{L^\infty(\mathbb{R}^n)}$.
Let $B_0$ be a ball centered at $\mathbf{0}$ such that
$\mathrm{supp\,}(|\nabla f|)\subset B_0$.
Let
$\widetilde{B_0}:=B(\mathbf{0},1+r_{B_0})$ with
radius $r_{B_0}$, the same as $B_0$.
From the same argument as that used in
the proof of \cite[Lemma~3.3(i)]{bsvy.arxiv},
we deduce that
\begin{align}\label{1921}
E_{\lambda,\frac{\gamma}{q}}[f]\subset\widetilde{B_0}\times\mathbb{R}^n.
\end{align}
On the other hand,
by the same argument as that used in the proof of \cite[(3.7)]{bsvy.arxiv},
we find that, if $(x,y)\in E_{\lambda,\frac{\gamma}{q}}[f]$ with
writing $y:=x+r\omega$, where
$r\in(0,\infty)$ and $\omega\in\mathbb{S}^{n-1}$,
then
\begin{align*}
r\leq\left\{\lambda^{-1}\left[|\nabla f(x)\cdot\omega|
+\rho\left(\left(\frac{L}{\lambda}\right)^\frac{q}{\gamma}
\right)\right]\right\}^\frac{q}{\gamma}=:r(x,\omega,\lambda),
\end{align*}
where, for any $t\in(0,\infty)$,
\begin{align*}
\rho(t):=\sup_{x\in\mathbb{R}^n}\sup_{|h|\leq t}
\left|\nabla f(x+h)-\nabla f(x)\right|.
\end{align*}
From this, \eqref{1921}, the polar coordinate, Lemma~\ref{1111},
and Definition~\ref{1659}(v), we deduce that,
for any given $\theta_1,\theta_2\in(0,1)$,
there exist $C_{\theta_1},C_{\theta_2}\in(0,\infty)$
such that
\begin{align*}
&\lambda\left\|\left[\int_{\mathbb{R}^n}
\mathbf{1}_{E_{\lambda,\frac{\gamma}{q}}[f]}(\cdot,y)
\left|\cdot-y\right|^{\gamma-n}\,dy\right]^\frac{1}{q}\right\|_X\\
&\quad\leq\lambda\left\|\left[\int_{\mathbb{S}^{n-1}}
\int_0^{r(\cdot,\omega,\lambda)}r^{\gamma-1}\,dr\,d\omega\right]^\frac{1}{q}
\mathbf{1}_{\widetilde{B_0}}(\cdot)\right\|_X\\
&\quad\leq\left\|\left\{\int_{\mathbb{S}^{n-1}}
(1+\theta_1)|\nabla f(\cdot)\cdot\omega|^q+C_{\theta_1}\left[\rho
\left(\left\{\frac{L}{\lambda}\right\}^\frac{q}{\gamma}
\right)\right]^q
\,d\omega\right\}^\frac{1}{q}
\mathbf{1}_{\widetilde{B_0}}(\cdot)\right\|_X\\
&\quad\leq(1+\theta_1)^\frac{1}{q}(1+\theta_2)
\left\|\left[\int_{\mathbb{S}^{n-1}}
|\nabla f(\cdot)\cdot\omega|^q
\,d\omega\right]^\frac{1}{q}
\mathbf{1}_{\widetilde{B_0}}(\cdot)\right\|_X\\
&\qquad+(C_{\theta_1})^\frac{1}{q}C_{\theta_2}
\rho\left(\left[\frac{L}{\lambda}\right]^\frac{q}{\gamma}
\right)\left\|
\mathbf{1}_{\widetilde{B_0}}(\cdot)\right\|_X\\
&\quad\leq(1+\theta_1)^\frac{1}{q}(1+\theta_2)
\left[\frac{\kappa(q,n)}{\gamma}\right]^\frac{1}{q}
\left\|\,\left|\nabla f\right|\,\right\|_{X}\\
&\qquad+(C_{\theta_1})^\frac{1}{q}C_{\theta_2}
\rho\left(\left[\frac{L}{\lambda}\right]^\frac{q}{\gamma}
\right)\left\|
\mathbf{1}_{\widetilde{B_0}}(\cdot)\right\|_X,
\end{align*}
which, combined with both Definition~\ref{1659}(iv)
and the obvious estimate that $\rho(t)\to0$
as $t\to\infty$ [because $|\nabla f|\in C_{\mathrm{c}}(\mathbb{R}^n)$],
via first letting $\lambda\to\infty$
and then letting $\theta_1,\theta_2\to0$,
further implies that
\eqref{2014} holds true.
This finishes the proof of (i).

Next, we show (ii). Let
$f\in C^1_{\mathrm{c}}(\mathbb{R}^n)$
and $B$ be a ball centered at $\mathbf{0}$ such that
$\mathrm{supp\,}(f)\subset B$. By Lemma~\ref{1111} and
Definition~\ref{1659}(v), we find that,
for any given $\eta_1,\eta_2\in(0,1)$,
there exist $C_{\eta_1},C_{\eta_2}\in(0,\infty)$
such that
\begin{align}\label{I123}
&\limsup_{\lambda\to0^+}\lambda
\left\|\left[\int_{\mathbb{R}^n}
\mathbf{1}_{E_{\lambda,\frac{\gamma}{q}}[f]}(\cdot,y)
\left|\cdot-y\right|^{\gamma-n}\,dy\right]^\frac{1}{q}\right\|_X\\
&\quad\leq\limsup_{\lambda\to0^+}\lambda
\left\|(1+\eta_1)\left[\int_{\mathbb{R}^n}
\mathbf{1}_{E_{\lambda,\frac{\gamma}{q}}[f]\cap(2B\times\mathbb{R}^n)}(\cdot,y)
\left|\cdot-y\right|^{\gamma-n}\,dy\right]^\frac{1}{q}\right.\nonumber\\
&\qquad+C_{\eta_1}(1+\eta_2)\left[\int_{\mathbb{R}^n}
\mathbf{1}_{E_{\lambda,\frac{\gamma}{q}}[f]
\cap[(2B)^\complement\times B]}(\cdot,y)
\left|\cdot-y\right|^{\gamma-n}\,dy\right]^\frac{1}{q}\nonumber\\
&\qquad+\left.C_{\eta_1}C_{\eta_2}\left[\int_{\mathbb{R}^n}
\mathbf{1}_{E_{\lambda,\frac{\gamma}{q}}[f]\cap
[(2B)^\complement\times B^\complement]}(\cdot,y)
\left|\cdot-y\right|^{\gamma-n}\,dy\right]^\frac{1}{q}
\right\|_X\nonumber\\
&\quad\leq(1+\eta_1)\limsup_{\lambda\to0^+}\lambda
\left\|\left[\int_{\mathbb{R}^n}
\mathbf{1}_{E_{\lambda,\frac{\gamma}{q}}[f]
\cap(2B\times\mathbb{R}^n)}(\cdot,y)
\left|\cdot-y\right|^{\gamma-n}\,dy\right]^\frac{1}{q}\right\|_X\nonumber\\
&\qquad+C_{\eta_1}(1+\eta_2)\limsup_{\lambda\to0^+}
\lambda\nonumber\\
&\qquad\times\left\|\left[\int_{\mathbb{R}^n}
\mathbf{1}_{E_{\lambda,\frac{\gamma}{q}}[f]
\cap[(2B)^\complement\times B]}(\cdot,y)
\left|\cdot-y\right|^{\gamma-n}\,dy\right]^\frac{1}{q}\right\|_X\nonumber\\
&\qquad+C_{\eta_1}C_{\eta_2}\limsup_{\lambda\to0^+}
\lambda\left\|\left[\int_{\mathbb{R}^n}
\mathbf{1}_{E_{\lambda,\frac{\gamma}{q}}[f]\cap
[(2B)^\complement\times B^\complement]}(\cdot,y)
\left|\cdot-y\right|^{\gamma-n}\,dy\right]^\frac{1}{q}
\right\|_X\nonumber\\
&\quad=:(1+\eta_1)\mathrm{I}_1
+C_{\eta_1}(1+\eta_2)\mathrm{I}_2
+C_{\eta_1}C_{\eta_2}\mathrm{I}_3.\nonumber
\end{align}
To deal with $\mathrm{I}_1$,
for any given $\varepsilon\in(0,\infty)$,
let $\delta(\varepsilon)\in(0,\infty)$
satisfy that $\rho(r)\leq\varepsilon$
for any $r\in(0,\delta(\varepsilon))$.
From the same argument as that
used in the proof of \cite[Lemma~3.3(ii)]{bsvy.arxiv},
it follows that, if $(x,y)\in E_{\lambda,\frac{\gamma}{q}}[f]$ with
writing $y:=x+r\omega$, where
$r\in(0,\infty)$ and $\omega\in\mathbb{S}^{n-1}$,
then
$$
r\ge\min\left\{\delta(\varepsilon),\,
\left[\frac{\lambda}{|\nabla f(x)\cdot
\omega|+\varepsilon}\right]^\frac{-q}{\gamma}\right\}.
$$
By this,  Lemma~\ref{1111},
and Definition~\ref{1659}(v),
we conclude that, for any given $\theta,\theta_1,\theta_2\in(0,1)$,
there exist positive constants $C_\theta$,
$C_{\theta_1}$, and $C_{\theta_2}$ such that
\begin{align*}
\mathrm{I}_1
&\leq\limsup_{\lambda\to0^+}\lambda
\left\|\left[\int_{\mathbb{S}^{n-1}}\int_{\min\{\delta(\varepsilon),\,
(\frac{\lambda}{|\nabla f\cdot\omega|+\varepsilon})^\frac{-q}{\gamma}\}}^\infty
r^{\gamma-1}\,dr\,d\omega\right]^{\frac{1}{q}}
\mathbf{1}_{2B}\right\|_{X}\nonumber\\
&=\left(-\frac{1}{\gamma}\right)^{\frac{1}{q}}
\limsup_{\lambda\to0^+}\lambda
\left\|\left[\int_{\mathbb{S}^{n-1}}
\max\left\{[\delta(\varepsilon)]^\gamma,\,
\left(\frac{\lambda}{|\nabla f\cdot\omega|+\varepsilon}\right)^{-q}\right\}
\,d\omega\right]^{\frac{1}{q}}\mathbf{1}_{2B}\right\|_{X}\nonumber\\
&\leq\left(-\frac{1}{\gamma}\right)^{
\frac{1}{q}}C_\theta\limsup_{\lambda\to0^+}\lambda
\left\|\left\{\int_{\mathbb{S}^{n-1}}[\delta(\varepsilon)]^\gamma
\,d\omega\right\}^{\frac{1}{q}}\mathbf{1}_{2B}\right\|_{X}\nonumber\\
&\quad+\left(-\frac{1}{\gamma}\right)^{
\frac{1}{q}}(1+\theta)\limsup_{\lambda\to0^+}\lambda
\left\|\left[\int_{\mathbb{S}^{n-1}}
\left(\frac{|\nabla f\cdot\omega|+\varepsilon}{\lambda}\right)^{q}
\,d\omega\right]^{\frac{1}{q}}\mathbf{1}_{2B}\right\|_{X}\nonumber\\
&=\left(-\frac{1}{\gamma}\right)^{\frac{1}{q}}(1+\theta)
\left\|\left[\int_{\mathbb{S}^{n-1}}
\left(|\nabla f\cdot\omega|+\varepsilon\right)^{q}
\,d\omega\right]^{\frac{1}{q}}\mathbf{1}_{2B}\right\|_{X}\nonumber\\
&\leq\left(-\frac{1}{\gamma}\right)^{\frac{1}{q}}(1+\theta)
\left\{(1+\theta_1)^{\frac{1}{q}}(1+\theta_2)
\left\|\left[\int_{\mathbb{S}^{n-1}}
|\nabla f\cdot\omega|^{q}
\,d\omega\right]^{\frac{1}{q}}\mathbf{1}_{2B}\right\|_{X}\right.\nonumber\\
&\quad\left.
+C_{\theta_1}^{\frac{1}{q}}C_{\theta_2}\left\|\left[\int_{\mathbb{S}^{n-1}}
\varepsilon^{q}
\,d\omega\right]^{\frac{1}{q}}\mathbf{1}_{2B}\right\|_{X}\right\},
\end{align*}
which, together with the arbitrariness of $\varepsilon\in(0,\infty)$,
further implies that
\begin{align*}
\mathrm{I_1}&\leq
\left(-\frac{1}{\gamma}\right)^{\frac{1}{q}}(1+\theta)
(1+\theta_1)^{\frac{1}{q}}(1+\theta_2)\left\|\left[\int_{\mathbb{S}^{n-1}}
|\nabla f\cdot\omega|^{q}
\,d\omega\right]^{\frac{1}{q}}\mathbf{1}_{2B}\right\|_{X}\\
&=\left(-\frac{1}{\gamma}\right)^{\frac{1}{q}}(1+\theta)
(1+\theta_1)^{\frac{1}{q}}(1+\theta_2)\left\|\left[\int_{\mathbb{S}^{n-1}}
\left|\frac{\nabla f}{|\nabla f|}\cdot\omega\right|^q
\,d\omega\right]^{\frac{1}{q}}
\left|\nabla f\right|\mathbf{1}_{2B}\right\|_{X}\\
&\leq\left(-\frac{1}{\gamma}\right)^{\frac{1}{q}}(1+\theta)
(1+\theta_1)^{\frac{1}{q}}(1+\theta_2)\left[\kappa(q,n)\right]^{\frac{1}{q}}
\left\|\,\left|\nabla f\right|\,\right\|_{X}.
\end{align*}
Letting $\theta,\theta_1,\theta_2\to0$, we then further obtain
\begin{align}\label{I1}
\mathrm{I}_1\leq
\left[-\frac{\kappa(q,n)}{\gamma}\right]^{\frac{1}{q}}
\left\|\,\left|\nabla f\right|\,\right\|_{X}.
\end{align}
To deal with $\mathrm{I}_2$,
we first claim that, for any given ball
$B_r:=B(\mathbf{0},r)$ with $r\in(0,\infty)$,
\begin{align}\label{yxst}
\left\||\cdot|^{\frac{\gamma-n}{q}}\mathbf{1}_{
(B_r)^\complement}(\cdot)\right\|_X<\infty.
\end{align}
Indeed, by Lemmas~\ref{4.6}, \ref{ApProperty}(ii),
and~\ref{4.5}(ii),
$q\in(0,\frac{n-\gamma}{n}p)$,
and Definition~\ref{1659}(iv),
we find that
\begin{align}\label{1511}
&\left\||\cdot|^{\frac{\gamma-n}{q}}
\mathbf{1}_{(B_r)^\complement}\right\|_X\\
&\quad\leq\sup_{\|g\|_{(X^\frac{1}{p})'}\leq1}
\left[\int_{(B_r)^\complement}|x|^{\frac{p(\gamma-n)}{q}}
R_{(X^\frac{1}{p})'}g(x)\,dx\right]^\frac{1}{p}\nonumber\\
&\quad=\sup_{\|g\|_{(X^\frac{1}{p})'}\leq1}
\left[\sum_{j=1}^\infty
\int_{2^jB_r\setminus2^{j-1}B_r}
|x|^{\frac{p(\gamma-n)}{q}}
R_{(X^\frac{1}{p})'}g(x)\,dx\right]^\frac{1}{p}\nonumber\\
&\quad\leq\sup_{\|g\|_{(X^\frac{1}{p})'}\leq1}
\left[\sum_{j=1}^\infty(2^{j-1}r)^{\frac{p(\gamma-n)}{q}}
\int_{2^jB_r}
R_{(X^\frac{1}{p})'}g(x)\,dx\right]^\frac{1}{p}\nonumber\\
&\quad\leq\sup_{\|g\|_{(X^\frac{1}{p})'}\leq1}
\left[\sum_{j=1}^\infty(2^{j-1}r)^{\frac{p(\gamma-n)}{q}}
C_{\left[[R_{(X^\frac{1}{p})'}g\right]_{A_1(\mathbb{R}^n)}}2^{jn}
\int_{B_r}
R_{(X^\frac{1}{p})'}g(x)\,dx\right]^\frac{1}{p}\nonumber\\
&\quad=2^{\frac{n-\gamma}{q}+\frac{1}{p}}
\left\|\mathcal{M}\right\|_{(X^\frac{1}{p})'\to(X^\frac{1}{p})'}^{\frac{1}{p}}
r^{\frac{\gamma-n}{q}}
\left\{\sum_{j=1}^\infty2^{j[\frac{p
(\gamma-n)}{q}+n]}\right\}^{\frac{1}{p}}\nonumber\\
&\qquad\times\sup_{\|g\|_{(X^\frac{1}{p})'}\leq1}\left[\int_{B_r}
R_{(X^\frac{1}{p})'}g(x)\,dx\right]^{\frac{1}{p}}\nonumber\\
&\quad\lesssim\left\|\mathcal{M}\right\|_{(X^\frac{1}{p})'
\to(X^\frac{1}{p})'}^{\frac{1}{p}}
r^{\frac{\gamma-n}{q}}
\left\|\mathbf{1}_{B_r}\right\|_X<\infty,\nonumber
\end{align}
where the implicit positive constants
depend only on $p$, $q$, $\gamma$, and $n$,
which completes the proof of the above claim.
Notice that, for any given $x\in(2B)^\complement$
and for any $y\in B$, we have $|x-y|\ge|x|-|y|\ge|x|/2$.
From this and \eqref{yxst} with $B_r$
replaced by $2B$,
we deduce that
\begin{align}\label{I2}
\mathrm{I_2}
&=\limsup_{\lambda\to0^+}\lambda^q
\left\|\left[\int_{B}
\left|\cdot-y\right|^{\gamma-n}\,dy\right]^{\frac{1}{q}}
\mathbf{1}_{(2B)^\complement}(\cdot)\right\|_{X}\\
&\leq2^{\frac{n-\gamma}{q}}\limsup_{\lambda\to0^+}
\lambda^q|B|^{\frac{1}{q}}
\left\||\cdot|^{\frac{\gamma-n}{q}}
\mathbf{1}_{(2B)^\complement}(\cdot)\right\|_{X}=0.\nonumber
\end{align}
To deal with $\mathrm{I}_3$,
notice that both $(x,y)\in(2B)^\complement
\times B^\complement$ and $\mathrm{supp\,}(f)\subset B$
imply that $f(x)=0=f(y)$. Thus,
$$
E_{\lambda,\frac{\gamma}{q}}[f]\cap\left[(2B)^\complement
\times B^\complement\right]=\emptyset,
$$
which further implies that
$\mathrm{I}_3=0$.
Combining this, \eqref{I123}, \eqref{I1}, and \eqref{I2}
and letting $\eta_1,\eta_2\to0$, we conclude that
$$
\limsup_{\lambda\to0^+}
\lambda
\left\|\left[\int_{\mathbb{R}^n}
\mathbf{1}_{E_{\lambda,\frac{\gamma}{q}}[f]}(\cdot,y)
\left|\cdot-y\right|^{\gamma-n}\,dy\right]^\frac{1}{q}\right\|_X
\leq\left[-\frac{\kappa(q,n)}{\gamma}\right]^\frac{1}{q}
\left\|\,\left|\nabla f\right|\,\right\|_{X}.
$$
This finishes the proof of (ii) and hence Theorem~\ref{LimFormulaSup}.
\end{proof}

\begin{remark}\label{1546}
When $n\in\mathbb{N}\cap[2,\infty)$,
from the proof of Theorem~\ref{LimFormulaSup}(ii),
we deduce that \eqref{1144} also holds true for any
$f\in C^1(\mathbb{R}^n)$ with $|\nabla f|\in C_{\mathrm{c}}(\mathbb{R}^n)$.
\end{remark}

When $n=1$,
we can also weaken the hypothesis
$f\in C^1_{\mathrm{c}}(\mathbb{R})$ in Theorem~\ref{LimFormulaSup}(ii)
to $f\in C^1(\mathbb{R})$ with $f'\in C_{\mathrm{c}}(\mathbb{R})$
in the case that $p\in[1,\infty)$, $\gamma\in(-\infty,0)$,
and $q\in(0,-p\gamma)$ as follows.

\begin{theorem}\label{n1}
Let $X$ be a ball Banach function space on $\mathbb{R}$,
$p\in[1,\infty)$, $\gamma\in(-\infty,0)$,
and $q\in(0,-p\gamma)$.
Assume that
\begin{enumerate}
\item[\textup{(i)}]
$X^\frac{1}{p}$ is a ball Banach function space;
\item[\textup{(ii)}]
the Hardy--Littlewood maximal operator $\mathcal{M}$
is bounded on $(X^\frac{1}{p})'$.
\end{enumerate}
Then, for any $f\in C^1(\mathbb{R})$
with $f'\in C_{\mathrm{c}}(\mathbb{R})$,
\begin{align*}
\lim_{\lambda\to0^+}
\lambda
\left\|\left[\int_{\mathbb{R}}
\mathbf{1}_{E_{\lambda,\frac{\gamma}{q}}[f]}(\cdot,y)
\left|\cdot-y\right|^{\gamma-1}\,dy\right]^\frac{1}{q}\right\|_X
\leq\left[-\frac{\kappa(q,1)}{\gamma}\right]^\frac{1}{q}
\left\|f'\right\|_{X},
\end{align*}
where $\kappa(q,1)$ is the same as in
\eqref{kappaqn} with $n$ replaced by $1$
and where $E_{\lambda,\frac{\gamma}{q}}[f]$
with $\lambda\in(0,\infty)$ is the same as in \eqref{Elambda}.
\end{theorem}

\begin{proof}
Let $\beta\in(0,\infty)$ be such that $\mathrm{supp\,}(f')\subset(-\beta,\beta)$.
By Lemma~\ref{1111} and
Definition~\ref{1659}(v), we find that,
for any given $\theta\in(0,1)$,
there exists a positive constant $C_{\theta}$
such that
\begin{align}\label{I12}
&\limsup_{\lambda\to0^+}
\lambda
\left\|\left[\int_{\mathbb{R}}
\mathbf{1}_{E_{\lambda,\frac{\gamma}{q}}[f]}(\cdot,y)
\left|\cdot-y\right|^{\gamma-1}\,dy\right]^\frac{1}{q}\right\|_X\\
&\quad\leq(1+\theta)\limsup_{\lambda\to0^+}\lambda
\left\|\left[\int_{\mathbb{R}}
\mathbf{1}_{E_{\lambda,\frac{\gamma}{q}}[f]\cap
[(-2\beta,2\beta)\times\mathbb{R}]}(\cdot,y)
\left|\cdot-y\right|^{\gamma-1}\,dy\right]^\frac{1}{q}\right\|_X\nonumber\\
&\qquad+C_{\theta}\limsup_{\lambda\to0^+}
\lambda\left\|\left[\int_{\mathbb{R}}
\mathbf{1}_{E_{\lambda,\frac{\gamma}{q}}[f]\cap
[(-2\beta,2\beta)^\complement\times\mathbb{R}]}(\cdot,y)
\left|\cdot-y\right|^{\gamma-1}\,dy\right]^\frac{1}{q}
\right\|_X\nonumber\\
&\quad=:(1+\theta)\mathrm{I}_1
+C_{\theta}\mathrm{I}_2.\nonumber
\end{align}
On the one hand, an argument similar to that
used in the estimation of \eqref{I1} implies that
\begin{align}\label{1814}
\mathrm{I}_1\leq
\left[-\frac{\kappa(q,1)}{\gamma}\right]^{\frac{1}{q}}
\left\|f'\right\|_{X}.
\end{align}
On the other hand, since $f$ is a constant on $(\beta,\infty)$
and another constant on $(-\infty,-\beta)$, it follows that,
if $(x,y)\in E_{\lambda,\frac{\gamma}{q}}[f]$ and $x\in[2\beta,\infty)$,
then $y\in(-\infty,\beta)$ and,
if $(x,y)\in E_{\lambda,\frac{\gamma}{q}}[f]$ and $x\in(-\infty,-2\beta]$,
then $y\in(-\beta,\infty)$. From this, Lemma~\ref{4.6},
an argument similar to that used in the estimation of \eqref{1906},
and $q\in(0,-p\gamma)$,
we deduce that
\begin{align}\label{xzl}
\mathrm{I}_2^p
&\leq\limsup_{\lambda\to0^+}
\lambda^p\sup_{\|g\|_{(X^\frac{1}{p})'}\leq1}\left\{\int_{-\infty}^{-2\beta}
\left[\int_{-\beta}^\infty
\left(y-x\right)^{\gamma-1}\,dy\right]^\frac{p}{q}
R_{(X^\frac{1}{p})'}g(x)\,dx\right.\\
&\quad\left.+\int_{2\beta}^\infty
\left[\int_{-\infty}^\beta
\left(x-y\right)^{\gamma-1}\,dy\right]^\frac{p}{q}
R_{(X^\frac{1}{p})'}g(x)\,dx\right\}\nonumber\\
&=\left(-\frac{1}{\gamma}\right)^{\frac{p}{q}}\limsup_{\lambda\to0^+}
\lambda^p\sup_{\|g\|_{(X^\frac{1}{p})'}\leq1}\left[\int_{-\infty}^{-2\beta}
\left(-x-\beta\right)^\gamma
R_{(X^\frac{1}{p})'}g(x)\,dx\right.\nonumber\\
&\quad\left.+\int_{2\beta}^\infty
\left(x-\beta\right)^\gamma
R_{(X^\frac{1}{p})'}g(x)\,dx\right]\nonumber\\
&\leq2^{3-\frac{p\gamma}{q}}
\left\|\mathcal{M}\right\|_{(X^\frac{1}{p})'\to(X^\frac{1}{p})'}
\left(-\frac{1}{\gamma}\right)^{\frac{p}{q}}\nonumber\\
&\quad\times\limsup_{\lambda\to0^+}
\lambda^p\sup_{\|g\|_{(X^\frac{1}{p})'}\leq1}
\sum_{j=1}^\infty2^{j(\frac{p\gamma}{q}+1)}\beta^\gamma
\int_{-\beta}^\beta R_{(X^\frac{1}{p})'}g(x)\,dx\nonumber\\
&\lesssim\left\|\mathcal{M}
\right\|_{(X^\frac{1}{p})'\to(X^\frac{1}{p})'}\limsup_{\lambda\to0^+}
\lambda^p\beta^\gamma
\left\|\mathbf{1}_{(-\beta,\beta)}\right\|_{X}^p=0,\nonumber
\end{align}
where the implicit positive constant depends only
on $p$, $q$, and $\gamma$,
which, together with both \eqref{I12} and \eqref{1814},
further implies that
\begin{align*}
&\limsup_{\lambda\to0^+}
\lambda
\left\|\left[\int_{\mathbb{R}}
\mathbf{1}_{E_{\lambda,\frac{\gamma}{q}}[f]}(\cdot,y)
\left|\cdot-y\right|^{\gamma-1}\,dy\right]^\frac{1}{q}\right\|_X\nonumber\\
&\quad\leq(1+\theta)\left[-\frac{\kappa(q,1)}{\gamma}\right]^{\frac{1}{q}}
\left\|f'\right\|_{X}.
\end{align*}
Letting $\theta\to0$, we then obtain the desired result.
This finishes the proof of Theorem~\ref{n1}.
\end{proof}

Using Remark~\ref{1546} and Theorems~\ref{LimFormulaInf},
\ref{LimFormulaSup}, and~\ref{n1}, we
immediately obtain the following conclusion;
we omit the details here.

\begin{corollary}\label{LimFormulaEq}
Let $X$ be a ball Banach function space.
\begin{enumerate}
\item[\textup{(i)}]
If $\gamma,q\in(0,\infty)$,
then, for any $f\in C^1(\mathbb{R}^n)$
with $|\nabla f|\in C_{\mathrm{c}}(\mathbb{R}^n)$,
\begin{align*}
\limsup_{\lambda\to\infty}
\lambda
\left\|\left[\int_{\mathbb{R}^n}
\mathbf{1}_{E_{\lambda,\frac{\gamma}{q}}[f]}(\cdot,y)
\left|\cdot-y\right|^{\gamma-n}\,dy\right]^\frac{1}{q}\right\|_X
=\left[\frac{\kappa(q,n)}{\gamma}\right]^\frac{1}{q}
\left\|\,\left|\nabla f\right|\,\right\|_{X},
\end{align*}
where $\kappa(q,n)$ and $E_{\lambda,\frac{\gamma}{q}}[f]$
with $\lambda\in(0,\infty)$ are the same as, respectively, in
\eqref{kappaqn} and \eqref{Elambda}.
\item[\textup{(ii)}]
Let $p\in[1,\infty)$, $\gamma\in(-\infty,0)$,
and $q\in(0,\frac{n-\gamma}{n}p)$.
Assume that $X^\frac{1}{p}$ is a ball Banach function space
and that the Hardy--Littlewood maximal operator
is bounded on $(X^\frac{1}{p})'$.
Then, for any $f\in C^1_{\mathrm{c}}(\mathbb{R}^n)$,
\begin{align}\label{1548}
\lim_{\lambda\to0^+}
\lambda
\left\|\left[\int_{\mathbb{R}^n}
\mathbf{1}_{E_{\lambda,\frac{\gamma}{q}}[f]}(\cdot,y)
\left|\cdot-y\right|^{\gamma-n}\,dy\right]^\frac{1}{q}\right\|_X
=\left[-\frac{\kappa(q,n)}{\gamma}\right]^\frac{1}{q}
\left\|\,\left|\nabla f\right|\,\right\|_{X};
\end{align}
moreover, if $n\in\mathbb{N}\cap[2,\infty)$
or if both $n=1$ and $q\in(0,-p\gamma)$,
then \eqref{1548} also holds true for any $f\in C^1(\mathbb{R}^n)$
with $|\nabla f|\in C_{\mathrm{c}}(\mathbb{R}^n)$.
\end{enumerate}
\end{corollary}

\begin{remark}
\begin{enumerate}
\item[\rm (i)]
Corollary~\ref{LimFormulaEq}(i) when $\gamma=n$
coincides with \cite[(3.26)]{dlyyz.arxiv}.
\item[\rm (ii)]
Let $p\in[1,\infty)$ and $q\in(0,\infty)$ be the same as in
Corollary~\ref{LimFormulaEq}.
If $X:=L^p(\mathbb{R}^n)$,
then both $X$ and $X^\frac{1}{p}=L^1(\mathbb{R}^n)$
are ball Banach function spaces
and the Hardy--Littlewood maximal operator
is bounded on $(X^\frac{1}{p})'=L^\infty(\mathbb{R}^n)$.
Thus, Corollary~\ref{LimFormulaEq} with $X:=L^p(\mathbb{R}^n)$
holds true, which when $q=p$ is just \cite[(1.7) and (1.8)]{bsvy.arxiv}
and which when $\gamma\in(0,\infty)$ and $q\in(0,\infty)$
with $q\neq p$,
or when $\gamma\in(-\infty,0)$, $n\in\mathbb{N}\cap[2,\infty)$, and
$q\in(0,\frac{n-\gamma}{n}p)$
with $q\neq p$, or when $\gamma\in(-\infty,0)$,
$n=1$, and $q\in(0,-p\gamma)$ with $q\neq p$
is new for any $f\in C^1(\mathbb{R}^n)$
with $|\nabla f|\in C_{\mathrm{c}}(\mathbb{R}^n)$.
In addition, Corollary~\ref{LimFormulaEq} with $X:=L^p(\mathbb{R}^n)$
when $\gamma\in(-\infty,0)$ and
$q\in(0,\frac{n-\gamma}{n}p)$
with $q\neq p$ is new for any
$f\in C^1_{\mathrm{c}}(\mathbb{R}^n)$.
\end{enumerate}
\end{remark}

\subsection{Extension to $\dot{W}^{1,X}(\mathbb{R}^n)$}
\label{ss3.3}

In this subsection,
under some additional mild assumptions about $X$,
we extend the main results of both Subsections~\ref{ss3.1}
and~\ref{ss3.2}
from any $f\in C^1(\mathbb{R}^n)$ with
$|\nabla f|\in C_{\mathrm{c}}(\mathbb{R}^n)$
to any $f\in\dot{W}^{1,X}(\mathbb{R}^n)$
via a density argument.
To this end, we begin with
the following density lemma
which is just \cite[Corollary~2.18]{dlyyz.arxiv}.

\begin{lemma}\label{828}
Let $X$ be a ball Banach function space
and $p\in[1,\infty)$.
Assume that $X^\frac{1}{p}$ is a ball Banach function space,
that the Hardy--Littlewood maximal operator $\mathcal{M}$ is
bounded on $(X^\frac{1}{p})'$, and that $X$ has an absolutely
continuous norm.
Then, for any $f\in\dot{W}^{1,X}(\mathbb{R}^n)$,
there exists a sequence
$\{f_k\}_{k\in\mathbb{N}}\subset C^\infty(\mathbb{R}^n)$
with $|\nabla f_k|\in C_{\mathrm{c}}(\mathbb{R}^n)$
for any $k\in\mathbb{N}$ such that, for any $R\in(0,\infty)$,
\begin{align}\label{dense}
\lim_{k\to\infty}\left\|f-f_k\right\|_{\dot{W}^{1,X}(\mathbb{R}^n)}=0
\quad\text{and}\quad
\lim_{k\to\infty}\left\|(f-f_k)\mathbf{1}_{B({\bf 0},R)}\right\|_X=0.
\end{align}
\end{lemma}

The first main result of this subsection reads as follows.

\begin{theorem}\label{2117}
Let $X$ be a ball Banach function space
and $p\in[1,\infty)$.
Assume that $X^\frac{1}{p}$ is a ball Banach function space,
that the Hardy--Littlewood maximal operator $\mathcal{M}$ is
bounded on $(X^\frac{1}{p})'$ with its operator norm denoted by
$\|\mathcal{M}\|_{(X^\frac{1}{p})'\to(X^\frac{1}{p})'}$,
and that $X$ has an absolutely
continuous norm. Assume that one of the following
four statements hold true:
\begin{enumerate}
\item[\rm(a)]
$\gamma\in\mathbb{R}\setminus\{0\}$, $q\in(0,p]$,
and $\mathcal{M}$ is bounded on $X$.
\item[\rm(b)]
$\gamma\in\mathbb{R}\setminus\{0\}$, $q\in(0,\infty)$,
and there exists an $r\in(n,\infty)$ such that
$\mathcal{M}$ is bounded on $X^\frac{1}{r}$.
\item[\rm(c)]
$\gamma\in(0,\infty)$ and $q\in(0,\infty)$ satisfies
$n(\frac{1}{p}-\frac{1}{q})<1$.
\item[\rm(d)]
$\gamma\in(-\infty,-1)$, $n=1$, and $q\in[1,p]$.
\end{enumerate}
Then the following statements hold true.
\begin{enumerate}
\item[\textup{(i)}]
There exists a positive constant $C$ such that,
for any $f\in\dot{W}^{1,X}(\mathbb{R}^n)$,
\begin{align*}
\left[\frac{\kappa(q,n)}{|\gamma|}\right]^{\frac{1}{q}}
\left\|\,|\nabla f|\,\right\|_X
&\leq\sup_{\lambda\in(0,\infty)}\lambda
\left\|\left[\int_{\mathbb{R}^n}
\mathbf{1}_{E_{\lambda,\frac{\gamma}{q}}[f]}(\cdot,y)
\left|\cdot-y\right|^{\gamma-n}\,dy\right]^\frac{1}{q}\right\|_X\\
&\leq C\left\|\,|\nabla f|\,\right\|_X,
\end{align*}
where $\kappa(q,n)$ and $E_{\lambda,\frac{\gamma}{q}}[f]$
with $\lambda\in(0,\infty)$ are the same as, respectively, in
\eqref{kappaqn} and \eqref{Elambda}
and where $C$ is continuous with
respect to $\|\mathcal{M}\|_{(X^\frac{1}{p})'\to(X^\frac{1}{p})'}$
and increases as $\|\mathcal{M}\|_{(X^\frac{1}{p})'\to(X^\frac{1}{p})'}$ increases.
\item[\textup{(ii)}]
If $\gamma\in(0,\infty)$,
then, for any $f\in\dot{W}^{1,X}(\mathbb{R}^n)$,
\begin{align*}
\lim_{\lambda\to\infty}
\lambda\left\|\left[\int_{\mathbb{R}^n}
\mathbf{1}_{E_{\lambda,\frac{\gamma}{q}}[f]}(\cdot,y)
\left|\cdot-y\right|^{\gamma-n}\,dy\right]^\frac{1}{q}\right\|_X
=\left[\frac{\kappa(q,n)}{\gamma}\right]^\frac{1}{q}
\left\|\,\left|\nabla f\right|\,\right\|_{X}.
\end{align*}
\item[\textup{(iii)}]
If $\gamma\in(-\infty,0)$,
assume further that both $n\in\mathbb{N}\cap[2,\infty)$
and $q\in(0,\frac{n-\gamma}{n}p)$
or that $n=1$, $q\in(0,-p\gamma)$,
and $\gamma\in(-\infty,-1)$.
Then, for any $f\in\dot{W}^{1,X}(\mathbb{R}^n)$,
\begin{align*}
\lim_{\lambda\to0^+}
\lambda
\left\|\left[\int_{\mathbb{R}^n}
\mathbf{1}_{E_{\lambda,\frac{\gamma}{q}}[f]}(\cdot,y)
\left|\cdot-y\right|^{\gamma-n}\,dy\right]^\frac{1}{q}\right\|_X
=\left[-\frac{\kappa(q,n)}{\gamma}\right]^\frac{1}{q}
\left\|\,\left|\nabla f\right|\,\right\|_{X}.
\end{align*}
\end{enumerate}
\end{theorem}

To prove Theorem~\ref{2117}, we need the following technical lemma
which is just \cite[Lemma~2.13]{dlyyz.arxiv}.

\begin{lemma}\label{1639}
Let $X$ be a ball Banach function space and $X'$ its associate
space. If $f\in X$ and $g\in X'$, then $fg$ is integrable and
$$
\int_{\mathbb{R}^n}\left|f(x)g(x)\right|\,dx
\leq\|f\|_X\|g\|_{X'}.
$$
\end{lemma}

\begin{proof}[Proof of Theorem~\ref{2117}]
Let $f\in\dot{W}^{1,X}(\mathbb{R}^n)$.
By Lemma~\ref{828}, we conclude that there exists a
sequence $\{f_k\}_{k\in\mathbb{N}}\subset C^\infty(\mathbb{R}^n)$
with $|\nabla f_k|\in C_{\mathrm{c}}(\mathbb{R}^n)$
for any $k\in\mathbb{N}$ such that \eqref{dense} holds true.
From Theorems~\ref{upperBound},~\ref{2001},~\ref{gamma>0},
and~\ref{n=1} and Corollary~\ref{LimFormulaEq},
we deduce that, for any $k\in\mathbb{N}$,
(i), (ii), and (iii) hold true for $f_k$.

Next, we show that the upper estimate in (i)
holds true for such an $f$.
To this end, for any $N\in\mathbb{N}$, let
$$
\mathcal{F}_N:=\left\{(x,y)\in B(\mathbf{0},N)\times
B(\mathbf{0},N):\ |x-y|\ge N^{-1}\right\}.
$$
Notice that, for any $\lambda\in(0,\infty)$ and $k,N\in\mathbb{N}$,
\begin{align*}
E_{\lambda,\frac{\gamma}{q}}[f]\cap\mathcal{F}_N
&\subset\left(E_{\frac{\lambda}{3},\frac{\gamma}{q}}[f_k]
\cap\mathcal{F}_N\right)\nonumber\\
&\quad\cup\left(\left\{(x,y)\in\mathbb{R}^n\times\mathbb{R}^n:\
x\neq y,\,\frac{|f_k(x)-f(x)|}{|x-y|^{1+\frac{\gamma}{q}}}
>\frac{\lambda}{3}\right\}
\cap\mathcal{F}_N\right)\nonumber\\
&\quad\cup\left(\left\{(x,y)\in\mathbb{R}^n\times\mathbb{R}^n:\
x\neq y,\,\frac{|f_k(y)-f(y)|}{|x-y|^{1+\frac{\gamma}{q}}}
>\frac{\lambda}{3}\right\}
\cap\mathcal{F}_N\right)\nonumber\\
&=:E_1\cup E_2\cup E_3,
\end{align*}
which, combined with Definition~\ref{1659}(v), further implies that,
for any $\lambda\in(0,\infty)$,
\begin{align}\label{YJ123}
&\lambda\left\|\left[\int_{\mathbb{R}^n}
\mathbf{1}_{E_{\lambda,\frac{\gamma}{q}}[f]\cap\mathcal{F}_N}(\cdot,y)
\left|\cdot-y\right|^{\gamma-n}\,dy\right]^\frac{1}{q}\right\|_X\\
&\quad\lesssim\sum_{i=1}^3\sup_{\lambda\in(0,\infty)}\lambda
\left\|\left[\int_{\mathbb{R}^n}
\mathbf{1}_{E_i}(\cdot,y)
\left|\cdot-y\right|^{\gamma-n}\,dy\right]^\frac{1}{q}\right\|_X
=:\sum_{i=1}^3\mathrm{I}_i.\nonumber
\end{align}
To deal with $\mathrm{I}_1$,
by (i) for $f_k$ and Definition~\ref{1659}(v),
we conclude that, for any $k\in\mathbb{N}$,
\begin{align}\label{YJ1}
\mathrm{I}_1
&\lesssim\left\|\,\left|\nabla f_k\right|\,\right\|_{X}
\leq\left\|\,\left|\nabla f\right|\,\right\|_{X}
+\left\|\,\left|\nabla(f_k-f)\right|\,\right\|_{X}.
\end{align}
To deal with $\mathrm{I}_2$, from the definitions of both
$E_2$ and $\mathcal{F}_N$,
we deduce that, for any $k,N\in\mathbb{N}$,
\begin{align}\label{YJ2}
\mathrm{I}_2
&\lesssim\left\|\left[\int_{\mathbb{R}^n}
\frac{|f_k(\cdot)-f(\cdot)|^q}{|\cdot-y|^{q+n}}
\mathbf{1}_{\mathcal{F}_N}(\cdot,y)\,dy\right]^\frac{1}{q}\right\|_{X}\\
&\lesssim N^\frac{q+2n}{q}\|(f-f_k)\mathbf{1}_{B(\mathbf{0},N)}\|_X.\nonumber
\end{align}
To deal with $\mathrm{I}_3$,
by the definitions of both $E_3$ and $\mathcal{F}_N$
and Lemma~\ref{1639},
we conclude that, for any $k,N\in\mathbb{N}$,
\begin{align*}
\mathrm{I}_3
&\lesssim\lambda^{1-\frac{1}{q}}\left\|\left[\int_{\mathbb{R}^n}
\frac{|f_k(y)-f(y)|}{|\cdot-y|^{1+\frac{\gamma}{q}+n-\gamma}}
\mathbf{1}_{\mathcal{F}_N}(\cdot,y)\,dy
\right]^\frac{1}{q}\right\|_{X}\nonumber\\
&\lesssim\lambda^{1-\frac{1}{q}}N^{|1+\frac{\gamma}{q}+n-\gamma|}
\left\|(f-f_k)\mathbf{1}_{B(\mathbf{0},N)}\right\|_X^\frac{1}{q}
\|\mathbf{1}_{B(\mathbf{0},N)}\|_{X'}^\frac{1}{q}
\|\mathbf{1}_{B(\mathbf{0},N)}\|_X,
\end{align*}
which, together with \eqref{YJ123}, \eqref{YJ1},
and \eqref{YJ2}, further implies that, for any $\lambda\in(0,\infty)$,
\begin{align*}
&\lambda\left\|\left[\int_{\mathbb{R}^n}
\mathbf{1}_{E_{\lambda,\frac{\gamma}{q}}[f]\cap\mathcal{F}_N}(\cdot,y)
\left|\cdot-y\right|^{\gamma-n}\,dy\right]^\frac{1}{q}\right\|_X\\
&\quad\lesssim\left\|\,\left|\nabla f\right|\,\right\|_{X}
+\left\|\,\left|\nabla(f_k-f)\right|\,\right\|_{X}
+N^\frac{q+2n}{q}\|(f-f_k)\mathbf{1}_{B(\mathbf{0},N)}\|_X\\
&\qquad+\lambda^{1-\frac{1}{q}}
N^{|1+\frac{\gamma}{q}+n-\gamma|}
\left\|(f-f_k)\mathbf{1}_{B(\mathbf{0},N)}\right\|_X^\frac{1}{q}
\|\mathbf{1}_{B(\mathbf{0},N)}\|_{X'}^\frac{1}{q}
\|\mathbf{1}_{B(\mathbf{0},N)}\|_X.
\end{align*}
Using this, Remark~\ref{2143},
Definition~\ref{1659}(iv),
and \eqref{dense} and letting $k\to\infty$,
we find that, for any $\lambda\in(0,\infty)$
and $N\in\mathbb{N}$,
\begin{align*}
\lambda\left\|\left[\int_{\mathbb{R}^n}
\mathbf{1}_{E_{\lambda,\frac{\gamma}{q}}[f]\cap\mathcal{F}_N}(\cdot,y)
\left|\cdot-y\right|^{\gamma-n}\,dy\right]^\frac{1}{q}\right\|_X
\lesssim\left\|\,\left|\nabla f\right|\,\right\|_{X}.
\end{align*}
By this and Definition~\ref{1659}(iii),
letting $N\to\infty$, and taking the supremum over all $\lambda\in(0,\infty)$,
we then conclude that
\begin{align}\label{upperW1X}
\sup_{\lambda\in(0,\infty)}
\lambda\left\|\left[\int_{\mathbb{R}^n}
\mathbf{1}_{E_{\lambda,\frac{\gamma}{q}}[f]}(\cdot,y)
\left|\cdot-y\right|^{\gamma-n}\,dy\right]^\frac{1}{q}\right\|_X
\lesssim\left\|\,\left|\nabla f\right|\,\right\|_{X}.
\end{align}
This finishes the proof the upper estimate in (i)
for any $f\in\dot{W}^{1,X}(\mathbb{R}^n)$.

Now, we show the lower estimate in (ii) for
any $f\in\dot{W}^{1,X}(\mathbb{R}^n)$.
From Lemma~\ref{1111}, Definition~\ref{1659}(v),
and \eqref{upperW1X}, we deduce that, for any $k\in\mathbb{N}$
and $\delta,\eta\in(0,1)$, there exists a $C_\eta\in(0,\infty)$ such that,
for any $\lambda\in(0,\infty)$,
\begin{align*}
&\lambda\left\|\left[\int_{\mathbb{R}^n}
\mathbf{1}_{E_{\lambda,\frac{\gamma}{q}}[f_k]}(\cdot,y)
\left|\cdot-y\right|^{\gamma-n}\,dy\right]^\frac{1}{q}\right\|_X\\
&\quad\leq(1+\eta)\lambda\left\|\left[\int_{\mathbb{R}^n}
\mathbf{1}_{E_{\delta\lambda,\frac{\gamma}{q}}[f]}(\cdot,y)
\left|\cdot-y\right|^{\gamma-n}\,dy\right]^\frac{1}{q}\right\|_X\\
&\qquad+C_{\eta}\lambda\left\|\left[\int_{\mathbb{R}^n}
\mathbf{1}_{E_{(1-\delta)\lambda,\frac{\gamma}{q}}[f_k-f]}(\cdot,y)
\left|\cdot-y\right|^{\gamma-n}\,dy\right]^\frac{1}{q}\right\|_X\\
&\quad\leq(1+\eta)\delta^{-1}\lambda\left\|\left[\int_{\mathbb{R}^n}
\mathbf{1}_{E_{\lambda,\frac{\gamma}{q}}[f]}(\cdot,y)
\left|\cdot-y\right|^{\gamma-n}\,dy\right]^\frac{1}{q}\right\|_X\\
&\qquad+\widetilde{C}C_{\eta}(1-\delta)^{-1}
\left\|\,\left|\nabla(f-f_k)\right|\,\right\|_{X},
\end{align*}
where $\widetilde{C}$ is the implicit positive constant in \eqref{upperW1X},
which, together with Definition~\ref{1659}(v),
the lower estimate in (ii) for $f_k$,
and \eqref{dense},
further implies that
\begin{align*}
\left\|\,\left|\nabla f\right|\,\right\|_{X}
&\leq\left\|\,\left|\nabla f_k\right|\,\right\|_{X}
+\left\|\,\left|\nabla(f-f_k)\right|\,\right\|_{X}\\
&\leq\left[\frac{\kappa(q,n)}{\gamma}\right]^{-\frac{1}{q}}
\liminf_{\lambda\to\infty}
\lambda\left\|\left[\int_{\mathbb{R}^n}
\mathbf{1}_{E_{\lambda,\frac{\gamma}{q}}[f_k]}(\cdot,y)
\left|\cdot-y\right|^{\gamma-n}\,dy\right]^\frac{1}{q}\right\|_X\\
&\quad+\left\|\,\left|\nabla(f-f_k)\right|\,\right\|_{X}\\
&\leq\left[\frac{\kappa(q,n)}{\gamma}\right]^{-\frac{1}{q}}
(1+\eta)\delta^{-1}\liminf_{\lambda\to\infty}
\lambda\left\|\left[\int_{\mathbb{R}^n}
\mathbf{1}_{E_{\lambda,\frac{\gamma}{q}}[f]}(\cdot,y)
\left|\cdot-y\right|^{\gamma-n}\,dy\right]^\frac{1}{q}\right\|_X\\
&\quad+\left\{\left[\frac{\kappa(q,n)}{\gamma}\right]^{-\frac{1}{q}}
\widetilde{C}C_{\eta}(1-\delta)^{-1}+1\right\}
\left\|\,\left|\nabla(f-f_k)\right|\,\right\|_{X}\\
&\to\left[\frac{\kappa(q,n)}{\gamma}\right]^{-\frac{1}{q}}
(1+\eta)\delta^{-1}\liminf_{\lambda\to\infty}
\lambda\left\|\left[\int_{\mathbb{R}^n}
\mathbf{1}_{E_{\lambda,\frac{\gamma}{q}}[f]}(\cdot,y)
\left|\cdot-y\right|^{\gamma-n}\,dy\right]^\frac{1}{q}\right\|_X
\end{align*}
as $k\to\infty$.
Letting $\delta\to1$ and $\eta\to0$, we further obtain
\begin{align}\label{2233}
\liminf_{\lambda\to\infty}
\lambda\left\|\left[\int_{\mathbb{R}^n}
\mathbf{1}_{E_{\lambda,\frac{\gamma}{q}}[f]}(\cdot,y)
\left|\cdot-y\right|^{\gamma-n}\,dy\right]^\frac{1}{q}\right\|_X
\ge\left[\frac{\kappa(q,n)}{\gamma}\right]^\frac{1}{q}
\left\|\,\left|\nabla f\right|\,\right\|_{X},
\end{align}
which completes the proof of the lower estimate in (ii) for such an $f$ and hence
the lower estimate in (i) for such an $f$ in the case $\gamma\in(0,\infty)$.
From an argument similar to that used in the estimation
of \eqref{2233}, we deduce that
both the lower estimate in (iii) and
the lower estimate in (i) in the case $\gamma\in(-\infty,0)$
also hold true for any $f\in\dot{W}^{1,X}(\mathbb{R}^n)$.

To complete the proof of the present theorem,
it remains to show the upper estimates in
both (ii) and (iii) for any $f\in\dot{W}^{1,X}(\mathbb{R}^n)$.
We only prove the former because the proof of the latter is similar
and hence we omit the details.
By the upper estimate in (ii) for $f_k$,
\eqref{upperW1X}, Lemma~\ref{1111}, and Definition~\ref{1659}(v),
we find that, for any $k\in\mathbb{N}$
and $\delta,\eta\in(0,1)$, there exists a $C_\eta\in(0,\infty)$ such that
\begin{align*}
&\limsup_{\lambda\to\infty}
\lambda\left\|\left[\int_{\mathbb{R}^n}
\mathbf{1}_{E_{\lambda,\frac{\gamma}{q}}[f]}(\cdot,y)
\left|\cdot-y\right|^{\gamma-n}\,dy\right]^\frac{1}{q}\right\|_X\\
&\quad\leq C_\eta\limsup_{\lambda\to\infty}
\lambda\left\|\left[\int_{\mathbb{R}^n}
\mathbf{1}_{E_{(1-\delta)\lambda,\frac{\gamma}{q}}[f-f_k]}(\cdot,y)
\left|\cdot-y\right|^{\gamma-n}\,dy\right]^\frac{1}{q}\right\|_X\\
&\qquad+(1+\eta)\limsup_{\lambda\to\infty}
\lambda\left\|\left[\int_{\mathbb{R}^n}
\mathbf{1}_{E_{\delta\lambda,\frac{\gamma}{q}}[f_k]}(\cdot,y)
\left|\cdot-y\right|^{\gamma-n}\,dy\right]^\frac{1}{q}\right\|_X\\
&\quad\leq\widetilde{C}(1-\delta)^{-1}C_\eta
\left\|\,\left|\nabla(f-f_k)\right|\,\right\|_{X}
+(1+\eta)\delta^{-1}\left[\frac{\kappa(q,n)}{\gamma}\right]^\frac{1}{q}
\left\|\,\left|\nabla f\right|\,\right\|_{X},
\end{align*}
which, together with \eqref{dense},
via first letting $k\to\infty$ and then
letting $\eta\to0$ and $\delta\to1$,
further implies that
$$
\limsup_{\lambda\to\infty}
\lambda\left\|\left[\int_{\mathbb{R}^n}
\mathbf{1}_{E_{\lambda,\frac{\gamma}{q}}[f]}(\cdot,y)
\left|\cdot-y\right|^{\gamma-n}\,dy\right]^\frac{1}{q}\right\|_X
\leq\left[\frac{\kappa(q,n)}{\gamma}\right]^\frac{1}{q}
\left\|\,\left|\nabla f\right|\,\right\|_{X}.
$$
This finishes the upper estimate in (ii) for
any $f\in\dot{W}^{1,X}(\mathbb{R}^n)$
and hence Theorem~\ref{2117}.
\end{proof}

\begin{remark}
Theorem~\ref{2117} when $\gamma=n$ coincides with
\cite[Theorem~4.5(ii)]{dlyyz.arxiv}.
\end{remark}

Using Theorem~\ref{2117} with $X:=L^p(\mathbb{R}^n)$
and $p\in[1,\infty)$,
we obtain the following corollary.

\begin{corollary}\label{2120}
Let $p\in[1,\infty)$, $\gamma\in\mathbb{R}\setminus\{0\}$,
and $q\in(0,\infty)$.
If one of the following four statements hold true:
\begin{enumerate}
\item[\rm(a)]
$p\in(n,\infty)$, $\gamma\in\mathbb{R}\setminus\{0\}$,
and $q\in(0,\infty)$;
\item[\rm(b)]
$p\in[1,n]$, $\gamma\in(0,\infty)$,
and $q\in(0,\infty)$ satisfies $n(\frac{1}{p}-\frac{1}{q})<1$;
\item[\rm(c)]
$p\in(1,n]$, $\gamma\in(-\infty,0)$,
and $q\in(0,p]$;
\item[\rm(d)]
$p=q=n=1$ and $\gamma\in(-\infty,-1)$.
\end{enumerate}
Then the following statements hold true.
\begin{enumerate}
\item[\textup{(i)}]
For any $f\in\dot{W}^{1,p}(\mathbb{R}^n)$,
\begin{align*}
\sup_{\lambda\in(0,\infty)}\lambda
\left\|\left[\int_{\mathbb{R}^n}
\mathbf{1}_{E_{\lambda,\frac{\gamma}{q}}[f]}(\cdot,y)
\left|\cdot-y\right|^{\gamma-n}\,dy\right]^\frac{1}{q}
\right\|_{L^p(\mathbb{R}^n)}\sim
\left\|\,|\nabla f|\,\right\|_{L^p(\mathbb{R}^n)},
\end{align*}
where the positive equivalence constants are independent of $f$
and where $\kappa(q,n)$ and $E_{\lambda,\frac{\gamma}{q}}[f]$
with $\lambda\in(0,\infty)$ are the same as, respectively, in
\eqref{kappaqn} and \eqref{Elambda}.
\item[\textup{(ii)}]
If $\gamma\in(0,\infty)$,
then, for any $f\in\dot{W}^{1,p}(\mathbb{R}^n)$,
\begin{align*}
\lim_{\lambda\to\infty}
\lambda\left\|\left[\int_{\mathbb{R}^n}
\mathbf{1}_{E_{\lambda,\frac{\gamma}{q}}[f]}(\cdot,y)
\left|\cdot-y\right|^{\gamma-n}\,dy
\right]^\frac{1}{q}\right\|_{L^p(\mathbb{R}^n)}
=\left[\frac{\kappa(q,n)}{\gamma}\right]^\frac{1}{q}
\left\|\,\left|\nabla f\right|\,\right\|_{L^p(\mathbb{R}^n)}.
\end{align*}
\item[\textup{(iii)}]
If $\gamma\in(-\infty,0)$,
assume further that both $n\in\mathbb{N}\cap[2,\infty)$
and $q\in(0,\frac{n-\gamma}{n}p)$
or that $n=1$, $q\in(0,-p\gamma)$,
and $\gamma\in(-\infty,-1)$.
Then, for any $f\in\dot{W}^{1,p}(\mathbb{R}^n)$,
\begin{align*}
\lim_{\lambda\to0^+}
\lambda
\left\|\left[\int_{\mathbb{R}^n}
\mathbf{1}_{E_{\lambda,\frac{\gamma}{q}}[f]}(\cdot,y)
\left|\cdot-y\right|^{\gamma-n}\,dy
\right]^\frac{1}{q}\right\|_{L^p(\mathbb{R}^n)}
=\left[-\frac{\kappa(q,n)}{\gamma}\right]^\frac{1}{q}
\left\|\,\left|\nabla f\right|\,\right\|_{L^p(\mathbb{R}^n)}.
\end{align*}
\end{enumerate}
\end{corollary}

\begin{proof}
Let $X:=L^p(\mathbb{R}^n)$.
Notice that $X^\frac{1}{p}=L^1(\mathbb{R}^n)$
which is obviously a ball Banach function space
and it is well known that the Hardy--Littlewood
maximal operator $\mathcal{M}$ is bounded
on $(X^\frac{1}{p})'=L^\infty(\mathbb{R}^n)$.
When $p\in[1,\infty)$, it is also easy to show that $X=L^p(\mathbb{R}^n)$
has an absolutely continuous norm.
Next, we consider the following four cases on $p$, $\gamma$, and $q$.

\emph{Case 1)}
$p\in(n,\infty)$, $\gamma\in\mathbb{R}\setminus\{0\}$,
and $q\in(0,\infty)$. In this case,
it is easy to find that there exists an $r\in(n,p)$
such that $\mathcal{M}$ is bounded
on $X^\frac{1}{r}=L^\frac{p}{r}(\mathbb{R}^n)$.
Thus, Theorem~\ref{2117}(b) with $X:=L^p(\mathbb{R}^n)$
is satisfied. By this and Theorem~\ref{2117},
we conclude that the present corollary
in this case holds true.

\emph{Case 2)}
$p\in[1,n]$, $\gamma\in(0,\infty)$,
and $q\in(0,\infty)$ satisfies $n(\frac{1}{p}-\frac{1}{q})<1$.
In this case, we find that Theorem~\ref{2117}(c)
with $X:=L^p(\mathbb{R}^n)$
is satisfied. From this and Theorem~\ref{2117},
we deduce that the present corollary
in this case holds true.

\emph{Case 3)}
$p\in(1,n]$, $\gamma\in(-\infty,0)$,
and $q\in(0,p]$. In this case, it is easy to see that
$\mathcal{M}$ is bounded on $X=L^p(\mathbb{R}^n)$
and hence Theorem~\ref{2117}(a)
with $X:=L^p(\mathbb{R}^n)$
is satisfied. By this and Theorem~\ref{2117},
we conclude that the present corollary
in this case holds true.

\emph{Case 4)}
$p=q=n=1$ and $\gamma\in(-\infty,-1)$.
In this case, we find that Theorem~\ref{2117}(d)
with $X:=L^1(\mathbb{R}^n)$
is satisfied. From this and Theorem~\ref{2117},
we deduce that the present corollary
in this case holds true.

This finishes the proof of Corollary~\ref{2120}.
\end{proof}

\begin{remark}
\begin{enumerate}
\item[(i)]
Corollary~\ref{2120} when
both $p=q\in[1,\infty)$ and $\gamma\in(0,\infty)$
or when both $p=q\in(1,\infty)$ and $\gamma\in(-\infty,0)$
coincides with \cite[both (a) and (b) of Theorem~1.1,
Theorem~1.3(i), and Theorem~1.4(i)]{bsvy.arxiv}
in the same case.
\item[(ii)]
Corollary~\ref{2120} when
both $p=q=n=1$ and $\gamma\in(-\infty,-1)$
is a part of \cite[Theorems~1.1(b) and~1.4(i)]{bsvy.arxiv}.
\item[(iii)]
Both (i) and (ii) of Corollary~\ref{2120} when
$p\in(n,\infty)$, $\gamma\in\mathbb{R}\setminus\{0\}$,
and $q\in(0,\infty)\setminus\{p\}$,
or when
$p\in[1,n]$, $\gamma\in(0,\infty)$,
and $q\in(0,\infty)$ with both
$n(\frac{1}{p}-\frac{1}{q})<1$ and $q\neq p$,
or when
$p\in(1,n]$, $\gamma\in(-\infty,0)$,
and $q\in(0,p)$
are new.
In particular, in the case $\gamma\in(-\infty,0)$,
Corollary~\ref{2120}(iii)
when $n\in\mathbb{N}\cap[2,\infty)$,
$p\in(n,\infty)$, and $q\in(0,\frac{n-\gamma}{n}p)$
with $q\neq p$,
or when $n\in\mathbb{N}\cap[2,\infty)$,
$p\in(1,n)$, and $q\in(0,p)$ is new.
\end{enumerate}
\end{remark}

Notice that, in Theorem~\ref{2117},
we assume that the Hardy--Littlewood maximal operator
is bounded on $(X^\frac{1}{p})'$ for some $p\in[1,\infty)$.
However, in certain critical cases of some specific function
spaces
(for example, the Morrey space $M_1^\alpha(\mathbb{R}^n)$
with $\alpha\in[1,\infty)$,
the variable Lebesgue space $L^{r(\cdot)}(\mathbb{R}^n)$ with
$\mathop\mathrm{\,ess\,inf\,}_{x\in\mathbb{R}^n}r(x)=1$,
and the Orlicz space $L^{\Phi}(\mathbb{R}^n)$
with lower type $r^-_{\Phi}=1$;
see the definitions of these spaces in Section~\ref{S5} below),
this assumption is not true even if $p=1$.
To establish \eqref{2213} in this case,
we begin with a density lemma.
Recall that, for any given $r\in(0,\infty)$,
the \emph{centered ball average operator}
$\mathcal{B}_r$ is defined by setting,
for any $f\in L_{{\mathrm{loc}}}^1(\mathbb{R}^n)$ and $x\in\mathbb{R}^n$,
\begin{align*}
\mathcal{B}_r(f)(x):=\frac{1}{|B(x,r)|}\int_{B(x,r)}\left|f(y)\right|\,dy.
\end{align*}
The following density lemma is just \cite[Theorem~2.6]{dlyyz.arxiv}.

\begin{lemma}\label{1516}
Let $X$ be a ball Banach function space.
Assume that $X$ has an absolutely continuous norm and that
the centered ball average operators $\{\mathcal{B}_r\}_{r\in(0,\infty)}$
are uniformly bounded on $X$.
Then, for any $f\in\dot{W}^{1,X}(\mathbb{R}^n)$, there exists a sequence
$\{f_k\}_{k\in\mathbb{N}}\subset C^\infty(\mathbb{R}^n)$ with
$|\nabla f_k|\in C_{\mathrm{c}}(\mathbb{R}^n)$
for any $k\in\mathbb{N}$ such that \eqref{dense} holds true
for any $R\in(0,\infty)$.
\end{lemma}

The following conclusion is
an alternative to Theorem~\ref{2117},
which proves \eqref{2213} without assuming the boundedness of
the Hardy--Littlewood maximal operator on $X'$.

\begin{theorem}\label{4.8}
Let $X$ be a ball Banach function space.
Assume that $X$ has an absolutely continuous norm and that
centered ball average operators $\{\mathcal{B}_r\}_{r\in(0,\infty)}$
are uniformly bounded on $X$.
Assume that there exists a sequence
$\{\theta_m\}_{m\in\mathbb{N}}\subset(0,1)$
satisfying $\lim_{m\to\infty}\theta_m=1$ such that,
for any $m\in\mathbb{N}$,
$X^\frac{1}{\theta_m}$ is a ball Banach function space,
the Hardy--Littlewood maximal operator $\mathcal{M}$
is bounded on $(X^\frac{1}{\theta_m})'$ with its operator norm
denoted by $\|\mathcal{M}\|_{(X^\frac{1}{\theta_m})'
\to(X^\frac{1}{\theta_m})'}$,
and
\begin{align}\label{2114}
\lim_{m\to\infty}\left\|\mathcal{M}\right\|_{(X^\frac{1}{\theta_m})'
\to(X^\frac{1}{\theta_m})'}<\infty.
\end{align}
Assume further that one of the following statements hold true:
\begin{enumerate}
\item[\rm(a)]
$\gamma\in\mathbb{R}\setminus\{0\}$, $q\in(0,1]$,
and $\mathcal{M}$ is bounded on $X$.
\item[\rm(b)]
$\gamma\in\mathbb{R}\setminus\{0\}$, $q\in(0,\infty)$,
and there exists an $r\in(n,\infty)$ such that
$\mathcal{M}$ is bounded on $X^\frac{1}{r}$.
\item[\rm(c)]
$\gamma\in(0,\infty)$ and $q\in(0,\infty)$ satisfies
$n(1-\frac{1}{q})<1$.
\item[\rm(d)]
$\gamma\in(-\infty,-1)$ and $n=q=1$.
\end{enumerate}
Then the following statements hold true.
\begin{enumerate}
\item[\textup{(i)}]
Then there exists a positive constant $C$ such that,
for any $f\in\dot{W}^{1,X}(\mathbb{R}^n)$,
\begin{align*}
\left[\frac{\kappa(q,n)}{|\gamma|}\right]^{\frac{1}{q}}
\left\|\,|\nabla f|\,\right\|_X
&\leq\sup_{\lambda\in(0,\infty)}\lambda
\left\|\left[\int_{\mathbb{R}^n}
\mathbf{1}_{E_{\lambda,\frac{\gamma}{q}}[f]}(\cdot,y)
\left|\cdot-y\right|^{\gamma-n}\,dy\right]^\frac{1}{q}\right\|_X\\
&\leq C\left\|\,|\nabla f|\,\right\|_X,
\end{align*}
where $\kappa(q,n)$ and $E_{\lambda,\frac{\gamma}{q}}[f]$
with $\lambda\in(0,\infty)$ are the same as, respectively, in
\eqref{kappaqn} and \eqref{Elambda}.
\item[\textup{(ii)}]
If $\gamma\in(0,\infty)$,
then, for any $f\in\dot{W}^{1,X}(\mathbb{R}^n)$,
\begin{align*}
\lim_{\lambda\to\infty}
\lambda\left\|\left[\int_{\mathbb{R}^n}
\mathbf{1}_{E_{\lambda,\frac{\gamma}{q}}[f]}(\cdot,y)
\left|\cdot-y\right|^{\gamma-n}\,dy\right]^\frac{1}{q}\right\|_X
=\left[\frac{\kappa(q,n)}{\gamma}\right]^\frac{1}{q}
\left\|\,\left|\nabla f\right|\,\right\|_{X}.
\end{align*}
\item[\textup{(iii)}]
If $\gamma\in(-\infty,0)$,
assume further that both $n\in\mathbb{N}\cap[2,\infty)$
and $q\in(0,\frac{n-\gamma}{n}p)$
or that $n=1$, $q\in(0,-p\gamma)$,
and $\gamma\in(-\infty,-1)$.
Then, for any $f\in\dot{W}^{1,X}(\mathbb{R}^n)$,
\begin{align*}
\lim_{\lambda\to0^+}
\lambda
\left\|\left[\int_{\mathbb{R}^n}
\mathbf{1}_{E_{\lambda,\frac{\gamma}{q}}[f]}(\cdot,y)
\left|\cdot-y\right|^{\gamma-n}\,dy\right]^\frac{1}{q}\right\|_X
=\left[-\frac{\kappa(q,n)}{\gamma}\right]^\frac{1}{q}
\left\|\,\left|\nabla f\right|\,\right\|_{X}.
\end{align*}
\end{enumerate}
\end{theorem}

\begin{proof}
We first claim that (i), (ii), and (iii) hold true for
any $f\in C^1(\mathbb{R}^n)$
with $|\nabla f|\in C_{\mathrm{c}}(\mathbb{R}^n)$.
We first show the upper estimate in (i) for such an $f$.
To this end,
by Theorems~\ref{gamma>0} and~\ref{n=1} with both $p:=1$
and $X$ replaced by $X^\frac{1}{\theta_m}$ with $m\in\mathbb{N}$
and by the assumptions that, for any $m\in\mathbb{N}$,
$X^\frac{1}{\theta_m}$ is a ball Banach function space and that
$\mathcal{M}$ is bounded on $(X^\frac{1}{\theta_m})'$,
we conclude that, for any $m\in\mathbb{N}$,
\begin{align*}
\sup_{\lambda\in(0,\infty)}\lambda
\left\|\left[\int_{\mathbb{R}^n}
\mathbf{1}_{E_{\lambda,\frac{\gamma}{q}}[f]}(\cdot,y)
\left|\cdot-y\right|^{\gamma-n}\,dy
\right]^\frac{1}{q}\right\|_{X^\frac{1}{\theta_m}}
\leq C_m\left\|\,|\nabla f|\,\right\|_{X^\frac{1}{\theta_m}},
\end{align*}
where the positive constant $C_m$
is continuous with respect to
$\|\mathcal{M}\|_{(X^\frac{1}{\theta_m})'\to(X^\frac{1}{\theta_m})'}$
and increases as
$\|\mathcal{M}\|_{(X^\frac{1}{\theta_m})'\to(X^\frac{1}{\theta_m})'}$ increases,
which further implies that, for any $\lambda\in(0,\infty)$,
\begin{align*}
\left\|\left[\int_{\mathbb{R}^n}
\mathbf{1}_{E_{\lambda,\frac{\gamma}{q}}[f]}(\cdot,y)
\left|\cdot-y\right|^{\gamma-n}\,dy\right]^\frac{1}{q\theta_m}\right\|_{X}
\leq C_m^\frac{1}{\theta_m}\lambda^{-\frac{1}{\theta_m}}\left\|\left|\nabla f\right|^{
\frac{1}{\theta_m}}\right\|_{X}.
\end{align*}
From this, Lemma~\ref{FatouX}, \eqref{2114},
and the Lebesgue dominated convergence theorem,
we infer that
\begin{align*}
&\sup_{\lambda\in(0,\infty)}\lambda
\left\|\left[\int_{\mathbb{R}^n}
\mathbf{1}_{E_{\lambda,\frac{\gamma}{q}}[f]}(\cdot,y)
\left|\cdot-y\right|^{\gamma-n}\,dy\right]^\frac{1}{q}\right\|_X\\
&\quad\leq\sup_{\lambda\in(0,\infty)}\lambda
\liminf_{m\to\infty}
\left\|\left[\int_{\mathbb{R}^n}
\mathbf{1}_{E_{\lambda,\frac{\gamma}{q}}[f]}(\cdot,y)
\left|\cdot-y\right|^{\gamma-n}\,dy\right]^\frac{1}{q\theta_m}\right\|_X\\
&\quad\lesssim\sup_{\lambda\in(0,\infty)}\lambda
\liminf_{m\to\infty}
C_m^\frac{1}{\theta_m}\lambda^{-\frac{1}{\theta_m}}
\left\|\left|\nabla f\right|^{\frac{1}{\theta_m}}\right\|_{X}\\
&\quad=\liminf_{m\to\infty}
C_m^\frac{1}{\theta_m}\left\|\,\left|\nabla f\right|\,\right\|_{X}
\lesssim\left\|\,\left|\nabla f\right|\,\right\|_{X}.
\end{align*}
This finishes the proof of the upper bound estimate in (i) for such an $f$.
On the other hand,
by Theorem~\ref{LimFormulaInf},
we conclude that the lower estimate in (i), (ii), and (iii)
obviously hold true for such an $f$.
This, together with the upper estimate in (i),
finishes the proof of (i) for any $f\in C^1(\mathbb{R}^n)$
with $|\nabla f|\in C_{\mathrm{c}}(\mathbb{R}^n)$.
Thus, to show both (ii) and (iii) for such an $f$,
it remains to show the upper estimates in both (ii)
and (iii) for such an $f$. Notice that, from Theorem~\ref{LimFormulaSup}(i),
it is easy to infer that the upper estimate in (ii) holds true for such an $f$.
Next, we show the upper estimate in (iii) for such an $f$.
By the proof of Theorem~\ref{LimFormulaSup}(ii),
we find that, to show this, it suffices to prove \eqref{1144}
for any $f\in C^1(\mathbb{R}^n)$ with
$|\nabla f|\in C_{\mathrm{c}}(\mathbb{R}^n)$.
To this end, we consider the following two cases on $n$.

\emph{Case 1)} $n\in\mathbb{N}\cap[2,\infty)$. In this case,
from the proof of Theorem~\ref{LimFormulaSup},
it follows that
we only need to show \eqref{yxst}.
By the estimation of \eqref{1511}
with both $X:=X^\frac{1}{\theta_m}$ and $p:=1$
and by the assumptions that,
for any $m\in\mathbb{N}$, $X^\frac{1}{\theta_m}$ is a ball
Banach function space and that $\mathcal{M}$ is bounded on
$(X^\frac{1}{\theta_m})'$, we conclude that, for any $m\in\mathbb{N}$
and any ball $B_r:=B(\mathbf{0},r)$ with $r\in(0,\infty)$,
$$
\left\||\cdot|^{\frac{\gamma-n}{q}}\mathbf{1}_{
(B_r)^\complement}(\cdot)\right\|_{X^\frac{1}{\theta_m}}
\leq C_{(q,\gamma,n)}
\left\|\mathcal{M}\right\|_{(X^\frac{1}{\theta_m})'
\to(X^\frac{1}{\theta_m})'}
r^{\frac{\gamma-n}{q}}
\left\|\mathbf{1}_{B_r}\right\|_{X^\frac{1}{\theta_m}},
$$
where $C_{(p,q,\gamma,n)}$ is a positive constant depending only
on $p$, $q$, $\gamma$, and $n$. This, combined with Definition~\ref{tuhua},
further implies that,
for any ball $B_r:=B(\mathbf{0},r)$ with $r\in(0,\infty)$,
$$
\left\|\left[|\cdot|^{\frac{\gamma-n}{q}}\mathbf{1}_{
(B_r)^\complement}(\cdot)\right]^{\frac{1}{\theta_m}}\right\|_{X}
\leq\left[C_{(q,\gamma,n)}
\left\|\mathcal{M}\right\|_{(X^\frac{1}{\theta_m})'
\to(X^\frac{1}{\theta_m})'}
r^{\frac{\gamma-n}{q}}\right]^{\frac{1}{\theta_m}}
\left\|\mathbf{1}_{B_r}\right\|_{X}.
$$
From this, Lemma~\ref{FatouX}, \eqref{2114},
and Definition~\ref{1659}(iv), we deduce that,
for any given ball $B_r:=B(\mathbf{0},r)$ with $r\in(0,\infty)$,
\begin{align*}
\left\||\cdot|^{\frac{\gamma-n}{q}}\mathbf{1}_{
(B_r)^\complement}(\cdot)\right\|_{X}
&\leq\liminf_{m\to\infty}
\left\|\left[|\cdot|^{\frac{\gamma-n}{q}}\mathbf{1}_{
(B_r)^\complement}(\cdot)\right]^{\frac{1}{\theta_m}}\right\|_{X}\\
&\lesssim\liminf_{m\to\infty}
\left[\left\|\mathcal{M}\right\|_{(X^\frac{1}{\theta_m})'
\to(X^\frac{1}{\theta_m})'}
r^{\frac{\gamma-n}{q}}\right]^{\frac{1}{\theta_m}}
\left\|\mathbf{1}_{B_r}\right\|_{X}\\
&=\left[\lim_{m\to\infty}\left\|\mathcal{M}\right\|_{
(X^\frac{1}{\theta_m})'\to(X^\frac{1}{\theta_m})'}\right]
r^{\frac{\gamma-n}{q}}
\left\|\mathbf{1}_{B_r}\right\|_{X}\\
&\sim r^{\frac{\gamma-n}{q}}
\left\|\mathbf{1}_{B_r}\right\|_{X}<\infty,
\end{align*}
which completes the proof of \eqref{yxst} and hence \eqref{1144}.
This finishes the proof of the upper estimate in (iii) for such an $f$
in this case.

\emph{Case 2)} $n=1$.
In this case, to prove \eqref{1144}
for any $f\in C^1(\mathbb{R})$ with $f'\in C_{\mathrm{c}}(\mathbb{R})$,
following the proof of Theorem~\ref{n1},
it suffices to show that
\begin{align}\label{ggg}
\limsup_{\lambda\to0^+}
\lambda\left\|\left[\int_{\mathbb{R}}
\mathbf{1}_{E_{\lambda,\frac{\gamma}{q}}[f]\cap
[(-2\beta,2\beta)^\complement\times\mathbb{R}]}(\cdot,y)
\left|\cdot-y\right|^{\gamma-1}\,dy\right]^\frac{1}{q}
\right\|_X=0,
\end{align}
where $\beta\in(0,\infty)$ is such that
$\mathrm{supp\,}(f')\subset(-\beta,\beta)$.
Since $f$ is a constant on $(\beta,\infty)$
and another constant on $(-\infty,-\beta)$, it follows that,
if $(x,y)\in E_{\lambda,\frac{\gamma}{q}}[f]$ and $x\in[2\beta,\infty)$,
then $y\in(-\infty,\beta)$ and,
if $(x,y)\in E_{\lambda,\frac{\gamma}{q}}[f]$ and $x\in(-\infty,-2\beta]$,
then $y\in(-\beta,\infty)$.
By this and Definition~\ref{1659}(v), we find that
\begin{align*}
&\limsup_{\lambda\to0^+}
\lambda\left\|\left[\int_{\mathbb{R}}
\mathbf{1}_{E_{\lambda,\frac{\gamma}{q}}[f]\cap
[(-2\beta,2\beta)^\complement\times\mathbb{R}]}(\cdot,y)
\left|\cdot-y\right|^{\gamma-1}\,dy\right]^\frac{1}{q}
\right\|_X\\
&\quad=\limsup_{\lambda\to0^+}
\lambda\left\|\left(\int_{-\infty}^\beta
\left|\cdot-y\right|^{\gamma-1}\,dy\right)^\frac{1}{q}
\mathbf{1}_{[2\beta,\infty)}(\cdot)\right.\\
&\qquad\left.+\left(\int_{-\beta}^\infty
\left|\cdot-y\right|^{\gamma-1}\,dy\right)^\frac{1}{q}
\mathbf{1}_{(-\infty,-2\beta]}(\cdot)
\right\|_X\\
&\quad\leq\limsup_{\lambda\to0^+}
\frac{\lambda}{(-\gamma)^\frac{1}{q}}
\left[\left\||\cdot-\beta|^{\frac{\gamma}{q}}
\mathbf{1}_{[2\beta,\infty)}(\cdot)\right\|_X+
\left\||\cdot+\beta|^{\frac{\gamma}{q}}
\mathbf{1}_{(-\infty,-2\beta]}(\cdot)
\right\|_X\right].
\end{align*}
Thus, to prove \eqref{ggg}, it suffices to show that
\begin{align}\label{1650}
\left\||\cdot-\beta|^{\frac{\gamma}{q}}
\mathbf{1}_{[2\beta,\infty)}(\cdot)\right\|_X+
\left\||\cdot+\beta|^{\frac{\gamma}{q}}
\mathbf{1}_{(-\infty,-2\beta]}(\cdot)
\right\|_X<\infty.
\end{align}
Indeed, from an argument similar to that used in both the proof in Case 1)
in the present proof
and the estimation of \eqref{xzl},
we infer that \eqref{1650} holds true; we omit the details here.
This finishes the proof of \eqref{1144}
for any $f\in C^1(\mathbb{R})$ with $f'\in C_{\mathrm{c}}(\mathbb{R})$
in the case $n=1$, which completes the proof of
the upper estimate in (iii) for such an $f$
and hence the above claim.

Now, let $f\in\dot{W}^{1,X}(\mathbb{R}^n)$.
By Lemma~\ref{1516}, we conclude that there exists a
sequence $\{f_k\}_{k\in\mathbb{N}}\subset C^\infty(\mathbb{R}^n)$
with $|\nabla f_k|\in C_{\mathrm{c}}(\mathbb{R}^n)$
for any $k\in\mathbb{N}$ such that \eqref{dense} holds true.
From this, the above claim, and a density argument similar to
that used in the proof of Theorem~\ref{2117},
we further deduce that (i), (ii), and (iii)
hold true for any $f\in\dot{W}^{1,X}(\mathbb{R}^n)$.
This finishes the proof of Theorem~\ref{4.8}.
\end{proof}

\begin{remark}\label{1409}
\begin{enumerate}
\item[(i)]
By the claim in the proof of Theorem~\ref{4.8},
we conclude that (i), (ii), and (iii) of Theorem~\ref{4.8} still
hold true for any $f\in C^1(\mathbb{R}^n)$
with $|\nabla f|\in C_{\mathrm{c}}(\mathbb{R}^n)$
without the assumption that $X$ has an absolutely continuous norm.
\item[(ii)]
Theorem~\ref{4.8} when $\gamma=n$ coincides
with \cite[Theorem~4.10]{dlyyz.arxiv}.
\item[(iii)]
We should point out that
the lower estimates of (i), (ii), and (iii) of Theorem~\ref{4.8} do not need
the assumption \eqref{2114}.
\end{enumerate}
\end{remark}

\section{Applications to Fractional Sobolev-Type and Fractional\\
Gagliardo--Nirenberg-Type Inequalities}
\label{section4}

In this section, we use the generalized
Brezis--Van Schaftingen--Yung formulae
\eqref{2213} established in Section~\ref{S3}
to prove the fractional Sobolev-type and the
fractional Gagliardo--Nirenberg-type inequalities on
the ball Banach function space  $X$.
The following fractional Gagliardo--Nirenberg-type
inequalities involving both $\dot{W}^{1,X}(\mathbb{R}^n)$
and $X^p(\mathbb{R}^n)$ with $p\in[1,\infty)$
[or $L^\infty(\mathbb{R}^n)$ when $p=\infty$]
is one of the two main theorems
of this section.

\begin{theorem}\label{2255}
Let $X$ be a ball Banach function space,
$p\in[1,\infty]$, $\gamma\in\mathbb{R}\setminus\{0\}$,
$s\in(0,1)$, and $q\in[1,p]$ satisfy $\frac{1}{q}=\frac{1-s}{p}+s$.
Assume that there exists a sequence
$\{\theta_m\}_{m\in\mathbb{N}}\subset(0,1)$
satisfying $\lim_{m\to\infty}\theta_m=1$ such that,
for any $m\in\mathbb{N}$, $X^\frac{1}{\theta_m}$ is a ball Banach function
space and that the Hardy--Littlewood maximal operator $\mathcal{M}$ is bounded on
$(X^\frac{1}{\theta_m})'$ with \eqref{2114} holding true.
If $\gamma\in(-\infty,0)$, assume further
that $\mathcal{M}$ is bounded on $X$
or both $n=1$ and $\gamma\in(-\infty,-1)$.
Then the following statements hold true.
\begin{enumerate}
\item[\textup{(i)}]
If $p\in[1,\infty)$, then there exists a positive constant $C$ such that,
for any $f\in C^1(\mathbb{R}^n)$ with
$|\nabla f|\in C_{\mathrm{c}}(\mathbb{R}^n)$,
\begin{align}\label{2053}
\sup_{\lambda\in(0,\infty)}\lambda
\left\|\left[\int_{\mathbb{R}^n}
\mathbf{1}_{D_{\lambda,\frac{\gamma}{q},s}[f]}(\cdot,y)
\left|\cdot-y\right|^{\gamma-n}\,dy\right]^\frac{1}{q}\right\|_{X^q}
\leq C\|f\|_{X^p}^{1-s}\left\|\,\left|\nabla f\right|\,\right\|_X^s,
\end{align}
where, for any $\lambda\in(0,\infty)$,
\begin{align}\label{Dlambdagammaqs}
D_{\lambda,\frac{\gamma}{q},s}[f]
:=\left\{(x,y)\in\mathbb{R}^n\times\mathbb{R}^n:\
x\neq y,\ \frac{|f(x)-f(y)|}{|x-y|^{s+\frac{\gamma}{q}}}>\lambda\right\}.
\end{align}
Assume further that $X$ has an absolutely continuous norm,
$X^p$ is a ball Banach function space, and the centered ball average
operators $\{\mathcal{B}_r\}_{r\in(0,\infty)}$ are uniformly bounded on $X$.
Then \eqref{2053} also holds true for any $f\in\dot{W}^{1,X}(\mathbb{R}^n)$.
\item[\textup{(ii)}]
If $p=\infty$, then there exists a
positive constant $\widetilde{C}$ such that,
for any $f\in C^1(\mathbb{R}^n)$ with
$|\nabla f|\in C_{\mathrm{c}}(\mathbb{R}^n)$,
\begin{align}\label{2054}
\sup_{\lambda\in(0,\infty)}\lambda
\left\|\left[\int_{\mathbb{R}^n}
\mathbf{1}_{D_{\lambda,\frac{\gamma}{q},s}[f]}(\cdot,y)
\left|\cdot-y\right|^{\gamma-n}\,dy\right]^\frac{1}{q}\right\|_{X^q}
\leq\widetilde{C}\|f\|_{L^\infty(
\mathbb{R}^n)}^{1-s}\left\|\,\left|\nabla f\right|\,\right\|_X^s.
\end{align}
Assume further both that $X$ has an absolutely continuous norm
and that the centered ball average operators
$\{\mathcal{B}_r\}_{r\in(0,\infty)}$ are uniformly bounded on $X$.
Then \eqref{2054} also holds true
for any $f\in\dot{W}^{1,X}(\mathbb{R}^n)$.
\end{enumerate}
\end{theorem}

To show Theorem~\ref{2255}, we need the following lemma
which is just \cite[Corollary~3.10]{dgpyyz2022}.

\begin{lemma}\label{938}
Let $X$ be a ball Banach function space having an absolutely continuous norm.
Assume that the centered ball average
operators $\{\mathcal{B}_r\}_{r\in(0,\infty)}$
are uniformly bounded on $X$. Then $C^\infty_{\mathrm{c}}(\mathbb{R}^n)$
is dense in $X$.
\end{lemma}

\begin{proof}[Proof of Theorem~\ref{2255}]
We first show (i). To this end, let $p\in[1,\infty)$.
We first prove that, for any
$f\in C^1(\mathbb{R}^n)$ with
$|\nabla f|\in C_{\mathrm{c}}(\mathbb{R}^n)$,
\eqref{2053} holds true.
Since $\frac{1}{q}=\frac{1-s}{p}+s$,
it follows that, for any $x,y\in\mathbb{R}^n$ with $x\neq y$,
$$
\frac{|f(x)-f(y)|}{|x-y|^{s+\frac{\gamma}{q}}}
=\left[\frac{|f(x)-f(y)|}{|x-y|^\frac{\gamma}{p}}\right]^{1-s}
\left[\frac{|f(x)-f(y)|}{|x-y|^{1+\gamma}}\right]^s,
$$
which further implies that, for any $\lambda\in(0,\infty)$,
\begin{align*}
D_{\lambda,\frac{\gamma}{q},s}[f]
&\subset D_{A^{-s}\lambda,\frac{\gamma}{p},0}[f]\cup
D_{A^{1-s}\lambda,\gamma,1}[f]\\
&=D_{A^{-s}\lambda,\frac{\gamma}{p},0}[f]\cup
E_{A^{1-s}\lambda,\gamma}[f],
\end{align*}
where $A\in(0,\infty)$ is a constant specified later and
$E_{A^{1-s}\lambda,\gamma}[f]$ is the same as in \eqref{Elambda}
with $\lambda$ and $\frac{\gamma}{q}$ replaced, respectively,
by $A^{1-s}\lambda$ and $\gamma$.
By this, Definition~\ref{1659}(v),
and Lemma~\ref{dis} with $r:=\frac{1}{q}\in(0,1]$, we conclude that,
for any $\lambda\in(0,\infty)$,
\begin{align}\label{1514}
&\lambda\left\|\int_{\mathbb{R}^n}
\mathbf{1}_{D_{\lambda,\frac{\gamma}{q},s}[f]}(\cdot,y)
\left|\cdot-y\right|^{\gamma-n}\,dy\right\|_{X}^{\frac{1}{q}}\\
&\quad\leq\lambda
\left\|\int_{\mathbb{R}^n}
\mathbf{1}_{D_{A^{-s}\lambda,\frac{\gamma}{p},0}[f]}(\cdot,y)
\left|\cdot-y\right|^{\gamma-n}\,dy\right\|_{X}^{\frac{1}{q}}\nonumber\\
&\qquad+\lambda\left\|\int_{\mathbb{R}^n}
\mathbf{1}_{E_{A^{1-s}\lambda,\gamma}[f]}(\cdot,y)
\left|\cdot-y\right|^{\gamma-n}\,dy\right\|_{X}^{\frac{1}{q}}\nonumber\\
&\quad\leq\lambda^{1-\frac{p}{q}}
\left(A^s\mathrm{G}\right)^{\frac{p}{q}}
+\lambda^{1-\frac{1}{q}}
\left(A^{s-1}\mathrm{H}\right)^{\frac{1}{q}},\nonumber
\end{align}
where
$$
\mathrm{G}:=\sup_{\lambda\in(0,\infty)}\lambda
\left\|\int_{\mathbb{R}^n}
\mathbf{1}_{D_{\lambda,\frac{\gamma}{p},0}[f]}(\cdot,y)
\left|\cdot-y\right|^{\gamma-n}\,dy\right\|_{X}^{\frac{1}{p}}
$$
and
\begin{align}\label{2251}
\mathrm{H}:=\sup_{\lambda\in(0,\infty)}\lambda
\left\|\int_{\mathbb{R}^n}
\mathbf{1}_{E_{\lambda,\gamma}[f]}(\cdot,y)
\left|\cdot-y\right|^{\gamma-n}\,dy\right\|_{X}.
\end{align}
Choose an $A\in(0,\infty)$ such that
$$
\lambda^{1-\frac{p}{q}}
\left(A^s\mathrm{G}\right)^{\frac{p}{q}}
=\lambda^{1-\frac{1}{q}}
\left(A^{s-1}\mathrm{H}\right)^{\frac{1}{q}}.
$$
From this, \eqref{1514},
$\frac{1}{q}=\frac{1-s}{p}+s$, Remark~\ref{1409}(i),
and Theorem~\ref{4.8}(i) with $q:=1$,
we infer that
\begin{align}\label{1550}
&\sup_{\lambda\in(0,\infty)}\lambda\left\|\int_{\mathbb{R}^n}
\mathbf{1}_{D_{\lambda,\frac{\gamma}{q},s}[f]}(\cdot,y)
\left|\cdot-y\right|^{\gamma-n}\,dy\right\|_{X}^{\frac{1}{q}}
\lesssim\mathrm{G}^{1-s}\mathrm{H}^s
\lesssim\mathrm{G}^{1-s}\left\|\,\left|\nabla f\right|\,\right\|_X^s.
\end{align}

Next, we claim that, for any
$f\in C^1(\mathbb{R}^n)$ with
$|\nabla f|\in C_{\mathrm{c}}(\mathbb{R}^n)$,
\begin{align}\label{4.6x}
\mathrm{G}\lesssim\|f\|_{X^p}.
\end{align}
By this claim and \eqref{1550},
we immediately conclude that \eqref{2053} holds true.
To show the above claim,
from Lemmas~\ref{1555} and~\ref{4.5}(i),
we deduce that, for any $\lambda\in(0,\infty)$ and $m\in\mathbb{N}$,
\begin{align}\label{1850}
&\left\|\int_{\mathbb{R}^n}
\mathbf{1}_{D_{\lambda,\frac{\gamma}{p},0}[f]}(\cdot,y)
\left|\cdot-y\right|^{\gamma-n}\,dy\right\|_{X^\frac{1}{\theta_m}}\\
&\quad=\sup_{\|g\|_{(X^\frac{1}{\theta_m})'}=1}
\int_{\mathbb{R}^n}\left[\int_{\mathbb{R}^n}
\mathbf{1}_{D_{\lambda,\frac{\gamma}{p},0}[f]}(x,y)
\left|x-y\right|^{\gamma-n}\,dy\right]g(x)\,dx\nonumber\\
&\quad\leq\sup_{\|g\|_{(X^\frac{1}{\theta_m})'}=1}
\int_{\mathbb{R}^n}\left[\int_{\mathbb{R}^n}
\mathbf{1}_{D_{\lambda,\frac{\gamma}{p},0}[f]}(x,y)
\left|x-y\right|^{\gamma-n}\,dy\right]
R_{(X^\frac{1}{\theta_m})'}g(x)\,dx\nonumber\\
&\quad\leq\sup_{\|g\|_{(X^\frac{1}{\theta_m})'}=1}
\int_{\mathbb{R}^n}\left[\int_{\mathbb{R}^n}
\mathbf{1}_{D_1}(x,y)
\left|x-y\right|^{\gamma-n}\,dy\right]
R_{(X^\frac{1}{\theta_m})'}g(x)\,dx\nonumber\\
&\qquad+\sup_{\|g\|_{(X^\frac{1}{\theta_m})'}=1}
\int_{\mathbb{R}^n}\left[\int_{\mathbb{R}^n}
\mathbf{1}_{D_2}(x,y)
\left|x-y\right|^{\gamma-n}\,dy\right]
R_{(X^\frac{1}{\theta_m})'}g(x)\,dx,\nonumber
\end{align}
where
$$
D_1:=\left\{(x,y)\in\mathbb{R}^n\times\mathbb{R}^n:\
x\neq y,\
\frac{|f(x)|}{|x-y|^\frac{\gamma}{p}}>\frac{\lambda}{2}\right\}
$$
and
$$
D_2:=\left\{(x,y)\in\mathbb{R}^n\times\mathbb{R}^n:\
x\neq y,\
\frac{|f(y)|}{|x-y|^\frac{\gamma}{p}}>\frac{\lambda}{2}\right\}.
$$
Then we consider the following two cases on $\gamma$.

\emph{Case 1)} $\gamma\in(0,\infty)$. In this case,
on the one hand, by the polar coordinate and
Lemmas~\ref{1639} and~\ref{4.5}(iii),
we find that, for any $\lambda\in(0,\infty)$,
\begin{align}\label{1851}
&\sup_{\|g\|_{(X^\frac{1}{\theta_m})'}=1}
\int_{\mathbb{R}^n}\left[\int_{\mathbb{R}^n}
\mathbf{1}_{D_1}(x,y)
\left|x-y\right|^{\gamma-n}\,dy\right]
R_{(X^\frac{1}{\theta_m})'}g(x)\,dx\\
&\quad=\sup_{\|g\|_{(X^\frac{1}{\theta_m})'}=1}
\int_{\mathbb{R}^n}\left(\int_{\mathbb{S}^{n-1}}
\int_0^{[\frac{2|f(x)|}{\lambda}]^\frac{p}{\gamma}}
r^{\gamma-1}\,dr\,d\sigma\right)
R_{(X^\frac{1}{\theta_m})'}g(x)\,dx\nonumber\\
&\quad\sim\lambda^{-p}\sup_{\|g\|_{(X^\frac{1}{\theta_m})'}=1}
\int_{\mathbb{R}^n}\left|f(x)\right|^p
R_{(X^\frac{1}{\theta_m})'}g(x)\,dx\nonumber\\
&\quad\leq\lambda^{-p}\|f\|_{X^\frac{p}{\theta_m}}^{p}
\sup_{\|g\|_{(X^\frac{1}{\theta_m})'}=1}\left\|
R_{(X^\frac{1}{\theta_m})'}g\right\|_{(X^\frac{1}{\theta_m})'}
\lesssim\lambda^{-p}\|f\|_{X^\frac{p}{\theta_m}}^{p}.\nonumber
\end{align}
On the other hand, for any given $y\in\mathbb{R}^n$,
let
$$
B_y:=B\left(y,\left[2\lambda^{-1}|f(y)|\right]^\frac{p}{\gamma}\right).
$$
From the Tonelli theorem, Lemma~\ref{ApProperty}(i),
both (ii) and (iii) of Lemma~\ref{4.5},
and Lemma~\ref{1639},
we deduce that, for any $\lambda\in(0,\infty)$,
\begin{align*}
&\sup_{\|g\|_{(X^\frac{1}{\theta_m})'}=1}
\int_{\mathbb{R}^n}\left[\int_{\mathbb{R}^n}
\mathbf{1}_{D_2}(x,y)
\left|x-y\right|^{\gamma-n}\,dy\right]
R_{(X^\frac{1}{\theta_m})'}g(x)\,dx\\
&\quad=\sup_{\|g\|_{(X^\frac{1}{\theta_m})'}=1}
\int_{\mathbb{R}^n}\sum_{j=0}^\infty\int_{2^{-j}B_y\setminus
2^{-j-1}B_y}
\left|x-y\right|^{\gamma-n}R_{(X^\frac{1}{\theta_m})'}
g(x)\,dx\,dy\nonumber\\
&\quad\lesssim\sup_{\|g\|_{(X^\frac{1}{\theta_m})'}=1}
\int_{\mathbb{R}^n}\sum_{j=0}^\infty\left\{2^{-j}
\left[\frac{|f(y)|}{\lambda}\right]^\frac{p}{\gamma}\right\}^{\gamma-n}
\int_{2^{-j}B_y}R_{(X^\frac{1}{\theta_m})'}g(x)\,dx\,dy\nonumber\\
&\quad\leq\sup_{\|g\|_{(X^\frac{1}{\theta_m})'}=1}
\left[R_{(X^\frac{1}{\theta_m})'}g
\right]_{A_1(\mathbb{R}^n)}
\int_{\mathbb{R}^n}\sum_{j=0}^\infty\left\{2^{-j}
\left[\frac{|f(y)|}{\lambda}\right]^\frac{p}{\gamma}
\right\}^{\gamma-n}\nonumber\\
&\qquad\times\left|2^{-j}B_y\right|
\inf_{x\in2^{-j}B_y}R_{(X^\frac{1}{\theta_m})'}g(x)\,dy\nonumber\\
&\quad\lesssim\left\|\mathcal{M}\right\|_{(X^\frac{1}{\theta_m})'\to
(X^\frac{1}{\theta_m})'}
\sup_{\|g\|_{(X^\frac{1}{\theta_m})'}=1}
\int_{\mathbb{R}^n}\sum_{j=0}^\infty2^{-\gamma j}
\left[\frac{|f(y)|}{\lambda}\right]^{p}
R_{(X^\frac{1}{\theta_m})'}g(y)\,dy\nonumber\\
&\quad\lesssim\left\|\mathcal{M}\right\|_{(X^\frac{1}{\theta_m})'\to
(X^\frac{1}{\theta_m})'}\lambda^{-p}\|f\|_{X^\frac{p}{\theta_m}}^{p}
\sup_{\|g\|_{(X^\frac{1}{\theta_m})'}=1}\left\|
R_{(X^\frac{1}{\theta_m})'}g\right\|_{(X^\frac{1}{\theta_m})'}\nonumber\\
&\quad\lesssim\left\|\mathcal{M}\right\|_{(X^\frac{1}{\theta_m})'\to
(X^\frac{1}{\theta_m})'}\lambda^{-p}
\|f\|_{X^\frac{p}{\theta_m}}^{p},\nonumber
\end{align*}
which, combined with \eqref{1850}, \eqref{1851},
and the dominated convergence theorem,
further implies that
\begin{align}\label{1917}
\mathrm{G}&=\sup_{\lambda\in(0,\infty)}\lambda
\left\|\lim_{m\to\infty}\left[\int_{\mathbb{R}^n}
\mathbf{1}_{D_{\lambda,\frac{\gamma}{p},0}[f]}(\cdot,y)
\left|\cdot-y\right|^{\gamma-n}\,dy\right]^\frac{1}{\theta_m}
\right\|_{X}^{\frac{1}{p}}\\
&\leq\sup_{\lambda\in(0,\infty)}\lambda
\liminf_{m\to\infty}\left\|\left[\int_{\mathbb{R}^n}
\mathbf{1}_{D_{\lambda,\frac{\gamma}{p},0}[f]}(\cdot,y)
\left|\cdot-y\right|^{\gamma-n}\,dy\right]^\frac{1}{\theta_m}
\right\|_{X}^{\frac{1}{p}}\nonumber\\
&\lesssim\sup_{\lambda\in(0,\infty)}\lambda
\liminf_{m\to\infty}\left\{
\left[1+\left\|\mathcal{M}\right\|_{(X^\frac{1}{\theta_m})'\to
(X^\frac{1}{\theta_m})'}\right]\lambda^{-p}\|f\|_{X^\frac{p}{\theta_m}}^{p}
\right\}^{\frac{1}{p\theta_m}}\nonumber\\
&\lesssim\liminf_{m\to\infty}
\left\|\left|f\right|^\frac{p}{\theta_m}\right\|_{X}^{\frac{\theta_m}{p}}
=\left\|f\right\|_{X^p}.\nonumber
\end{align}
This finishes the proof of the above claim in the case
$\gamma\in(0,\infty)$.

\emph{Case 2)} $\gamma\in(-\infty,0)$. In this case,
on the one hand, by the polar coordinate and
Lemmas~\ref{1639} and~\ref{4.5}(iii),
we find that, for any $\lambda\in(0,\infty)$,
\begin{align}\label{1911}
&\sup_{\|g\|_{(X^\frac{1}{\theta_m})'}=1}
\int_{\mathbb{R}^n}\left[\int_{\mathbb{R}^n}
\mathbf{1}_{D_1}(x,y)
\left|x-y\right|^{\gamma-n}\,dy\right]
R_{(X^\frac{1}{\theta_m})'}g(x)\,dx\\
&\quad=\sup_{\|g\|_{(X^\frac{1}{\theta_m})'}=1}
\int_{\mathbb{R}^n}\left[\int_{\mathbb{S}^{n-1}}
\int_{[\frac{2|f(x)|}{\lambda}]^\frac{p}{\gamma}}^\infty
r^{\gamma-1}\,dr\,d\sigma\right]
R_{(X^\frac{1}{\theta_m})'}g(x)\,dx\nonumber\\
&\quad\sim\lambda^{-p}\sup_{\|g\|_{(X^\frac{1}{\theta_m})'}=1}
\int_{\mathbb{R}^n}\left|f(x)\right|^p
R_{(X^\frac{1}{\theta_m})'}g(x)\,dx\nonumber\\
&\quad\leq\lambda^{-p}\|f\|_{X^\frac{p}{\theta_m}}^{p}
\sup_{\|g\|_{(X^\frac{1}{\theta_m})'}=1}\left\|
R_{(X^\frac{1}{\theta_m})'}g\right\|_{(X^\frac{1}{\theta_m})'}
\lesssim\lambda^{-p}\|f\|_{X^\frac{p}{\theta_m}}^{p}.\nonumber
\end{align}
On the other hand,
from the Tonelli theorem, Lemma~\ref{ApProperty}(i),
both (ii) and (iii) of Lemma~\ref{4.5},
and Lemma~\ref{1639},
we deduce that, for any $\lambda\in(0,\infty)$,
\begin{align*}
&\sup_{\|g\|_{(X^\frac{1}{\theta_m})'}=1}
\int_{\mathbb{R}^n}\left[\int_{\mathbb{R}^n}
\mathbf{1}_{D_2}(x,y)
\left|x-y\right|^{\gamma-n}\,dy\right]
R_{(X^\frac{1}{\theta_m})'}g(x)\,dx\\
&\quad=\sup_{\|g\|_{(X^\frac{1}{\theta_m})'}=1}
\int_{\mathbb{R}^n}\sum_{j=1}^\infty\int_{2^{j}B_y\setminus
2^{j-1}B_y}
\left|x-y\right|^{\gamma-n}
R_{(X^\frac{1}{\theta_m})'}g(x)\,dx\,dy\nonumber\\
&\quad\lesssim\sup_{\|g\|_{(X^\frac{1}{\theta_m})'}=1}
\int_{\mathbb{R}^n}\sum_{j=1}^\infty\left\{2^{j-1}
\left[\frac{|f(y)|}{\lambda}\right]^\frac{p}{\gamma}\right\}^{\gamma-n}
\int_{2^{j}B_y}R_{(X^\frac{1}{\theta_m})'}g(x)\,dx\,dy\nonumber\\
&\quad\leq\sup_{\|g\|_{(X^\frac{1}{\theta_m})'}=1}
\left[R_{(X^\frac{1}{\theta_m})'}g
\right]_{A_1(\mathbb{R}^n)}
\int_{\mathbb{R}^n}\sum_{j=1}^\infty\left\{2^j
\left[\frac{|f(y)|}{\lambda}\right]^\frac{p}{\gamma}
\right\}^{\gamma-n}\nonumber\\
&\qquad\quad\times
\left|2^{j}B_y\right|
\inf_{x\in2^{j}B_y}R_{(X^\frac{1}{\theta_m})'}g(x)\,dy\nonumber\\
&\quad\lesssim\left\|\mathcal{M}\right\|_{(X^\frac{1}{\theta_m})'\to
(X^\frac{1}{\theta_m})'}
\sup_{\|g\|_{(X^\frac{1}{\theta_m})'}=1}
\int_{\mathbb{R}^n}\sum_{j=1}^\infty2^{\gamma j}
\left[\frac{|f(y)|}{\lambda}\right]^{p}
R_{(X^\frac{1}{\theta_m})'}g(y)\,dy\nonumber\\
&\quad\lesssim\left\|\mathcal{M}\right\|_{(X^\frac{1}{\theta_m})'\to
(X^\frac{1}{\theta_m})'}\lambda^{-p}\|f\|_{X^\frac{p}{\theta_m}}^{p}
\sup_{\|g\|_{(X^\frac{1}{\theta_m})'}=1}\left\|
R_{(X^\frac{1}{\theta_m})'}g\right\|_{(X^\frac{1}{\theta_m})'}\nonumber\\
&\quad\lesssim\left\|\mathcal{M}\right\|_{(X^\frac{1}{\theta_m})'\to
(X^\frac{1}{\theta_m})'}\lambda^{-p}\|f\|_{X^\frac{p}{\theta_m}}^{p},\nonumber
\end{align*}
which, together with \eqref{1850}, \eqref{1911},
and an argument similar to that used in the estimation of \eqref{1917},
further implies that $\mathrm{G}\lesssim\left\|f\right\|_{X^p}$.
This, combined with \eqref{1917}, then finishes the proof
of the above claim and hence \eqref{2053} for any
$f\in C^1(\mathbb{R}^n)$ with $|\nabla f|\in C_{\mathrm{c}}(\mathbb{R}^n)$.

Now, we prove \eqref{2053} for any $f\in\dot{W}^{1,X}(\mathbb{R}^n)$
under the additional assumptions in (i).
From an argument similar to that used in the estimation of
\eqref{1550} and Theorems~\ref{4.8}(i)
with $q:=1$,
we infer that, for any $f\in\dot{W}^{1,X}(\mathbb{R}^n)$,
$$
\sup_{\lambda\in(0,\infty)}\lambda\left\|\int_{\mathbb{R}^n}
\mathbf{1}_{D_{\lambda,\frac{\gamma}{q},s}[f]}(\cdot,y)
\left|\cdot-y\right|^{\gamma-n}\,dy\right\|_{X}^{\frac{1}{q}}
\lesssim\mathrm{G}^{1-s}\left\|\,\left|\nabla f\right|\,\right\|_X^s.
$$
Thus, to complete the proof of (i),
it suffices to show that
$\mathrm{G}\lesssim\|f\|_{X^p}$
for any $f\in\dot{W}^{1,X}(\mathbb{R}^n)\cap X^p$.
Indeed, if $f\notin X^p$, then \eqref{2053} holds true automatically.
Let $f\in\dot{W}^{1,X}(\mathbb{R}^n)\cap X^p$.
By the assumptions that both
$X$ has an absolutely continuous norm and
$\{\mathcal{B}_r\}_{r\in(0,\infty)}$ are uniformly bounded on $X$,
it is easy to prove that
both $X^p$ also has an absolutely continuous norm
and
$\{\mathcal{B}_r\}_{r\in(0,\infty)}$ are also uniformly bounded on $X^p$
via their definitions directly.
These, together with both that $X^p$ is a ball Banach function space
and Lemma~\ref{938},
further imply that $C_{\mathrm{c}}^\infty(\mathbb{R}^n)$ is
dense in $X^p$. From this, we deduce that there exists a
sequence $\{f_k\}_{k\in\mathbb{N}}\subset
C_{\mathrm{c}}^\infty(\mathbb{R}^n)$ such that
\begin{align}\label{2103}
\lim_{k\to\infty}\left\|f-f_k\right\|_{X^p}=0.
\end{align}
For any $N\in\mathbb{N}$, let
$$
\mathcal{F}_N:=\left\{(x,y)\in B(\mathbf{0},N)\times
B(\mathbf{0},N):\ |x-y|\ge\frac{1}{N}\right\}.
$$
Notice that, for any $\lambda\in(0,\infty)$ and $k\in\mathbb{N}$,
\begin{align*}
D_{\lambda,\frac{\gamma}{p},0}[f]\cap\mathcal{F}_N
&\subset\left[\left\{(x,y)\in\mathbb{R}^n\times\mathbb{R}^n:\
x\neq y,\,\frac{|f(x)-f_k(x)|}{|x-y|^\frac{\gamma}{p}}>\frac{\lambda}{3}\right\}
\cap\mathcal{F}_N\right]\nonumber\\
&\quad\cup\left[\left\{(x,y)\in\mathbb{R}^n\times\mathbb{R}^n:\
x\neq y,\,\frac{|f_k(x)-f_k(y)|}{|x-y|^\frac{\gamma}{p}}>\frac{\lambda}{3}\right\}
\cap\mathcal{F}_N\right]\nonumber\\
&\quad\cup\left[\left\{(x,y)\in\mathbb{R}^n\times\mathbb{R}^n:\
x\neq y,\,\frac{|f_k(y)-f(y)|}{|x-y|^\frac{\gamma}{p}}>\frac{\lambda}{3}\right\}
\cap\mathcal{F}_N\right]\nonumber\\
&=:E_1\cup E_2\cup E_3,
\end{align*}
which, combined with both Definition~\ref{1659}(v) and
Lemma~\ref{dis} with $r:=\frac{1}{p}\in(0,1]$, implies that
\begin{align}\label{201}
&\sup_{\lambda\in(0,\infty)}\lambda\left\|\int_{\mathbb{R}^n}
\mathbf{1}_{D_{\lambda,\frac{\gamma}{p},0}[f]\cap\mathcal{F}_N}(\cdot,y)
\left|\cdot-y\right|^{\gamma-n}\,dy\right\|_{X}^{\frac{1}{p}}\\
&\quad\leq\sum_{i=1}^3\sup_{\lambda\in(0,\infty)}\lambda\left\|\int_{\mathbb{R}^n}
\mathbf{1}_{E_i}(\cdot,y)
\left|\cdot-y\right|^{\gamma-n}\,dy\right\|_{X}^{\frac{1}{p}}
=:\sum_{i=1}^3\mathrm{I}_i.\nonumber
\end{align}
To deal with $\mathrm{I}_1$,
from the definitions of both $E_1$ and $\mathcal{F}_N$
and $p\in[1,\infty)$,
we infer that, for any $k,N\in\mathbb{N}$,
\begin{align}\label{202}
\mathrm{I}_1
&\lesssim\sup_{\lambda\in(0,\infty)}
\lambda\left\|\int_{B(\mathbf{0},N)\cap \mathcal{F}_N}
\left[\frac{|f(\cdot)-f_k(\cdot)|}{\lambda|\cdot-y|^\frac{\gamma}{p}}\right]^p
\left|\cdot-y\right|^{\gamma-n}\,dy\right\|_{X}^{\frac{1}{p}}\\
&=\left\|\left|f(\cdot)-f_k(\cdot)\right|^p
\int_{B(\mathbf{0},N)\cap \mathcal{F}_N}
|\cdot-y|^{-n}\,dy\right\|_{X}^{\frac{1}{p}}
\lesssim N^{\frac{2p}{n}}\left\|f-f_k\right\|_{X^p}.\nonumber
\end{align}
To deal with $\mathrm{I}_2$, by \eqref{4.6x} for
$\{f_k\}_{k\in\mathbb{N}}\subset
C_{\mathrm{c}}^\infty(\mathbb{R}^n)$
and Definition~\ref{1659}(v),
we find that, for any $k,N\in\mathbb{N}$,
\begin{align}\label{203}
\mathrm{I}_2
&\leq\sup_{\lambda\in(0,\infty)}\lambda\left\|\int_{\mathbb{R}^n}
\mathbf{1}_{D_{\lambda,\frac{\gamma}{p},0}[f_k]}(\cdot,y)
\left|\cdot-y\right|^{\gamma-n}\,dy\right\|_{X}^{\frac{1}{p}}\\
&\lesssim\left\|f_k\right\|_{X^p}
\leq\left\|f-f_k\right\|_{X^p}+\left\|f\right\|_{X^p}.\nonumber
\end{align}
To deal with $\mathrm{I}_3$,
from the definitions of both $E_3$ and $\mathcal{F}_N$,
$p\in[1,\infty)$,
and Lemma~\ref{1639},
we deduce that, for any $k,N\in\mathbb{N}$,
\begin{align*}
\mathrm{I}_3
&\lesssim\sup_{\lambda\in(0,\infty)}
\lambda\left\|\int_{B(\mathbf{0},N)\cap\mathcal{F}_N}
\left[\frac{|f(y)-f_k(y)|}{\lambda|x-y|^\frac{\gamma}{p}}\right]^p
\left|\cdot-y\right|^{\gamma-n}\,dy
\mathbf{1}_{B(\mathbf{0},N)}(\cdot)\right\|_{X}^{\frac{1}{p}}\nonumber\\
&\lesssim N^n\left\|\mathbf{1}_{B(\mathbf{0},N)}\right\|_{X}^{\frac{1}{p}}
\left[\int_{B(\mathbf{0},N)}\left|f(y)-f_k(y)\right|^p
\,dy\right]^{\frac{1}{p}}\\
&\leq N^n\left\|\mathbf{1}_{B(\mathbf{0},N)}\right\|_{X}^{\frac{1}{p}}
\left\|\mathbf{1}_{B(\mathbf{0},N)}\right\|_{X'}^{\frac{1}{p}}
\left\|f-f_k\right\|_{X^p},
\end{align*}
which, together with \eqref{201}, \eqref{202},
and \eqref{203}, further implies that
\begin{align*}
&\sup_{\lambda\in(0,\infty)}\lambda\left\|\int_{\mathbb{R}^n}
\mathbf{1}_{D_{\lambda,\frac{\gamma}{p},0}[f]\cap\mathcal{F}_N}(\cdot,y)
\left|\cdot-y\right|^{\gamma-n}\,dy\right\|_{X}^{\frac{1}{p}}\\
&\quad\lesssim\left[1+N^{\frac{2p}{n}}
+N^n
\left\|\mathbf{1}_{B(\mathbf{0},N)}\right\|_{X'}^{\frac{1}{p}}
\left\|\mathbf{1}_{B(\mathbf{0},N)}
\right\|_{X}^{\frac{1}{p}}
\right]\left\|f-f_k\right\|_{X^p}+\left\|f\right\|_{X^p}.
\end{align*}
Using this and \eqref{2103}
and letting $k\to\infty$ and $N\to\infty$, we conclude that
$\mathrm{G}\lesssim\|f\|_{X^p}$
for any $f\in\dot{W}^{1,X}(\mathbb{R}^n)\cap X^p$ and hence
\eqref{2053} holds true for any $f\in\dot{W}^{1,X}(\mathbb{R}^n)$.
This finishes the proof of (i).

Next, we prove (ii). Let $f\in\dot{W}^{1,X}(\mathbb{R}^n)$.
If $\|f\|_{L^\infty(\mathbb{R}^n)}=0$
or $\|f\|_{L^\infty(\mathbb{R}^n)}=\infty$, then
\eqref{2054} holds true automatically.
Now, assume that $\|f\|_{L^\infty(\mathbb{R}^n)}\in(0,\infty)$.
Observe that, for any $(x,y)\in D_{\lambda,\gamma s,s}[f]$,
\begin{align*}
\left[\frac{|f(x)-f(y)|}{|x-y|^{\gamma+1}}\right]^{s}
\ge\frac{|f(x)-f(y)|}{|x-y|^{(\gamma+1)s}}
\left[2\|f\|_{L^\infty(\mathbb{R}^n)}\right]^{s-1}
>\frac{\lambda}{[2\|f\|_{L^\infty(\mathbb{R}^n)}]^{1-s}},
\end{align*}
which further implies that
\begin{align}\label{735}
D_{\lambda,\gamma s,s}[f]
\subset D_{\frac{\lambda^\frac{1}{s}}
{[2\|f\|_{L^\infty(\mathbb{R}^n)}]^\frac{1-s}{s}},\gamma,1}[f].
\end{align}
Notice that $p=\infty$ implies $qs=1$.
From this, \eqref{735}, and Definition~\ref{1659}(ii),
it follows that
\begin{align}\label{800}
&\sup_{\lambda\in(0,\infty)}\lambda
\left\|\left[\int_{\mathbb{R}^n}
\mathbf{1}_{D_{\lambda,\frac{\gamma}{q},s}[f]}(\cdot,y)
\left|\cdot-y\right|^{\gamma-n}\,dy
\right]^\frac{1}{q}\right\|_{X^q}\\
&\quad=\sup_{\lambda\in(0,\infty)}\lambda
\left\|\int_{\mathbb{R}^n}
\mathbf{1}_{D_{\lambda,\gamma s,s}[f]}(\cdot,y)
\left|\cdot-y\right|^{\gamma-n}\,dy\right\|_{X}^s\nonumber\\
&\quad\leq\sup_{\lambda\in(0,\infty)}\lambda
\left\|\int_{\mathbb{R}^n}
\mathbf{1}_{D_{\frac{\lambda^\frac{1}{s}}
{[2\|f\|_{L^\infty(\mathbb{R}^n)}]^\frac{1-s}{s}},\gamma,1}[f]}(\cdot,y)
\left|\cdot-y\right|^{\gamma-n}\,dy\right\|_{X}^s\nonumber\\
&\quad=\left[2\|f\|_{L^\infty(\mathbb{R}^n)}\right]^{1-s}
\sup_{\lambda\in(0,\infty)}\lambda^s
\left\|\int_{\mathbb{R}^n}
\mathbf{1}_{D_{\lambda,\gamma,1}[f]}(\cdot,y)
\left|\cdot-y\right|^{\gamma-n}\,dy\right\|_{X}^s\nonumber\\
&\quad=\left[2\|f\|_{L^\infty(\mathbb{R}^n)}\right]^{1-s}
\left[\sup_{\lambda\in(0,\infty)}\lambda
\left\|\int_{\mathbb{R}^n}
\mathbf{1}_{E_{\lambda,\gamma}[f]}(\cdot,y)
\left|\cdot-y\right|^{\gamma-n}\,dy\right\|_{X}\right]^s.\nonumber
\end{align}
This, combined with both Remark~\ref{1409}(i)
and Theorem~\ref{4.8}(i) with $q:=1$, further implies that,
for any $f\in C^1(\mathbb{R}^n)$ with
$|\nabla f|\in C_{\mathrm{c}}(\mathbb{R}^n)$,
\begin{align}\label{810}
\sup_{\lambda\in(0,\infty)}\lambda
\left\|\left[\int_{\mathbb{R}^n}
\mathbf{1}_{D_{\lambda,\frac{\gamma}{q},s}[f]}(\cdot,y)
\left|\cdot-y\right|^{\gamma-n}\,dy\right]^\frac{1}{q}\right\|_{X^q}
\lesssim
\|f\|_{L^\infty(\mathbb{R}^n)}^{1-s}
\left\|\,\left|\nabla f\right|\,\right\|_X^s.
\end{align}
Next, if  $X$ has an absolutely continuous norm
and
$\{\mathcal{B}_r\}_{r\in(0,\infty)}$ are uniformly bounded on $X$,
then Theorems~\ref{4.8}(i) with $q:=1$, together with \eqref{800},
further implies that \eqref{2054} holds true for any
$f\in\dot{W}^{1,X}(\mathbb{R}^n)$,
which, combined with \eqref{810}, then completes the proof of (ii).
This finishes the proof of Theorem~\ref{2255}.
\end{proof}

\begin{remark}\label{4.3}
\begin{enumerate}
\item[(i)]
Theorem~\ref{2255} when $\gamma=n$ coincides
with \cite[Corollary~4.13]{dlyyz.arxiv}. In this case,
if $p=\infty$, $q\in(1,\infty)$,
$X:=L^1(\mathbb{R}^n)$,
and $f\in C^\infty_{\mathrm{c}}(\mathbb{R}^n)$,
then Theorem~\ref{2255}(ii) is just \cite[Corollary~5.1]{bvy2021}.
To the best of our knowledge,
Theorem~\ref{2255} when $\gamma\in\mathbb{R}\setminus\{0,n\}$
is new.
\item[(ii)]
If $p=\infty$, $q\in(1,\infty)$,
and $X:=L^1(\mathbb{R})$,
then Theorem~\ref{2255}(ii)
implies that, for any $f\in C^\infty_{\mathrm{c}}(\mathbb{R}^n)$,
$$
\left\|\frac{f(x)-f(y)}{|x-y|^\frac{2}{q}}
\right\|_{L^{q,\infty}(\mathbb{R}\times\mathbb{R},
\nu_\gamma)}
\lesssim\|f\|_{L^\infty(\mathbb{R})}^{1-\frac{1}{q}}
\left\|f'\right\|_{L^1(\mathbb{R})}^\frac{1}{q}\lesssim
\left\|f'\right\|_{L^1(\mathbb{R})}
$$
with the measure $\nu_\gamma$ the same as in \eqref{nu}
and $\gamma$ the same as in Theorem~\ref{2255},
which when $\gamma=1$ is just \cite[Corollary~4.1]{bvy2021}
and which when $\gamma=1$ and $q=2$
is originally due to \cite[Theorem~3.1]{gs2020}.
\item[(iii)]
As was pointed out in \cite[(ii) and (iii) of Remark~4.14]{dlyyz.arxiv},
even when $\gamma=n$ and $p\in[1,\infty)$,
it is still unclear whether or not Theorem~\ref{2255}(i) holds true
with the left-hand side of \eqref{2053}
replaced by
$$
\left\|\left[\int_{\mathbb{R}^n}
\frac{|f(\cdot)-f(y)|^q}{|\cdot-y|^{n+sq}}\,dy
\right]^{\frac{1}{q}}\right\|_{X^q}.
$$
However, when $\gamma=n$ and $p=\infty$,
\eqref{2054} may not hold true with the same modification as above
(see, for instance, \cite[(5.3)]{bvy2021}).
In this sense, \eqref{2054} seems to be \emph{sharp}.
\item[(iv)]
Let $X$ be the same as in Theorem~\ref{2117}.
Then both \eqref{2053} and \eqref{2054} hold true
for any $f\in\dot{W}^{1,X}(\mathbb{R}^n)$,
which can be proved by a slight modification of the proof of
Theorem~\ref{2255} with Theorem~\ref{4.8}(i)
replaced by Theorem~\ref{2117}(i).
\end{enumerate}
\end{remark}

The following fractional Gagliardo--Nirenberg-type
inequalities involving both $\dot{W}^{1,X}(\mathbb{R}^n)$
and its fractional version
is another main theorem
of this section.

\begin{theorem}\label{2256}
Let $\gamma\in\mathbb{R}\setminus\{0\}$, $\eta\in(0,1)$,
and
$0\leq s_0<s<1<q<q_0<\infty$ satisfy
\begin{align}\label{2228}
s=(1-\eta)s_0+\eta
\quad\text{and}\quad
\frac{1}{q}=\frac{1-\eta}{q_0}+\eta.
\end{align}
Let X be a ball Banach function space.
Assume that there exists a sequence
$\{\theta_m\}_{m\in\mathbb{N}}\subset(0,1)$
satisfying $\lim_{m\to\infty}\theta_m=1$ such that,
for any $m\in\mathbb{N}$, $X^\frac{1}{\theta_m}$
is a ball Banach function space and that
the Hardy--Littlewood maximal operator $\mathcal{M}$
is bounded on $(X^\frac{1}{\theta_m})'$ with
\eqref{2114} holding true.
If $\gamma\in(-\infty,0)$, assume further
that $\mathcal{M}$ is bounded on $X$
or both $n=1$ and $\gamma\in(-\infty,-1)$.
Then there exists a positive constant $C$ such that,
for any $f\in C^1(\mathbb{R}^n)$ with
$|\nabla f|\in C_{\mathrm{c}}(\mathbb{R}^n)$,
\begin{align}\label{939}
&\sup_{\lambda\in(0,\infty)}\lambda
\left\|\int_{\mathbb{R}^n}
\mathbf{1}_{D_{\lambda,\frac{\gamma}{q},s}[f]}(\cdot,y)
\left|\cdot-y\right|^{\gamma-n}\,dy\right\|_{X}^{\frac{1}{q}}\\
&\quad\leq C
\sup_{\lambda\in(0,\infty)}\lambda
\left\|\int_{\mathbb{R}^n}
\mathbf{1}_{D_{\lambda,\frac{\gamma}{q_0},s_0}[f]}(\cdot,y)
\left|\cdot-y\right|^{\gamma-n}\,dy\right\|_{X}^{\frac{1-\eta}{q_0}}
\left\|\,\left|\nabla f\right|\,\right\|_X^\eta,\nonumber
\end{align}
where, for any $\lambda\in(0,\infty)$,
$D_{\lambda,\frac{\gamma}{q},s}[f]$ is the same as in \eqref{Dlambdagammaqs}.
Assume further both that $X$ has an
absolutely continuous norm and that
the centered ball average operators
$\{\mathcal{B}_r\}_{r\in(0,\infty)}$ are uniformly bounded on $X$.
Then \eqref{939} also holds true
for any $f\in\dot{W}^{1,X}(\mathbb{R}^n)$.
\end{theorem}

\begin{proof}
First, let $f\in C^1(\mathbb{R}^n)$
with $|\nabla f|\in C_{\mathrm{c}}(\mathbb{R}^n)$.
From \eqref{2228},
it follows that, for any $x,y\in\mathbb{R}^n$ with $x\neq y$,
$$
\frac{|f(x)-f(y)|}{|x-y|^{s+\frac{\gamma}{q}}}
=\left[\frac{|f(x)-f(y)|}{|x-y|^{s_0+\frac{\gamma}{q_0}}}\right]^{1-\eta}
\left[\frac{|f(x)-f(y)|}{|x-y|^{\gamma+1}}\right]^\eta,
$$
which implies that, for any $\lambda\in(0,\infty)$,
$$
D_{\lambda,\frac{\gamma}{q},s}[f]
\subset\left(D_{A^{-\eta}\lambda,\frac{\gamma}{q_0},s_0}[f]\cup
E_{A^{1-\eta}\lambda,\gamma}[f]\right)
$$ with
$A\in(0,\infty)$ being a constant specified later,
where $E_{A^{1-\eta}\lambda,\gamma}[f]$ is the
same as in \eqref{Elambda} with $\lambda$ and $\frac{\gamma}{q}$ replaced,
respectively, by $A^{1-\eta}\lambda$ and $\gamma$.
By this, Definition~\ref{1659}(v),
and Lemma~\ref{dis} with $r:=\frac{1}{q}\in(0,1)$, we find that,
for any $\lambda\in(0,\infty)$,
\begin{align}\label{2232}
&\lambda\left\|\int_{\mathbb{R}^n}
\mathbf{1}_{D_{\lambda,\frac{\gamma}{q},s}[f]}(\cdot,y)
\left|\cdot-y\right|^{\gamma-n}\,dy\right\|_{X}^{\frac{1}{q}}\\
&\quad\leq\lambda
\left\|\int_{\mathbb{R}^n}
\mathbf{1}_{D_{A^{-\eta}\lambda,\frac{\gamma}{q_0},s_0}[f]}(\cdot,y)
\left|\cdot-y\right|^{\gamma-n}\,dy\right\|_{X}^{\frac{1}{q}}\nonumber\\
&\qquad+\lambda\left\|\int_{\mathbb{R}^n}
\mathbf{1}_{E_{A^{1-\eta}\lambda,\gamma}[f]}(\cdot,y)
\left|\cdot-y\right|^{\gamma-n}\,dy\right\|_{X}^{\frac{1}{q}}\nonumber\\
&\quad\leq\lambda^{1-\frac{q_0}{q}}
\left(A^\eta\mathrm{G}\right)^{\frac{q_0}{q}}
+\lambda^{1-\frac{1}{q}}
\left(A^{\eta-1}\mathrm{H}\right)^{\frac{1}{q}},\nonumber
\end{align}
where
$$
\mathrm{G}:=\sup_{\lambda\in(0,\infty)}\lambda
\left\|\int_{\mathbb{R}^n}
\mathbf{1}_{D_{\lambda,\frac{\gamma}{q_0},s_0}[f]}(\cdot,y)
\left|\cdot-y\right|^{\gamma-n}\,dy\right\|_{X}^{\frac{1}{q_0}}
$$
and $\mathrm{H}$ is the same as in \eqref{2251}.
Choose an $A\in(0,\infty)$ such that
$$
\lambda^{1-\frac{q_0}{q}}
\left(A^\eta\mathrm{G}\right)^{\frac{q_0}{q}}
=\lambda^{1-\frac{1}{q}}
\left(A^{\eta-1}\mathrm{H}\right)^{\frac{1}{q}}.
$$
From this, \eqref{2232},
\eqref{2228}, Remark~\ref{1409}(i),
and Theorem~\ref{4.8}(i) with $q:=1$,
we deduce that,
for any $f\in C^1(\mathbb{R}^n)$ with
$|\nabla f|\in C_{\mathrm{c}}(\mathbb{R}^n)$,
\begin{align}\label{2242}
&\sup_{\lambda\in(0,\infty)}\lambda\left\|\int_{\mathbb{R}^n}
\mathbf{1}_{D_{\lambda,\frac{\gamma}{q},s}[f]}(\cdot,y)
\left|\cdot-y\right|^{\gamma-n}\,dy\right\|_{X}^{\frac{1}{q}}\\
&\quad\lesssim\sup_{\lambda\in(0,\infty)}\lambda^{1-\frac{q_0}{q}}
\left(A^\eta\mathrm{G}\right)^{\frac{q_0}{q}}
=\mathrm{G}^{1-\eta}\mathrm{H}^\eta
\lesssim\mathrm{G}^{1-\eta}
\left\|\,\left|\nabla f\right|\,\right\|_X^\eta.\nonumber
\end{align}

Next, we prove \eqref{939} for any $f\in\dot{W}^{1,X}(\mathbb{R}^n)$
under additional assumptions that $X$ has an
absolutely continuous norm and that
centered ball average operators
$\{\mathcal{B}_r\}_{r\in(0,\infty)}$ are uniformly bounded on $X$.
Indeed, by Theorem~\ref{4.8}(i) with $q:=1$
and an argument similar to that used in the estimation of \eqref{2242},
we conclude that
$$
\sup_{\lambda\in(0,\infty)}\lambda\left\|\int_{\mathbb{R}^n}
\mathbf{1}_{D_{\lambda,\frac{\gamma}{q},s}[f]}(\cdot,y)
\left|\cdot-y\right|^{\gamma-n}\,dy\right\|_{X}^{\frac{1}{q}}
\lesssim\mathrm{G}^{1-\eta}\mathrm{H}^\eta
\lesssim\mathrm{G}^{1-\eta}
\left\|\,\left|\nabla f\right|\,\right\|_X^\eta,
$$
which completes the proof of \eqref{939} for any
$f\in\dot{W}^{1,X}(\mathbb{R}^n)$ and hence Theorem~\ref{2256}.
\end{proof}

\begin{remark}\label{2021}
\begin{enumerate}
\item[(i)]
Theorem~\ref{2256} when $\gamma=n$ coincides
with \cite[Corollary~4.15]{dlyyz.arxiv}. In this case,
if $X:=L^p(\mathbb{R}^n)$ with $p\in[1,\infty)$,
then Theorem~\ref{2256} with $X$ replaced
by $L^p(\mathbb{R}^n)$ holds true.
Moreover, when $\gamma=n$, $X:=L^1(\mathbb{R}^n)$,
and $f\in C^\infty_{\mathrm{c}}(\mathbb{R}^n)$,
Theorem~\ref{2256} is just \cite[Corollary~5.2]{bvy2021}.
To the best of our knowledge,
Theorem~\ref{2256} when $\gamma\in\mathbb{R}\setminus\{0,n\}$
is new.
\item[(ii)]
Let $X$ be the same as in Theorem~\ref{2117}.
Then \eqref{939} holds true
for any $f\in\dot{W}^{1,X}(\mathbb{R}^n)$,
which can be proved by a slight modification of the proof of
Theorem~\ref{2256} with Theorem~\ref{4.8}(i)
replaced by Theorem~\ref{2117}(i).
\end{enumerate}
\end{remark}

\section{Applications in Specific Function Spaces}
\label{S5}

In this section, we apply Theorems~\ref{2117},
\ref{4.8},~\ref{2255}, and~\ref{2256},
respectively, to seven specific examples of ball
Banach function spaces, namely Morrey spaces
(see Subsection~\ref{5.1} below),
mixed-norm Lebesgue spaces (see Subsection~\ref{5.2} below),
variable Lebesgue spaces (see Subsection~\ref{5.3} below),
weighted Lebesgue spaces (see Subsection~\ref{5.4} below),
Lorentz spaces (see Subsection~\ref{5.7} below),
Orlicz spaces (see Subsection~\ref{5.5} below),
and Orlicz-slice spaces (see Subsection~\ref{5.6} below).

\subsection{Morrey Spaces}\label{5.1}

Recall that the \emph{Morrey space} $M_r^\alpha(\mathbb{R}^n)$,
with $0<r\leq\alpha<\infty$, is defined to be the set of
all the $f\in\mathscr{M}(\mathbb{R}^n)$ with
\begin{align*}
\|f\|_{M_r^\alpha(\mathbb{R}^n)}:=\sup_{B\in\mathbb{B}}
\left|B\right|^{\frac{1}{\alpha}-\frac{1}{r}}\|f\|_{L^r(B)}<\infty,
\end{align*}
where $\mathbb{B}$ is the same as in \eqref{1654}.
In 1938, Morrey \cite{m1938} introduced the above Morrey space
to study the regularity of the solution of partial differential equations.
These spaces nowadays have important applications in the theory of
elliptic partial differential equations, potential theory,
and harmonic analysis
(see, for instance, \cite{a2015,cf1987,hms2017,hms2020,hss2018,
hs2017,jw2009,sdh20201,sdh20202,tyy2019}).
As was indicated in \cite[p.\,87]{shyy2017},
$M_r^\alpha(\mathbb{R}^n)$ for any $1\leq r\leq\alpha<\infty$
is a ball Banach function space, but is not a Banach function space
in the terminology of Bennett and Sharpley \cite{bs1988}.

Using Theorems~\ref{upperBound},~\ref{2001},
~\ref{gamma>0},~\ref{n=1}, and~\ref{4.8},
Remark~\ref{1409}(i),
and Corollary~\ref{LimFormulaEq},
we obtain the following conclusion.

\begin{theorem}\label{1524}
Let $1\leq r\leq\alpha<\infty$.
If one of the following statements holds true:
\begin{enumerate}
\item[\textup{(a)}]
$r\in(n,\infty)$, $\gamma\in\mathbb{R}\setminus\{0\}$,
and $q\in(0,\infty)$;
\item[\textup{(b)}]
$r\in[1,n]$, $\gamma\in(0,\infty)$, and
$n(\frac{1}{r}-\frac{1}{q})<1$;
\item[\textup{(c)}]
$r\in(1,n]$, $\gamma\in(-\infty,0)$, and
$q\in(0,r)$;
\item[\textup{(d)}]
$r=n=q=1$ and $\gamma\in(-\infty,-1)$,
\end{enumerate}
then, for any $f\in C^1(\mathbb{R}^n)$ with
$|\nabla f|\in C_{\mathrm{c}}(\mathbb{R}^n)$,
\begin{align*}
\sup_{\lambda\in(0,\infty)}\lambda
\left\|\left[\int_{\mathbb{R}^n}
\mathbf{1}_{E_{\lambda,\frac{\gamma}{q}}[f]}(\cdot,y)
\left|\cdot-y\right|^{\gamma-n}\,dy\right]^\frac{1}{q}
\right\|_{M_r^\alpha(\mathbb{R}^n)}
\sim\left\|\,|\nabla f|\,\right\|_{M_r^\alpha(\mathbb{R}^n)},
\end{align*}
where the positive equivalence constants are independent of $f$
and $E_{\lambda,\frac{\gamma}{q}}[f]$
for any $\lambda\in(0,\infty)$
is the same as in \eqref{Elambda}.
Moreover,
for any $f\in C^1(\mathbb{R}^n)$ with
$|\nabla f|\in C_{\mathrm{c}}(\mathbb{R}^n)$,
\begin{enumerate}
\item[\textup{(i)}]
if $\gamma\in(0,\infty)$, then
\begin{align*}
&\lim_{\lambda\to\infty}\lambda\left\|\left[\int_{\mathbb{R}^n}
\mathbf{1}_{E_{\lambda,\frac{\gamma}{q}}[f]}(\cdot,y)
\left|\cdot-y\right|^{\gamma-n}\,dy\right]^\frac{1}{q}
\right\|_{M_r^\alpha(\mathbb{R}^n)}
=\left[\frac{\kappa(q,n)}{\gamma}\right]^\frac{1}{q}
\left\|\,\left|\nabla f\right|\,\right\|_{M_r^\alpha(\mathbb{R}^n)};
\end{align*}
\item[\textup{(ii)}]
if $\gamma\in(-\infty,0)$,
assume further that both $n\in\mathbb{N}\cap[2,\infty)$
and $q\in(0,\frac{n-\gamma}{n}r)$
or that $n=1$, $\gamma\in(-\infty,-1)$,
and $q\in(0,-\gamma r)$,
then
\begin{align*}
&\lim_{\lambda\to0^+}\lambda\left\|\left[\int_{\mathbb{R}^n}
\mathbf{1}_{E_{\lambda,\frac{\gamma}{q}}[f]}(\cdot,y)
\left|\cdot-y\right|^{\gamma-n}\,dy\right]^\frac{1}{q}
\right\|_{M_r^\alpha(\mathbb{R}^n)}
=\left[-\frac{\kappa(q,n)}{\gamma}\right]^\frac{1}{q}
\left\|\,\left|\nabla f\right|\,\right\|_{M_r^\alpha(\mathbb{R}^n)}.
\end{align*}
\end{enumerate}
\end{theorem}

\begin{proof}
We prove the present theorem by considering the following five cases
on both $r$ and $\gamma$.

\emph{Case 1)} $r\in(n,\infty)$ and $\gamma\in\mathbb{R}\setminus\{0\}$.
In this case, we find that there exists an $s\in(n,r)$,
which, combined with \cite[p.\,410, Theorem~382]{sdh20201},
further implies that
the Hardy--Littlewood maximal operator $\mathcal{M}$ is bounded on
$[M_r^\alpha(\mathbb{R}^n)]^\frac{1}{s}$
for any $\alpha\in[r,\infty)$.
When proving (ii),
if $\gamma\in(-\infty,0)$ and $n\in\mathbb{N}\cap[2,\infty)$,
then we choose a $p\in[1,r)$ such that $q\in(0,\frac{n-\gamma}{n}p)$
and, if $\gamma\in(-\infty,0)$ and $n=1$,
then we choose a $p\in[1,r)$ such that $q\in(0,-\gamma p)$.
By these, Theorem~\ref{2001},
and Corollary~\ref{LimFormulaEq}, we conclude that
the present theorem in this case holds true.

\emph{Case 2)} $r\in(1,n]$ and $\gamma\in(0,\infty)$.
In this case, let $X:=M^{\alpha}_r(\mathbb{R}^n)$.
Since $n(\frac{1}{r}-\frac{1}{q})<1$, it follows that
there exists a $p\in[1,r)$ satisfying $n(\frac{1}{p}-\frac{1}{q})<1$.
It is known that, for any $1<r\leq\alpha<\infty$, the associate space
$[M^\alpha_r(\mathbb{R}^n)]'$ is a block space
(see, for instance, \cite[Theorem~4.1]{st2015}), on which
$\mathcal{M}$ is bounded
(see, for instance, \cite[Theorem~3.1]{ch2014}
and \cite[Lemma~5.7]{h2015}).
By this and \cite[p.\,87]{shyy2017}, we find that both $X$ and
$X^\frac{1}{p}$ are ball Banach function spaces and $\mathcal{M}$
is bounded on $(X^\frac{1}{p})'$.
From these, Theorem~\ref{gamma>0},
and Corollary~\ref{LimFormulaEq}, we infer that
the present theorem in this case holds true.

\emph{Case 3)} $r=1$ and $\gamma\in(0,\infty)$.
In this case, by the proof of \cite[Theorem~3.1]{ch2014},
we find that, when $\alpha\in[1,\infty)$ and $\theta\in(0,1)$,
for any $f\in(X^\frac{1}{\theta})'$ with $X:=M_1^\alpha(\mathbb{R}^n)$,
$$
\left\|\mathcal{M}f\right\|_{(X^\frac{1}{\theta})'}
\lesssim\theta^{-1}
\left\|f\right\|_{(X^\frac{1}{\theta})'},
$$
where the implicit positive constant depends only on $n$.
From this, it easily follows that \eqref{2114} holds true.
Thus, by Remark~\ref{1409}(i),
we find that the present theorem in this case holds true.

\emph{Case 4)} $r\in(1,n]$ and $\gamma\in(-\infty,0)$.
In this case, from \cite[p.\,410, Theorem~382]{sdh20201}, we deduce that
$\mathcal{M}$ is bounded on $M_r^\alpha(\mathbb{R}^n)$
for any $\alpha\in[r,\infty)$.
Since $q\in(0,r)$,
it follows that there exists a $p\in[\max\{1,\,q\},r)$ and,
when proving (ii), if $n=1$, we
choose a $p\in[\max\{1,\,q\},r)$ satisfying $q\in(0,-\gamma p)$.
This, together with Theorem~\ref{upperBound}, Corollary~\ref{1409}(i),
and an argument similar to that used in Case 2),
further implies that the present theorem in this case holds true.

\emph{Case 5)} $r=1$ and $\gamma\in(-\infty,-1)$.
In this case, by an argument similar to that used in Case 3)
and Remark~\ref{1409}(i),
we conclude that the present theorem in this case holds true.
This finishes the proof of Theorem~\ref{1524}.
\end{proof}

Using both Theorems~\ref{2255} and~\ref{2256}
and Remarks~\ref{4.3}(iv) and~\ref{2021}(ii),
we obtain the following conclusions;
since their proofs are similar to that of Theorem~\ref{1524},
we omit the details here.

\begin{theorem}\label{2009}
Let $1\leq r\leq\alpha<\infty$, $p\in[1,\infty]$,
$\gamma\in\mathbb{R}\setminus\{0\}$,
and $s\in(0,1)$.
Let $q\in[1,p]$ satisfy $\frac{1}{q}=\frac{1-s}{p}+s$.
If $r=1$ and $\gamma\in(-\infty,0)$, assume further that both
$\gamma\in(-\infty,-1)$ and $n=1$.
\begin{enumerate}
\item[\textup{(i)}]
If $p\in[1,\infty)$, then, for any $f\in C^1(\mathbb{R}^n)$
with $|\nabla f|\in C_{\mathrm{c}}(\mathbb{R}^n)$,
\eqref{2053} with $X$ replaced by $M_r^\alpha(\mathbb{R}^n)$ holds true.
\item[\textup{(ii)}]
If $p=\infty$, then, for any $f\in C^1(\mathbb{R}^n)$
with $|\nabla f|\in C_{\mathrm{c}}(\mathbb{R}^n)$,
\eqref{2054} with $X$ replaced by $M_r^\alpha(\mathbb{R}^n)$ holds true.
\end{enumerate}
\end{theorem}

\begin{theorem}\label{2011}
Let $1\leq r\leq\alpha<\infty$,
$\gamma\in\mathbb{R}\setminus\{0\}$, $\eta\in(0,1)$,
and $0\leq s_0<s<1<q<q_0<\infty$ satisfy \eqref{2228}.
If $r=1$ and $\gamma\in(-\infty,0)$, assume further that both
$\gamma\in(-\infty,-1)$ and $n=1$.
Then, for any $f\in C^1(\mathbb{R}^n)$
with $|\nabla f|\in C_{\mathrm{c}}(\mathbb{R}^n)$,
\eqref{939} with $X$ replaced by $M_r^\alpha(\mathbb{R}^n)$ holds true.
\end{theorem}

\begin{remark}
\begin{enumerate}
\item[\textup{(i)}]
From Corollary~\ref{LimFormulaEq}(i),
we deduce that Theorem~\ref{1524}(i)
holds true for any given $q\in(0,\infty)$.
\item[\textup{(ii)}]
Theorem~\ref{1524} when $\gamma=n$ coincides with
\cite[Theorem~5.1]{dlyyz.arxiv}.
To the best of our knowledge,
Theorem~\ref{1524}
when $\gamma\in\mathbb{R}\setminus\{0,n\}$
is new.
\item[\textup{(iii)}]
Theorems~\ref{2009} and~\ref{2011} when $\gamma=n$
coincide with \cite[Corollaries~5.3 and~5.4]{dlyyz.arxiv}, respectively.
To the best of our knowledge,
Theorems~\ref{2009} and~\ref{2011}
(the Gagliardo--Nirenberg-type inequalities
on the Sobolev--Morrey space)
when $\gamma\in\mathbb{R}\setminus\{0,n\}$
are new.
\item[\textup{(iv)}]
Let $1\leq r<\alpha<\infty$.
Since the Morrey space $M_r^\alpha(\mathbb{R}^n)$
does not have an absolutely continuous
norm, it is still unclear whether or not
Theorems~\ref{1524},
\ref{2009}, and~\ref{2011} for any
$f\in\dot{W}^{1,M_r^\alpha(\mathbb{R}^n)}(\mathbb{R}^n)$
hold true.
\end{enumerate}
\end{remark}

\subsection{Mixed-Norm Lebesgue Spaces}\label{5.2}

For a given vector $\vec{r}:=(r_1,\ldots,r_n)
\in(0,\infty]^n$, the \emph{mixed-norm Lebesgue
space $L^{\vec{r}}(\mathbb{R}^n)$} is defined to be the
set of all the $f\in\mathscr{M}(\mathbb{R}^n)$ with the finite quasi-norm
\begin{equation*}
\|f\|_{L^{\vec{r}}(\mathbb{R}^n)}:=\left\{\int_{\mathbb{R}}
\cdots\left[\int_{\mathbb{R}}\left|f(x_1,\ldots,
x_n)\right|^{r_1}\,dx_1\right]^{\frac{r_2}{r_1}}
\cdots\,dx_n\right\}^{\frac{1}{r_n}},
\end{equation*}
where the usual modifications are made when $r_i=
\infty$ for some $i\in\{1,\ldots,n\}$.
Throughout this subsection, we always let
$r_-:=\min\{r_1, \ldots , r_n\}$.
The study of mixed-norm Lebesgue spaces
can be traced back to H\"ormander \cite{h1960}
and Benedek and Panzone \cite{bp1961}.
For more studies on mixed-norm Lebesgue spaces,
we refer the reader
to \cite{cgn2017,hlyy2019,hlyy2019b,hy2021} for the Hardy space
associated with mixed-norm Lebesgue spaces,
to \cite{cgn2017b,gjn2017,gn2016} for the Triebel--Lizorkin
and Besov spaces associated with mixed-norm Lebesgue spaces,
to \cite{cgn2017,cgn2019} for
the (anisotropic) mixed-norm Lebesgue space, and
to \cite{n2019,noss2021} for the mixed Morrey space.
Moreover,
when $\vec{r}\in(0,\infty)^n$,
from the definition of $L^{\vec{r}}(\mathbb{R}^n)$,
we easily deduce that
$L^{\vec{r}}(\mathbb{R}^n)$
is a ball quasi-Banach function space.
But $L^{\vec{r}}(\mathbb{R}^n)$ may not be a quasi-Banach function space
(see, for instance, \cite[Remark 7.20]{zwyy2021}).
When $X:=L^{\vec{r}}(\mathbb{R}^n)$, we
simply write $\dot{W}^{1,\vec{r}}(\mathbb{R}^n)
:=\dot{W}^{1,X}(\mathbb{R}^n)$.

Using Theorem~\ref{2117},
we obtain the following conclusion.

\begin{theorem}\label{2003}
Let $\vec{r}:=(r_1,\ldots,r_n)\in(1,\infty)^n$,
$\gamma\in\mathbb{R}\setminus\{0\}$,
and $q\in(0,\infty)$.
If one of the following statements holds true:
\begin{enumerate}
\item[\textup{(a)}]
$r_-\in(n,\infty)$, $\gamma\in\mathbb{R}\setminus\{0\}$,
and $q\in(0,\infty)$;
\item[\textup{(b)}]
$r_-\in(1,n]$, $\gamma\in(0,\infty)$, and
$n(\frac{1}{r_-}-\frac{1}{q})<1$;
\item[\textup{(c)}]
$r_-\in(1,n]$, $\gamma\in(-\infty,0)$, and
$q\in(0,r_-)$,
\end{enumerate}
then, for any $f\in\dot{W}^{1,\vec{r}}(\mathbb{R}^n)$,
\begin{align*}
\sup_{\lambda\in(0,\infty)}\lambda
\left\|\left[\int_{\mathbb{R}^n}
\mathbf{1}_{E_{\lambda,\frac{\gamma}{q}}[f]}(\cdot,y)
\left|\cdot-y\right|^{\gamma-n}\,dy\right]^\frac{1}{q}
\right\|_{L^{\vec{r}}(\mathbb{R}^n)}
\sim\left\|\,|\nabla f|\,\right\|_{L^{\vec{r}}(\mathbb{R}^n)},
\end{align*}
where the positive equivalence constants are independent of $f$
and $E_{\lambda,\frac{\gamma}{q}}[f]$
for any $\lambda\in(0,\infty)$
is the same as in \eqref{Elambda}.
Moreover, for any $f\in\dot{W}^{1,\vec{r}}(\mathbb{R}^n)$,
\begin{enumerate}
\item[\textup{(i)}]
if $\gamma\in(0,\infty)$, then
\begin{align*}
&\lim_{\lambda\to\infty}\lambda\left\|\left[\int_{\mathbb{R}^n}
\mathbf{1}_{E_{\lambda,\frac{\gamma}{q}}[f]}(\cdot,y)
\left|\cdot-y\right|^{\gamma-n}\,dy\right]^\frac{1}{q}
\right\|_{L^{\vec{r}}(\mathbb{R}^n)}
=\left[\frac{\kappa(q,n)}{\gamma}\right]^\frac{1}{q}
\left\|\,\left|\nabla f\right|\,\right\|_{L^{\vec{r}}(\mathbb{R}^n)};
\end{align*}
\item[\textup{(ii)}]
if $\gamma\in(-\infty,0)$,
assume further that both $n\in\mathbb{N}\cap[2,\infty)$
and $q\in(0,\frac{n-\gamma}{n}r_-)$
or that $n=1$, $\gamma\in(-\infty,-1)$,
and $q\in(0,-\gamma r_-)$,
then
\begin{align*}
&\lim_{\lambda\to0^+}\lambda\left\|\left[\int_{\mathbb{R}^n}
\mathbf{1}_{E_{\lambda,\frac{\gamma}{q}}[f]}(\cdot,y)
\left|\cdot-y\right|^{\gamma-n}\,dy\right]^\frac{1}{q}
\right\|_{L^{\vec{r}}(\mathbb{R}^n)}
=\left[-\frac{\kappa(q,n)}{\gamma}\right]^\frac{1}{q}
\left\|\,\left|\nabla f\right|\,\right\|_{L^{\vec{r}}(\mathbb{R}^n)}.
\end{align*}
\end{enumerate}
\end{theorem}

\begin{proof}
To prove the present theorem,
we consider the following three cases on both $r_-$ and $\gamma$.

\emph{Case 1)} $r_-\in(n,\infty)$ and
$\gamma\in\mathbb{R}\setminus\{0\}$.
In this case, we find that there exists an $s\in(n,r_-)$,
which, combined with
\cite[Lemma~3.5]{hlyy2019},
further implies that the Hardy--Littlewood
maximal operator $\mathcal{M}$ is bounded on
$[L^{\vec{r}}(\mathbb{R}^n)]^\frac{1}{s}$.
When proving (ii),
if $\gamma\in(-\infty,0)$ and $n\in\mathbb{N}\cap[2,\infty)$,
then we choose a $p\in[1,r_-)$ such that $q\in(0,\frac{n-\gamma}{n}p)$
and, if $\gamma\in(-\infty,0)$ and $n=1$,
then we choose a $p\in[1,r_-)$ such that $q\in(0,-\gamma p)$.
From \cite[Lemma~4.1]{gp1965},
we deduce that $L^{\vec{r}}(\mathbb{R}^n)$
has an absolutely continuous norm.
By these and Theorem~\ref{2117}, we conclude that
the present theorem in this case holds true.

\emph{Case 2)} $r_-\in(1,n]$ and $\gamma\in(0,\infty)$. In this case,
since $(r_1,\ldots,r_n)\in(1,\infty)^n$
and $n(\frac{1}{r_-}-\frac{1}{q})<1$,
it follows that there
exists a $p\in[1,r_-)$ satisfying $n(\frac{1}{p}-\frac{1}{q})<1$.
By this, $[L^{\vec{r}}(\mathbb{R}^n)]^{\frac{1}{p}}
=L^{\frac{\vec{r}}{p}}(\mathbb{R}^n)$,
and \cite[p.\,304, Theorem 1.b)]{bp1961}
(see also \cite[Remark~2.8(ii)]{hlyy2019}),
we find that both $L^{\vec{r}}(\mathbb{R}^n)$ and
$[L^{\vec{r}}(\mathbb{R}^n)]^{\frac{1}{p}}$
are ball Banach function spaces.
On the other hand, from \cite[p.\,304, Theorem~1.a)]{bp1961}
and $1<\frac{r_-}{p}\leq r_+<\infty$,
we deduce that
$$
\left(\left[L^{\vec{r}}(\mathbb{R}^n)\right]^{\frac{1}{p}}\right)'
=L^{(\frac{\vec{r}}{p})'}(\mathbb{R}^n),
$$
where $(\frac{\vec{r}}{p})':=((\frac{r_1}{p})',\ldots,(\frac{r_n}{p})')$,
which, combined with \cite[Lemma~3.5]{hlyy2019},
further implies that $\mathcal{M}$ is bounded
on $([L^{\vec{r}}(\mathbb{R}^n)]^{\frac{1}{p}})'$.
In addition, by \cite[Lemma~4.1]{gp1965},
we find that $L^{\vec{r}}(\mathbb{R}^n)$
has an absolutely continuous norm.
Thus, by Theorem~\ref{2117},
we conclude that the present theorem in this case
holds true.

\emph{Case 3)} $r_-\in(1,n]$ and $\gamma\in(-\infty,0)$. In this case,
since $q\in(0,r_-)$,
it follows that there exists a $p\in[\max\{1,\,q\},r_-)$ and,
when proving (ii), if $n=1$, we
choose a $p\in[\max\{1,\,q\},r_-)$ satisfying $q\in(0,-\gamma p)$.
By this and an argument similar to that used in Case 2), to prove
the present theorem in this case,
it remains to show that $\mathcal{M}$ is bounded on
$L^{\vec{r}}(\mathbb{R}^n)$. Indeed, this follows
directly from \cite[Lemma~3.5]{hlyy2019}.
Thus, Theorem~\ref{2117} implies that
the present theorem in this case also holds true.
This finishes the proof of Theorem~\ref{2003}.
\end{proof}

Using both Theorems~\ref{2255} and~\ref{2256}
and Remarks~\ref{4.3}(iv) and~\ref{2021}(ii),
we obtain the following conclusions;
since their proofs are similar to that of Theorem~\ref{2003},
we omit the details here.

\begin{theorem}\label{2108}
Let $\vec{r}:=(r_1,\ldots,r_n)\in(1,\infty)^n$,
$\gamma\in\mathbb{R}\setminus\{0\}$,
$p\in[1,\infty]$,
and $s\in(0,1)$.
Let $q\in[1,p]$ satisfy $\frac{1}{q}=\frac{1-s}{p}+s$.
\begin{enumerate}
\item[\textup{(i)}]
If $p\in[1,\infty)$, then, for any
$f\in\dot{W}^{1,\vec{r}}(\mathbb{R}^n)$,
\eqref{2053} with $X$ replaced by $L^{\vec{r}}(\mathbb{R}^n)$ holds true.
\item[\textup{(ii)}]
If $p=\infty$, then, for any $f\in\dot{W}^{1,\vec{r}}(\mathbb{R}^n)$,
\eqref{2054} with $X$ replaced by $L^{\vec{r}}(\mathbb{R}^n)$ holds true.
\end{enumerate}
\end{theorem}

\begin{theorem}\label{21008}
Let $\vec{r}:=(r_1,\ldots,r_n)\in(1,\infty)^n$,
$\gamma\in\mathbb{R}\setminus\{0\}$, $\eta\in(0,1)$,
and
$$
0\leq s_0<s<1<q<q_0<\infty
$$
satisfy \eqref{2228}.
Then, for any $f\in\dot{W}^{1,\vec{r}}(\mathbb{R}^n)$,
\eqref{939} with $X$ replaced by $L^{\vec{r}}(\mathbb{R}^n)$ holds true.
\end{theorem}

\begin{remark}
\begin{enumerate}
\item[(i)]
Theorem~\ref{2003} when $\gamma=n$
coincides with \cite[Theorem~5.6]{dlyyz.arxiv}.
To the best of our knowledge,
Theorem~\ref{2003}
when $\gamma\in\mathbb{R}\setminus\{0,n\}$
is new.
\item[(ii)]
Theorems~\ref{2108} and~\ref{21008} when $\gamma=n$
coincide with \cite[Corollaries~5.7 and~5.8]{dlyyz.arxiv}, respectively.
To the best of our knowledge,
Theorems~\ref{2108} and~\ref{21008}
(the Gagliardo--Nirenberg-type inequalities
on the mixed-norm Sobolev space)
when $\gamma\in\mathbb{R}\setminus\{0,n\}$
are new.
\item[(iii)]
Let $\vec{r}:=(r_1,\ldots,r_n)\in[1,\infty)^n$
with $r_-=1$. In this case,
since the Hardy--Littlewood maximal operator $\mathcal{M}$
may not be bounded on $[L^{\vec{r}}(\mathbb{R}^n)]'$
and since \eqref{2114}
with $X$ replaced by $L^{\vec{r}}(\mathbb{R}^n)$
may also not hold true,
it is still unclear whether or not
Theorem~\ref{2003}
and Corollaries~\ref{2108} and~\ref{21008}
with $r_-=1$ hold true.
\end{enumerate}
\end{remark}

\subsection{Variable Lebesgue Spaces}\label{5.3}

Let $r:\ \mathbb{R}^n\to(0,\infty)$ be a nonnegative
measurable function. Let
\begin{equation*}
\widetilde{r}_-:=\underset{x\in\mathbb{R}^n}{
\mathop\mathrm{\,ess\,inf\,}}\,r(x)\ \text{and}\
\widetilde{r}_+:=\underset{x\in\mathbb{R}^n}{
\mathop\mathrm{\,ess\,sup\,}}\,r(x).
\end{equation*}
A function $r:\ \mathbb{R}^n\to(0,\infty)$ is said to be \emph{globally
log-H\"older continuous} if there exists an
$r_{\infty}\in\mathbb{R}$ and a positive constant $C$ such that, for any
$x,y\in\mathbb{R}^n$,
\begin{equation*}
|r(x)-r(y)|\le \frac{C}{\log(e+\frac{1}{|x-y|})}\ \ \text{and}\ \
|r(x)-r_\infty|\le \frac{C}{\log(e+|x|)}.
\end{equation*}
The \emph{variable Lebesgue space
$L^{r(\cdot)}(\mathbb{R}^n)$} associated with the function
$r:\ \mathbb{R}^n\to(0,\infty)$ is defined to be the set
of all the $f\in\mathscr{M}(\mathbb{R}^n)$ with the finite quasi-norm
\begin{equation*}
\|f\|_{L^{r(\cdot)}(\mathbb{R}^n)}:=\inf\left\{\lambda
\in(0,\infty):\ \int_{\mathbb{R}^n}\left[\frac{|f(x)|}
{\lambda}\right]^{r(x)}\,dx\le1\right\}.
\end{equation*}
By the definition of $L^{r(\cdot)}(\mathbb{R}^n)$,
it is easy to show that $L^{r(\cdot)}(\mathbb{R}^n)$
is a ball quasi-Banach function space
(see, for instance, \cite[Section~7.8]{shyy2017}).
In particular,
when $1\leq\widetilde r_-\le \widetilde r_+<\infty$,
$(L^{r(\cdot)}(\mathbb{R}^n), \|\cdot\|_{
L^{r(\cdot)}(\mathbb{R}^n)})$ is
a Banach function space
in the terminology of Bennett and Sharpley \cite{bs1988}
and hence also a ball Banach function space.
For more related results on variable Lebesgue spaces,
we refer the reader to
\cite{cf2013,cw2014,dhr2009,kr1991,ns2012,n1950,n1951}.
When $X:=L^{{r}(\cdot)}(\mathbb{R}^n)$,
we simply write
$\dot{W}^{1,r(\cdot)}(\mathbb{R}^n):=\dot{W}^{1,X}(\mathbb{R}^n)$.

The following
boundedness of centered ball average operators
on variable Lebesgue spaces
is just \cite[Lemma~5.10]{dlyyz.arxiv}.

\begin{lemma}\label{BCM}
Let $r:\ \mathbb{R}^n\to(0,\infty)$ be globally
log-H\"older continuous and $1\leq\widetilde{r}_-
\leq\widetilde{r}_+<\infty$.
Then the centered ball average operators
$\{\mathcal{B}_r\}_{r\in(0,\infty)}$
are uniformly bounded on $L^{r(\cdot)}(\mathbb{R}^n)$;
moreover, there exists a positive constant $C$
such that, for any $t\in(0,\infty)$ and
$f\in L^{r(\cdot)}(\mathbb{R}^n)$,
\begin{align*}
\left\|\mathcal{B}_t(f)\right\|_{L^{r(\cdot)}(\mathbb{R}^n)}
\leq C\|f\|_{L^{r(\cdot)}(\mathbb{R}^n)}.
\end{align*}
\end{lemma}

Using Lemma~\ref{BCM} and
Theorems~\ref{2117} and~\ref{4.8},
we obtain the following conclusion.

\begin{theorem}\label{1022}
Let $\gamma\in\mathbb{R}\setminus\{0\}$
and $r:\ \mathbb{R}^n\to(0,\infty)$ be globally
log-H\"older continuous.
Let $1\leq\widetilde{r}_-
\leq\widetilde{r}_+<\infty$ and $q\in(0,\infty)$.
If one of the following statements holds true:
\begin{enumerate}
\item[\textup{(a)}]
$\widetilde{r}_-\in(n,\infty)$, $\gamma\in\mathbb{R}\setminus\{0\}$,
and $q\in(0,\infty)$;
\item[\textup{(b)}]
$\widetilde{r}_-\in[1,n]$, $\gamma\in(0,\infty)$, and
$n(\frac{1}{\widetilde{r}_-}-\frac{1}{q})<1$;
\item[\textup{(c)}]
$\widetilde{r}_-\in(1,n]$, $\gamma\in(-\infty,0)$, and
$q\in(0,\widetilde{r}_-)$;
\item[\textup{(d)}]
$\widetilde{r}_-=n=q=1$ and $\gamma\in(-\infty,-1)$,
\end{enumerate}
then, for any $f\in\dot{W}^{1,r(\cdot)}(\mathbb{R}^n)$,
\begin{align*}
\sup_{\lambda\in(0,\infty)}\lambda
\left\|\left[\int_{\mathbb{R}^n}
\mathbf{1}_{E_{\lambda,\frac{\gamma}{q}}[f]}(\cdot,y)
\left|\cdot-y\right|^{\gamma-n}\,dy\right]^\frac{1}{q}
\right\|_{L^{{r}(\cdot)}(\mathbb{R}^n)}
\sim\left\|\,|\nabla f|\,\right\|_{L^{{r}(\cdot)}(\mathbb{R}^n)},
\end{align*}
where the positive equivalence constants are independent of $f$
and $E_{\lambda,\frac{\gamma}{q}}[f]$
for any $\lambda\in(0,\infty)$
is the same as in \eqref{Elambda}.
Moreover, for any $f\in\dot{W}^{1,r(\cdot)}(\mathbb{R}^n)$,
\begin{enumerate}
\item[\textup{(i)}]
if $\gamma\in(0,\infty)$, then
\begin{align*}
&\lim_{\lambda\to\infty}\lambda\left\|\left[\int_{\mathbb{R}^n}
\mathbf{1}_{E_{\lambda,\frac{\gamma}{q}}[f]}(\cdot,y)
\left|\cdot-y\right|^{\gamma-n}\,dy\right]^\frac{1}{q}
\right\|_{L^{{r}(\cdot)}(\mathbb{R}^n)}
=\left[\frac{\kappa(q,n)}{\gamma}\right]^\frac{1}{q}
\left\|\,\left|\nabla f\right|\,\right\|_{L^{{r}(\cdot)}(\mathbb{R}^n)};
\end{align*}
\item[\textup{(ii)}]
if $\gamma\in(-\infty,0)$,
assume further that both $n\in\mathbb{N}\cap[2,\infty)$
and $q\in(0,\frac{n-\gamma}{n}\widetilde{r}_-)$
or that $n=1$, $\gamma\in(-\infty,-1)$,
and $q\in(0,-\gamma\widetilde{r}_-)$,
then
\begin{align*}
&\lim_{\lambda\to0^+}\lambda\left\|\left[\int_{\mathbb{R}^n}
\mathbf{1}_{E_{\lambda,\frac{\gamma}{q}}[f]}(\cdot,y)
\left|\cdot-y\right|^{\gamma-n}\,dy\right]^\frac{1}{q}
\right\|_{L^{{r}(\cdot)}(\mathbb{R}^n)}
=\left[-\frac{\kappa(q,n)}{\gamma}\right]^\frac{1}{q}
\left\|\,\left|\nabla f\right|\,\right\|_{L^{{r}(\cdot)}(\mathbb{R}^n)}.
\end{align*}
\end{enumerate}
\end{theorem}

\begin{proof}
We prove the present theorem by considering the following
five cases on both $\gamma$ and $\widetilde{r}_-$.

\emph{Case 1)} $\gamma\in\mathbb{R}\setminus\{0\}$ and
$\widetilde{r}_-\in(n,\infty)$. In this case,
we find that there exists an $s\in(n,\widetilde{r}_-)$,
which, combined with \cite[Theorem~1.7]{ahh2015},
further implies that the Hardy--Littlewood
maximal operator $\mathcal{M}$ is bounded on
$[L^{{r}(\cdot)}(\mathbb{R}^n)]^\frac{1}{s}$.
When proving (ii),
if $\gamma\in(-\infty,0)$ and $n\in\mathbb{N}\cap[2,\infty)$,
then we choose a $p\in[1,\widetilde{r}_-)$
such that $q\in(0,\frac{n-\gamma}{n}p)$
and, if $\gamma\in(-\infty,0)$ and $n=1$,
then we choose a $p\in[1,\widetilde{r}_-)$ such that $q\in(0,-\gamma p)$.
From \cite[p.\,73]{cf2013}, we infer that $L^{r(\cdot)}(\mathbb{R}^n)$
has an absolutely continuous norm.
By these and Theorem~\ref{2117}, we conclude that
the present theorem in this case holds true.

\emph{Case 2)}
$\gamma\in(0,\infty)$ and $\widetilde{r}_-\in(1,n]$. In this case,
since $n(\frac{1}{\widetilde{r}_-}-\frac{1}{q})<1$,
it follows that there exists a $p\in[1,\widetilde{r}_-)$ such that
$n(\frac{1}{p}-\frac{1}{q})<1$.
By this, $[L^{r(\cdot)}(\mathbb{R}^n)]^{\frac{1}{p}}
=L^{\frac{r(\cdot)}{p}}(\mathbb{R}^n)$,
and \cite[Theorems~2.17 and~2.71]{cf2013},
we conclude that both
$L^{r(\cdot)}(\mathbb{R}^n)$ and $L^{\frac{r(\cdot)}{p}}(\mathbb{R}^n)$
are ball Banach function spaces.
From \cite[Theorem~2.80]{cf2013} and
$1<\frac{\widetilde{r}_-}{p}\leq\widetilde{r}_+<\infty$,
we deduce that
$$
\left(\left[L^{r(\cdot)}(\mathbb{R}^n)\right]^{\frac{1}{p}}\right)'
=L^{(\frac{r(\cdot)}{p})'}(\mathbb{R}^n),
$$
where, for any $x\in\mathbb{R}^n$,
$(\frac{r(x)}{p})'$ is the conjugate index
of $\frac{r(x)}{p}$, which, together with \cite[Theorem~1.7]{ahh2015},
further implies that $\mathcal{M}$
is bounded on
$([L^{r(\cdot)}(\mathbb{R}^n)]^{\frac{1}{p}})'$.
Moreover,
by \cite[p.\,73]{cf2013}, we conclude that $L^{r(\cdot)}(\mathbb{R}^n)$
has an absolutely continuous norm.
Thus, Theorem~\ref{2117}
implies that the present theorem in this case holds true.

\emph{Case 3)}
$\gamma\in(0,\infty)$ and $\widetilde{r}_-=1$.
In this case, from \cite[Theorem~4.3.8]{dhhr2011},
we deduce that, for any given $\theta\in(\frac{1}{2},1)$
and for any $f\in[L^\frac{r(\cdot)}{\theta}(\mathbb{R}^n)]'$,
$$
\left\|\mathcal{M}(f)\right\|_{[L^\frac{r(\cdot)}{\theta}(\mathbb{R}^n)]'}
\lesssim\frac{1}{\theta}
\left\|f\right\|_{[L^\frac{r(\cdot)}{\theta}(\mathbb{R}^n)]'},
$$
where the implicit positive constant depends only on both $n$ and $r$,
which further implies that \eqref{2114}
with $X:=L^{r(\cdot)}(\mathbb{R}^n)$ holds true.
Moreover, by Lemma~\ref{BCM}, we find that
centered ball average operators $\{\mathcal{B}_r\}_{r\in(0,\infty)}$
are uniformly bounded on $L^{r(\cdot)}(\mathbb{R}^n)$.
From these, Lemma~\ref{BCM}, Theorem~\ref{4.8},
and an argument similar to that used in Case 2),
we deduce that the present theorem in this case holds true.

\emph{Case 4)}
$\gamma\in(-\infty,0)$ and $\widetilde{r}_-\in(1,n]$.
In this case,
since $q\in(0,\widetilde{r}_-)$,
it follows that there exists a $p\in[\max\{1,\,q\},\widetilde{r}_-)$ and,
when proving (ii), if $n=1$
and $\widetilde{r}_-\in(1,\infty)$, we
choose a $p\in[\max\{1,\,q\},\widetilde{r}_-)$
satisfying $q\in(0,-\gamma p)$.
By this and an argument similar to that used in Case 2), to prove
the present theorem in this case,
it remains to show that $\mathcal{M}$ is bounded on
$L^{r(\cdot)}(\mathbb{R}^n)$. Indeed, this follows
directly from \cite[Theorem~1.7]{ahh2015}.
Thus, Theorem~\ref{2117}
implies that the present theorem in this case holds true.

\emph{Case 5)}
$\gamma\in(-\infty,-1)$ and $\widetilde{r}_-=1$. In this case,
from an argument similar to that
used in Case 3), Lemma~\ref{BCM}, and Theorem~\ref{4.8},
we deduce that the present theorem in this case still holds true.
This finishes the proof of Theorem~\ref{1022}.
\end{proof}

Using Lemma~\ref{BCM}, Theorems~\ref{2255} and~\ref{2256},
and Remarks~\ref{4.3}(iv) and~\ref{2021}(ii),
we obtain the following conclusions;
since their proofs are similar to that of Theorem~\ref{1022},
we omit the details here.

\begin{theorem}\label{10251}
Let $\gamma\in\mathbb{R}\setminus\{0\}$
and $r:\ \mathbb{R}^n\to(0,\infty)$ be globally
log-H\"older continuous.
Let $1\leq\widetilde{r}_-\leq\widetilde{r}_+<\infty$,
$p\in[1,\infty]$,
and $s\in(0,1)$.
Let $q\in[1,p]$ satisfy $\frac{1}{q}=\frac{1-s}{p}+s$.
Moreover, if $\widetilde{r}_-=1$ and $\gamma\in(-\infty,0)$,
assume further that
both $\gamma\in(-\infty,-1)$ and $n=1$.
\begin{enumerate}
\item[\textup{(i)}]
If $p\in[1,\infty)$, then, for any
$f\in\dot{W}^{1,r(\cdot)}(\mathbb{R}^n)$,
\eqref{2053} with $X$ replaced by $L^{r(\cdot)}(\mathbb{R}^n)$ holds true.
\item[\textup{(ii)}]
If $p=\infty$, then, for any $f\in\dot{W}^{1,r(\cdot)}(\mathbb{R}^n)$,
\eqref{2054} with $X$ replaced by $L^{r(\cdot)}(\mathbb{R}^n)$ holds true.
\end{enumerate}
\end{theorem}

\begin{theorem}\label{10261}
Let $\gamma\in\mathbb{R}\setminus\{0\}$
and $r:\ \mathbb{R}^n\to(0,\infty)$ be globally
log-H\"older continuous.
Let $1\leq\widetilde{r}_-\leq\widetilde{r}_+<\infty$,
$\eta\in(0,1)$,
and
$$
0\leq s_0<s<1<q<q_0<\infty
$$
satisfy \eqref{2228}.
Moreover, if $\widetilde{r}_-=1$ and $\gamma\in(-\infty,0)$,
assume further that
both $\gamma\in(-\infty,-1)$ and $n=1$.
Then, for any $f\in\dot{W}^{1,r(\cdot)}(\mathbb{R}^n)$,
\eqref{939} with $X$ replaced by $L^{r(\cdot)}(\mathbb{R}^n)$ holds true.
\end{theorem}

\begin{remark}
\begin{enumerate}
\item[(i)]
Theorem~\ref{1022} when $\gamma=n$ coincides with
\cite[Theorem~5.11]{dlyyz.arxiv}.
To the best of our knowledge,
Theorem~\ref{1022}
when $\gamma\in\mathbb{R}\setminus\{0,n\}$
is new.
\item[(ii)]
Theorems~\ref{10251} and~\ref{10261} when $\gamma=n$
coincide with \cite[Corollaries~5.12 and~5.13]{dlyyz.arxiv}, respectively.
To the best of our knowledge,
Theorems~\ref{10251} and~\ref{10261}
(the Gagliardo--Nirenberg-type inequalities
on the variable Sobolev space)
when $\gamma\in\mathbb{R}\setminus\{0,n\}$
are new.
\end{enumerate}
\end{remark}

\subsection{Weighted Lebesgue Spaces}\label{5.4}

Recall that, for any given $r\in(0,\infty)$
and a weight $\omega$ on $\mathbb{R}^n$,
$L^r_\omega(\mathbb{R}^n)$ denotes the weighted Lebesgue space
with respect to the measure $\omega(x)\,dx$ on $\mathbb{R}^n$;
see Definition~\ref{1556}(i).
As was pointed out in \cite[p.\,86]{shyy2017},
the weighted Lebesgue space $L^r_\omega(\mathbb{R}^n)$
is a ball quasi-Banach function space,
but it may not be a Banach function space.
When $X:=L^r_{\omega}(\mathbb{R}^n)$,
we simply write $\dot{W}^{1,r}_{\omega}
(\mathbb{R}^n):=\dot{W}^{1,X}(\mathbb{R}^n)$.

Using Theorem~\ref{2117},
we obtain the following conclusion.

\begin{theorem}\label{1657}
Let $\gamma\in\mathbb{R}\setminus\{0\}$,
$1\leq p\leq r<\infty$, $\omega\in A_{\frac{r}{p}}(\mathbb{R}^n)$,
and $q\in(0,\infty)$.
If one of the following statements holds true:
\begin{enumerate}
\item[\textup{(a)}]
$r\in(n,\infty)$, $p\in[n,r)$,
$\gamma\in\mathbb{R}\setminus\{0\}$,
and $q\in(0,\infty)$;
\item[\textup{(b)}]
$r\in[1,\infty)$, $p\in[1,r]$, $\gamma\in(0,\infty)$, and
$n(\frac{1}{p}-\frac{1}{q})<1$;
\item[\textup{(c)}]
$r\in(1,\infty)$, $p\in[1,r]$, $\gamma\in(-\infty,0)$, and
$q\in(0,p]$;
\item[\textup{(d)}]
$r=n=p=q=1$ and $\gamma\in(-\infty,-1)$,
\end{enumerate}
then, for any $f\in\dot{W}^{1,r}_{\omega}(\mathbb{R}^n)$,
\begin{align*}
\sup_{\lambda\in(0,\infty)}\lambda
\left\|\left[\int_{\mathbb{R}^n}
\mathbf{1}_{E_{\lambda,\frac{\gamma}{q}}[f]}(\cdot,y)
\left|\cdot-y\right|^{\gamma-n}\,dy\right]^\frac{1}{q}
\right\|_{L^r_\omega(\mathbb{R}^n)}
\sim\left\|\,|\nabla f|\,\right\|_{L^r_\omega(\mathbb{R}^n)},
\end{align*}
where the positive equivalence constants are independent of $f$
and $E_{\lambda,\frac{\gamma}{q}}[f]$
for any $\lambda\in(0,\infty)$
is the same as in \eqref{Elambda}.
Moreover, for any $f\in\dot{W}^{1,r}_{\omega}(\mathbb{R}^n)$,
\begin{enumerate}
\item[\textup{(i)}]
if $\gamma\in(0,\infty)$, then
\begin{align*}
&\lim_{\lambda\to\infty}\lambda\left\|\left[\int_{\mathbb{R}^n}
\mathbf{1}_{E_{\lambda,\frac{\gamma}{q}}[f]}(\cdot,y)
\left|\cdot-y\right|^{\gamma-n}\,dy\right]^\frac{1}{q}
\right\|_{L^r_\omega(\mathbb{R}^n)}
=\left[\frac{\kappa(q,n)}{\gamma}\right]^\frac{1}{q}
\left\|\,\left|\nabla f\right|\,\right\|_{L^r_\omega(\mathbb{R}^n)};
\end{align*}
\item[\textup{(ii)}]
if $\gamma\in(-\infty,0)$,
assume further that both $n\in\mathbb{N}\cap[2,\infty)$
and $q\in(0,\frac{n-\gamma}{n}p)$
or that $n=1$, $\gamma\in(-\infty,-1)$,
and $q\in(0,-\gamma p)$,
then
\begin{align*}
&\lim_{\lambda\to0^+}\lambda\left\|\left[\int_{\mathbb{R}^n}
\mathbf{1}_{E_{\lambda,\frac{\gamma}{q}}[f]}(\cdot,y)
\left|\cdot-y\right|^{\gamma-n}\,dy\right]^\frac{1}{q}
\right\|_{L^r_\omega(\mathbb{R}^n)}
=\left[-\frac{\kappa(q,n)}{\gamma}\right]^\frac{1}{q}
\left\|\,\left|\nabla f\right|\,\right\|_{L^r_\omega(\mathbb{R}^n)}.
\end{align*}
\end{enumerate}
\end{theorem}

\begin{proof}
From both Definition~\ref{1556}(i) and \cite[p.\,86]{shyy2017},
it easily follows that
$[L^r_\omega(\mathbb{R}^n)]^\frac{1}{p}=L^\frac{r}{p}_\omega(\mathbb{R}^n)$
and that both $L^r_\omega(\mathbb{R}^n)$ and
$[L^r_\omega(\mathbb{R}^n)]^\frac{1}{p}$
are ball Banach function spaces.

Next, we claim that the Hardy--Littlewood
maximal operator $\mathcal{M}$ is bounded
on $([L^r_\omega(\mathbb{R}^n)]^\frac{1}{p})'$.
On the one hand, when $p\in[1,r)$,
by \cite[Theorem~2.7.4]{dhhr2011}
and $\frac{r}{p}\in(1,\infty)$, we conclude that
\begin{align}\label{1616}
\left(\left[L^r_\omega(\mathbb{R}^n)\right]^\frac{1}{p}\right)'=
L^{(\frac{r}{p})'}_{\omega^{1-(\frac{r}{p})'}}(\mathbb{R}^n).
\end{align}
From $\omega\in A_{\frac{r}{p}}(\mathbb{R}^n)$, $\frac{r}{p}>1$, and
Lemma~\ref{ApProperty}(v),
we deduce that
$\omega^{1-(\frac{r}{p})'}\in A_{(\frac{r}{p})'}(\mathbb{R}^n)$.
By this, \eqref{1616}, and Lemma~\ref{ApProperty}(vi),
we further conclude that $\mathcal{M}$ is bounded on
$([L^r_\omega(\mathbb{R}^n)]^\frac{1}{p})'$.
On the other hand, when $p=r$, from \cite[p.\,9]{ins2019},
we infer that
$$
\left(\left[L^r_\omega(\mathbb{R}^n)\right]^\frac{1}{p}\right)'
=\left[L^1_\omega(\mathbb{R}^n)\right]'
=L^\infty_{\omega^{-1}}(\mathbb{R}^n),
$$
which, together with both \cite[Theorem~3.1(b)]{aj1980/81}
and \cite[p.\,9]{ins2019},
further implies that $\mathcal{M}$ is bounded on
$([L^r_\omega(\mathbb{R}^n)]^\frac{1}{p})'$.
This finishes the proof of the above claim.

Now, we consider the following four cases on
$r$, $p$, $\gamma$, and $q$.

\emph{Case 1)}
$r\in(n,\infty)$, $p\in[n,r)$,
$\gamma\in\mathbb{R}\setminus\{0\}$,
and $q\in(0,\infty)$.
In this case,
since $\omega\in A_{\frac{r}{p}}(\mathbb{R}^n)$, $p\in[n,r)$,
and \cite[Corollary~7.2.6]{g2014},
it follows that there exists a $\widetilde{p}\in(p,r)$
such that $\omega\in A_{\frac{r}{\widetilde{p}}}(\mathbb{R}^n)$.
Choose an $s\in(n,\widetilde{p})$ and hence
$\omega\in A_{\frac{r}{s}}(\mathbb{R}^n)$ since Lemma~\ref{ApProperty}(iv).
Then Lemma~\ref{ApProperty}(vi) implies that $\mathcal{M}$
is bounded on $[L^r_\omega(\mathbb{R}^n)]^\frac{1}{s}$.
Moreover,
from \cite[Theorem~1.34]{rudin},
we deduce that $L^r_\omega(\mathbb{R}^n)$
has an absolutely continuous norm.
Thus, Theorem~\ref{2117} implies that
the present theorem in this case holds true.

\emph{Case 2)}
$r\in[1,\infty)$, $p\in[1,r]$, $\gamma\in(0,\infty)$, and
$n(\frac{1}{p}-\frac{1}{q})<1$.
In this case, by the above claim, we find that $\mathcal{M}$
is bounded on $([L^r_\omega(\mathbb{R}^n)]^\frac{1}{p})'$.
Moreover,
from \cite[Theorem~1.34]{rudin},
we deduce that $L^r_\omega(\mathbb{R}^n)$
has an absolutely continuous norm.
Thus, Theorem~\ref{2117} implies that
the present theorem in this case holds true.

\emph{Case 3)}
$r\in(1,\infty)$, $p\in[1,r]$, $\gamma\in(-\infty,0)$, and
$q\in(0,p]$.
In this case, by an argument similar to that used
in Case 1), we find that $\mathcal{M}$
is bounded on $L^r_\omega(\mathbb{R}^n)$,
which, combined with both an argument similar
to that used in Case 2) and
Theorem~\ref{2117},
further implies that
the present theorem in this case holds true.

\emph{Case 4)}
$r=n=p=q=1$ and $\gamma\in(-\infty,-1)$.
In this case, from an argument similar to that used in Case 2),
and Theorem~\ref{2117},
we conclude that
the present theorem in this case holds true.
This finishes the proof of Theorem~\ref{1657}.
\end{proof}

Using Theorems~\ref{2255} and~\ref{2256}
and Remarks~\ref{4.3}(iv) and~\ref{2021}(ii),
we obtain the following conclusions;
since their proofs are similar to that of Theorem~\ref{1657},
we omit the details here.

\begin{theorem}\label{1714}
Let $\gamma\in\mathbb{R}\setminus\{0\}$,
$r\in[1,\infty)$, $\omega\in A_r(\mathbb{R}^n)$,
$p\in[1,\infty]$,
and $s\in(0,1)$.
Let $q\in[1,p]$ satisfy $\frac{1}{q}=\frac{1-s}{p}+s$.
Moreover, if $r=1$ and $\gamma\in(-\infty,0)$,
assume further that
both $\gamma\in(-\infty,-1)$ and $n=1$.
\begin{enumerate}
\item[\textup{(i)}]
If $p\in[1,\infty)$, then, for any
$f\in\dot{W}^{1,r}_\omega(\mathbb{R}^n)$,
\eqref{2053} with $X$
replaced by $L^{r}_\omega(\mathbb{R}^n)$ holds true.
\item[\textup{(ii)}]
If $p=\infty$, then, for any $f\in\dot{W}^{1,r}_\omega(\mathbb{R}^n)$,
\eqref{2054} with $X$
replaced by $L^{r}_\omega(\mathbb{R}^n)$ holds true.
\end{enumerate}
\end{theorem}

\begin{theorem}\label{1715}
Let $\gamma\in\mathbb{R}\setminus\{0\}$,
$r\in[1,\infty)$, $\omega\in A_r(\mathbb{R}^n)$,
$\eta\in(0,1)$,
and
$$
0\leq s_0<s<1<q<q_0<\infty
$$
satisfy \eqref{2228}.
Moreover, if $r=1$ and $\gamma\in(-\infty,0)$,
assume further that
both $\gamma\in(-\infty,-1)$ and $n=1$.
Then, for any $f\in\dot{W}^{1,r}_\omega(\mathbb{R}^n)$,
\eqref{939} with $X$
replaced by $L^{r}_\omega(\mathbb{R}^n)$ holds true.
\end{theorem}

\begin{remark}
\begin{enumerate}
\item[(i)]
Theorem~\ref{1657} when $\gamma=n$ coincides with
\cite[Theorem~5.15]{dlyyz.arxiv}.
To the best of our knowledge,
Theorem~\ref{1657}
when $\gamma\in\mathbb{R}\setminus\{0,n\}$
is new.
\item[(ii)]
Theorems~\ref{1714} and~\ref{1715} when $\gamma=n$
coincide with
\cite[Corollaries~5.16 and~5.17]{dlyyz.arxiv}, respectively.
To the best of our knowledge,
Theorems~\ref{1714} and~\ref{1715}
(the Gagliardo--Nirenberg-type inequalities
on the weighted
Sobolev space)
when $\gamma\in\mathbb{R}\setminus\{0,n\}$
are new.
\end{enumerate}
\end{remark}

\subsection{Lorentz Spaces}\label{5.7}

Recall that the \emph{Lorentz space $L^{r,\tau}(\mathbb{R}^n)$},
with any given $r,\tau\in(0,\infty)$,
is defined to be the set of all the
$f\in\mathscr{M}(\mathbb{R}^n)$ such that
\begin{equation*}
\|f\|_{L^{r,\tau}(\mathbb{R}^n)}
:=\left\{\int_0^{\infty}
\left[t^{\frac{1}{r}}f^*(t)\right]^\tau
\frac{\,dt}{t}\right\}^{\frac{1}{\tau}}
<\infty,
\end{equation*}
where $f^*$ denotes the \emph{decreasing rearrangement of $f$},
defined by setting, for any $t\in[0,\infty)$,
\begin{equation*}
f^*(t):=\inf\left\{s\in(0,\infty):\ \left|
\left\{x\in\mathbb{R}^n:\ |f(x)|>s\right\}\right|\leq t\right\}.
\end{equation*}
We adopt the convention $\inf \emptyset = \infty$,
thus having $f^*(t)=\infty$ whenever
$|\{x\in\mathbb{R}^n:\ |f(x)|>s\}|>t$ for all $s\in[0,\infty)$.
When $r,\tau\in(0,\infty)$,
$L^{r,\tau}(\mathbb{R}^n)$ is a quasi-Banach function space
and hence a ball quasi-Banach function space
(see, for instance, \cite[Theorem 1.4.11]{g2014});
when $r,\tau\in(1,\infty)$,
the Lorentz space $L^{r,\tau}(\mathbb{R}^n)$ is a Banach function space
and hence a ball Banach function space
(see, for instance, \cite[p.\,87]{shyy2017} and \cite[p.\,74]{g2014}).
When $X:=L^{r,\tau}(\mathbb{R}^n)$,
we simply write
$\dot{W}^{1,L^{r,\tau}}(\mathbb{R}^n):=\dot{W}^{1,X}(\mathbb{R}^n)$.

Using Theorem~\ref{2117},
we obtain the following conclusion.

\begin{theorem}\label{2151}
Let $\gamma\in\mathbb{R}\setminus\{0\}$,
$r,\tau\in(1,\infty)$, and $q\in(0,\infty)$.
If one of the following statements holds true:
\begin{enumerate}
\item[\textup{(a)}]
$\min\{r,\,\tau\}\in(n,\infty)$,
$\gamma\in\mathbb{R}\setminus\{0\}$,
and $q\in(0,\infty)$;
\item[\textup{(b)}]
$\min\{r,\,\tau\}\in(1,n]$, $\gamma\in(0,\infty)$, and
$n(\frac{1}{\min\{r,\,\tau\}}-\frac{1}{q})<1$;
\item[\textup{(c)}]
$\min\{r,\,\tau\}\in(1,n]$, $\gamma\in(-\infty,0)$, and
$q\in(0,\min\{r,\,\tau\})$,
\end{enumerate}
then, for any $f\in\dot{W}^{1,L^{r,\tau}}(\mathbb{R}^n)$,
\begin{align*}
\sup_{\lambda\in(0,\infty)}\lambda
\left\|\left[\int_{\mathbb{R}^n}
\mathbf{1}_{E_{\lambda,\frac{\gamma}{q}}[f]}(\cdot,y)
\left|\cdot-y\right|^{\gamma-n}\,dy\right]^\frac{1}{q}
\right\|_{L^{r,\tau}(\mathbb{R}^n)}
\sim\left\|\,|\nabla f|\,\right\|_{L^{r,\tau}(\mathbb{R}^n)},
\end{align*}
where the positive equivalence constants are independent of $f$
and $E_{\lambda,\frac{\gamma}{q}}[f]$
for any $\lambda\in(0,\infty)$
is the same as in \eqref{Elambda}.
Moreover, for any $f\in\dot{W}^{1,L^{r,\tau}}(\mathbb{R}^n)$,
\begin{enumerate}
\item[\textup{(i)}]
if $\gamma\in(0,\infty)$, then
\begin{align*}
&\lim_{\lambda\to\infty}\lambda\left\|\left[\int_{\mathbb{R}^n}
\mathbf{1}_{E_{\lambda,\frac{\gamma}{q}}[f]}(\cdot,y)
\left|\cdot-y\right|^{\gamma-n}\,dy\right]^\frac{1}{q}
\right\|_{L^{r,\tau}(\mathbb{R}^n)}
=\left[\frac{\kappa(q,n)}{\gamma}\right]^\frac{1}{q}
\left\|\,\left|\nabla f\right|\,\right\|_{L^{r,\tau}(\mathbb{R}^n)};
\end{align*}
\item[\textup{(ii)}]
if $\gamma\in(-\infty,0)$,
assume further that both $n\in\mathbb{N}\cap[2,\infty)$
and $q\in(0,\frac{n-\gamma}{n}\min\left\{r,\,\tau\right\})$
or that $n=1$, $\gamma\in(-\infty,-1)$,
and $q\in(0,-\gamma\min\left\{r,\,\tau\right\})$,
then
\begin{align*}
&\lim_{\lambda\to0^+}\lambda\left\|\left[\int_{\mathbb{R}^n}
\mathbf{1}_{E_{\lambda,\frac{\gamma}{q}}[f]}(\cdot,y)
\left|\cdot-y\right|^{\gamma-n}\,dy\right]^\frac{1}{q}
\right\|_{L^{r,\tau}(\mathbb{R}^n)}
=\left[-\frac{\kappa(q,n)}{\gamma}\right]^\frac{1}{q}
\left\|\,\left|\nabla f\right|\,\right\|_{L^{r,\tau}(\mathbb{R}^n)}.
\end{align*}
\end{enumerate}
\end{theorem}

\begin{proof}
We consider the following three cases on both $\min\{r,\,\tau\}$ and $\gamma$.

\emph{Case 1)} $\min\{r,\,\tau\}\in(n,\infty)$ and
$\gamma\in\mathbb{R}\setminus\{0\}$. In this case,
we find that there exists an $s\in(n,\min\{r,\,\tau\})$,
which, combined with \cite[Lemma~3.5]{zhy2020},
further implies that the Hardy--Littlewood
maximal operator $\mathcal{M}$ is bounded on
$[L^{r,\tau}(\mathbb{R}^n)]^\frac{1}{s}$.
From \cite[Remark~3.4(iii)]{wyy2020},
we deduce that $L^{r,\tau}(\mathbb{R}^n)$ has an
absolutely continuous norm. Moreover,
when proving (ii),
if $\gamma\in(-\infty,0)$ and $n\in\mathbb{N}\cap[2,\infty)$,
then we choose a $p\in[1,\min\{r,\,\tau\})$
such that $q\in(0,\frac{n-\gamma}{n}p)$
and, if $\gamma\in(-\infty,0)$ and $n=1$,
then we choose a $p\in[1,\min\{r,\,\tau\})$ such that $q\in(0,-\gamma p)$.
Thus, Theorem~\ref{2117} implies that
the present theorem in this case holds true.

\emph{Case 2)} $\min\{r,\,\tau\}\in(1,n]$ and
$\gamma\in(0,\infty)$. In this case,
since $r,\tau\in(1,\infty)$ and
$n(\frac{1}{\min\{r,\,\tau\}}-\frac{1}{q})<1$, it follows that
there exists a $p\in[1,\min\{r,\,\tau\})$
such that $n(\frac{1}{p}-\frac{1}{q})<1$.
Notice that
$[L^{r,\tau}(\mathbb{R}^n)]^\frac{1}{p}
=L^{\frac{r}{p},\frac{\tau}{p}}(\mathbb{R}^n)$.
By this and the conclusion in \cite[p.\,87]{shyy2017},
it follows that both $L^{r,\tau}(\mathbb{R}^n)$
and $[L^{r,\tau}(\mathbb{R}^n)]^\frac{1}{p}$ are ball Banach function spaces.
From \cite[Theorem~1.4.16(vi)]{g2014}, we deduce that
$([L^{r,\tau}(\mathbb{R}^n)]^\frac{1}{p})'
=L^{(\frac{r}{p})',(\frac{\tau}{p})'}(\mathbb{R}^n)$,
which, combined with \cite[Lemma~3.5]{zhy2020}, further implies that
$\mathcal{M}$ is bounded on $([L^{r,\tau}(\mathbb{R}^n)]^\frac{1}{p})'$.
By \cite[Remark~3.4(iii)]{wyy2020},
we find that $L^{r,\tau}(\mathbb{R}^n)$ has an
absolutely continuous norm.
Thus, Theorem~\ref{2117} implies that
the present theorem in this case holds true.

\emph{Case 3)} $\min\{r,\,\tau\}\in(1,n]$ and
$\gamma\in(-\infty,0)$. In this case,
by \cite[Lemma~3.5]{zhy2020}, we find that $\mathcal{M}$
is bounded on $L^{r,\tau}(\mathbb{R}^n)$.
Since $r,\tau\in(1,\infty)$ and
$q\in(0,\min\{r,\,\tau\})$, it follows that
there exists a $p\in[1,\min\{r,\,\tau\})$ such that $q\leq p$ and,
when proving (ii), if $n=1$, we
choose a $p\in[\max\{1,\,q\},\min\{r,\,\tau\})$ satisfying
$q\in(0,-\gamma p)$.
This, together with an
argument similar to that used in Case 2), further implies that
all the assumptions of Theorem~\ref{2117} with
$X:=L^{r,\tau}(\mathbb{R}^n)$ are satisfied.
Thus, Theorem~\ref{2117} implies that
the present theorem in this case still holds true.
This finishes the proof of Theorem~\ref{2151}.
\end{proof}

Using both Theorems~\ref{2255} and~\ref{2256}
and Remarks~\ref{4.3}(iv) and~\ref{2021}(ii),
we obtain the following conclusions;
since their proofs are similar to that of Theorem~\ref{2151},
we omit the details here.

\begin{theorem}\label{2288}
Let $\gamma\in\mathbb{R}\setminus\{0\}$,
$r,\tau\in(1,\infty)$,
$p\in[1,\infty]$,
and $s\in(0,1)$.
Let $q\in[1,p]$ satisfy $\frac{1}{q}=\frac{1-s}{p}+s$.
\begin{enumerate}
\item[\textup{(i)}]
If $p\in[1,\infty)$, then, for any
$f\in\dot{W}^{1,L^{r,\tau}}(\mathbb{R}^n)$,
\eqref{2053} with $X$
replaced by $L^{r,\tau}(\mathbb{R}^n)$ holds true.
\item[\textup{(ii)}]
If $p=\infty$, then, for any $f\in\dot{W}^{1,L^{r,\tau}}(\mathbb{R}^n)$,
\eqref{2054} with $X$
replaced by $L^{r,\tau}(\mathbb{R}^n)$ holds true.
\end{enumerate}
\end{theorem}

\begin{theorem}\label{2289}
Let $\gamma\in\mathbb{R}\setminus\{0\}$,
$r,\tau\in(1,\infty)$,
$\eta\in(0,1)$,
and
$$
0\leq s_0<s<1<q<q_0<\infty
$$
satisfy \eqref{2228}.
Then, for any $f\in\dot{W}^{1,L^{r,\tau}}(\mathbb{R}^n)$,
\eqref{939} with $X$ replaced by
$L^{r,\tau}(\mathbb{R}^n)$ holds true.
\end{theorem}

\begin{remark}
\begin{enumerate}
\item[(i)]
Theorem~\ref{2151} when $\gamma=n$
can be deduced from \cite[Theorem~4.5(ii)]{dlyyz.arxiv}
with $X:=L^{r,\tau}(\mathbb{R}^n)$ for any $r,\tau\in(1,\infty)$.
To the best of our knowledge,
Theorem~\ref{2151}
when $\gamma\in\mathbb{R}\setminus\{0,n\}$
is new.
\item[(ii)]
Theorems~\ref{2288} and~\ref{2289} when $\gamma=n$
can be obtained, respectively,
by \cite[Corollaries~4.13 and~4.15]{dlyyz.arxiv}
with $X:=L^{r,\tau}(\mathbb{R}^n)$ for any $r,\tau\in(1,\infty)$.
To the best of our knowledge,
Theorems~\ref{2288} and~\ref{2289}
(the Gagliardo--Nirenberg-type inequalities
on the Sobolev--Lorentz space)
when $\gamma\in\mathbb{R}\setminus\{0,n\}$
are new.
\end{enumerate}
\end{remark}

\subsection{Orlicz Spaces}\label{5.5}

We first give several basic concepts on Orlicz spaces.
A non-decreasing function $\Phi:\ [0,\infty)
\ \to\ [0,\infty)$ is called an \emph{Orlicz function}
if $\Phi$ satisfies that
\begin{enumerate}
\item[(i)]
$\Phi(0)= 0$;
\item[(ii)]
for any $t\in(0,\infty)$,
$\Phi(t)\in(0,\infty)$;
\item[(iii)]
$\lim_{t\to\infty}\Phi(t)=\infty$.
\end{enumerate}
An Orlicz function $\Phi$ is said to be of \emph{lower}
(resp. \emph{upper}) \emph{type} $r$ for some
$r\in\mathbb{R}$ if there exists a positive constant
$C_{(r)}$ such that,
for any $t\in[0,\infty)$ and
$s\in(0,1)$ [resp. $s\in[1,\infty)$],
\begin{equation*}
\Phi(st)\le C_{(r)} s^r\Phi(t).
\end{equation*}
In the remainder of this subsection, we always assume that
$\Phi:\ [0,\infty)\ \to\ [0,\infty)$
is an Orlicz function with both positive lower
type $r_{\Phi}^-$ and positive upper type $r_{\Phi}^+$.
The \emph{Orlicz space $L^\Phi(\mathbb{R}^n)$}
is defined to be the set of all the $f\in\mathscr{M}(\mathbb{R}^n)$
with the finite quasi-norm
\begin{equation*}
\|f\|_{L^\Phi(\mathbb{R}^n)}:=\inf\left\{\lambda\in
(0,\infty):\ \int_{\mathbb{R}^n}\Phi\left(\frac{|f(x)|}
{\lambda}\right)\,dx\le1\right\}.
\end{equation*}
It is easy to prove that $L^\Phi(\mathbb{R}^n)$
is a quasi-Banach function space
(see \cite[Section~7.6]{shyy2017}).
For more related results on Orlicz spaces,
we refer the reader to \cite{dfmn2021,ns2014,rr2002}.
When $X:=L^{\Phi}(\mathbb{R}^n)$,
we simply write
$\dot{W}^{1,\Phi}(\mathbb{R}^n):=\dot{W}^{1,X}(\mathbb{R}^n)$.
The following boundedness of centered ball average operators
on Orlicz spaces is just \cite[Lemma~5.19]{dlyyz.arxiv}.

\begin{lemma}\label{1910}
Let $\Phi$ be an Orlicz function with both
positive lower type $r^-_{\Phi}\in[1,\infty)$
and positive upper type $r^+_\Phi$.
Then centered ball average operators
$\{\mathcal{B}_r\}_{r\in(0,\infty)}$
are uniformly bounded on $L^\Phi(\mathbb{R}^n)$;
moreover, there exists a positive constant $C$
such that, for any $r\in(0,\infty)$ and $f\in L^\Phi(\mathbb{R}^n)$,
\begin{align*}
\left\|\mathcal{B}_r(f)\right\|_{L^\Phi(\mathbb{R}^n)}
\leq C\|f\|_{L^\Phi(\mathbb{R}^n)}.
\end{align*}
\end{lemma}

Using Lemma~\ref{1910} and
Theorems~\ref{2117} and~\ref{4.8}, we obtain the following conclusion.

\begin{theorem}\label{1912}
Let $\Phi$ be an Orlicz function with both
positive lower type $r^-_{\Phi}$
and positive upper type $r^+_\Phi$.
Let $1\leq r^-_{\Phi}\leq r^+_{\Phi}<\infty$
and $q\in(0,\infty)$.
If one of the following statements holds true:
\begin{enumerate}
\item[\textup{(a)}]
$r^-_{\Phi}\in(n,\infty)$, $\gamma\in\mathbb{R}\setminus\{0\}$,
and $q\in(0,\infty)$;
\item[\textup{(b)}]
$r^-_{\Phi}\in[1,n]$, $\gamma\in(0,\infty)$, and
$n(\frac{1}{r^-_{\Phi}}-\frac{1}{q})<1$;
\item[\textup{(c)}]
$r^-_{\Phi}\in(1,n]$, $\gamma\in(-\infty,0)$, and
$q\in(0,r^-_{\Phi})$;
\item[\textup{(d)}]
$r^-_{\Phi}=n=q=1$ and $\gamma\in(-\infty,-1)$,
\end{enumerate}
then, for any $f\in\dot{W}^{1,\Phi}(\mathbb{R}^n)$,
\begin{align*}
\sup_{\lambda\in(0,\infty)}\lambda
\left\|\left[\int_{\mathbb{R}^n}
\mathbf{1}_{E_{\lambda,\frac{\gamma}{q}}[f]}(\cdot,y)
\left|\cdot-y\right|^{\gamma-n}\,dy\right]^\frac{1}{q}
\right\|_{L^\Phi(\mathbb{R}^n)}
\sim\left\|\,|\nabla f|\,\right\|_{L^\Phi(\mathbb{R}^n)},
\end{align*}
where the positive equivalence constants are independent of $f$
and $E_{\lambda,\frac{\gamma}{q}}[f]$
for any $\lambda\in(0,\infty)$
is the same as in \eqref{Elambda}.
Moreover, for any $f\in\dot{W}^{1,\Phi}(\mathbb{R}^n)$,
\begin{enumerate}
\item[\textup{(i)}]
if $\gamma\in(0,\infty)$, then
\begin{align*}
&\lim_{\lambda\to\infty}\lambda\left\|\left[\int_{\mathbb{R}^n}
\mathbf{1}_{E_{\lambda,\frac{\gamma}{q}}[f]}(\cdot,y)
\left|\cdot-y\right|^{\gamma-n}\,dy\right]^\frac{1}{q}
\right\|_{L^\Phi(\mathbb{R}^n)}
=\left[\frac{\kappa(q,n)}{\gamma}\right]^\frac{1}{q}
\left\|\,\left|\nabla f\right|\,\right\|_{L^\Phi(\mathbb{R}^n)};
\end{align*}
\item[\textup{(ii)}]
if $\gamma\in(-\infty,0)$,
assume further that both $n\in\mathbb{N}\cap[2,\infty)$
and $q\in(0,\frac{n-\gamma}{n}r^-_{\Phi})$
or that $n=1$, $\gamma\in(-\infty,-1)$,
and $q\in(0,-\gamma r^-_{\Phi})$,
then
\begin{align*}
&\lim_{\lambda\to0^+}\lambda\left\|\left[\int_{\mathbb{R}^n}
\mathbf{1}_{E_{\lambda,\frac{\gamma}{q}}[f]}(\cdot,y)
\left|\cdot-y\right|^{\gamma-n}\,dy\right]^\frac{1}{q}
\right\|_{L^\Phi(\mathbb{R}^n)}
=\left[-\frac{\kappa(q,n)}{\gamma}\right]^\frac{1}{q}
\left\|\,\left|\nabla f\right|\,\right\|_{L^\Phi(\mathbb{R}^n)}.
\end{align*}
\end{enumerate}
\end{theorem}

\begin{proof}
To prove the present theorem, we consider the following
five cases on both $\gamma$ and $r^-_\Phi$.

\emph{Case 1)} $\gamma\in\mathbb{R}\setminus\{0\}$ and $r^-_\Phi\in(n,\infty)$.
In this case,
we find that there exists an $s\in(n,r^-_\Phi)$,
which, combined with \cite[Theorem~1.2.1]{kk1991},
further implies that the Hardy--Littlewood
maximal operator $\mathcal{M}$ is bounded on
$[L^\Phi(\mathbb{R}^n)]^\frac{1}{s}$.
From the proof of \cite[Lemma~4.5]{zyyw2019},
we deduce that $L^\Phi(\mathbb{R}^n)$ has
an absolutely continuous norm.
Moreover,
when proving (ii),
if $\gamma\in(-\infty,0)$ and $n\in\mathbb{N}\cap[2,\infty)$,
then we choose a $p\in[1,r^-_\Phi)$
such that $q\in(0,\frac{n-\gamma}{n}p)$
and, if $\gamma\in(-\infty,0)$ and $n=1$,
then we choose a $p\in[1,r^-_\Phi)$ such that $q\in(0,-\gamma p)$.
Thus, Theorem~\ref{2117} implies that
the present theorem in this case holds true.

\emph{Case 2)} $\gamma\in(0,\infty)$ and $r^-_\Phi\in(1,n]$.
In this case,
since $n(\frac{1}{r^-_\Phi}-\frac{1}{q})<1$,
it follows that there exists a $p\in[1,r^-_\Phi)$ such that
$n(\frac{1}{p}-\frac{1}{q})<1$.
Let $\Phi_p(t):=\Phi(t^{\frac{1}{p}})$ for any $t\in[0,\infty)$.
Notice that
$[L^\Phi(\mathbb{R}^n)]^\frac{1}{p}=L^{\Phi_p}(\mathbb{R}^n)$.
From this and the definition of the Orlicz space,
it is easy to see that both $L^\Phi(\mathbb{R}^n)$
and $[L^\Phi(\mathbb{R}^n)]^\frac{1}{p}$
are ball Banach function spaces.
Moreover, by
\cite[Theorem~1.2.1]{kk1991} and
the dual theorem of $L^\Phi(\mathbb{R}^n)$
(see, for instance, \cite[Theorem~13]{rr2002}),
we further conclude that
$\mathcal{M}$ is bounded on
$([L^\Phi(\mathbb{R}^n)]^{\frac{1}{p}})'$.
In addition, from the proof of \cite[Lemma~4.5]{zyyw2019},
we deduce that $L^\Phi(\mathbb{R}^n)$ has
an absolutely continuous norm.
Thus, Theorem~\ref{2117}
implies that the present theorem in this case holds true.

\emph{Case 3)} $\gamma\in(0,\infty)$ and $r^-_\Phi=1$. In this case,
by the proof of \cite[Theorem~1.2.1]{kk1991},
we find that, when $1=r^-_\Phi\leq r^+_\Phi<\infty$ and $\theta\in(0,1)$,
for any $f\in([L^\Phi(\mathbb{R}^n)]^{\frac{1}{\theta}})'$,
\begin{align*}
\left\|\mathcal{M}f\right\|_{([L^\Phi(\mathbb{R}^n)]^{\frac{1}{\theta}})'}
\lesssim\left[3C_{(r^+_\Phi)}\right]^\frac{3r^+_\Phi}{\theta}
\left\|f\right\|_{([L^\Phi(\mathbb{R}^n)]^{\frac{1}{\theta}})'},
\end{align*}
where the implicit positive constant
depends only on $n$. Thus, \eqref{2114}
with $X:=L^\Phi(\mathbb{R}^n)$ holds true.
From this, an argument similar to that used in Case 2),
Lemma~\ref{1910}, and Theorem~\ref{4.8},
we deduce that the present theorem in this case holds true.

\emph{Case 4)} $\gamma\in(-\infty,0)$ and $r^-_\Phi\in(1,n]$.
In this case, by \cite[Theorem~1.2.1]{kk1991},
we find that $\mathcal{M}$ is bounded on $L^\Phi(\mathbb{R}^n)$.
Since $q\in(0,r^-_\Phi)$, it follows that
there exists a $p\in[\max\{1,\,q\},r^-_\Phi)$ and,
when proving (ii), if $n=1$, we always choose a
$p\in[\max\{1,\,q\},r^-_\Phi)$ satisfying $q\in(0,-\gamma p)$.
By these, an argument similar to that used in Case 2),
and Theorem~\ref{2117}, we find that
the present theorem in this case holds true.

\emph{Case 5)} $\gamma\in(-\infty,-1)$ and $r^-_\Phi=1$. In this case,
from an argument similar to that used in Case 3), Lemma~\ref{1910},
and Theorem~\ref{4.8}, we deduce that
the present theorem in this case also holds true.
This finishes the proof Theorem~\ref{1912}.
\end{proof}

Using Lemma~\ref{1910}, Theorems~\ref{2255} and~\ref{2256},
and Remarks~\ref{4.3}(iv) and~\ref{2021}(ii),
we obtain the following
conclusions; since their proofs are similar to that of Theorem~\ref{1912},
we omit the details here.

\begin{theorem}\label{1922}
Let $\gamma\in\mathbb{R}\setminus\{0\}$ and
$\Phi$ be an Orlicz function with both positive lower type $r^-_{\Phi}$
and positive upper type $r^+_\Phi$.
Let $1\leq r^-_{\Phi}\leq r^+_{\Phi}<\infty$,
$p\in[1,\infty]$,
and $s\in(0,1)$.
Let $q\in[1,p]$ satisfy $\frac{1}{q}=\frac{1-s}{p}+s$.
If $r^-_{\Phi}=1$ and $\gamma\in(-\infty,0)$,
assume further that both $\gamma\in(-\infty,-1)$ and $n=1$.
\begin{enumerate}
\item[\textup{(i)}]
If $p\in[1,\infty)$, then, for any $f\in\dot{W}^{1,\Phi}(\mathbb{R}^n)$,
\eqref{2053} with $X$ replaced by $L^{\Phi}(\mathbb{R}^n)$ holds true.
\item[\textup{(ii)}]
If $p=\infty$, then, for any $f\in\dot{W}^{1,\Phi}(\mathbb{R}^n)$,
\eqref{2054} with $X$ replaced by $L^{\Phi}(\mathbb{R}^n)$ holds true.
\end{enumerate}
\end{theorem}

\begin{theorem}\label{1923}
Let $\gamma\in\mathbb{R}\setminus\{0\}$ and
$\Phi$ be an Orlicz function with both positive lower type $r^-_{\Phi}$
and positive upper type $r^+_\Phi$.
Let $1\leq r^-_{\Phi}\leq r^+_{\Phi}<\infty$,
$\eta\in(0,1)$,
and
$$
0\leq s_0<s<1<q<q_0<\infty
$$
satisfy \eqref{2228}.
If $r^-_{\Phi}=1$ and $\gamma\in(-\infty,0)$,
assume further that both $\gamma\in(-\infty,-1)$ and $n=1$.
Then, for any $f\in\dot{W}^{1,\Phi}(\mathbb{R}^n)$,
\eqref{939} with $X$ replaced by $L^{\Phi}(\mathbb{R}^n)$ holds true.
\end{theorem}

\begin{remark}
\begin{enumerate}
\item[(i)]
Theorem~\ref{1912} when $\gamma=n$
coincides with \cite[Theorem~5.20]{dlyyz.arxiv}.
To the best of our knowledge,
Theorem~\ref{1912}
when $\gamma\in\mathbb{R}\setminus\{0,n\}$
is new.
\item[(ii)]
Theorems~\ref{1922} and~\ref{1923} when $\gamma=n$
coincide with \cite[Corollaries~5.21 and~5.22]{dlyyz.arxiv}, respectively.
To the best of our knowledge,
Theorems~\ref{1922} and~\ref{1923}
(the Gagliardo--Nirenberg-type inequalities
on the Sobolev--Orlicz space)
when $\gamma\in\mathbb{R}\setminus\{0,n\}$
are new.
\end{enumerate}
\end{remark}

\subsection{Orlicz-Slice Spaces}\label{5.6}

We recall the definition of Orlicz-slice
spaces and briefly describe some related facts.
Throughout this subsection, we always assume
that $\Phi: [0,\infty)\to [0,\infty)$
is an Orlicz function with both positive
lower type $r_{\Phi}^-$ and positive upper
type $r_{\Phi}^+$. For any given $t,r\in(0,\infty)$,
the \emph{Orlicz-slice space}
$(E_\Phi^r)_t(\mathbb{R}^n)$ is defined to be the set of all
the $f\in\mathscr{M}(\mathbb{R}^n)$ with the
finite quasi-norm
\begin{equation*}
\|f\|_{(E_\Phi^r)_t(\mathbb{R}^n)} :=\left\{\int_{\mathbb{R}^n}
\left[\frac{\|f\mathbf{1}_{B(x,t)}\|_{L^\Phi(\mathbb{R}^n)}}
{\|\mathbf{1}_{B(x,t)}\|_{L^\Phi(\mathbb{R}^n)}}\right]
^r\,dx\right\}^{\frac{1}{r}}.
\end{equation*}
The Orlicz-slice spaces were introduced in
\cite{zyyw2019} as a generalization of both
the slice space of Auscher and Mourgoglou
\cite{am2019,ap2017} and the Wiener amalgam space
in \cite{h2019,h1975,kntyy2007}. According to both
\cite[Lemma 2.28]{zyyw2019} and \cite[Remark 7.41(i)]{zwyy2021},
the Orlicz-slice space $(E_\Phi^r)_t(\mathbb{R}^n)$ is a
ball Banach function space, but in general is not a
Banach function space.
When $X:=(E_\Phi^r)_t(\mathbb{R}^n)$,
we simply write $\dot{W}^{1,(E_\Phi^r)_t}
(\mathbb{R}^n):=\dot{W}^{1,X}(\mathbb{R}^n)$.

The following boundedness of centered ball average operators
on Orlicz-slice spaces
is just \cite[Lemma~5.24]{dlyyz.arxiv}.

\begin{lemma}\label{2043}
Let $t\in(0,\infty)$, $r\in[1,\infty)$,
and $\Phi$ be an Orlicz function with
both positive lower type $r^-_{\Phi}\in[1,\infty)$
and positive upper type $r^+_\Phi$.
Then centered ball average operators
$\{\mathcal{B}_r\}_{r\in(0,\infty)}$
are uniformly bounded on $(E_\Phi^r)_t(\mathbb{R}^n)$;
moreover, there exists a positive constant $C$
such that, for any $t,s\in(0,\infty)$
and $f\in(E_\Phi^r)_t(\mathbb{R}^n)$,
\begin{align*}
\left\|\mathcal{B}_sf\right\|_{(E_\Phi^r)_t(\mathbb{R}^n)}
\leq C\|f\|_{(E_\Phi^r)_t(\mathbb{R}^n)}.
\end{align*}
\end{lemma}

Using Lemma~\ref{2043} and
Theorems~\ref{2117} and~\ref{4.8},
we obtain the following conclusion.

\begin{theorem}\label{2045}
Let $t\in(0,\infty)$, $r\in[1,\infty)$,
and $\Phi$ be an Orlicz function with
both positive lower type $r^-_{\Phi}$
and positive upper type $r^+_\Phi$.
Let $1\leq r^-_{\Phi}\leq r^+_{\Phi}<\infty$
and $q\in(0,\infty)$.
If one of the following statements holds true:
\begin{enumerate}
\item[\textup{(a)}]
$\min\{r^-_{\Phi},\,r\}\in(n,\infty)$,
$\gamma\in\mathbb{R}\setminus\{0\}$,
and $q\in(0,\infty)$;
\item[\textup{(b)}]
$r^-_{\Phi}\in[1,\infty)$, $r\in[1,\infty)$,
$\gamma\in(0,\infty)$, and
$n(\frac{1}{\min\{r^-_\Phi,\,r\}}-\frac{1}{q})<1$;
\item[\textup{(c)}]
$r^-_{\Phi}\in(1,\infty)$, $r\in(1,\infty)$,
$\gamma\in(-\infty,0)$, and
$q\in(0,\min\{r^-_\Phi,\,r\})$;
\item[\textup{(d)}]
$r^-_{\Phi}\in(1,\infty)$, $r=n=q=1$,
and $\gamma\in(-\infty,-1)$;
\item[\textup{(e)}]
$r^-_{\Phi}=n=q=1$, $r\in[1,\infty)$,
and $\gamma\in(-\infty,-1)$,
\end{enumerate}
then, for any $f\in\dot{W}^{1,(E_\Phi^r)_t}(\mathbb{R}^n)$,
\begin{align*}
\sup_{\lambda\in(0,\infty)}\lambda
\left\|\left[\int_{\mathbb{R}^n}
\mathbf{1}_{E_{\lambda,\frac{\gamma}{q}}[f]}(\cdot,y)
\left|\cdot-y\right|^{\gamma-n}\,dy\right]^\frac{1}{q}
\right\|_{(E_\Phi^r)_t(\mathbb{R}^n)}
\sim\left\|\,|\nabla f|\,\right\|_{(E_\Phi^r)_t(\mathbb{R}^n)},
\end{align*}
where the positive equivalence constants are independent of $f$
and $E_{\lambda,\frac{\gamma}{q}}[f]$
for any $\lambda\in(0,\infty)$
is the same as in \eqref{Elambda}.
Moreover, for any $f\in\dot{W}^{1,(E_\Phi^r)_t}(\mathbb{R}^n)$,
\begin{enumerate}
\item[\textup{(i)}]
if $\gamma\in(0,\infty)$, then
\begin{align*}
&\lim_{\lambda\to\infty}\lambda\left\|\left[\int_{\mathbb{R}^n}
\mathbf{1}_{E_{\lambda,\frac{\gamma}{q}}[f]}(\cdot,y)
\left|\cdot-y\right|^{\gamma-n}\,dy\right]^\frac{1}{q}
\right\|_{(E_\Phi^r)_t(\mathbb{R}^n)}
=\left[\frac{\kappa(q,n)}{\gamma}\right]^\frac{1}{q}
\left\|\,\left|\nabla f\right|\,\right\|_{(E_\Phi^r)_t(\mathbb{R}^n)};
\end{align*}
\item[\textup{(ii)}]
if $\gamma\in(-\infty,0)$,
assume further that both $n\in\mathbb{N}\cap[2,\infty)$
and $q\in(0,\frac{n-\gamma}{n}\min\{r^-_\Phi,\,r\})$
or that $n=1$, $\gamma\in(-\infty,-1)$,
and $q\in(0,-\gamma\min\{r^-_\Phi,\,r\})$,
then
\begin{align*}
&\lim_{\lambda\to0^+}\lambda\left\|\left[\int_{\mathbb{R}^n}
\mathbf{1}_{E_{\lambda,\frac{\gamma}{q}}[f]}(\cdot,y)
\left|\cdot-y\right|^{\gamma-n}\,dy\right]^\frac{1}{q}
\right\|_{(E_\Phi^r)_t(\mathbb{R}^n)}
=\left[-\frac{\kappa(q,n)}{\gamma}\right]^\frac{1}{q}
\left\|\,\left|\nabla f\right|\,\right\|_{(E_\Phi^r)_t(\mathbb{R}^n)}.
\end{align*}
\end{enumerate}
\end{theorem}

\begin{proof}
To prove the present theorem, we consider the following six cases
on $r^-_{\Phi}$, $r$, $\gamma$, and $q$.

\emph{Case 1)}
$\min\{r^-_{\Phi},\,r\}\in(n,\infty)$,
$\gamma\in\mathbb{R}\setminus\{0\}$,
and $q\in(0,\infty)$. In this case,
we find that there exists an $s\in(n,\min\{r^-_{\Phi},\,r\})$,
which, combined with \cite[Lemmas~4.3]{zyyw2019},
further implies that the Hardy--Littlewood
maximal operator $\mathcal{M}$ is bounded on
$[(E_\Phi^r)_t(\mathbb{R}^n)]^\frac{1}{s}$.
From \cite[Lemma~4.5]{zyyw2019},
we deduce that $(E_\Phi^r)_t(\mathbb{R}^n)$
has an absolutely continuous norm.
Moreover,
when proving (ii),
if $\gamma\in(-\infty,0)$ and $n\in\mathbb{N}\cap[2,\infty)$,
then we choose a $p\in[1,\min\{r^-_\Phi,\,r\})$
such that $q\in(0,\frac{n-\gamma}{n}p)$
and, if $\gamma\in(-\infty,0)$ and $n=1$,
then we choose a $p\in[1,\min\{r^-_\Phi,\,r\})$ such that $q\in(0,-\gamma p)$.
Thus, Theorem~\ref{2117} implies that
the present theorem in this case holds true.

\emph{Case 2)}
$r^-_{\Phi}\in(1,\infty)$, $r\in[1,\infty)$,
$\gamma\in(0,\infty)$, and
$n(\frac{1}{\min\{r^-_\Phi,\,r\}}-\frac{1}{q})<1$. In this case,
since $n(\frac{1}{\min\{r^-_\Phi,\,r\}}-\frac{1}{q})<1$,
it follows that there exists a $p\in[1,\min\{r^-_\Phi,\,r\})$ such
that $n(\frac{1}{p}-\frac{1}{q})<1$.
By \cite[Lemma~2.31]{zyyw2019}, we find that
$[(E_\Phi^r)_t(\mathbb{R}^n)]^\frac{1}{p}
=(E_{\Phi_p}^\frac{r}{p})_t(\mathbb{R}^n)$,
where $\Phi_p(t):=\Phi(t^{\frac{1}{p}})$ for any $t\in[0,\infty)$.
From this, \cite[Lemma~2.28]{zyyw2019}, and \cite[Remark~7.41(i)]{zwyy2021},
we infer that both $(E_\Phi^r)_t(\mathbb{R}^n)$
and $[(E_\Phi^r)_t(\mathbb{R}^n)]^\frac{1}{p}$
are ball Banach function spaces.
Moreover, by \cite[Theorem~2.26 and Lemma~4.4]{zyyw2019},
we conclude that $\mathcal{M}$ is
bounded on $([(E_\Phi^r)_t(\mathbb{R}^n)]^\frac{1}{p})'$.
In addition, from \cite[Lemma~4.5]{zyyw2019},
we deduce that $(E_\Phi^r)_t(\mathbb{R}^n)$
has an absolutely continuous norm.
Thus, Theorem~\ref{2117} implies that
the present theorem in this case holds true.

\emph{Case 3)}
$r^-_{\Phi}=1$, $r\in[1,\infty)$,
$\gamma\in(0,\infty)$, and
$n(\frac{1}{\min\{r^-_\Phi,\,r\}}-\frac{1}{q})<1$. In this case,
by the proof of \cite[Theorem~2.20]{zyyw2019}, we conclude that,
when $1=r^-_{\Phi}\leq r^+_{\Phi}<\infty$,
$r\in[1,\infty)$, and $\theta\in(0,1)$,
for any $f\in(X^\frac{1}{\theta})'$ with $X:=(E_\Phi^r)_t(\mathbb{R}^n)$,
\begin{align*}
\left\|\mathcal{M}(f)\right\|_{(X^\frac{1}{\theta})'}
\lesssim\left[\left(3C_{r^+_\Phi}\right)^{\frac{3r^+_\Phi}{\theta}}
+\frac{r}{\theta}\right]
\|f\|_{(X^\frac{1}{\theta})'},
\end{align*}
where the implicit positive constant depends only on $n$.
Thus, \eqref{2114} with $X:=(E_\Phi^r)_t(\mathbb{R}^n)$ holds true.
From this, Lemma~\ref{2043},
an argument similar to that used in Case 2),
and Theorem~\ref{4.8},
we deduce that the present theorem in this case holds true.

\emph{Case 4)}
$r^-_{\Phi}\in(1,\infty)$, $r\in(1,\infty)$,
$\gamma\in(-\infty,0)$, and
$q\in(0,\min\{r^-_\Phi,\,r\})$.
In this case, by \cite[Lemmas~4.3]{zyyw2019},
we find that $\mathcal{M}$ is bounded on $(E_\Phi^r)_t(\mathbb{R}^n)$.
Since $q\in(0,r^-_\Phi)$, it follows that
there exists a $p\in[\max\{1,\,q\},r^-_\Phi)$ and,
when proving (ii), if $n=1$, we always choose a
$p\in[\max\{1,\,q\},r^-_\Phi)$ satisfying $q\in(0,-\gamma p)$.
By these, an argument similar to that used in Case 2),
and Theorem~\ref{2117}, we find that
the present theorem in this case holds true.

\emph{Case 5)}
$r^-_{\Phi}\in(1,\infty)$, $r=q=1$,
and $\gamma\in(-\infty,-1)$.
In this case,
from an argument similar to that used in Case 2)
and Theorem~\ref{2117}, we deduce that
the present theorem in this case holds true.

\emph{Case 6)}
$r^-_{\Phi}=q=1$, $r\in[1,\infty)$,
and $\gamma\in(-\infty,-1)$. In this case,
from an argument similar to that used in Case 3), Lemma~\ref{2043},
and Theorem~\ref{4.8}, we deduce that
the present theorem in this case holds true.
This finishes the proof of Theorem~\ref{2045}.
\end{proof}

Using Lemma~\ref{2043}, Theorems~\ref{2255} and~\ref{2256},
and Remarks~\ref{4.3}(iv) and~\ref{2021}(ii),
we obtain the following conclusions;
since their proofs are similar to that of Theorem~\ref{2045},
we omit the details here.

\begin{theorem}\label{2048}
Let $t\in(0,\infty)$, $r\in[1,\infty)$,
and $\Phi$ be an Orlicz function with
both positive lower type $r^-_{\Phi}$
and positive upper type $r^+_\Phi$.
Let $1\leq r^-_{\Phi}\leq r^+_{\Phi}<\infty$,
$p\in[1,\infty]$,
and $s\in(0,1)$.
Let $q\in[1,p]$ satisfy $\frac{1}{q}=\frac{1-s}{p}+s$.
Moreover, if $r^-_{\Phi}=1$ and $\gamma\in(-\infty,0)$,
assume further that
both $\gamma\in(-\infty,-1)$ and $n=1$.
\begin{enumerate}
\item[\textup{(i)}]
If $p\in[1,\infty)$, then, for any
$f\in\dot{W}^{1,(E_\Phi^r)_t}(\mathbb{R}^n)$,
\eqref{2053} with $X$
replaced by $(E_\Phi^r)_t(\mathbb{R}^n)$ holds true.
\item[\textup{(ii)}]
If $p=\infty$, then, for any
$f\in\dot{W}^{1,(E_\Phi^r)_t}(\mathbb{R}^n)$,
\eqref{2054} with $X$
replaced by $(E_\Phi^r)_t(\mathbb{R}^n)$ holds true.
\end{enumerate}
\end{theorem}

\begin{theorem}\label{2049}
Let $t\in(0,\infty)$, $r\in[1,\infty)$,
and $\Phi$ be an Orlicz function with
both positive lower type $r^-_{\Phi}$
and positive upper type $r^+_\Phi$.
Let $1\leq r^-_{\Phi}\leq r^+_{\Phi}<\infty$,
$\eta\in(0,1)$,
and
$$
0\leq s_0<s<1<q<q_0<\infty
$$
satisfy \eqref{2228}.
Moreover, if $r^-_{\Phi}=1$ and $\gamma\in(-\infty,0)$,
assume further that
both $\gamma\in(-\infty,-1)$ and $n=1$.
Then, for any $f\in\dot{W}^{1,(E_\Phi^r)_t}(\mathbb{R}^n)$,
\eqref{939} with $X$ replaced by $(E_\Phi^r)_t(\mathbb{R}^n)$ holds true.
\end{theorem}

\begin{remark}
\begin{enumerate}
\item[(i)]
Theorem~\ref{2045} when $\gamma=n$
coincides with \cite[Theorem~5.25]{dlyyz.arxiv}.
To the best of our knowledge,
Theorem~\ref{2045}
when $\gamma\in\mathbb{R}\setminus\{0,n\}$
is new.
\item[(ii)]
Theorems~\ref{2048} and~\ref{2049} when $\gamma=n$
coincide with \cite[Corollaries~5.26 and~5.27]{dlyyz.arxiv}, respectively.
To the best of our knowledge,
Theorems~\ref{2048} and~\ref{2049}
(the Gagliardo--Nirenberg-type inequalities
on the Sobolev--Orlicz-slice space)
when $\gamma\in\mathbb{R}\setminus\{0,n\}$
are new.
\end{enumerate}
\end{remark}

\noindent\textbf{Acknowledgements}\quad
Chenfeng Zhu would like to thank Yangyang Zhang for some
useful discussions and suggestions on the subject of this article.

\medskip

\noindent\textbf{Author Contributions}\quad All authors developed and discussed the results and contributed to the final manuscript.

\medskip

\noindent\textbf{Data Availibility Statement}\quad Data sharing is not applicable to this article as no data sets were generated or
analysed.

\medskip

\noindent\textbf{Declarations}

\medskip

\noindent\textbf{Conflict of interest}\quad All authors state no conflict of interest.

\medskip

\noindent\textbf{Informed consent}\quad Informed consent has been obtained from all individuals included in this research work.

\bigskip

\noindent  Chenfeng Zhu, Dachun Yang (Corresponding author) and Wen Yuan

\medskip

\noindent Laboratory of Mathematics
and Complex Systems (Ministry of Education of China),
School of Mathematical Sciences,
Beijing Normal University,
Beijing 100875, The People's Republic of China

\smallskip

\noindent{\it E-mails:} \texttt{cfzhu@mail.bnu.edu.cn} (C. Zhu)

\noindent\phantom{{\it E-mails:} }\texttt{dcyang@bnu.edu.cn} (D. Yang)

\noindent\phantom{{\it E-mails:} }\texttt{wenyuan@bnu.edu.cn} (W. Yuan)


\begin{thebibliography}{999}

\bibitem{ahh2015}
T. Adamowicz, P. Harjulehto and P. H\"ast\"o,
Maximal operator in variable exponent Lebesgue spaces
on unbounded quasimetric measure spaces,
Math. Scand. 116 (2015), 5--22.

\vspace{-0.3cm}

\bibitem{a2015}
D. R. Adams, Morrey Spaces, Lecture Notes in
Applied and Numerical Harmonic Analysis,
Birkh\"auser/Springer, Cham, 2015.

\vspace{-0.3cm}

\bibitem{aj1980/81}
K. F. Andersen and R. T. John,
Weighted inequalities for vector-valued maximal functions
and singular integrals,
Studia Math. 69 (1980/81), 19--31.

\vspace{-0.3cm}

\bibitem{am2019}
P. Auscher and M. Mourgoglou,
Representation and uniqueness for boundary value elliptic
problems via first order systems,
Rev. Mat. Iberoam. 35 (2019), 241--315.

\vspace{-0.3cm}

\bibitem{ap2017}
P. Auscher and C. Prisuelos-Arribas,
Tent space boundedness via extrapolation,
Math. Z. 286 (2017), 1575--1604.

\vspace{-0.3cm}

\bibitem{bp1961}
A. Benedek and R. Panzone,
The space $L^p$, with mixed norm,
Duke Math. J. 28 (1961), 301--324.

\vspace{-0.3cm}

\bibitem{bs1988}
C. Bennett and R. Sharpley,
Interpolation of Operators, Pure and Applied Mathematics 129,
Academic Press, Inc., Boston, MA, 1988.

\vspace{-0.3cm}

\bibitem{bbm2000}
J. Bourgain, H. Brezis and P. Mironescu,
Lifting in Sobolev spaces,
J. Anal. Math. 80 (2000), 37--86.


\vspace{-0.3cm}

\bibitem{bbm2001}
J. Bourgain, H. Brezis and P. Mironescu,
Another look at Sobolev spaces, in: Optimal
Control and Partial Differential Equations,
IOS, Amsterdam, 2001, pp. 439--455.

\vspace{-0.3cm}

\bibitem{bbm2002}
J. Bourgain, H. Brezis and P. Mironescu,
Limiting embedding theorems for $W^{s,p}$
when $s\uparrow 1$ and applications,
J. Anal. Math. 87 (2002), 77--101.

\vspace{-0.3cm}

\bibitem{bsy2023}
D. Brazke, A. Schikorra and P.-L. Yung,
Bourgain--Brezis--Mironescu convergence via Triebel--Lizorkin
spaces,
Calc. Var. Partial Differential Equations 62 (2023),
Paper No. 41, 33 pp.

\vspace{-0.3cm}

\bibitem{b2002}
H. Brezis,
How to recognize constant functions. A connection with Sobolev spaces,
Russian Math. Surveys 57 (2002), 693--708.

\vspace{-0.3cm}

\bibitem{bm2018}
H. Brezis and P. Mironescu,
Gagliardo--Nirenberg inequalities and non-inequalities: the full story,
Ann. Inst. H. Poincar\'e C Anal. Non Lin\'eaire 35 (2018), 1355--1376.

\vspace{-0.3cm}

\bibitem{bm2019}
H. Brezis and P. Mironescu,
Where Sobolev interacts with Gagliardo--Nirenberg,
J. Funct. Anal. 277 (2019), 2839--2864.

\vspace{-0.3cm}

\bibitem{bn2016}
H. Brezis and H.-M. Nguyen,
The BBM formula revisited,
Atti Accad. Naz. Lincei Rend. Lincei Mat. Appl. 27 (2016), 515--533.

\vspace{-0.3cm}

\bibitem{bn2018}
H. Brezis and H.-M. Nguyen,
Non-local functionals related to the total variation
and connections with image processing,
Ann. PDE 4 (2018), Paper No. 9, 77 pp.

\vspace{-0.3cm}

\bibitem{bn2020}
H. Brezis and H.-M. Nguyen,
Non-local, non-convex functionals converging to Sobolev norms,
Nonlinear Anal. 191 (2020), 111626, 9 pp.

\vspace{-0.3cm}

\bibitem{bsvy.arxiv}
H. Brezis, A. Seeger, J. Van Schaftingen and P.-L. Yung,
Families of functionals representing Sobolev norms,
Anal. PDE (to appear) or arXiv: 2109.02930.

\vspace{-0.3cm}

\bibitem{bsvy}
H. Brezis, A. Seeger, J. Van Schaftingen and P.-L. Yung,
Sobolev spaces revisited,
Atti Accad. Naz. Lincei Rend. Lincei Mat. Appl.
33 (2022), 413--437.

\vspace{-0.3cm}

\bibitem{bvy2021}
H. Brezis, J. Van Schaftingen and P.-L. Yung,
A surprising formula for Sobolev norms,
Proc. Natl. Acad. Sci. USA 118 (2021),
Paper No. e2025254118, 6 pp.

\vspace{-0.3cm}

\bibitem{bvyCVPDE}
H. Brezis, J. Van Schaftingen and P.-L. Yung,
Going to Lorentz when fractional Sobolev,
Gagliardo and Nirenberg estimates fail,
Calc. Var. Partial Differential Equations 60 (2021),
Paper No. 129, 12 pp.

\vspace{-0.3cm}

\bibitem{crs2010}
L. Caffarelli, J.-M. Roquejoffre and O. Savin,
Nonlocal minimal surfaces,
Comm. Pure Appl. Math. 63 (2010), 1111--1144.

\vspace{-0.3cm}

\bibitem{cv2011}
L. Caffarelli and E. Valdinoci,
Uniform estimates and limiting
arguments for nonlocal minimal surfaces,
Calc. Var. Partial Differential Equations 41 (2011), 203--240.

\vspace{-0.3cm}

\bibitem{cwyz2020}
D.-C. Chang, S. Wang, D. Yang and Y. Zhang,
Littlewood--Paley characterizations of Hardy-type spaces
associated with ball quasi-Banach function spaces,
Complex Anal. Oper. Theory 14 (2020), Paper No. 40, 33 pp.

\vspace{-0.3cm}

\bibitem{ch2014}
K. Cheung and K.-P. Ho,
Boundedness of Hardy--Littlewood maximal operator on block
spaces with variable exponent,
Czechoslovak Math. J. 64(139) (2014), 159--171.

\vspace{-0.3cm}

\bibitem{cf1987}
F. Chiarenza and M. Frasca,
Morrey spaces and Hardy--Littlewood maximal function,
Rend. Mat. Appl. (7) 7 (1987), 273--279 (1988).

\vspace{-0.3cm}

\bibitem{cgn2017}
G. Cleanthous, A. G. Georgiadis and M. Nielsen,
Anisotropic mixed-norm Hardy spaces,
J. Geom. Anal. 27 (2017), 2758--2787.

\vspace{-0.3cm}

\bibitem{cgn2017b}
G. Cleanthous, A. G. Georgiadis and M. Nielsen,
Discrete decomposition of homogeneous
mixed-norm Besov spaces, in: Functional Analysis,
Harmonic Analysis, and Image Processing:
A Collection of Papers in Honor of Bj\"orn Jawerth,
167--184, Contemp. Math. 693,
Amer. Math. Soc. Providence, RI, 2017.

\vspace{-0.3cm}

\bibitem{cgn2019}
G. Cleanthous, A. G. Georgiadis and M. Nielsen,
Fourier multipliers on anisotropic mixednorm spaces of distributions,
Math. Scand. 124 (2019), 289--304.

\vspace{-0.3cm}

\bibitem{cf2013}
D. V. Cruz-Uribe and A. Fiorenza,
Variable Lebesgue Spaces. Foundations and Harmonic Analysis,
Appl. Number. Harmon. Anal., Birkh\"auser/Springer, Heidelberg, 2013.

\vspace{-0.3cm}

\bibitem{cw2014}
D. V. Cruz-Uribe and L.-A. D. Wang,
Variable Hardy spaces,
Indiana Univ. Math. J. 63 (2014), 447--493.

\vspace{-0.3cm}

\bibitem{dgpyyz2022}
F. Dai, L. Grafakos, Z. Pan, D. Yang, W. Yuan and Y. Zhang,
The Bourgain--Brezis--Mironescu formula on ball Banach
function spaces, Math. Ann. (2023),
https://doi.org/10.1 007/s00208-023-02562-5.

\vspace{-0.3cm}

\bibitem{dlyyz2022}
F. Dai, X. Lin, D. Yang, W. Yuan and Y. Zhang,
Poincar\'e inequality meets Brezis--Van Schaftingen--Yung formula
on metric measure spaces,
J. Funct. Anal. 283 (2022), Paper No. 109645, 52 pp.

\vspace{-0.3cm}

\bibitem{dlyyz.arxiv}
F. Dai, X. Lin, D. Yang, W. Yuan and Y. Zhang,
Brezis--Van Schaftingen--Yung formulae in\\ ball Banach function
spaces with applications to fractional Sobolev and
Gagliardo--Nirenb-\\erg inequalities,
Calc. Var. Partial Differential Equations 62 (2023),
Paper No. 56, 73 pp.

\vspace{-0.3cm}

\bibitem{dfmn2021}
R. del Campo, A. Fern\'andez, F. Mayoral and F. Naranjo,
Orlicz spaces associated to a
quasi-Banach function space: applications to
vector measures and interpolation,
Collect. Math. 72 (2021), 481--499.

\vspace{-0.3cm}

\bibitem{dhhr2011}
L. Diening, P. Harjulehto, P. H\"ast\"o and
M. R$\mathring{\mathrm{u}}$\v{z}i\v{c}ka,
Lebesgue and Sobolev Spaces with Variable Exponents,
Lecture Notes in Math. 2017, Springer, Heidelberg, 2011.

\vspace{-0.3cm}

\bibitem{dhr2009}
L. Diening, P. H\"ast\"o and S. Roudenko,
Function spaces of variable smoothness and integrability,
J. Funct. Anal. 256 (2009), 1731--1768.



\vspace{-0.3cm}

\bibitem{dm.arXiv1}
O. Dom\'inguez and M. Milman,
New Brezis--Van
Schaftingen--Yung--Sobolev type inequalities
connected with maximal inequalities
and one parameter families of operators,
Adv. Math. 411 (2022), Paper No. 108774, 76 pp.


\vspace{-0.3cm}

\bibitem{dm.arXiv}
O. Dom\'inguez and M. Milman,
Bourgain--Brezis--Mironescu--Maz'ya--Shaposhnikova
lim-it formulae for fractional Sobolev spaces
via interpolation and extrapolation,
Calc. Var. Partial Differential Equations 62 (2023),
Paper No. 43, 37 pp.

\vspace{-0.3cm}

\bibitem{dssvy}
O. Dom\'inguez, A. Seeger, B. Street, J. Van Schaftingen and P.-L. Yung,
Spaces of Besov--Sobolev type and a problem on nonlinear
approximation, J. Funct. Anal. 284 (2023),
Paper No. 109775, 50 pp.

\vspace{-0.3cm}

\bibitem{dt.arXiv}
O. Dom\'inguez and S. Tikhonov,
Sobolev embeddings, extrapolations, and related inequalities,
arXiv: 1909.12818.

\vspace{-0.3cm}

\bibitem{D2000}
J. Duoandikoetxea, Fourier Analysis,
Graduate Studies in Mathematics 29,
American Mathematical Society Providence, RI, 2001.

\vspace{-0.3cm}

\bibitem{eg2015}
L. C. Evans and R. F. Gariepy,
Measure Theory and Fine Properties of Functions,
Textbooks in Mathematics. CRC Press, Boca Raton, FL, 2015.

\vspace{-0.3cm}

\bibitem{g1958}
E. Gagliardo,
Propriet\`a di alcune classi di funzioni in pi\`u variabili,
Ricerche Mat. 7 (1958), 102--137.

\vspace{-0.3cm}

\bibitem{gp1965}
A. R. Galmarino and R. Panzone,
$L^p$-spaces with mixed norm, for $P$ a sequence,
J. Math. Anal. Appl. 10 (1965), 494--518.

\vspace{-0.3cm}

\bibitem{gjn2017}
A. G. Georgiadis, J. Johnsen and M. Nielsen,
Wavelet transforms for homogeneous
mixed-norm Triebel--Lizorkin spaces,
Monatsh. Math. 183 (2017), 587--624.

\vspace{-0.3cm}

\bibitem{gn2016}
A. G. Georgiadis and M. Nielsen,
Pseudodifferential operators on mixed-norm
Besov and Triebel--Lizorkin spaces,
Math. Nachr. 289 (2016), 2019--2036.

\vspace{-0.3cm}

\bibitem{g2014}
L. Grafakos,
Classical Fourier Analysis, Third edition,
Graduate Texts in Math. 249. Springer, New York, 2014.

\vspace{-0.3cm}

\bibitem{gs2020}
L. Greco and R. Schiattarella,
An embedding theorem for BV-functions,
Commun. Contemp. Math. 22 (2020), 1950032, 13 pp.

\vspace{-0.3cm}

\bibitem{gy2021}
Q. Gu and P.-L. Yung,
A new formula for the $L^p$ norm,
J. Funct. Anal. 281 (2021), Paper No. 109075, 19 pp.

\vspace{-0.3cm}

\bibitem{h2008}
D. D. Haroske,
Sobolev spaces with Muckenhoupt weights,
singularities and inequalities,
Georgian Math. J. 15 (2008), 263--280.

\vspace{-0.3cm}

\bibitem{hms2017}
D. D. Haroske, S. D. Moura, C. Schneider and L. Skrzypczak,
Unboundedness properties of smoothness
Morrey spaces of regular distributions on domains,
Sci. China Math. 60 (2017), 2349--2376.

\vspace{-0.3cm}

\bibitem{hms2020}
D. D. Haroske, S. D. Moura and L. Skrzypczak,
Some embeddings of Morrey spaces with
critical smoothness,
J. Fourier Anal. Appl. 26 (2020), Paper No. 50, 31 pp.

\vspace{-0.3cm}

\bibitem{hss2018}
D. D. Haroske, C. Schneider and L. Skrzypczak,
Morrey spaces on domains: different
approaches and growth envelopes,
J. Geom. Anal. 28 (2018), 817--841.

\vspace{-0.3cm}

\bibitem{hs2017}
D. D. Haroske and L. Skrzypczak,
Embeddings of weighted Morrey spaces,
Math. Nachr. 290 (2017), 1066--1086.

\vspace{-0.3cm}

\bibitem{ht2008}
D. D. Haroske and H. Triebel,
Distributions, Sobolev Spaces, Elliptic Equations, EMS
Textbooks in Mathematics,
European Mathematical Society (EMS), Z\"urich, 2008.

\vspace{-0.3cm}

\bibitem{h2015}
K.-P. Ho,
Atomic decomposition of Hardy--Morrey spaces with variable exponents,
Ann. Acad. Sci. Fenn. Math. 40 (2015), 31--62.

\vspace{-0.3cm}

\bibitem{h2019}
K.-P. Ho,
Dilation operators and integral operators on amalgam space $(L_p,l_q)$,
Ric. Mat. 68 (2019), 661--677.

\vspace{-0.3cm}

\bibitem{h2021}
K.-P. Ho,
Erd\'elyi--Kober fractional integral operators
on ball Banach function spaces,
Rend. Semin. Mat. Univ. Padova 145 (2021), 93--106.

\vspace{-0.3cm}

\bibitem{h1975}
F. Holland,
Harmonic analysis on amalgams of $L^p$ and $l^q$,
J. London Math. Soc. (2) 10 (1975), 295--305.

\vspace{-0.3cm}

\bibitem{h1960}
L. H\"ormander,
Estimates for translation invariant operators in $L^p$ spaces,
Acta Math. 104 (1960), 93--140.

\vspace{-0.3cm}

\bibitem{h2022}
M. Hovemann,
Triebel--Lizorkin--Morrey spaces and differences,
Math. Nachr. 295 (2022), 725--761.

\vspace{-0.3cm}

\bibitem{hcy2021}
L. Huang, D.-C. Chang and D. Yang,
Fourier transform of Hardy spaces associated with ball
quasi-Banach function spaces,
Appl. Anal. (2021), 3825--3840.

\vspace{-0.3cm}

\bibitem{hlyy2019}
L. Huang, J. Liu, D. Yang and W. Yuan,
Atomic and Littlewood--Paley characterizations
of anisotropic mixed-norm Hardy spaces and their applications,
J. Geom. Anal. 29 (2019), 1991--2067.

\vspace{-0.3cm}

\bibitem{hlyy2019b}
L. Huang, J. Liu, D. Yang and W. Yuan,
Dual spaces of anisotropic mixed-norm Hardy spaces,
Proc. Amer. Math. Soc. 147 (2019), 1201--1215.

\vspace{-0.3cm}

\bibitem{hy2021}
L. Huang and D. Yang,
On function spaces with mixed norms--a survey,
J. Math. Study 54 (2021), 262--336.

\vspace{-0.3cm}

\bibitem{ins2019}
M. Izuki, T. Noi and Y. Sawano,
The John--Nirenberg inequality in ball Banach function
spaces and application to characterization of BMO,
J. Inequal. Appl. 2019, Paper No. 268, 11 pp.

\vspace{-0.3cm}

\bibitem{is2017}
M. Izuki and Y. Sawano,
Characterization of BMO via ball Banach function spaces,
Vestn. St.-Peterbg. Univ. Mat. Mekh. Astron. 4(62) (2017), 78--86.

\vspace{-0.3cm}

\bibitem{jw2009}
H. Jia and H. Wang,
Decomposition of Hardy--Morrey spaces,
J. Math. Anal. Appl. 354 (2009), 99--110.

\vspace{-0.3cm}

\bibitem{kntyy2007}
N. Kikuchi, E. Nakai, N. Tomita, K. Yabuta and T. Yoneda,
Calder\'on--Zygmund operators
on amalgam spaces and in the discrete case,
J. Math. Anal. Appl. 335 (2007), 198--212.

\vspace{-0.3cm}

\bibitem{kk1991}
V. Kokilashvili and M. Krbec,
Weighted Inequalities in Lorentz and Orlicz Spaces,
World Scientific Publishing Co. Inc., River Edge, NJ, 1991.

\vspace{-0.3cm}

\bibitem{k2015}
K. A. Kopotun,
Polynomial approximation with doubling
weights having finitely many
zeros and singularities,
J. Approx. Theory 198 (2015), 24--62.

\vspace{-0.3cm}

\bibitem{kr1991}
O. Kov\'a$\check{\mathrm{c}}$ik and J. R\'akosn\'ik,
On spaces $L^{p(x)}$ and $W^{k,p(x)}$,
Czechoslovak Math. J. 41(116) (1991), 592--618.

\vspace{-0.3cm}

\bibitem{lsu2012}
M. Lacey, E. T. Sawyer and I. Uriarte-Tuero,
A characterization of two weight norm inequalities
for maximal singular integrals with one doubling measure,
Anal. PDE 5 (2012), 1--60.

\vspace{-0.3cm}

\bibitem{lyh2320}
Y. Li, D. Yang and L. Huang,
Real-Variable Theory of Hardy Spaces Associated with Generalized
Herz Spaces of Rafeiro and Samko,
Lecture Notes in Math. 2320, Springer,
Cham, 2022.

\vspace{-0.3cm}

\bibitem{mt2001}
G. Mastroianni and V. Totik,
Best approximation and moduli of smoothness for doubling weights,
J. Approx. Theory 110 (2001), 180--199.

\vspace{-0.3cm}

\bibitem{ma2011}
V. Maz'ya,
Sobolev Spaces with Applications to
Elliptic Partial Differential Equations, second,
revised and augmented edition,
Grundlehren der MathematischenWissenschaften 342,
Springer, Heidelberg, 2011.

\vspace{-0.3cm}

\bibitem{m2011}
G. Mingione,
Gradient potential estimates,
J. Eur. Math. Soc. 13 (2011), 459--486.

\vspace{-0.3cm}

\bibitem{m1938}
C. B. Morrey, On the solutions of quasi-linear
elliptic partial differential equations,
Trans. Amer. Math. Soc. 43 (1938), 126--166.

\vspace{-0.3cm}

\bibitem{ns2012}
E. Nakai and Y. Sawano,
Hardy spaces with variable exponents
and generalized Campanato spaces,
J. Funct. Anal. 262 (2012), 3665--3748.

\vspace{-0.3cm}

\bibitem{ns2014}
E. Nakai and Y. Sawano,
Orlicz--Hardy spaces and their duals,
Sci. China Math. 57 (2014), 903--962.

\vspace{-0.3cm}

\bibitem{n1950}
H. Nakano,
Modulared Semi-Ordered Linear Spaces,
Maruzen Co. Ltd., Tokyo, 1950.

\vspace{-0.3cm}

\bibitem{n1951}
H. Nakano,
Topology of Linear Topological Spaces,
Maruzen Co. Ltd., Tokyo, 1951.

\vspace{-0.3cm}

\bibitem{npv2012}
E. Di Nezza, G. Palatucci and E. Valdinoci,
Hitchhiker's guide to the fractional Sobolev spaces,
Bull. Sci. Math. 136 (2012), 521--573.

\vspace{-0.3cm}

\bibitem{n2006}
H.-M. Nguyen,
Some new characterizations of Sobolev spaces,
J. Funct. Anal. 237 (2006), 689--720.

\vspace{-0.3cm}

\bibitem{n2019}
T. Nogayama,
Mixed Morrey spaces,
Positivity 23 (2019), 961--1000.

\vspace{-0.3cm}

\bibitem{noss2021}
T. Nogayama, T. Ono, D. Salim and Y. Sawano,
Atomic decomposition for mixed Morrey spaces,
J. Geom. Anal. 31 (2021), 9338--9365.

\vspace{-0.3cm}

\bibitem{rr2002}
M. M. Rao and Z. D. Ren,
Applications of Orlicz Spaces, Monographs and Textbooks in
Pure and Applied Mathematics 250,
Marcel Dekker, Inc., New York, 2002.

\vspace{-0.3cm}

\bibitem{rudin}
W. Rudin,
Real and Complex Analysis,
Third edition. McGraw-Hill Book Co., New York, 1987.

\vspace{-0.3cm}

\bibitem{s2018}
Y. Sawano,
Theory of Besov Spaces,
Developments in Mathematics 56, Springer, Singapore, 2018.

\vspace{-0.3cm}

\bibitem{sdh20201}
Y. Sawano, G. Di Fazio and D. Hakim,
Morrey Spaces--Introduction and Applications to
Integral Operators and PDE's. Vol. I,
Monographs and Research Notes in Mathematics.
CRC Press, Boca Raton, FL, 2020.

\vspace{-0.3cm}

\bibitem{sdh20202}
Y. Sawano, G. Di Fazio and D. Hakim,
Morrey Spaces--Introduction and Applications to
Integral Operators and PDE's. Vol. II,
Monographs and Research Notes in Mathematics.
CRC Press, Boca Raton, FL, 2020.

\vspace{-0.3cm}

\bibitem{shyy2017}
Y. Sawano, K.-P. Ho, D. Yang and S. Yang,
Hardy spaces for ball quasi-Banach function spaces,
Dissertationes Math. 525 (2017), 1--102.

\vspace{-0.3cm}

\bibitem{st2015}
Y. Sawano and H. Tanaka,
The Fatou property of block spaces,
J. Math. Sci. Univ. Tokyo 22 (2015), 663--683.

\vspace{-0.3cm}

\bibitem{tyy2019}
J. Tao, Da. Yang and Do. Yang,
Boundedness and compactness characterizations of Cauchy
integral commutators on Morrey spaces,
Math. Methods Appl. Sci. 42 (2019), 1631--1651.

\vspace{-0.3cm}

\bibitem{tyyz2021}
J. Tao, D. Yang, W. Yuan and Y. Zhang,
Compactness characterizations of commutators on
ball Banach function spaces,
Potential Anal (2021), https://doi.org/10.1007/s11118-021-09953-w.

\vspace{-0.3cm}

\bibitem{t1983}
H. Triebel,
Theory of Function Spaces, Monographs in Mathematics 78,
Birkh\"auser Verlag, Basel, 1983.

\vspace{-0.3cm}

\bibitem{wyy.arxiv}
F. Wang, D. Yang and W. Yuan,
Riesz transform characterization of Hardy spaces
associated with ball quasi-Banach function spaces,
arXiv: 2208.10309.

\vspace{-0.3cm}

\bibitem{wyy2020}
F. Wang, D. Yang and S. Yang, Applications of Hardy spaces associated
with ball quasi-Banach function spaces,
Results Math. 75 (2020), Paper No. 26, 58 pp.

\vspace{-0.3cm}

\bibitem{wyyz2021}
S. Wang, D. Yang, W. Yuan and Y. Zhang,
Weak Hardy-type spaces associated with ball
quasi-Banach function spaces II: Littlewood--Paley
characterizations and real interpolation,
J. Geom. Anal. 31 (2021), 631--696.

\vspace{-0.3cm}

\bibitem{yhyy2022-1}
X. Yan, Z. He, D. Yang and W. Yuan,
Hardy spaces associated with ball quasi-Banach
function spaces on spaces of homogeneous type:
Littlewood--Paley characterizations with
applications to boundedness of Calder\'on--Zygmund operators,
Acta Math. Sin. (Engl. Ser.) 38 (2022), 1133--1184.

\vspace{-0.3cm}

\bibitem{yhyy2022-2}
X. Yan, Z. He, D. Yang and W. Yuan,
Hardy spaces associated with ball quasi-Banach
function spaces on spaces of homogeneous type:
characterizations of maximal functions,
decompositions, and dual spaces,
Math. Nachr. (2022), https://doi.org/10.1002/mana.2021004 32.

\vspace{-0.3cm}

\bibitem{yyy2020}
X. Yan, D. Yang and W. Yuan,
Intrinsic square function characterizations of Hardy spaces
associated with ball quasi-Banach function spaces,
Front. Math. China 15 (2020), 769--806.

\vspace{-0.3cm}

\bibitem{zhyy2022}
Y. Zhang, L. Huang, D. Yang and W. Yuan,
New ball Campanato-type function spaces and their applications,
J. Geom. Anal. 32 (2022), Paper No. 99, 42 pp.

\vspace{-0.3cm}

\bibitem{zyyw2019}
Y. Zhang, D. Yang, W. Yuan and S. Wang,
Real-variable characterizations of Orlicz-slice Hardy spaces,
Anal. Appl. (Singap.) 17 (2019), 597--664.


\vspace{-0.3cm}

\bibitem{zwyy2021}
Y. Zhang, D. Yang, W. Yuan and S. Wang,
Weak Hardy-type spaces associated with
ball quasi-Banach function spaces I:
Decompositions with applications to boundedness of
Calder\'on--Zygmund operators, Sci. China Math. 64 (2021), 2007--2064.

\vspace{-0.3cm}

\bibitem{zhy2020}
X. Zhou, Z. He and D. Yang,
Real-variable characterizations of Hardy--Lorentz spaces
on spaces of homogeneous type with applications to
real interpolation and boundedness of Calder\'on--Zygmund operators,
Anal. Geom. Metr. Spaces 8 (2020), 182--260.

\vspace{-0.3cm}

\bibitem{zyy2023}
C. Zhu, D. Yang and W. Yuan,
New characterization of homogeneous ball Banach Sobolev
spaces, Submitted.

\end{thebibliography}
\end{document}